\documentclass[11pt, reqno, oneside]{amsart}

\usepackage{amssymb, bm}

\usepackage{microtype}

\usepackage{pdfsync} 
\usepackage{setspace} 

\usepackage{mathabx}
\usepackage{color, xcolor}

\usepackage[utf8]{inputenc}

\usepackage{abstract}
\usepackage[all]{xy}
\usepackage{lastpage}

\usepackage[left= 3 cm, right= 3 cm, top= 2.7 cm, bottom=2.9 cm, foot= 1.2 cm, marginparwidth=2.3 cm, marginparsep=0.3 cm]{geometry}

\usepackage[inline]{enumitem}

\usepackage{hyperref}
\hypersetup{colorlinks}
\definecolor{darkred}{rgb}{0.5,0,0}
\definecolor{darkgreen}{rgb}{0, 0.5, 0}
\definecolor{darkblue}{rgb}{0,0,0.5}

\hypersetup{colorlinks, linkcolor=darkblue, filecolor=darkgreen, urlcolor=darkred, citecolor=darkblue}
\usepackage{braket}

\makeatletter 
\@addtoreset{equation}{section}
\makeatother  

\numberwithin{equation}{section}

\usepackage{tikz-cd}
\usepackage{tikz}
\usetikzlibrary{arrows}
\usepackage{float}
\usepackage{ulem}

\usepackage{marginnote}

\newtheorem{thm}{Theorem}[section]

\newtheorem{cor}[thm]{Corollary}

\newtheorem{prop}[thm]{Proposition}
\newtheorem{lemma}[thm]{Lemma}

\theoremstyle{definition}
\newtheorem{defn}[thm]{Definition}
\theoremstyle{remark}
\newtheorem{rem}[thm]{Remark}
\newtheorem{example}[thm]{Example}

\newcounter{notes}
{\end{list}}

\makeatletter
\@addtoreset{thm}{section}
\makeatother

\newcommand{\ev}{{\rm ev}}

\renewcommand{\setminus}{\smallsetminus}

\newcommand\qu{/\hspace{-0.1cm}/} 

\newcommand{\beq}{\begin{equation}}
\newcommand{\eeq}{\end{equation}}

\newcommand{\beqn}{\begin{equation*}}
\newcommand{\eeqn}{\end{equation*}}

\newcommand{\dual}{*}
\newcommand{\ov}{\overline}
\newcommand{\ol}{\overline}
\newcommand{\mb}{\mathbb}

\newcommand{\mc}{\mathcal}
\newcommand{\mf}{\mathfrak}
\newcommand{\fk}{ {\mf k}}
\newcommand{\g}{ {\mf g}}

\renewcommand{\txt}[1]{\rm{\texttt{#1}}}

\newcommand{\wt}{\widetilde}

\newcommand{\uds}[1]{\underline{\smash{#1}}}

\newcommand{\weight}{{{\color{lightgray}{\bullet}}}}

\newcommand{\fuk}{{\rm Fuk}}

\newcommand{\End}{{\rm{End}}}
\newcommand{\m}{{\bm m}}
\newcommand{\crit}{{\bf crit}}

\renewcommand{\ss}{{\rm st}}

\newcommand{\us}{{\rm us}}

\newcommand{\Gammait}{{\mathit{\Gamma}}}

\begin{document}

\title{An Open Quantum Kirwan Map}

\author[Woodward]{Chris Woodward}
\address{Department of Mathematics, Rutgers University, Hill Center for Mathematical Sciences, 110 Frelinghuysen Rd., Piscataway, NJ 08854, USA}
\email{ctw@math.rutgers.edu}

\author[Xu]{Guangbo Xu}
\address{
Department of Mathematics, Texas A{\&}M University, College Station, TX 77843, USA}
\email{guangboxu@math.tamu.edu}

\date{\today}

\maketitle

\begin{abstract}
We construct a morphism from the equivariant Fukaya algebra of a Lagrangian brane in the zero level set of a moment map of a Hamiltonian action to the Fukaya algebra of the quotient brane. This morphism induces a map between Maurer--Cartan solution spaces, and intertwines the disk potentials. As an application, we show under some technical hypotheses that weak unobstructedness of an invariant Lagrangian brane implies weak unobstructedness of its quotient. For semi-Fano toric manifolds we give a different proof of the open mirror theorem of Chan--Lau--Leung--Tseng \cite{CLLT} by showing that the potential of a Lagrangian toric orbit in a toric manifold is related to the Givental--Hori--Vafa potential by a change of variable. We also reprove the results of Fukaya--Oh--Ohta--Ono \cite{FOOO_toric_1} on weak unobstructedness of these toric orbits. In the case of polygon spaces we show the existence of weakly unobstructed and Floer nontrivial products of spheres.
\end{abstract}

\setcounter{tocdepth}{2}
\tableofcontents

\section{Introduction}

The {\it gauged linear sigma model} (GLSM) in physics, introduced by Witten \cite{Witten_LGCY}, has been an influential framework in mathematics and physics.  In addition to providing important applications in mirror symmetry (see \cite{Hori_Vafa}), the GLSM provides an alternate approach to holomorphic curves in symplectic geometry.  Several authors have  defined Gromov--Witten type invariants for the GLSM as counterparts of ordinary Gromov--Witten invariants of the symplectic quotient: In algebraic geometry the  {\it quasimap invariants} (\cite{CKM_quasimap}) and in symplectic geometry  the {\it gauged Gromov--Witten invariants} or {\it Hamiltonian Gromov--Witten} or {\it vortex invariants} \cite{Cieliebak_Gaio_Salamon_2000} \cite{Cieliebak_Gaio_Mundet_Salamon_2002} \cite{Mundet_thesis, Mundet_2003}. 

In physical terms, the GLSM theory and the corresponding nonlinear
sigma model (NLSM) theory are related via a {\it renormalization}
procedure. In the study of gauged Gromov--Witten theory in symplectic
geometry, a mathematical version of this relationship was observed by Gaio--Salamon
\cite{Gaio_Salamon_2005} via an adiabatic limit of the
symplectic vortex equation. The phenomena appearing in the adiabatic limit process 
suggests there should be a quantum correction of the
classical Kirwan map that intertwines with the two kinds of invariants. As explained in
\cite{Witten_LGCY}\cite{Morrison_Plesser_1995}\cite{Gaio_Salamon_2005}, 
the correction is enumerative in nature, counting the so-called point-like instantons or {\it affine vortices}. The quantum Kirwan map conjecture has been studied in symplectic geometric setting by Ziltener \cite{Ziltener_Decay, Ziltener_thesis, Ziltener_book} and has been proved in algebraic geometric setting by the first-named author \cite{Woodward_15}. Remarkably, as explained in \cite{Hori_Vafa}, the quantum correction is essentially equivalent to the mirror map accounted for the nontrivial relation between A-model and B-model theories.

In this paper we extend the idea of quantum Kirwan map to the open-string situation. Our construction leads to various applications in Lagrangian Floer theory. We first state the general picture without listing the precise assumptions. Let $V$ be a symplectic manifold acted by a compact Lie group $K$. There are two ways of studying the Lagrangian Floer theory of the symplectic quotient $V \qu K$. The ordinary approach associates to each Lagrangian submanifold $L \subset V \qu K$ and a Floer cochain complex, or more generally an $A_\infty$ algebra, whose structure maps count pseudoholomorphic curves in $V \qu K$ with boundaries mapped into $L$. An alternative construction, introduced by the first-named author in \cite{Woodward_toric}, replaces holomorphic curves in $V \qu K$ by holomorphic curves in $V$ which are counted equivariantly. The equivariant approach brings in additional convenience in both theory and calculation, since many symplectic manifolds arise as the GIT quotients of rather simple $V$ such as a vector space. For example, in \cite{Woodward_toric} certain Lagrangian tori in toric manifolds were proved to be Hamiltonian non-displaceable without appealing to the formidable Kuranishi structure, in contrast to the approach of \cite{FOOO_toric_1, FOOO_toric_2}. One conceptual motivation of this paper is to compare these two different constructions as an analogue of the quantum Kirwan map for the closed-string case.

Another motivation of our paper is to understand the ``coordinate change'' that relates the two kinds of potential functions of Lagrangian Floer theory, particularly in the toric case. For each toric manifold $X$ and a class of Lagrangian torus $L \subset X$, there is the so-called Givental--Hori--Vafa potential (see \cite{Givental_potential}\cite{Hori_Vafa}) that counts a naive set of holomorphic disks. Cho--Oh \cite{Cho_Oh} first showed that when $X$ is Fano, the Givental--Hori--Vafa potential agrees with the Lagrangian Floer potential. Beyond the Fano case, many people \cite{Chan_Lau_2014} \cite{FOOO_mirror} \cite{CLLT} \cite{Gonzalez_Iritani} have shown that a nontrivial coordinate change is needed in order to identify the two potential functions. On the other hand, as indicated by physicists' work and rigorously shown by the first-named author \cite{Woodward_toric}, the Givental--Hori--Vafa potential is indeed the potential function from the equivariant approach. The coordinate change was then conjectured in \cite{Woodward_toric} to come from counts of affine vortices over both the complex plane and the upper half plane. Our construction of the open quantum Kirwan map verifies this conjecture for a large class of manifolds and unifies the previous results for the toric case. We also applied our construction to two other situations and obtain nontrivial Floer-theoretic implications.

\subsection{Main results}

The geometric setup is as follows (see more details in Subsection \ref{subsection21}). Let $V$ be a K\"ahler manifold with symplectic form $\omega_V$ and complex structure $J_V$. Suppose $V$ is acted upon by a complex reductive Lie group $G$ with a maximal compact subgroup $K$. Let ${\mf k}$ resp. ${\mf g}$ denote the Lie algebras of $K$ resp. $G$. 
Let
\beqn
\mu: V \to {\mf k}^*
\eeqn
be a moment map. We assume that $K$ acts freely on
  $\mu^{-1}(0)$ so that the symplectic quotient of $V$ by $K$ (
  homeomorphic to the geometric invariant theory quotient with respect
  to a stability condition corresponding to $\mu$) is a compact
  K\"ahler manifold
\beqn X = V \qu G = \mu^{-1}(0)/K.  \eeqn
We assume that there exists a holomorphic line bundle over $V$ with connection whose curvature is a positive multiple of the equivariant symplectic form, see \eqref{eqn22} below, so that the quotient may also be defined using geometric invariant theory.

We study Fukaya algebras associated to branes in the quotient.
Let $L\subset X$ be a compact Lagrangian submanifold, equipped with an
orientation, a spin structure and a local system. Let $\fuk ( L )$
denote the associated Fukaya $A_\infty$ algebra with composition maps
\beqn {\bm m}_k: \fuk( L )^{\otimes k} \to \fuk( L),\ \ k = 0, 1,
\ldots \eeqn 
defined by counting treed holomorphic disks in $X$ (see
\cite{FOOO_Book} \cite{Charest_Woodward_2015}). The Fukaya algebra has
a family of deformations parametrized by cycles
${\mf a}$ in $X$, denoted by $\fuk(L, {\mf a})$, called the
{\it bulk deformation} as in \cite{FOOO_toric_2}, defined by counting
holomorphic disks with interior markings mapping to  ${\mf a}$.  On the other hand, the preimage of $L$
under the projection $\mu^{-1}(0) \to X$ is a $K$-invariant Lagrangian
submanifold $L_V \subset V$. In \cite{Woodward_toric} the first-named
author introduced the equivariant Fukaya algebra
$\fuk^{\it eq} (L)$, \footnote{It was called the {\it quasimap Fukaya
    algebra} in that paper.} whose composition maps \beqn {\bm
  m}_k^{\it eq}: \fuk^{\it eq} ( L )^{\otimes k} \to \fuk^{\it eq} (
L),\ \ k = 0, 1, \ldots\eeqn are defined by counting holomorphic disks
in $V$ modulo the $K$-action.

Our main result of this paper is about the relation between these two
kinds of $A_\infty$ algebras. The following describes the technical
hypotheses we will make on our targets in order to regularize the
moduli spaces of vortices.

\begin{defn}\label{defn11} Define the following conditions  on the target datum $(V, G, \mu)$.
\begin{enumerate}
\item[(T1)] The element $0 \in {\mf k}^*$ is a regular value of $\mu$
  and the $G$ acts freely on $\mu^{-1}(0)$ so that the quotient
  $X$ is a smooth manifold;
\item[(T2)] the unstable locus $V^{\rm us} \subset V$ is an algebraic
  subvariety of complex codimension at least two and so real
  codimension at least four;
\item[(T3)] the manifold $V$ is either compact or convex at infinity
  (see Definition \ref{defn21}) and either aspherical or monotone with
  minimal Chern number $2$.
\end{enumerate}
\end{defn}

To construct the morphisim between Fukaya algebras we count {\it
  affine vortices} as in Ziltener \cite{Ziltener_book}. These are the objects that appear as bubbles in the adiabatic limit  studied in Gaio-Salamon \cite{Gaio_Salamon_2005}. An affine vortex with domain ${\bm C}$, the complex plane, is a pair consisting of a map with connection
\[ (u, A), \quad u: {\bm C} \to V , \quad  A \in \Omega^1({\bm C},
{\mf k}) \] 
satisfying the equations
\begin{align}\label{eqn11}
&\ \ov\partial_A u = 0,\ &\ * F_A + \mu(u) = 0.
\end{align}
As explained in Ziltener \cite{Ziltener_book}, these equations are
elliptic modulo gauge and their solutions form finite-dimensional
moduli spaces.  In the case of domain ${\bm H}$, the upper half plane, we require that the
map $u$ satisfy the boundary condition
$u(\partial {\bm H}) \subset L_V$. Ziltener \cite{Ziltener_book}
  shows that  affine vortices $(u,A)$ with domain ${\bm C}$ have a
well-defined equivariant homology class $[(u,A)] \in H_2^K (V)$, and so
a well-defined Chern number $c_1([u,A]) \in \mb{Z}$  given by pairing
with $c_1^K (V) \in H^2_K (V)$.

To count affine vortices we wish to choose generic perturbations so that the moduli space is smooth. In this paper, we construct such perturbations in a restricted semi-positive setting, although we expect that the same arguments work in general for a sufficiently general perturbation scheme.  To regularize affine vortices mapped  into  a stabilizing divisor,  we impose the following conditions which give us some control over the spheres of zero or   negative Chern numbers.  We expect that these conditions can be   removed using a more sophisticated equivariant perturbation scheme.

\begin{defn}\label{defn12}   Define the following conditions on the target  $(V, G, \mu)$ :
\begin{enumerate}

\item[(S1)] \label{S1} All $J_V$-affine vortices over ${\bm C}$ have nonnegative
  Chern number.

\item[(S2)] \label{S2} All $I_X$-holomorphic spheres in $X$ have nonnegative
  Chern number.

\item[(S3)]  \label{S3} All nonconstant $J_V$-affine vortices of Chern number zero
  are contained in a normal crossing $G$-invariant divisor
  \beqn S = S_1 \cup \cdots \cup S_N \subset V \eeqn
  with each $S_i$ smooth.

\item[(S4)]  \label{S4}  All nonconstant $I_X$-holomorphic spheres in $X$ of Chern
  number zero are contained in the geometric invariant theory quotient
  $\bar S:=S \qu G$.
  
\end{enumerate}
\end{defn}

\begin{example}\label{example13} 
These conditions are satisfied in the toric case under the following combinatorial condition that the first equivariant Chern class lie in the closure of the ``positive cone'' defined by the equivariant symplectic class. Suppose $V \cong {\mb C}^k$ is a vector space with an action of torus $G$ with Lie algebra $\g$. We identify 
\[ H_2^G(V) \cong \g^\dual, \quad H^2_G(V) \cong \g .\] 
Suppose the $G$-action has weights 
\[ {\bf m}_1,\ldots, {\bf m}_k \in \g^\dual \] 
and the equivariant symplectic class is denoted 
\[ [\omega_G]
\in H^2_G(V) \cong \g^\dual . \]  
For each subset
$I \subset \{ 1, \ldots , k\} $ let
\[ C_I =  \operatorname{span}_+ \{ {\bf m}_i, k \in I \} \]
denote the closed cone generated by ${\bf m}_i, i \in I$.
Let 
\[ \g^{\dual,\circ} = \g^\dual - \bigcup_{I, \operatorname{codim}(C_I) > 0 } C_I  \] 
denote the complement of the cones $C_I$ of positive codimension.  The {\it positive cone} 
\[ \g^{\dual,+} \subset \g^{\dual,\circ} \] 
defined by $[\omega_G]$ is the connected component of $\g^{\dual,\circ}$ containing $[\omega_G]$, and the 
{\it semipositive cone} is the closure of $\g^{\dual,+}$.
If the equivariant Chern class $c_1^G(V)$ lies in semipositive
cone $ \ol{\g^{\dual,+}}$, then $V$ is semi-Fano in the sense above.  Indeed
any homology class $d$ of affine vortices
pairs positively with any $\xi \in \g^{\dual,+}$
since the classes of vortices are independent
of the symplectic class $\xi$ by the Hitchin-Kobayashi correspondence  described in Example \ref{hn} below.
Then $c_1^G(V)$ pairs non-negatively with $d$ by continuity.
This implies (S1) and (S2) is similar. 
For (S3), we take the hyperplane $S_j$ to be the sum of the weight subspaces for weights other 
than ${\bf m}_j$, namely the $j$-th coordinate hyperplane.   If a vortex is not contained in $S$, then all of the pairings of $d$ with 
${\bf m}_j$ are non-negative, which implies non-negative Chern number and zero Chern number only if the vortex is zero energy, hence constant. The argument for (S4) is similar.
This ends the example.
\end{example}

We need a Novikov field to sum the counts of vortices. Since we are in the rational case, we use the Novikov field with integral energy filtrations, namely the field of Laurent series in a formal variable ${\bm q}$:
\beqn 
{\bm \Lambda} = {\mb C} (({\bm q})) = \Set{ \sum_{i=m}^{+\infty} a_i {\bm q}^i\ |\ a_i \in {\mb C},\ m \in {\mb Z} }.  
\eeqn
Let ${\bm \Lambda}_+$ be the subring of power series containing only positive powers of ${\bm q}$.

Now we state the first main theorem. 

\begin{thm}\label{mainthm} 
  Let $(V, G, \mu)$ satisfy (T1)---(T3) of Definition \ref{defn11} and (S1)---(S4) of Definition \ref{defn12}. Let $S \subset V$ be the divisor in Definition \ref{defn12}. Let
  $L \subset X \setminus \bar S$ be a connected and spin Lagrangian
  submanifold that satisfies (L1)---(L2) of Definition \ref{defn27}. For
  each local system on $L$ there are an $A_\infty$ algebra $\fuk^0(L)$
  (which is formally the same as $\fuk^{\it eq}(L)$), an $A_\infty$
  algebra $\fuk^\infty(L)$ (which is homotopy equivalent to a version
  of bulk-deformed Fukaya algebra), and a unital $A_\infty$ morphism
  (called the {\sc open quantum Kirwan morphism})
\beqn
{\bm \kappa}: \fuk^0 ( L ) \to \fuk^\infty ( L ) 
\eeqn
defined by counting affine vortices over the upper half plane ${\bm H}$. Moreover, ${\bm \kappa}$ is a higher order deformation of the identity.
\end{thm}

The unitality of the open quantum Kirwan map has the following consequences for disk potentials of Lagrangians in symplectic quotients. Recall that an $A_\infty$ algebra ${\mc A}$
with compositions ${\bm m}_l$ and strict unit ${\bm e}$ is called {\it
  weakly unobstructed} if the following {\it weak Maurer--Cartan
  equation} has a solution:
\beqn 
\sum_{l\geq 0} {\bm m}_l \big({\bm b}, \ldots, {\bm b} \big) \in {\bm \Lambda} {\bm e}.  
\eeqn
On its solution set $MC({\mc A})$ there is the {\it potential function} ${\bf W}_{\mc A}$ defined by 
\beqn {\bf W}_{\mc A} ({\bm b}) {\bm e} = \sum_{l \geq 0} {\bm m}_l
\big( {\bm b}, \ldots, {\bm b} \big)
 \ \ \ \forall {\bm b} \in MC({\mc
  A}).  \eeqn
In the situation of Theorem \ref{mainthm}, the $A_\infty$ morphism ${\bm \kappa}$ induces a total map (formally)
\begin{align*}
&\ \uds \kappa: \fuk^0  (L) \to \fuk^\infty ( L),\ &\  \uds \kappa ( {\bm a}) = \sum_{l \geq 0} \kappa_l \big(  {\bm a}, \ldots, {\bm a} \big).
\end{align*}
By the $A_\infty$ axiom for ${\bm \kappa}$ and unitality, $\uds \kappa$ maps the Maurer--Cartan  spaces of $\fuk^0 (L)$ into the Maurer--Cartan space of $\fuk^\infty (L)$ and intertwines with the potential functions
\begin{align*}
&\ {\bf W}^0: MC(\fuk^0 (L)) \to {\bm \Lambda},\ &\  {\bf W}^\infty: MC (\fuk^\infty ( L) ) \to {\bm\Lambda}.
\end{align*}

\begin{cor}\label{cor15}
The map $\uds \kappa$ maps $MC(\fuk^0 (L))$ injectively into $MC(\fuk^\infty (L))$. Therefore, $\fuk^\infty ( L)$ is weakly unobstructed if $\fuk^0 (L)$ is weakly unobstructed. Moreover, for each weakly bounding cochain ${\bm b} \in MC(\fuk^0  (L))$ there holds
\beqn
{\bf W}^\infty ( \uds \kappa ({\bm b})) =  {\bf W}^0 ({\bm b}).
\eeqn
\end{cor}

Another consequence of Theorem \ref{mainthm} is a relation between the equivariant Floer cohomology upstairs and a  bulk-deformed Floer cohomology of the quotient. Given a weakly bounding cochain ${\bm b}^0 \in MC( \fuk^0 (L))$ and its image ${\bm b}^\infty = \uds \kappa ({\bm b}^0 )$, the deformed ${\bm m}_1^0$ resp. ${\bm m}_1^\infty$ are differentials on $\fuk{}^0(L)$ resp. $\fuk{}^\infty(L)$. While the cohomology of ${\bm m}_1^0$ can be identified with the {\it quasimap Floer cohomology} $QHF^*( (L, {\bm b}^0); {\bm \Lambda})$ introduced by the first-named author \cite{Woodward_toric}, the cohomology of ${\bm m}_1^\infty$ can be viewed as a version of Lagrangian Floer cohomology of the brane $(L, {\bm b}^\infty)$ with a bulk deformation ${\mf c}$ (see the discussion of Remark \ref{rem78}), denoted by $HF^* ((L, {\bm b}^\infty; {\mf c}); {\bm \Lambda})$. Since the $A_\infty$ morphism is a higher order deformation of the identity, our construction gives an isomorphism between these two cohomology groups.

\begin{cor}\label{cor16}
There is an isomorphism 
\beqn 
QHF^*( (L, {\bm b}^0); {\bm \Lambda}) \cong HF^* ( (L, {\bm b}^\infty; {\mf c}); {\bm \Lambda}).
\eeqn
\end{cor}

Perhaps the most important application of the open quantum Kirwan map
is a simple criterion of weak unobstructedness of Lagrangians in a symplectic  quotient as above. We introduce the following positivity condition.

\begin{itemize}
\item[(P1)] Any nonconstant $J_V$-holomorphic disk in $V$ with
  boundary mapped into $L_V$ has positive Maslov indices.
\end{itemize}

\begin{thm}\label{thm17}
Under (P1), for all local systems on $L$, both $\fuk{}^0 (L)$ and $\fuk{}^\infty (L)$ are weakly unobstructed. 
\end{thm}

Furthermore, we show that under (P1) there is a natural section 
\beqn
H^1(L; {\bm \Lambda}_0) \to \bigsqcup_{b \in H^1(L; {\bm \Lambda}_0)} MC( \fuk{}^0(L, b)).
\eeqn
Here ${\bm \Lambda}_0 \subset {\bm \Lambda}$ is the subring of power
series starting from nonnegative powers of ${\bm q}$ and $b \in H^1(
L; {\bm \Lambda}_0)$ parametrizes local systems on $L$. In other
words, for each local system on $L$ there is a canonical weakly
bounding cochain. One can then transfer the potential function ${\bf
  W}^0$ to a function defined over the space $H^1(L; {\bm
  \Lambda}_0)$. We call this function the {\it restricted potential
  function} and by abuse of terminology denote the restricted potential by
\beqn
{\bf W}^0: H^1(L; {\bm \Lambda}_0) \to {\bm \Lambda}.
\eeqn
If we assume a similar positivity condition for holomorphic disks in
$X$ 
(see (P2) in Subsection \ref{subsection75}) one
obtains another restricted potential function 
\beqn
{\bf W}^\infty: H^1(L; {\bm \Lambda}_0) \to {\bm \Lambda}.
\eeqn
Formally, ${\bf W}^0$ is the same as the {\it quasimap potential
  function} introduced in
\cite{Woodward_toric} and ${\bf W}^\infty$ is the bulk-deformed
Lagrangian Floer potential function defined by Fukaya--Oh--Ohta--Ono
\cite{FOOO_Book}. Then under the third positivity condition 
(see (P3) in Subsection \ref{subsection75}), as a special
case of Corollary \ref{cor15}, one has
 \beqn {\bf W}^0 = {\bf
  W}^\infty.  \eeqn

\subsection{Examples}

We now turn to concrete examples in which we prove the weak
unobstructedness of certain Lagrangians in symplectic quotients.

\subsubsection{Toric manifolds}

Any compact symplectic toric manifold can be realized as the GIT
quotient of a Euclidean space by a torus action, by a result of
Delzant \cite{Delzant}. We use the quantum Kirwan morphism to study
the Lagrangian Floer theory of toric manifolds. In Section
\ref{section8} we prove the following theorem.

\begin{thm}\label{thm18}
Let $X$ be an $n$-dimensional rational compact toric manifold whose  moment polytope $P$ has $N$ faces. Let $L \subset X$ be the   Lagrangian torus that is the preimage under the projection   $X \to P$ of a rational interior point of $P$.

\begin{enumerate}
\item We view $X$ as the GIT quotient of $V = {\mb C}^N$ by the action of $G = ({\mb C}^*)^{N-n}$ and assume that $V$ satisfies the conditions of Definition \ref{defn12} (see Example \ref{example13} for more clarification).  Then for any local system on $L$, the $A_\infty$ algebras $\fuk^0(L)$ and
  $\fuk^\infty(L)$ are both weakly unobstructed.

\item The restricted potential function
  ${\bf W}^0: H^1(L; {\bm \Lambda}_0) \to {\bm \Lambda}$ coincides
  with the Givental--Hori--Vafa potential (see
  \cite{Givental_potential} \cite{Hori_Vafa}).

\item The above two restricted potential functions ${\bf W}^0$ and ${\bf W}^\infty$ are equal.
\end{enumerate}
\end{thm}

Item (c) can be viewed as equivalent to the ``open mirror theorem'' of Chan--Lau--Leung--Tseng \cite{CLLT}, which says that the disk potential of $L$ coincides with the Givental--Hori--Vafa potential after a coordinate change induced from the mirror map. The mirror map should be equivalent to the counts of affine vortex over the complex plane, which is independent of the Lagrangian $L$. If we can remove the semi-Fano conditions (S1)---(S4) in Definition \ref{defn12} and can still regularize related moduli spaces, then we will see that one also needs a possibly nontrivial transformation on the variable ${\bm b}$ to identify the two potential functions (see also Fukaya--Oh--Ohta--Ono \cite{FOOO_mirror} for more abstract statement).  Another situation where the quasimap disk potential is computed is that of SYZ fibrations of toric Calabi--Yau varieties in Chan \cite{chan:quasi}; the results here
imply that the quasimap potential is related to the 
disk potential of the quotient by the open quantum Kirwan map.

\subsubsection{Polygon spaces}

A second example is the case of polygon spaces, that is, quotients of
products of two-spheres. The classical cohomology rings of these spaces were studied
by Kirwan \cite{Kirwan_book}.  The set $S^2$ of unit vectors
  $v \in {\mb R}^3$ is acted upon by the group of rotations $SO(3)$.
The inclusion $S^2 \hookrightarrow {\mb R}^3 \cong \mf{so}(3)^*$ is a
moment map for the action. Let $V$ be the product of $(2n+3)$ copies
of $S^2$ acted by the diagonal $SO(3)$-action, with a moment map
\beqn 
\mu(v_1, \ldots, v_{2n+3}) = v_1 + \cdots + v_{2n+3}.  
\eeqn 
The corresponding symplectic quotient can be viewed as the moduli space of equilateral $(2n+3)$-gons in ${\mb R}^3$ up to rigid body motion. Let $L_V \subset V$ be the Lagrangian 
\beqn 
L_V = \underbrace{ \bar \Delta_{S^2} \times \cdots \times \bar \Delta_{S^2}}_{n} \times \Delta_3 \eeqn 
where $\bar \Delta_{S^2} \subset S^2 \times S^2$ is the anti-diagonal
and $\Delta_3 \subset S^2 \times S^2 \times S^2$ is the set of triples
of unit vectors $(v_1, v_2, v_3)$ in ${\mb R}^3$ satisfying
$v_1 + v_2 + v_3 = 0$. The submanifold $L_V$ is an $SO(3)$-
Lagrangian of $V$ which projects to a Lagrangian $L \subset X$ that is
diffeomorphic to the product of $n$ spheres. In Section
\ref{section9} we prove the following theorem.

\begin{thm} \label{polyex} The Fukaya algebras $\fuk^0(L)$ and
  $\fuk^\infty(L)$ are weakly unobstructed and the corresponding Floer
  cohomologies for a canonical weakly bounding cochain
  are isomorphic to $H^*(L; {\bm \Lambda})$.
\end{thm}

\subsection{Further remarks}

There have been many works on symplectic vortex equation and related
invariants. In \cite{Cieliebak_Gaio_Mundet_Salamon_2002} and
\cite{Mundet_2003} the so-called {\it Hamiltonian Gromov--Witten
  invariants} were defined in certain cases, followed by the project
of Mundet--Tian \cite{Mundet_Tian_2009}. An alternative approach due
to Venugopalan \cite{Venugopalan_quasi} aims at constructing the
so-called {\it symplectic quasimap invariants}. An approach to the
quantum Kirwan map conjecture was initiated in
\cite{Gaio_Salamon_2005}, followed by a series of works of Ziltener
\cite{Ziltener_Decay, Ziltener_thesis, Ziltener_book}. Under certain
assumptions, Frauenfelder \cite{Frauenfelder_thesis,
  Frauenfelder_2004} and the second  author \cite{Xu_VHF}
constructed Lagrangian Floer cohomology and Hamiltonian Floer
cohomology respectively, using the vortex equation. The current paper
also relies on results proved in \cite{Wang_Xu} \cite{Venugopalan_Xu}
\cite{Xu_glue}.

The transversality of various relevant moduli spaces is achieved by
generalizing the stabilizing divisor technique introduced by
Cieliebak--Mohnke \cite{Cieliebak_Mohnke}. Using this method they
construct genus-zero Gromov--Witten invariants for compact rational
symplectic manifolds.  Their method is extended by
\cite{Charest_Woodward_2017} and \cite{Charest_Woodward_2015} to
Lagrangian Floer theory for rational Lagrangian submanifolds. In all
the previous approaches, a single stabilizing divisor is
sufficient. In our case we use the K\"ahler hypothesis to construct
multiple stabilizing divisors stabilizing objects having positive
energy whose images are contained in the divisor (for example, to
avoid treating tangency conditions at the infinity for affine
vortices).

The organization of the paper is as follows. In Section \ref{section2}
we set up the geometric background and review basic facts of
vortices. In Section \ref{section3} we give detailed definitions of
combinatorial types used in the paper. In Section \ref{section4} we
define the set of perturbation data. In Section \ref{section5} we
define the relevant moduli spaces and recall the fundamental
compactness theorem about affine vortices. In Section \ref{section6}
we show that generic perturbations can regularize moduli spaces for a
restricted class of combinatorial types. In Section \ref{section7} we construct the $A_\infty$
algebras and the $A_\infty$ morphism and prove Theorem \ref{mainthm},
Corollary \ref{cor15}, Corollary \ref{cor16}, and Theorem
\ref{thm17}. In Section \ref{section8} and Section \ref{section9} we
consider the examples of toric manifolds and polygon spaces. In
Section \ref{sectiona1} we show how to construct strict units in the
$A_\infty$ algebras.

\subsection{Acknowledgements}

We would like to thank Kai Cieliebak, Kenji Fukaya, Siu-Cheong Lau,
Paul Seidel, and Gang Tian for helpful discussions, and to thank
Yuchen Liu and Ziquan Zhuang for kindly answering algebraic geometry
questions. During the preparation of this paper, we learn that Kenji
Fukaya has a different approach that would lead to a construction of a
similar $A_\infty$ morphism, using the Lagrangian correspondence given
by the zero level set of the moment map. One advantage of our approach
is that it shows that the correction is given (in good cases) by
counting affine vortices, so that the coordinate change for the
potential functions is the same, at least in principle\footnote{A
  complete argument would require a proof that the symplectic vortex
  invariants defined here are the same as the algebraic vortex
  invariants.} as one appearing in the closed case.

\section{Symplectic vortices}\label{section2}

In this section we give the basic geometric setup and review necessary
analytic results about holomorphic curves and affine vortices.

\subsection{The target manifold}\label{subsection21}

Our target manifolds are K\"ahler manifolds with Hamiltonian group actions. Let $V$ be a smooth manifold with symplectic form $\omega_V \in \Omega^2(V)$ and complex structure $J_V$, so that the triple $(V, \omega_V, J_V)$ is a K\"ahler manifold. Let $G$ be a connected complex reductive Lie group acting on $V$. Let $K \subset G$ be a maximal compact subgroup. Suppose there is a linearization of the $G$-action, namely a $G$-equivariant
holomorphic line bundle $\tilde V \to V$.  Moreover, suppose
$\tilde V$ is equipped with a $K$-invariant Hermitian metric with
induced Chern connection $\nabla^{\tilde V}$.  Let
\beqn
F_{\nabla^{\tilde V}} = (\nabla^{\tilde V})^2\in \Omega^2(V, \End(\tilde V)) 
\eeqn
denote its curvature.  We assume that there is a positive integer
$k_0$ such that
\beqn
- \frac{1}{2\pi {\bf i}} {\rm tr} (F_{\nabla^{\tilde V}}) = k_0 \omega_V. 
\eeqn
In particular, $(V, \omega_V)$ is a rational symplectic manifold. The
rationality should also hold in an equivariant sense. Recall that the
moment map on $\tilde V$, denoted by
\beqn \tilde \mu \in \Gammait ( V, {\rm End} (\tilde V) \otimes {\mf
  k}^*) \eeqn 
is defined as 
\beq\label{eqn21} \tilde \mu (\xi)(v) = \nabla_{{\mc X}_\xi}^{\tilde
  V} v - {\mc L}_{\xi}^{\tilde V} v,\ \forall \xi \in {\mf k},\ v\in
\tilde V.  \eeq 
Here ${\mc X}_\xi \in \Gamma( T V) $ is the infinitesimal action and
${\mc L}_{\xi}^{\tilde V}$ is the Lie derivative on $\tilde V$. Then
\beqn \mu:= - \frac{1}{2k_0 \pi {\bf i}} {\rm tr}(\tilde \mu): V \to
{\mf k}^* \eeqn 
is a moment map of the Hamiltonian $K$-action on $(V, \omega_V)$ and so
\beq \label{eqn22} 
- \frac{1}{2 \pi {\bf i}} {\rm tr} ( F_{\nabla^{\tilde V}} +
\tilde \mu ) = k_0( \omega_V + \mu).  \eeq 
Recall that $\omega_V + \mu$ is an equivariant closed differential
form and represents an equivariant cohomology class
$[\omega_V + \mu] \in H_K^2(V; {\mb R})$. By rescaling the symplectic
form, we assume that $k_0 = 1$. Therefore
$[\omega_V + \mu] \in H_K^2 (V; {\mb Z})$.

The Kempf--Ness theorem relates the symplectic quotient with Mumford's geometric invariant theory (GIT) quotient. Assume that $\mu$ is a proper map,
$0\in {\mf k}^*$ is a regular value of $\mu$, and the $K$-action on
$\mu^{-1}(0)$ is free. The {\it symplectic quotient}
\beqn X:= V \qu K := \mu^{-1}(0) / K \eeqn
is a smooth compact K\"ahler manifold with induced K\"ahler form
$\omega_X$ and complex structure $I_X$. The semistable locus of $V$ is defined as
\beqn V^{\rm ss} = \Set{ x \in V \ | \ \exists k>0,\ s \in H^0( \tilde
  V{}^{\otimes k})^G\ {\rm such\ that\ } s(x) \neq 0 }.  \eeqn 
The complement of the semistable locus $V^{\rm ss}$ is the unstable
locus, denoted by $V^{\rm us}$.  Inside $V^{\rm ss}$, the
polystable locus $V^{\rm ps}$ consists of points $x$ for which the
orbits $Gx$ is closed in $V^{\rm ss}$, and the stable locus
$V^{\rm s} \subset V^{\rm ps}$ consists of polystable points $x$ whose
stabilizers $G_x$ are finite. The geometric invariant theory quotient is then obtained from
the semistable locus $V^{\rm ss}$ by quotienting by the orbit
equivalence relation 
\beqn x_1 \sim x_2 \iff \ol{G x_1} \cap \ol{G
  x_2} \cap V^{\rm ss} \neq \emptyset .  \eeqn 
The assumption that $0$ is a regular value of the moment map is
equivalent to 
\beqn V^{\rm ss}  = V^{\rm st} = V^{\rm ps} = G
\mu^{-1}(0).  \eeqn 
We will then only use the symbol $V^\ss$.
The Kempf--Ness Theorem \cite{Kempf_Ness} states that the symplectic reduction is homeomorphic to the GIT quotient:
\beqn X:= \mu^{-1}(0)/K \cong V \qu G:= V^\ss/ G.
\eeqn 
In this paper we will assume that the unstable locus $V^{\rm us}$ has
complex codimension at least two so that in particular 
\[ H_2(X) \cong H_2^G(V^{\ss}) \cong H_2^G(V) \] 
by Kirwan's book \cite{Kirwan_book}.


\subsection{Symplectic vortices}\label{subsection22}

Symplectic vortices are equivariant generalizations of pseudoholomorphic maps, obtained from what we call {\it gauged maps} by minimizing a certain energy functional. Here we recall basic notations and results about vortices.

Gauged maps are pairs of a connection and a holomorphic section of an associated bundle. Let $\Sigma$ be a Riemann surface and $P \to \Sigma$ be a $K$-bundle.  Denote by ${\mc A}(P)$ the space of connections on $P$, i.e., invariant one-forms $A \in \Omega^1(P,{\mf k})^K$ whose pairing with the generating vector field of any element $\xi \in {\mf k }$ is equal to $\xi$.  A {\it gauged map} from $\Sigma$ to $V$ is a triple $(P, A, u)$ where $P \to \Sigma$ is a $K$-bundle, $A \in {\mc A}(P)$
  is a connection, and $u: \Sigma \to P(V)$ is a section of the associated fibre bundle $P(V) = (P \times V)/ K$.  When $\Sigma$ has boundary $\partial \Sigma$, we also impose the boundary condition
\beqn
u(\partial \Sigma) \subset P(L_V)
\eeqn
where $L_V \subset V$ is a $K$-invariant Lagrangian submanifold. The group of gauge transformations on the principal bundle $P$ is the
space of sections of $P( K )$ where $K$ acts on itself by the adjoint
action.

Gauged maps over closed surfaces represent certain equivariant
homology classes. When $\Sigma$ is closed, a continuous gauged map
$(P, A, u)$ pushes forward the fundamental class of $\Sigma$ to a
class $[ P,A,u ] \in H_2^K(V; {\mb Z})$ (which is independent of the connection $A$). When $\Sigma$ is a disk, a gauged map
represents a relative equivariant class
$[P,A,u] \in H_2^K(V,L_V; {\mb Z})$. Since bundles over a disk can
always be trivialized, the class represented by the gauged map is in
$H_2(V, L_V; {\mb Z})$.

    The energy of a gauged map involves the curvature, moment map, and anti-holomorphic part of the twisted derivative. The curvature of a connection $A$ is denoted $F_A \in \Omega^2({\rm ad} P)$.  Given a section $u$ of $P(V)$, its composition with the moment map is denoted $\mu(u) \in \Gammait( {\rm ad}P^*)$. The covariant derivative of a section $u$ is denoted $d_A u \in \Omega^1( P(TV))$.  Choose a $K$-invariant, $\omega_V$-compatible almost complex structure $J$ on $V$ (possibly parametrized by points on $\Sigma$), an area form $\nu \in \Omega^2(\Sigma)$, and an ${\rm Ad}$-invariant metric on ${\mf k}$.  Let $\ov\partial_A u$ denote the $(0,1)$-part of the covariant derivative of $d_A u$, and $*: \Omega^2( \Sigma, {\rm ad} P) \cong \Omega^0( \Sigma, {\rm ad} P^*)$ is the Hodge-star operator induced by the area form and the identification ${\rm ad} P \cong {\rm ad} P^*$ induced by the metric on the Lie algebra.  For a gauged map $(P, A, u)$ from $\Sigma$ to $V$, define its energy by
\beqn
E(P, A, u) = \frac{1}{2} \left( \| d_A u\|_{L^2}^2 + \| F_A \|_{L^2}^2  + \| \mu(u)\|_{L^2}^2 \right).
\eeqn   
Here the norms are defined with respect to the metric on $\Sigma$
induced from the complex structure and the area form, the metric on
the Lie algebra we used for the vortex equation, and the Hermitian
metric on $P(TV)$ induced by $\omega_V$ and $J$. When $\Sigma$ is
closed and $(P, A, u)$ represents a class $B \in H_2^K(V; {\mb Z})$,
one has 
\beqn 
E(P, A, u) = \langle [\omega_V + \mu], B \rangle + \| * F_A +\mu(u)\|_{L^2}^2 + 2 \| \ov\partial_A u \|_{L^2}^2.  
\eeqn

Minimizers of the energy are solutions to the {\it vortex equation} \eqref{eqn11} and called {\it vortices} on $\Sigma$.  The vortex equation \eqref{eqn11} is invariant under gauge transformations.  In this paper we mainly consider the vortex equation over the complex
plane or the upper half plane equipped with the standard complex
structure and area form. Let ${\bm A}$ denote either ${\bm C}$ or
${\bm H}$. Let $z = s + {\bf i} t$ be the standard complex
coordinate. Since ${\bm A}$ is contractible, one can always regard the
bundle $P \to {\bm A}$ in a gauged map as the trivial bundle. Then a
gauged map can be identified with a triple ${\bm v} = (u, \phi, \psi)$
where $u: {\bm A} \to V$ is a map and $\phi, \psi$ are
${\mf k}$-valued functions on ${\bm A}$, regarded as the $ds$ and $dt$
components of the connection form. In this way, the symplectic vortex
equation \eqref{eqn11} can be written as
\begin{align*}
&\ \partial_s u + {\mc X}_\phi + J( \partial_t u + {\mc X}_\psi) = 0,\ &\ \partial_s \psi - \partial_t \phi + [\phi, \psi] + \mu(u) = 0
\end{align*}
A solution of this equation is called an {\it $J$-affine vortex} over ${\bm A}$. We only consider solutions with finite energy.

To count affine vortices, we must compactify their moduli spaces. We first recall the following equivariant convexity condition which guarantees the $C^0$-compactness of vortex moduli spaces. 

\begin{defn}\label{defn21}(cf. \cite[Section  2.5]{Cieliebak_Gaio_Mundet_Salamon_2002}) Let $\nabla$ be the  Levi--Civita connection of the K\"ahler metric on $V$. The manifold  $V$ is called {\it convex at infinity} if there exist a  $K$-invariant smooth proper function $f_V: V \to [0, +\infty)$ and $C >0$ such that 
\beqn 
\forall \ x \in V, \ f_V (x) \geq C \Longrightarrow \left\{ \begin{array}{c} \langle \nabla_\xi \nabla f_V (x), \xi \rangle  + \langle \nabla_{J_V \xi} \nabla f_V (x), J_V \xi \rangle \geq 0,\ \forall \xi\in T_x V,\\
d f_V (x) \cdot J_V {\mc X}_{\mu(x)} \geq 0. 
\end{array} \right. 
\eeqn 
\end{defn}

Assume for the rest of the paper that $V$ is compact or convex at infinity. In particular it follows from this assumption that $X = V \qu K$ is compact.

We recall the basic results about removal of singularities for affine vortices. For domain ${\bm C}$ the removal singularity theorem is due to Ziltener (see \cite{Ziltener_Decay, Ziltener_thesis,  Ziltener_book}.  A more general result in the recent \cite{Chen_Wang_Wang}); for affine vortices over ${\bm H}$ the result is due to Wang and the second
 named author \cite[Theorem 2.11]{Wang_Xu}. 

\begin{prop}\label{prop22}
  Let ${\bm v} = (u, \phi, \psi)$ be an affine vortex over ${\bm C}$
  or ${\bm H}$.  There exists $x \in X = \mu^{-1}(0)/K$ such that as
  points in the orbit space $V / K$, 
  \beqn \lim_{z \to \infty} K \cdot u(z) = x.  \eeqn
\end{prop}

The above proposition allows one to associate an equivariant homology
class to any affine vortex so that an affine vortex ${\bm v}$ over
${\bm C}$ represents a curve class $B \in H_2^K(V; {\mb Z})$. By
\cite[page 5]{Ziltener_book} one has the energy identity
\beq\label{eqn23} E({\bm v}) = \omega(B):= \langle [\omega_V + \mu], B
\rangle.  \eeq 
Similarly, an affine vortex over ${\bm H}$ represents a class
$B\in H_2(V, L_V; {\mb Z})$ and 
\beqn E({\bm v}) = \omega(B):= \langle [\omega_V], B \rangle.  \eeqn

A topology on the moduli space of affine vortices is defined as follows. Given a class $B \in H_2^K(V; {\mb Z})$ resp. $B \in H_2(V, L_V; {\mb Z})$ and a family of $K$-invariant almost complex structures $J$ on $V$ parametrized by points on ${\bm A}$, let ${\mc M}_J ({\bm A}; B)$ be the set of gauge equivalence classes of $J$-affine vortices over ${\bm A}$. There is a natural topology, called the {\it compact convergence topology}, abbreviated as c.c.t., on ${\mc M}_J({\bm A}; B)$ defined as follows. We say that a sequence of smooth gauged maps ${\bm v}_i = (u_i, \phi_i, \psi_i)$ {\it converges} in c.c.t. to a gauged map ${\bm v}_\infty = (u_\infty, \phi_\infty, \psi_\infty)$ if for any precompact open subset $Z \subset {\bm A}$, ${\bm v}_i|_Z$ converges to ${\bm v}_\infty|_Z$ uniformly with all derivatives. We say that ${\bm v}_i$ converges in c.c.t. to ${\bm v}_\infty$ modulo gauge transformation if there exist a sequence of smooth gauge transformations $g_i: {\bm A} \to K$ such that $g_i^* {\bm v}_i$ converges to ${\bm v}_\infty$ in c.c.t. The notion of sequential convergence in c.c.t. modulo gauge descends to a notion of sequential convergence in the moduli space ${\mc M}_J({\bm A}; B)$, and induces a topology on ${\mc M}_J({\bm A}; B)$ (see \cite[Section 5.6]{McDuff_Salamon_2004} and the recent erratum \cite{McDuff_Salamon_erratum_2}). This topology is Hausdorff and second countable, but however not necessarily compact. We  recall the compactification of affine vortices over ${\bm A}$ constructed in by Ziltener \cite{Ziltener_thesis, Ziltener_book} and of affine vortices over ${\bm H}$ constructed by Wang and the second author \cite{Wang_Xu} later in this paper (see Section \ref{subsection53}).

\begin{example} \label{hn} {\rm (The toric case)}  
There is a complete classification of affine vortices in the toric case. Suppose $V = {\mb C}^n$ and $G \hookrightarrow ({\mb C}^*)^n$ is a torus acting on $V$ with a compact GIT quotient. An affine vortices in $V$ can be identified via a certain  Hitchin--Kobayashi correspondence with a polynomial map from ${\bm C}$ to $V$. When $V = {\mb C}$ and $G = {\mb C}^*$, the classification was given by Taubes \cite{Taubes_vortex}. When $V = {\mb C}^n$ and $G = {\mb C}^* \subset ({\mb C}^*)^n$ is the diagonal, the classification was due to the second author \cite{Guangbo_vortex}. The classification for the general situation was provided by \cite{VW_affine}. 
\end{example} 

\subsection{Local model of affine vortices}\label{subsection23}

In this subsection we recall the result of \cite{Venugopalan_Xu} on local models of moduli spaces of affine vortices that we will need later for regularization of the moduli spaces. Unlike the case of holomorphic curves or vortices over compact domains, the case of affine vortices does not follow from standard Fredholm theory due to the noncompactness of the domain and unusual asymptotic behavior at the infinity.

\subsubsection{Sobolev spaces}

We first recall the particular Sobolev spaces we used to construct the
local model. Let ${\bm A}$ be either the complex plane ${\bm C}$ or the upper half plane ${\bm H}$. Choose a smooth function $\rho: {\bm A} \to [1, +\infty)$ that coincides with the radial coordinate $r$ outside a compact subset
of ${\bm A}$. For $\tau \in {\mb R}$ introduce the following weighted
Sobolev norms on functions on ${\bm A}$:
\beqn
\| f \|_{L^{p, \tau}} = \left[ \int_{\bm A} |f(z)|^p \rho(z)^{p \tau} d s d t \right]^{\frac{1}{p}}.  
\eeqn
\begin{align*}
&\ \| f \|_{L_g^{1, p, \tau}} = \| f \|_{L^{p,\tau}} + \| \nabla f\|_{L^{p,\tau}},\ &\ \| f \|_{L_h^{1,p,\tau}} = \| f \|_{L^\infty} + \| \nabla f \|_{L^{p,\tau}}.
\end{align*}
These norms have the following properties:  The space $L_g^{1,p, \tau}({\bm A})$ embeds into the space of continuous functions on ${\bm A}$ that converge to zero at infinity, while the space $L_h^{1,p,\tau}({\bm A})$ embeds into the space of continuous functions on ${\bm A}$ that converge to a finite number at infinity. To simplify notations, we use the specializations of the Sobolev exponents as in \cite{Xu_glue}. Choose and set 
\beqn 
p \in (2, 4), \quad \tau = \tau_p: = 2- \frac{4}{p}.  
\eeqn 
In this case the above three norms are abbreviated as $\tilde L^p$, $\tilde L^{1,p}_g$, $\tilde L_h^{1,p}$.

We use the above weighted norms to define Sobolev norms on the tangent space to the space of gauged maps.  Given a gauged map
${\bm v} = (u, \phi, \psi)$ from ${\bm A}$ to $V$, consider the space of formal infinitesimal deformations 
\beqn 
\Gammait ( {\bm A}, u^* TV \oplus {\fk} \oplus {\fk})_{L_V}
\eeqn
whose elements are denoted by ${\bm \xi} = (\xi, \eta, \zeta)$. Here the subscript ${}_{L_V}$ means that the elements are required to satisfy the boundary condition 
\begin{align*}
&\ \xi|_{\partial {\bm A}} \subset u^* TL_V,\ &\ \zeta|_{\partial {\bm A}} \equiv 0.
\end{align*}
Define the covariant derivative of ${\bm \xi}$ by
\beqn 
\tilde \nabla {\bm \xi} = d s \otimes \tilde \nabla_s {\bm \xi} + d t \otimes \tilde \nabla_t {\bm \xi} 
\eeqn
where
\begin{align*}
&\ \tilde \nabla_s {\bm \xi} = \left[ \begin{array}{c} \nabla_s \xi + \nabla_\xi {\mc X}_\phi \\
                                                  \nabla_s \eta + [ \phi, \eta] \\
																									\nabla_s \zeta + [\phi, \zeta] \end{array} \right],\ &\ \tilde \nabla_t {\bm \xi} = \left[ \begin{array}{c} \nabla_t \xi + \nabla_\xi {\mc X}_\psi \\
                                                  \nabla_t \eta + [ \psi, \eta] \\
																									\nabla_t \zeta + [\psi, \zeta] \end{array} \right].
\end{align*}
Define the following norm on ${\bm \xi} = (\xi, \eta, \zeta)$: 
\beq\label{mixednorm} 
\| {\bm \xi}\|_{\tilde L_m^{1,p}}:= \| \xi \|_{L^\infty} + \| \tilde \nabla {\bm \xi}\|_{\tilde L^p} + \| d\mu(u) \cdot \xi \|_{\tilde L^p} + \| d\mu(u) \cdot J_V \xi \|_{\tilde L^p} + \| \eta \|_{\tilde L^p} + \| \zeta \|_{\tilde L^p}.
\eeq 
Denote by
\beqn
\tilde L_m^{1,p}( {\bm A}, u^* TV \oplus {\mf k} \oplus {\mf k})_{L_V}
\eeqn
the completion of the space $\Gammait({\bm A}, u^* TV \oplus {\mf k} \oplus {\mf k})_{L_V}$ with respect to the norm \eqref{mixednorm}.

There are two important features of the above norm. First, this norm is gauge invariant in the following sense. Suppose ${\bm v}$ and ${\bm v}'$ are two smooth gauged maps and $g: {\bm A} \to K$ is a gauge transformation such that ${\bm v}' = g \cdot {\bm v}$. For every infinitesimal deformation ${\bm \xi}$ of ${\bm v}$, the triple ${\bm \xi}':=g\cdot {\bm \xi}$ is an infinitesimal deformation of ${\bm v}'$ and
\beqn
\| {\bm \xi}\|_{\tilde L_m^{1,p}} = \| {\bm \xi}'\|_{\tilde L^{1,p}_m} .
\eeqn
Second, the subscript ${}_m$, which represents ``mixed,'' indicates that this
norm is a combination of the norms $\tilde L_h^{1,p}$ and
$\tilde L_g^{1,p}$. More precisely, near the infinity of ${\bm A}$,
the value of the map $u: {\bm A} \to V$ is contained in the stable
locus $V^{\rm st} \subset V$ where one has the orthogonal
decomposition 
\beqn 
TV|_{V^{\rm st}} = {\rm ker} (d\mu \oplus d\mu \circ J) \oplus G_V 
\eeqn 
where $G_V \subset TV$ is the distribution
spanned by infinitesimal $G$-actions. Then the finiteness of the norm
$\| (\xi, \eta, \zeta)\|_{\tilde L_m^{1,p}}$ implies that the
$G_V$-direction of $\xi$ together with $\eta$ and $\zeta$ has finite
$\tilde L_g^{1,p}$-norm, which further implies that their values converge to zero at infinity, while the
${\rm ker}( d\mu \oplus d\mu\circ J)$-direction of $\xi$ has finite
$\tilde L_h^{1,p}$-norm and have finite but possibly nonzero limit at infinity. Furthermore, we know that $\mu(u) \to 0$ as
$z \to \infty$ and $u(z)$ has a limit $x \in \mu^{-1}(0)/K$ as a
$K$-orbit. Each tuple
\[ {\bm \xi} = (\xi, \eta, \zeta) \in \tilde L_m^{1,p}({\bm A}, u^* TV
\oplus {\mf k} \oplus {\mf k})_{L_V} \] 
has a limit in $T_x X$ at infinity.

\subsubsection{The local model}

We summarize the result on local models of moduli spaces of affine vortices proved by Venugopalan and the second  author (see \cite{Venugopalan_Xu}). 

\begin{thm}\cite{Venugopalan_Xu}\label{localmodel} 
Given $B \in H_2^K(V; {\mb Z})$ (resp. $B \in H_2(V, L_V; {\mb Z})$) and $p \in (2, 4)$, there exist a Banach manifold ${\mc B} = {\mc B}_{\bm C}^{\rm vor}(B)$ (resp. ${\mc B} = {\mc B}_{\bm H}^{\rm vor}(B)$), a Banach vector bundle ${\mc E} = {\mc E}_{\bm C}^{\rm vor}(B)$ (resp. ${\mc E}_{\bm H}^{\rm vor}(B)$) over ${\mc B}$, and for a $K$-invariant $\omega_V$-compatible almost complex structure $J$, a smooth section 
\beqn
{\mc F}_J: {\mc B} \to {\mc E} 
\eeqn
satisfying the following conditions.
\begin{enumerate}

\item Every element of ${\mc B}$ is a gauge equivalence class of gauged maps from ${\bm C}$ (resp. ${\bm H}$) to $V$ of regularity $W^{1, p}$ having limits at infinity as $K$-orbits. Moreover, the evaluation map at infinity $\ev: {\mc B} \to X$ (resp. the evaluation map $\ev_z: {\bm H} \to L$ at every point $z \in \partial {\bm H}$ including $z = \infty$) is smooth.

\item The zero locus of ${\mc F}_J$ consists of gauge equivalence classes of $J$-affine vortices representing the class $B$ and the natural map ${\mc F}_J^{-1}(0) \cong {\mc M}_J({\bm C}, B)$ (resp. ${\mc F}_J^{-1}(0) \cong {\mc M}_J({\bm H}, B)$) is a homeomorphism between Banach manifold topology on the former and the compact convergence topology on the latter.

\item For every element of ${\mc F}_J^{-1}(0)$, choose a smooth representative ${\bm v} = (u, \phi, \psi)$, the tangent space of ${\mc B}$ at ${\bm v}$ is isomorphic to 
\beq\label{eqn25}
\Big\{ {\bm \xi} = (\xi, \eta, \zeta) \in \tilde L_m^{1,p}( {\bm C}, u^* TV \oplus {\mf k} \oplus {\mf k}) \ |\ \tilde \nabla_s \eta + \tilde \nabla_t \zeta + d\mu(u) \cdot J \xi = 0 \Big\},
\eeq
\beq\label{eqn26}
\left({\rm resp.}\ \Big\{ {\bm \xi} = (\xi, \eta, \zeta) \in \tilde L_m^{1,p} ( {\bm H}, u^*TV \oplus {\mf k} \oplus {\mf k})_{L_V}\ |\  \tilde \nabla_s \eta + \tilde \nabla_t \zeta + d\mu(u) \cdot J \xi = 0 \Big\} \right),
\eeq
the fibre of ${\mc E}$ at ${\bm v}$ is isomorphic to $\tilde L^p( {\bm C}, u^* TV \oplus {\mf k} )$ (resp. $\tilde L^p({\bm H}, u^* TV \oplus {\mf k})$), and the linearization of ${\mc F}_J$ at ${\bm v}$ reads
\beqn
D_{\bm v}({\bm \xi}) = \left[ \begin{array}{c} \tilde \nabla_s \xi + J \tilde \nabla_t \xi + (\nabla_\xi J) {\bm v}_t + {\mc X}_\eta + J {\mc X}_\zeta \\   \tilde \nabla_s \zeta - \tilde \nabla_t \eta + d\mu(u) \cdot \xi  \end{array}\right]. \eeqn

\item (see \cite[(1.27)]{Ziltener_book} for the case over ${\bm C}$) The linearized operator $D_{\bm v}$ is a Fredholm operator with index given by the formula
\beqn
{\rm index} D_{\bm v} = \left\{ \begin{array}{lc} {\rm dim} X + 2\langle c_1^K(TV), B \rangle, &\ {\bm A} = {\bm C};\\
                                                  {\rm dim} L + {\rm Maslov}(B), &\  {\bm A} = {\bm H}.
                                                  \end{array} \right.
\eeqn
\end{enumerate}
\end{thm}

The linearized operator in the theorems above is equivalent to another operator, called the {\it augmented linearized operator}, that incorporates the gauge-fiing condition.  Fix a smooth $J$-affine vortex ${\bm v} = (u, \phi, \psi)$ over ${\bm A}$ representing an element of
${\mc B}_{\bm A}^{\rm vor}(B)$. The condition in \eqref{eqn25} and
\eqref{eqn26} is called the {\it Coulomb gauge condition}. Using the
basic treatment of gauge theory, define the {\it augmented linearized
  operator}
\beqn D_{\bm v}^+: \tilde L_m^{1,p}( {\bm A}, u^* TV \oplus {\mf k}
\oplus {\mf k})_{L_V} \to \tilde L^p( {\bm A}, u^* TV \oplus {\mf k}
\oplus {\mf k}) \eeqn 
by 
\beq\label{eqn27} D_{\bm v}^+
\left[ \begin{array}{c} \xi \\ \eta \\ \zeta
 \end{array} \right] = \left[ \begin{array}{c}  \tilde \nabla_s \xi + J \tilde \nabla_t \xi  + ( \nabla_\xi J) (\partial_t u + {\mc X}_\phi) + {\mc X}_\eta + J {\mc X}_\zeta \\
                            \tilde \nabla_s \eta + \tilde \nabla_t \zeta + d\mu(u) \cdot J \xi \\
                            \tilde \nabla_s \zeta - \tilde \nabla_t
                            \eta + d\mu(u) \cdot \xi
 \end{array} \right].
\eeq
(As a convention we put the gauge-fixing condition in the second coordinate.) The operator $D_{\bm v}^+$ is a Cauchy--Riemann type operator. In fact one can prove that $D_{\bm v}^+$ is also Fredholm and has the same index as $D_{\bm v}$ (see \cite{Venugopalan_Xu}). Moreover, ${\rm ker} D_{\bm v} = {\rm ker} D_{\bm v}^+$ and $D_{\bm v}$ is surjective if and only if $D_{\bm v}^+$ is surjective, since these two operators differ by an invertible one.

\subsection{Stabilizing divisors}\label{subsection24}

To regularize the moduli spaces of vortices we choose invariant
divisors so that the additional intersection points stabilize the
domains. In the symplectic case the existence of stabilizing divisors
relies on the rationality of the symplectic class, thanks to the
theorems of Donaldson \cite{Donaldson_96} and Auroux--Gayet--Mohsen
\cite{Auroux_Gayet_Mohsen}. In our setting, the
quotient $X$ is automatically rational. Moreover, one can choose
holomorphic divisors rather than almost complex ones.

Donaldson's hypersurfaces are defined as sections of a line bundle-with-connection whose curvature is the symplectic form.  Recall that by assumption the linearization $\tilde V \to V$ descends to a holomorphic line bundle
\beqn
\tilde X \to X
\eeqn
with an induced Hermitian metric and Chern connection
$\nabla^{\tilde X}$. Take a sufficiently large integer $k$ and denote
\begin{align*}
&\ \Gammait_V^k:= H^0( \tilde V{}^{k})^G,\ &\ \Gammait_X^k:= H^0( \tilde X{}^{k}).
\end{align*}
Restriction to the stable locus defines a natural isomorphism
\beq\label{eqn28} \Gammait_V^k \cong \Gammait_X^k.  \eeq
To see this, notice that every section of $\tilde X^k$ can be pulled
back to a holomorphic section of $\tilde V^k$ over the unstable locus.
By (T2) of Definition \ref{defn11}, the unstable locus
has complex codimension two or higher.  By using Riemann's extension
theorem (see \cite[(6.4) Corollary]{Demailly_book}), the pullback can
be extended holomorphically to the unstable locus. Then for any
$s \in \Gammait_V^k$, define 
\beqn D_s = s^{-1}(0) \subset V.  \eeqn
The corresponding section in $\Gammait_X^k$ is denoted by $\bar s$ and
the corresponding divisor in $X$ is denoted by
\beqn
\bar D_s = \bar s^{-1}(0) \subset X.
\eeqn
By Bertini's theorem \cite[II.8.18]{Hartshorne}), for sufficiently large $k$, a generic $\bar s \in \Gammait_X^k$ is transverse to the zero section and define a smooth divisor $\bar D_s \subset X$. Moreover, $\bar D_s$ represents the Poincar\'e dual of $k [\omega_X]$. For such $\bar s$, the divisor $D_s \subset V$ is smooth in the stable locus. We note, for later use, the fact that any stabilizing divisor $D_s$ contains the unstable locus. Indeed, by Mumford's definition \cite{Mumford_GIT}, $V^{\rm us}$ is the locus where all equivariant holomorphic sections of $\tilde V^k$ vanish.

A result of Clemens and Voisin \cite{Clemens_1986}\cite{Voisin_1996}
prevents holomorphic spheres from appearing in the stabilizing
divisors.  Recall that a hypersurface of $\mb{CP}^n$ of degree $d$ is
{\it general} if its defining section belongs to a nonempty Zariski
open subset of $PH^0( \mb{CP}^n ,{\mc O}(d))$.

\begin{thm}\label{thm25}\cite{Clemens_1986}\cite{Voisin_1996}\cite{Voisin_1998} A general hypersurface of $\mb{CP}^n$  of degree $d \geq 2n-1$ contains no rational curves.
\end{thm}

These results imply the existence of stabilizing divisors in the
following sense:

\begin{cor}\label{cor26}
There exists $k_0 >0$ such that for a generic section $\bar s \in \Gammait_X^k$ for $k \geq k_0$, $\bar D_s \subset X$ contains no nonconstant $I_X$-holomorphic sphere.
\end{cor}

\subsection{Lagrangian submanifolds}

Now we describe the assumptions on the Lagrangian
submanifold. Consider a closed embedded Lagrangian submanifold
\beqn
L \subset X
\eeqn
which is oriented and equipped with a spin structure. Its
preimage under the projection $\mu^{-1}(0) \to X$ is a $K$-invariant
Lagrangian submanifold 
\beqn L_V \subset V.  \eeqn 
This is the class of Lagrangians studied in Woodward
\cite{Woodward_toric} and Xu \cite{Guangbo_compactness}, but is not
the same as that studied in Frauenfelder \cite{Frauenfelder_thesis,
  Frauenfelder_2004}.

In order to apply the stabilizing divisor technique to Lagrangian
Floer theory, we need an additional rationality assumption on
Lagrangian submanifolds. The notion of rational Lagrangians used in
\cite{Charest_Woodward_2017} generalizes to the equivariant case as
follows.

\begin{defn}(cf. \cite[Definition
  3.5]{Charest_Woodward_2017})\label{defn27} The Lagrangian
  $L \subset X$ is {\it strongly rational} if for some positive
  integer $k$ the following conditions are satisfied.

\begin{itemize}
\item[(L1)] \label{defn27a} $L$ is Bohr--Sommerfeld with respect to
  $(\tilde X{}^k, \nabla^{\tilde X{}^k})$. Namely, the restriction of
  $( \tilde X{}^k, \nabla^{{\tilde X}{}^k} )$ to $L$ is isomorphic to
  a trivial bundle with the trivial connection.
\item[(L2)] \label{defn27b} There is a holomorphic section of $\tilde X{}^k$ that is nonvanishing along $L$ and its restriction to $L$ induces a trivialization of $\tilde X{}^k |_L$ that is homotopic to the trivialization in the (L1). \footnote{It is possible to remove the second condition. Indeed, the argument of Auroux--Gayet--Mohsen. \cite{Auroux_Gayet_Mohsen} can be extended to construct holomorphic sections concentrated along $L$ in every homotopy class. However we add the second part as an assumption for simplicity.} See also \cite{Borthwick_Paul_Uribe}.
\end{itemize}
\end{defn}

\begin{rem} 
The strong rationality of $L$ implies that there is a nonempty open subset
\beqn
\Gammait_X^k(L) \subset \Gammait_X^k
\eeqn
consisting of sections satisfying (L2). Let
\beqn
\Gammait_V^k(L) \subset \Gammait_V^k
\eeqn
be the corresponding open subset under the identification \eqref{eqn28}. 
\end{rem} 

\begin{lemma}\label{lemma29}
For sufficiently large $k$, for any $s \in \Gammait_V^k(L)$ that defines a divisor $D_s \subset V$, $L_V$ is an exact Lagrangian submanifold of $(V \setminus D_s, \omega_V)$ and $L$ is an exact Lagrangian submanifold of $( X \setminus \bar D_s, \omega_X)$.
\end{lemma}

\begin{proof}
  We first show that $L$ is exact in $X \setminus \bar D_s$. Take a
  holomorphic section $\bar s \in \Gammait_X^k(L)$. Without loss of
  generality, assume $k = 1$. Then define
  $\theta_{\bar s} \in \Omega^1(X \setminus \bar D_s)$ by
\beqn
\nabla^{\tilde X} \left( \frac{ \bar s}{\| \bar s \|} \right) = - \frac{\theta_s}{2\pi {\bf i}} \frac{ \bar s }{\| \bar s\|}.
\eeqn
Then there holds $\omega_X = d \theta_{\bar s}$. We would like to show that $\theta_{\bar s}|_{L}$ is an exact form. Indeed, by condition (L1), there is a smooth section $\bar s{}'$ of $\tilde X|_L$ that is nowhere vanishing and covariantly constant with respect to the connection $\nabla^{\tilde X}$, and there is a function $f: L \to {\mb R}$ such that
\beqn
\frac{\bar s}{\| \bar s\|} = e^{{\bf i}f} \frac{\bar s{}'}{\| \bar s{}'\|}.
\eeqn
Extend $\bar s'$ smoothly over $X$ and define $\theta_{\bar s'}$ in the same way as defining $\theta_{\bar s}$. Then over $L$ we have $\theta_{\bar s'} = 0$ and
\beqn
\theta_{\bar s} = \theta_{\bar s'}  + d f = d f.
\eeqn
Hence $L$ is exact on the complement of $\bar D_s$. 

To show that $L_V$ is also exact, consider the projection $\pi: L_V \to L$. The section $\bar s'$ lifts to a $K$-invariant section $s'$ of $\tilde V|_{L_V}$ and on $L_V$ there holds
\beqn
\frac{s}{\| s\|} = e^{{\bf i} \pi^* f} \frac{s'}{\| s'\|}.
\eeqn
By \eqref{eqn21}, covariant derivatives of $s'$ in $K$-orbit directions are the same as Lie derivatives of $s'$, which vanish since $s'$ is $K$-invariant. Hence $s'$ is covariantly constant with respect to $\nabla^{\tilde V}$. Then the exactness of $L_V$ in the complement of $s^{-1}(0)$ follows.
\end{proof}

The exactness implies that the morphism
\beqn
H_2(V, L_V; {\mb Z}) \to {\mb R},\ B\mapsto \langle [\omega_V], B \rangle
\eeqn
takes rational values. From now on, by rescaling the symplectic form, one assume that the areas of disk classes are all integers.

Now we choose finitely many stabilizing divisors in general position.

\begin{lemma}\label{lemma210}
Let $
h \geq 2 {\rm dim}_{\mb C} X + 1$
be an integer. 
There exists a collection of invariant prime divisors $D_1,\ldots, D_h \subset V \setminus L_V$ whose intersection is the unstable
  locus $V^{\rm us}$  intersecting transversely in $V^{\rm ss}$ and such that no nonconstant holomorphic sphere
  $u: {\bm S}^2 \to X$ is contained in $(D_1 \cup \ldots \cup D_h) \qu G$.
\end{lemma}

\begin{proof} 
Choose $k$ sufficiently large. Since $\Gammait_V^k(L)$ is nonempty and open, Corollary \ref{cor26} implies for a generic $s \in \Gammait_V^k(L)$, $D_s$ is disjoint from $L_V$ and $\bar D_s$ contains no $I_X$-holomorphic spheres. Choose $h$ generic such sections defining invariant divisors 
\beqn
D_1, \ldots, D_h  \subset V
\eeqn
possible not smooth over the unstable locus. 
Denote their union by
\beq \label{Dh} D = D_1 \cup \cdots \cup D_h.  \eeq
Since the number of components is more than the dimension of $X$, the
intersection $D_1 \cap \cdots \cap D_h$ is the unstable locus
$V^{\rm us}$.
\end{proof}

\subsection{Almost complex structures}\label{subsection26}	

The almost complex structures used for the construction of perturbations are given as follows. Let ${\mc J}_V$ be the set of smooth $K$-invariant $\omega_V$-compatible almost complex structures on $V$. We would like to use almost complex structures that differ from the integrable one $J_V$ only near $\mu^{-1}(0)$. Fix a small $K$-invariant open neighborhood $U$ of $\mu^{-1}(0)$ such as
\beq\label{eqn210}
U:= \Big\{  x \in V\ \Big|\ |\mu(x)| < \epsilon_U  \Big. \Big\}, \ {\rm and}\ \epsilon_U>0\ {\rm being\ sufficiently\ small}.
\eeq
We require an almost complex structure $J \in {\mc J}_V$ that makes each divisor $D_a$ almost complex and agrees with $J_V$ on the normal bundle $ND_a \to D_a$. Define
\beq\label{eqn211}
 {\mc J}_{D, U}  = \Big\{ J \in {\mc J}_V \ |\ J|_{V \setminus U} = J_V,\  J|_{D_a} = J_V,\  1 \leq a \leq h \Big\}.
\eeq
We often abbreviate ${\mc J}_{D, U}$ by ${\mc J}_D$. Each $J \in {\mc J}_D$ induces an almost complex structure on $X$, generally denoted by $I$, making the quotient divisors $\bar D_a \subset X$ almost complex. 

Almost complex structures in the above class guarantee nonempty intersections between nonconstant pseudoholomorphic disks or vortices with the divisors. We first consider the notion of intersection multiplicities and its relation with tangent orders. The case of holomorphic curves in the quotient intersecting the induced divisors in the quotient $X$, this notion is well-understood. For curves or vortices in $V$, since $D$ has codimension two with singularity having codimension four, there is a well-defined intersection number between a map $u: \Sigma \to V$ with $u(\partial \Sigma) \cap D = \emptyset$ and $D$. Moreover, since $D$ is $G$-invariant, for any smooth map $g: \Sigma \to K$, the intersection number $u \cap D$ and $(gu) \cap D$ are the same. Therefore one can define the intersection number for gauge equivalence classes of gauged maps. 

\begin{lemma}\label{lemma211}
Let $J: {\bm D}^2 \to {\mc J}_{D, U}$ be a smooth family of almost complex structures. Let ${\bm v} = (u, \phi, \psi)$ be a $J$-holomorphic gauged map from ${\bm D}^2$ to $V$, namely satisfying
\beqn
\partial_s u + {\mc X}_\phi + J( \partial_t u +{\mc X}_\psi) = 0.
\eeqn
Suppose $u({\bm D}^2)$ is not contained in $D$, then $u^{-1}(D) \cap {\rm Int} {\bm D}^2$ is discrete. Moreover, for each intersection $p \in u^{-1}(D) \cap {\rm Int} {\bm D}^2$, the local intersection number is positive. 
\end{lemma}

\begin{proof}
By covering ${\bm D}^2$ by smaller disks, one can reduce the problem to the situation that either $u({\bm D}^2) \subset V \setminus U$ or $u({\bm D}^2)  \cap V^{\rm us} = \emptyset$. In the former case, $u$ is $J_V$-holomorphic while $J_V$ is integrable. Then for each defining section $s_a$ of $D_a$, the section $u^* s_a$ is holomorphic and doesn't vanish identically. Hence $u^{-1}(D)$ is discrete. In the latter case, using the gauge field $\phi ds + \psi dt$, one can define an almost complex structure $J_{\phi, \psi}$ on ${\bm D}^2 \times V$ such that the map $\tilde u: {\bm D}^2 \to {\bm D}^2 \times Y$ defined by $\tilde u(z) = (z, u(z))$ is $J_{\phi, \psi}$-holomorphic. Moreover, ${\bm D}^2 \times D$ is $J_{\phi, \psi}$-complex. Then by the proof of \cite[Proposition 7.1]{Cieliebak_Mohnke} using the Carleman similarity principle proved in \cite{Floer_Hofer_Salamon}, the assertion is proved. 
\end{proof}

On the other hand, the exactness of $L$ resp. $L_V$ in the complement of $\bar D_a$ resp. of $D_a$ (see Lemma \ref{lemma29}) implies that the energy of holomorphic disks in $X$ resp. in $V$ are proportional to the intersection numbers with $\bar D_a$ resp. with $D_a$. Moreover, similar fact holds for affine vortices over ${\bm H}$. 

\begin{lemma}\label{lemma212}
Given $J \in {\mc J}_D$ and a $J$-affine vortex ${\bm v}$ over ${\bm H}$ with positive energy, for each component $D_a \subset D$, the intersection $u^{-1}(D_a)$ is nonempty. Indeed, 
\beq\label{eqn212}
\# ({\bm v} \cap D_a) = {\rm deg} D_a \cdot E({\bm v}).
\eeq
and similarly for other objects including affine vortices over ${\bm C}$, holomorphic disks in $V$, holomorphic disks and spheres in $X$.
\end{lemma}

\begin{proof}
Suppose ${\bm v} = (u, \phi, \psi)$ is a $J$-affine vortex over ${\bm H}$. Assume that in some gauge (see \cite{Wang_Xu}), the connection part $a \in \Omega^1({\bm H})$ of ${\bm v}$ is well-behaved at infinity. Then by the energy identity for affine vortices \eqref{eqn23}, one has
\beqn
E({\bm v}) = \int_{{\bm H}} \big[ u^* \omega_V + d (\mu \cdot (\phi ds + \psi dt)) \big] = \int_{{\bm H}} u^* \omega_V.
\eeqn
On the other hand, after generic perturbation, $u$ intersects the stable part of $D_a$ transversely at a finite subset $Z \subset {\bm H}$ and does not intersect the unstable locus. By Lemma \ref{lemma29}, there is a 1-form $\theta_a \in \Omega^1(V \setminus D_a)$ such that $d\theta_a = \omega_V$. It follows that 
\beqn
E({\bm v}) = \lim_{\epsilon \to 0} \int_{{\bm H} \setminus B_\epsilon(Z)} u^* \omega_V = \lim_{\epsilon \to 0} \int_{\partial ({\bm H} \setminus B_\epsilon(Z))} u^* \theta_a = \lim_{\epsilon \to 0} \int_{\partial B_\epsilon(Z)} u^* \theta_a. 
\eeqn
The last equality follows from the fact that $\theta_a |_{L_V}$ is exact (Lemma \ref{lemma29}). Indeed, the last term above is exactly the intersection number. Hence \eqref{eqn212} follows. In particular, when ${\bm v}$ has positive energy, there is at least one intersection point between ${\bm v}$ and $D_a$. 
The other cases are left to the reader.
\end{proof}

\subsection{Orientations}\label{subsection27}

Moduli spaces of objects defined over closed domains, i.e., holomorphic spheres and affine vortices over ${\bm C}$, have canonical orientations. Orientations of moduli spaces of maps defined over bordered domains are achieved as in Fukaya--Oh--Ohta--Ono \cite{FOOO_Book}.  The orientation and the spin structure on the Lagrangian boundary condition $L$ induces an orientation on the linearization of the Cauchy--Riemann operators at all pseudoholomorphic disks in $X$ with boundary in $L$. The orientation and the spin structure on $L$ also induce orientations for disks in $V$ modulo $K$-action, which were used in \cite{Woodward_toric} to count quasidisks. 

Here we show that the moduli space of affine vortices with Lagrangian boundary condition is oriented in a similar fashion.

\begin{prop}\label{prop213}
Given a spin structure and an orientation on $L$, there is a naturally induced orientation for every affine vortex over ${\bm H}$.
\end{prop}

Before proving the proposition we introduce the following notation. Let $E \to {\bm D}^2$ be a complex vector bundle and $\ov\partial{}^E: \Omega^0(E) \to \Omega^{0,1}(E)$ be a real linear Cauchy--Riemann operator. Let $\lambda \to \partial {\bm D}^2$ be a totally real subbundle of $E|_{\partial {\bm D}^2}$, meaning that $J_E \lambda \cap \lambda = \{0\}$ and $\lambda$ has maximal dimension, where $J_E$ is the complex structure on $E$. The pair $(E, \lambda)$ is called a {\it bundle pair} over $({\bm D}^2, \partial {\bm D}^2)$. Let $p > 2$ and $W^{1, p} (E)_\lambda$ be the Sobolev space of $W^{1, p}$-sections of $E$ whose boundary values lie in $\lambda$. Then we have a Fredholm operator 
\beqn 
\ov\partial{}_E: W^{1, p}(E)_\lambda \to
L^p(E). 
\eeqn

\begin{lemma}\label{lemma214}\cite[Proposition 8.1.4]{FOOO_Book}
Suppose $\lambda$ is trivializable. Then each trivialization of $\lambda$ canonically induces an orientation on $\ov\partial{}_E: W^{1,p}(E)_\lambda \to L^p(E)$.
\end{lemma}

A similar discussion holds in the case of boundary value problems on the half-space, fixed outside a compact set.  Let $(E, \lambda)$ be a bundle pair over ${\bm H}$.  Fix a trivialization of $E$ outside a compact set $Z \subset {\bm H}$, i.e., 
\beqn 
\rho: E|_{{\bm H} \setminus Z} \cong ({\bm H} \setminus Z) \times {\mb C}^m 
\eeqn 
whose restriction to $\partial {\bm H} \setminus Z$ maps $\lambda$ to ${\mb R}^m$. If we regard ${\bm D}^2 \cong {\bm H} \cup \{\infty\}$, then $\rho$ induces a bundle pair on ${\bm D}^2$, which is still denoted by $(E, \lambda)$.
Choose a Hermitian metric $h_E$ on $E$ such that  
\beqn
h_E - h_0 \circ \rho \in \tilde L^{1,p} ({\bm H} \setminus Z, {\mb C}^{m \times m} )
\eeqn
where $h_0$ is the standard metric on ${\mb C}^m$.\footnote{This particular type of metric comes from the particular convergence rate of affine vortices at infinity (see \cite{Venugopalan_Xu}).} Here $\tilde L^{1,p}$ is the weighted Sobolev space used in the previous section for $p \in (2, 4)$ and weighted by $|z|^{\tau_p}$ at infinity where $\tau_p = 2 - \frac{4}{p}$. Choose a metric connection $\nabla^E$ on $E$ such that with respect to the fixed trivialization near infinity, the connection agrees with the ordinary differentiation of ${\mb C}^m$-valued functions up to matrix-valued functions in the space $\tilde L^{1,p}$. Then, using $h_E$ and this metric connection, one can define weighted Sobolev space of sections of $E$
\beqn
\tilde L_h^{1,p} ({\bm H}, E)_\lambda = \{ s \in W^{1,p}_{\rm loc}({\bm H}, E)\ |\ s|_{\partial {\bm H}} \subset \lambda,\ \|s \|_{L^\infty} + \| \nabla^E s \|_{L^{p, \tau_p}} < \infty \}.
\eeqn
Consider a real linear Cauchy--Riemann operator $\ov\partial{}_E$ on $E$ such that with respect to the trivialization $\rho$ of $E$ near infinity,
\beqn
\ov\partial{}_E = \ov\partial + \alpha,\ {\rm where}\ \alpha \in W^{1, p, \tau}({\bm H} \setminus Z, \Lambda^{0,1} \otimes {\mb C}^{m \times m}),\ \tau > \tau_p = 2 - \frac{4}{p}.
\eeqn
Consider the bounded linear operator
\beq\label{eqn213}
\ov\partial{}_E: \tilde L_h^{1,p} ({\bm H}, E)_\lambda \to \tilde L^p ({\bm H}, \Lambda^{0,1} \otimes E).
\eeq
By an argument similar to the proof of the Fredholm property of linearizations of affine vortices over ${\bm H}$ given in
\cite[Appendix]{Venugopalan_Xu}, the operator \eqref{eqn213} is Fredholm. Moreover, by straightforward verification, we have the isomorphisms of Banach spaces
\begin{align*}
&\ \tilde L_h^{1,p}({\bm H}) \cong W^{1,p}({\bm D}), &\ \tilde L^p({\bm H}) \cong L^p({\bm D}).
\end{align*}
Hence Lemma \ref{lemma214} implies the following lemma.

\begin{lemma}\label{lemma215}
Suppose the real bundle $\lambda$ over $\partial {\bm D}^2$ is trivializable. Then each trivialization of $\lambda$ canonically induces an orientation of the operator \eqref{eqn213}.
\end{lemma}

Now we prove Proposition \ref{prop213}. Given an affine vortex ${\bm v} = (u, \phi, \psi)$ over ${\bm H}$. By Proposition \ref{prop22}, $u$ has a well-defined limit at $\infty \in {\bm H}$  modulo $K$-action. Since a neighborhood of $\infty$ is contractible, up to gauge transformations we can assume that $u$ has a limit $x_\infty \in L_V$ at $\infty$. Moreover, the convergence can be made stronger. By a result used in \cite[Appendix]{Venugopalan_Xu} (analogous to \cite[Lemma 6.1]{Venugopalan_Xu}), up to gauge transformation, we may assume that in local coordinates around $x_\infty$, 
\beqn 
\nabla u \in W^{1, p, \tau} ({\bm H} \setminus Z),\ \forall \tau \in ( \tau_p, 1 ).  
\eeqn
Therefore, $E':= u^* TV$ has a continuous extension to ${\bm D}^2$, still denoted by $E'$. Define 
\beqn
E:= E' \oplus E'':= E' \oplus ({\mf k} \oplus {\mf k}).
\eeqn
The complex structure on $E''$ is defined by viewing ${\mf k} \oplus {\mf k} \cong {\mf k} \otimes {\mb C}$. On the other hand, the totally real subbundle $\lambda \subset E|_{\partial {\bm D}^2}$ is defined by
\beqn
\lambda = \lambda' \oplus \lambda'' = (u|_{\partial {\bm D}^2})^* TL_V \oplus ({\mf k} \oplus \{0\}) \subset (u|_{\partial {\bm D}^2})^* TV \oplus ({\mf k} \oplus {\mf k}).
\eeqn

To study the orientation of the operator, we make a further decomposition of the first summand of $\lambda$. Let $\lambda_K \subset {\mf \lambda}'$ be the subbundle spanned by infinitesimal $K$-actions. Then $\lambda_K \oplus J_{E'} \lambda_K'$ is canonically trivialized. Moreover, $\lambda_K' \oplus J_{E'}\lambda_K'$ extends to a trivial subbundle $E_K \subset E'$ over ${\bm D}^2$, hence is isomorphic to $E''$. The extension is unique up to homotopy. Let $E_K^\bot\subset E'$ (resp. $\lambda_K^\bot \subset \lambda'$) be the orthogonal complement of $E_K$ (resp. $\lambda_K$). Notice that if we denote by $l: \partial {\bm D}^2 \to L$ the composition of $u|_{\partial {\bm D}^2}$ and the projection $L_V \to L$, then $\lambda_K^\bot \cong l^* TL$.  With respect to the splitting $E \cong E_K^\bot \oplus E_K \oplus E''$ and the isomorphism $E'' \cong E_K$, the linear operator \eqref{eqn27} may be written, up to compact operators, as
\beqn
\left[ \begin{array}{cc} \ov\partial{}_{E_K^\bot} & 0 \\ 0 & D_K \end{array}\right] = \left[\begin{array}{ccc} \ov\partial{}_{E_K^\bot} & 0  & 0 \\
        0            & \ov\partial{}_{E_K} & {\rm Id}_{E_K} \\
        0 & {\rm Id}_{E_K} & \partial{}_{E_K}
\end{array}  \right].
\eeqn
Here $\ov\partial{}_{E_K^\bot}$ is the operator 
\beqn
\ov\partial{}_{E_K^\bot}: L_h^{1,p,\tau_p} ( {\bm H}, E_K^\bot)_{\lambda_K^\bot} \to L^{p, \tau_p}( {\bm H}, E_K^\bot)
\eeqn
of the form considered in Lemma \ref{lemma215} and
\beqn
D_K: L_g^{1,p, \tau_p} ({\bm H}, E_K \oplus E_K) \to L^{p, \tau_p}( {\bm H}, E_K \oplus E_K).
\eeqn
By Lemma \ref{lemma215}, $\ov\partial{}_{E_K^\bot}$ already has an orientation.  Proposition \ref{prop213} follows from the following lemma.

\begin{lemma}\label{lemma216}
The operator $D_K$ has a canonical orientation.
\end{lemma}

\begin{proof}
By simple generalization of \cite[Proposition 97]{Ziltener_book} to the case over ${\bm H}$, for all $\tau \geq 0$, the operator
\beqn
D_K^\tau:= D_K : L_g^{1,p, \tau}(E_K\oplus E_K) \to L^{p, \tau}(E_K
  \oplus E_K)
\eeqn
  is Fredholm with index zero. Let $m_\tau$ be the multiplication by
  the weight function $\rho^\tau$. By explicit calculation (see
  the proof of \cite[Proposition 97]{Ziltener_book}), \beqn D_K^\tau =
  m_\tau^{-1}\circ ( D_K^0 + A_\tau )\circ m_\tau \eeqn where $A_\tau$
  is a compact operator. The orientation of $D_K^\tau$ is induced
  from the orientation of $D_K^0$ (which is an isomorphism) and the
  path $D_K^0 + t A_\tau$.
\end{proof}

This completes the proof of Proposition \ref{prop213}.

\begin{rem}\label{coherence}
The gluing construction requires orienting moduli spaces in a coherent way. For the discussion of the case of holomorphic disks, see \cite{FOOO_Book} and \cite{Katic_Milinkovic}. The orientation specified above is also coherent with respect to two kinds of gluing constructions: (1) bubbling off holomorphic disks in $V$ and (2)  degeneration into several affine vortices over ${\bm H}$ and several affine vortices over ${\bm C}$ connected by a holomorphic disk in the quotient $X$. As our orientation of affine vortices over ${\bm H}$ is via the Cauchy--Riemann operator and essentially topological, the coherence under gluing for both kind of degenerations holds for the same reason as the case of holomorphic disks. 
\end{rem}

\section{Scaled curves}\label{section3}

In this section we describe the moduli spaces of domains used in this paper.  These domains are combinations of spheres, disks, and trees, and their combinatorial types are trees with additional data we call {\it scalings.}

\subsection{Trees}\label{subsection31}

We begin with the combinatorics of trees. A {\it rooted tree} $\Gamma$ consists of a set of vertices $V_\Gamma$, a set of edges $E_\Gamma$, and a distinguished vertex $v_\infty \in V_\Gamma$ called the root. There is a canonical partial order among vertices $V_\Gamma$ so that $v_\infty$ is the minimal vertex. We write $v_\alpha \succ v_\beta$ if $v_\alpha$ and $v_\beta$ are adjacent and $v_\beta$ is closer to the root $v_\infty$. A subtree $\Gamma'$ of $\Gamma$ has an induced root $v_\infty'$ which agrees with $v_\infty$ if $v_\infty \in V_{\Gamma'}$. In the remainder of this paper, all trees are rooted and the term ``rooted'' will be omitted. On the other hand, recall that a {\it ribbon tree}, denoted by $\uds \Gamma$, is a rooted tree $\uds\Gamma$ together with an isotopy class of embeddings of $\uds\Gamma$ into ${\mb C}$.


The combinatorial type of a nodal disk with both disk and sphere components is a {\it based tree}, in which a distinguished subtree corresponds to the disk components: A {\it based tree} consists of a rooted tree $\Gamma$, a nonempty subtree $\uds\Gamma$ containing $v_\infty$ called the {\it base}, with $\uds\Gamma$ equipped with the structure of a ribbon tree. A rooted tree $\Gamma$ without a base is called a {\it base-free tree}.
We refer to the set of semi-infinite interior edges as the set of {\it leaves} $L_\Gamma$, and the set of semi-infinite boundary
edges as the set of {\it tails} $T_\Gamma$, equipped with attaching maps $L_\Gamma \to V_\Gamma$ and $T_\Gamma \to V_{\uds\Gamma}$.  When $V_{\uds\Gamma} \neq \emptyset$, we require that $T_\Gamma$ is nonempty with a distinguished element $t_{\rm out}$ called the {\it output} attached to $v_\infty$, with elements in $T_{\Gamma} \setminus \{t_{\rm out}\}$ (called {\it inputs}) ordered in a way compatible with the ribbon tree structure. See Figure \ref{basedtree} for an illustration of a typical based tree.

\begin{figure}[ht]
    \centering
    \includegraphics{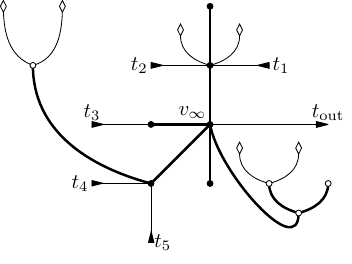}
    \caption{A based tree. The part drawn in straight segments is the base which has a ribbon tree structure. The thick segments are edges and the thin segments are tails. The curly part are vertices, edges (thickened), and leaves which are not part of the the base. The circles are vertices not in the base and the diamonds are leaves.}
    \label{basedtree}
\end{figure}

Inclusions of strata in the moduli spaces of vortices correspond to
certain morphisms of trees.  A {\it morphism} between two based
resp. base-free trees $\Gamma'$ and $\Gamma$, denoted by
$\rho: \Gamma' \to \Gamma$, consists of a surjective map
$\rho_V: V_{\Gamma'} \to V_\Gamma$, a bijection
$\rho_T: T_{\Gamma'} \to T_{\Gamma}$, and a bijection
$\rho_L: L_{\Gamma'} \to L_\Gamma$. They need to satisfy the following
conditions.

\begin{enumerate}

\item The map between sets of vertices $\rho_V: V_{\Gamma'} \to V_\Gamma$ is a tree morphism such that $\rho_V ( V_{\uds \Gamma'} ) = V_{\uds \Gamma}$ and the restriction $\uds\rho_V: V_{\uds\Gamma'} \to V_{\uds\Gamma}$ respects the ribbon tree structures of the bases.

\item If a tail $t'\in T_{\Gamma'}$ resp. a leaf $l' \in L_{\Gamma'}$ is attached to $v' \in V_{\Gamma'}$, then $\rho_T(t')$ resp. $\rho_L(l')$ is attached to $\rho_V (v')$.

\item The map between sets of tails preserves the output, i.e. $\rho_T ( t_{\rm out}') = t_{\rm out}$. 
\end{enumerate}

It is straightforward to generalize the notion of based trees to {\it
  broken} based trees by allowing certain edges in the base to
break. We regard a broken based tree $\Gamma$ as the union of several
unbroken trees, called {\it unbroken parts}, with certain pairs of
ends of tails identified at the breakings. The partial order among
vertices induced by the position of the root $v_\infty$ also induces a partial
order among unbroken parts: if $\Gamma_i$ and $\Gamma_j$ are unbroken
parts of a broken tree $\Gamma$, then denote $\Gamma_i \succ \Gamma_j$
if $\Gamma_i$ and $\Gamma_j$ are adjacent and $\Gamma_j$ is closer to
the root $v_\infty$. Given a broken based tree $\Gamma$, one can {\it glue} at a
subset of breakings and obtain a tree with less breakings.

In some of the following discussions, we also consider disconnected graphs, namely forests. One example is the {\it superstructure} of a based tree. For a based tree $\Gamma$, the superstructure of $\Gamma$, denoted by $\Gamma^{\rm \sup}$, is the subgraph consisting of all vertices not in the base and all edges connecting these vertices. 

We put various extra discrete structures to a based or base-free tree $\Gamma$, consisting of a {\it scale}, a {\it metric type}, and a {\it decoration}.

\begin{defn}\label{defn31} {\rm (Scaled trees)}  
Let $0$, $1$, $\infty$ be ordered as usual $0 \leq  1 \leq \infty$.

\begin{enumerate}

\item Let $\Gamma$ be an unbroken based or base-free tree. A {\it
    scale} on $\Gamma$ is an order-reversing map
  ${\mf s}: V_\Gamma \to \{ 0, 1, \infty \}$ satisfying the following
  condition: Within any non-self-crossing path
  $v_\alpha \succ \cdots \succ v_\beta$ in $\Gamma$ with
  ${\mf s}( v_\alpha) \in \{0, 1\}$,
  ${\mf s}(v_\beta) \in \{1, \infty\}$, there is a unique $v_\gamma$
  in this path with ${\mf s}(v_\gamma) = 1$. A {\it scaled tree}
  $(\Gamma, {\mf s})$ consists of a tree $\Gamma$ and a scale
  ${\mf s}$.

\item Given a scaled tree $(\Gamma, {\mf s})$. For $\lambda = 0, 1, \infty$, we denote $V_\Gamma^\lambda = {\mf s}^{-1}(\lambda)$. The scale also induces a partition on the set of edges
\beqn
E_\Gamma = E_\Gamma^0 \sqcup E_\Gamma^1 \sqcup E_\Gamma^\infty
\eeqn
where $E_\Gamma^0$ are edges connecting vertices in $V_\Gamma^0 \cup V_\Gamma^1$, $E_\Gamma^\infty$ consists of edges connecting vertices in $V_\Gamma^\infty$, and $E_\Gamma^1$ consists of edges connecting one vertex in $E_\Gamma^1$ and one vertex in $E_\Gamma^\infty$. An edge in $E_\Gamma^1$ is called a {\it special edge}; other edges are called {\it non-special edges}. 

\item A scaled based tree $(\Gamma, {\mf s})$ is of {\it scale $0$} if
  ${\mf s} \equiv 0$; is of {\it scale $\infty$} if ${\mf s}|_{V_{\uds \Gamma}} \equiv \infty$; otherwise we say that $(\Gamma, {\mf s})$ is of {\it mixed scale} or {\it scale $1$}.

\item The set of maximal vertices is denoted as follows. When an unbroken scaled tree $(\Gamma, {\mf s})$ is of mixed scale, denote by $\partial V_\Gamma^\infty \subset V_\Gamma^\infty$ the set of vertices in $V_\Gamma^\infty$ which are maximal with respect to the partial order among vertices.  When $(\Gamma, {\mf s})$ is of scale $\infty$, define $\partial V_\Gamma^\infty \subset V_\Gamma^\infty \setminus V_{\uds \Gamma}$ to be the set of vertices in $V_\Gamma^\infty$ which are maximal. When $(\Gamma, {\mf s})$ is of scale $0$, define $\partial V_\Gamma^\infty = \emptyset$.

\item If $\Gamma$ broken, then a scale on $\Gamma$ consists of scales
  ${\mf s}_1, \ldots, {\mf s}_m$ on its unbroken parts
  $\Gamma_1, \ldots, \Gamma_m$, satisfying the following conditions:
  1) the induced map ${\mf s}: V_\Gamma \to \{ 0, 1, \infty \}$ is
  order-reversing; 2) if $\Gamma_1 \succ \cdots \succ \Gamma_l$ is a
  chain of unbroken parts, $(\Gamma_1, {\mf s}_1)$ is of scale $0$ or
  $1$, $(\Gamma_l, {\mf s}_l)$ is of scale $1$ or $\infty$, then there
  is a unique $\Gamma_j$ in this chain such that
  $(\Gamma_j, {\mf s}_j)$ is of scale $1$.
\end{enumerate}
\end{defn}

Now we introduce certain special scaled trees. 

\begin{defn}\label{defn32}(Special domain types)

\begin{enumerate}

\item An {\it infinite edge} $\Gamma$ has $V_\Gamma = E_\Gamma = \emptyset$ and $T_\Gamma = \{ t_{\rm in}, t_{\rm out}\}$.

\item A {\it{\rm Y}-shape} is a scale tree $(\Gamma, {\mf s})$ with $\Gamma = \uds\Gamma$, $V_\Gamma   =  \{v_\infty\}$, $E_\Gamma = \emptyset$, $T_\Gamma = \{ t_{\rm in}', t_{\rm in}'', t_{\rm out}\}$ and ${\mf s} (v_\infty) \in \{0, \infty\}$.

\item A {\it$\Phi$-shape} is a scaled tree $(\Gamma, {\mf s})$ with $\Gamma = \uds\Gamma$, $V_\Gamma = \{ v_\infty\}$, $E_\Gamma = \emptyset$, $T_\Gamma =  \{ t_{\rm in}, t_{\rm out}\}$ and ${\mf s}(v_\infty) = 1$.
\end{enumerate}
\end{defn}

The moduli space of vortices we consider will be a union of strata corresponding to {\it stable trees}, defined as follows. For each $v_\alpha \in V_\Gamma$, denote by $\mathring{d}(v_\alpha)$ the number of vertices $v_\beta \in V_{\Gamma^{\rm sup}}$ with $v_\beta \succ v_\alpha$, by $\uds d(v_\alpha)$ the number of edges in $E_{\uds\Gamma}$ or tails attached to $v_\alpha$, and by $L_{v_\alpha} \subset L_\Gamma$ the set of leaves attached to $v_\alpha$.

\begin{defn}\label{defn33} (Stability) 
A scaled tree $(\Gamma, {\mf s})$ is called {\it stable} if $E_\Gamma$ does not contain an infinite-length edge and satisfy the following additional conditions.
\begin{enumerate}
\item For each $v_\alpha \in V_{\Gamma^{\rm sup}}^\infty \sqcup V_{\Gamma^{\rm sup}}^0$ (a vertex corresponding to a spherical component) there holds $\mathring{d}(v_\alpha) + \# L_{v_\alpha} \geq 2$.

\item For each $v_\alpha \in V_{\uds \Gamma}^\infty \sqcup V_{\uds\Gamma}^0$ (a vertex corresponding to a disk component) there holds $2 \mathring{d}(v_\alpha) + \uds d(v_\alpha) + 2 \# L_{v_\alpha}
  \geq 3$.

\item For each $v_\alpha \in V_\Gamma^1$ (a vertex corresponding to an affine vortex) there holds  $2\mathring{d}(v_\alpha) + \uds d(v_\alpha) + 2 \# L_{v_\alpha} \geq 2$.

\end{enumerate}
\end{defn}

As in \cite{Biran_Cornea, Biran_Cornea_09} we allow edges which connect disk components to acquire length and impose gradient flow equation on the edges. This gives rise the notion of metric trees and metric treed. 

\begin{defn}\label{defn34} {\rm (Metric)} 
Let $(\Gamma, {\mf s})$ be an unbroken scaled tree. A {\it metric} on $(\Gamma, {\mf s})$ is a function (called the {\it length function})
    \beqn
    \ell: E_{\uds\Gamma} \to [0, +\infty)
    \eeqn
    satisfying the following condition
\begin{itemize}

\item {\rm (Balanced condition)} Suppose $(\Gamma, {\mf s})$ has scale   1. For each $v_{\alpha} \in V_{\uds \Gamma}^1 \sqcup V_{\uds \Gamma}^\infty$,   let $P(v_\alpha, v_\infty) \subset E_{\uds\Gamma}$ be the unique non-self-crossing path in $\uds \Gamma$ connecting $v_{ \alpha}$ with the root $v_\infty$. Consider the function $\tilde \ell: V_{\uds \Gamma}^1 \sqcup V_{\uds \Gamma}^\infty \to {\mb R}$ defined by %
\beqn
\tilde \ell (v_{ \alpha}):= \sum_{e  \in P(v_\alpha, v_\infty)} \ell ( e ).
\eeqn
We require that $\tilde \ell$ restricted to $V_{\uds \Gamma}^1$ is a constant and 
\beqn
\tilde \ell|_{V_{\uds\Gamma}^\infty} \leq \tilde \ell|_{V_{\uds\Gamma}^1}.
\eeqn
\end{itemize}
The balanced condition is nonvacuous only for trees of scale 1; for a broken tree, the balanced condition only applies to its unbroken parts of scale $1$ (See Figure \ref{metrictree} for an illustration of a metric tree and the balanced condition).
\end{defn}

\begin{figure}[ht]
    \centering
    \includegraphics{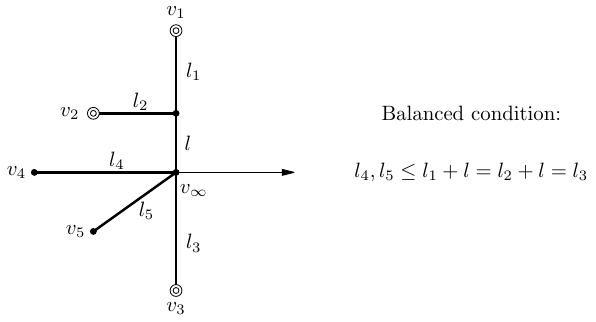}
    \caption{A metric tree of scale 1. Among all vertices, $v_1, v_2, v_3$ have scale $1$ and $v_4, v_5, v_\infty$ have scale $\infty$. If all edges have positive lengths and $l_5 < l_4 = l_3$, then $\partial^* V_{\uds\Gamma}^\infty = \{v_4\}$.}
    \label{metrictree}
\end{figure}

In what follows we define the notion of {\it metric types} that are the discrete data underlying the length functions.

\begin{defn}[Metric type]
Let $(\Gamma, {\mf s})$ be an unbroken scaled tree. A {\it metric type} on $(\Gamma, {\mf s})$ consists of two functions ${\mf m} = ({\mf m}', {\mf m}'')$, ${\mf m}': E_{\uds \Gamma} \to  \{ * , + \}$ and ${\mf m}'': \partial V_{\uds \Gamma}^\infty \to \{ *, -\}$. These functions induce two partitions
\begin{align*}
&\ E_{\uds \Gamma} = E_{\uds \Gamma}^* \sqcup E_{\uds \Gamma}^+,\ &\ \partial V_{\uds \Gamma}^\infty = \partial^* V_{\uds \Gamma}^\infty \sqcup \partial^- V_{\uds \Gamma}^\infty.
\end{align*}
A metric $\ell$ on $(\Gamma, {\mf s})$ is said to be of type ${\mf m}$ if 
\begin{enumerate}
    \item The set $E_{\uds \Gamma}^*$ (resp. $E_{\uds\Gamma}^+$) consists of edges of zero (resp. positive) lengths.
    
    \item When $(\Gamma, {\mf s})$ is of scale $1$, a vertex $v_\alpha$ is in $\partial^* V_{\uds\Gamma}^\infty$ if and only if 
    \beqn
    \tilde \ell(v_\alpha) = \tilde \ell|_{V_{\uds\Gamma}^1}
    \eeqn
\end{enumerate}
We call the elements of $\partial^* V_{\uds \Gamma}^\infty$ {\it bordered} disk vertices and the remaining elements $ \partial^- V_{\uds \Gamma}^\infty$ {\it unbordered}; later these vertices will correspond to components connected to the root $v_\infty$ component by an edge of maximal resp. non-maximal length. 
\end{defn}

We need to record the contact orders of holomorphic curves or vortices
at interior markings with respect to a (singular) divisor. For a tree
$\Gamma$, a {\it decoration} is a function
\beqn
{\mf p}: L_\Gamma \to (\{0\} \cup {\mb N})^{\{ D_1, \ldots, D_h\} }.
\eeqn
We use a monomial to describe the value of ${\mf p}$:  Define
\beqn
{\mf p}(l_i) = p_i = D_1^{a_1} \cdots D_h^{a_h}.
\eeqn
The {\it  degree}  of a variable $D_a$ in $p_i$ is denoted by 
\beqn
{\rm deg}_{D_a} (p_i) \in \{0\}\cup {\mb N}.
\eeqn
Moreover, for each subtree $\Pi\subset \Gamma$, define 
\beqn
{\rm deg}_{D_a} \Pi = \sum_{l_i \in L_\Pi} \frac{ {\rm deg}_{D_a} p_i}{{\rm deg} D_a} \in {\mb Q}
\eeqn
and define
\beq\label{eqn31}
{\rm deg}^{\rm max} \Pi = \max_a {\rm deg}_{D_a} \Pi.
\eeq

\begin{defn} A {\it domain type} is a quadruple $(\Gamma, {\mf s}, {\mf m} , {\mf p})$, often abbreviated by  $\Gamma$, where $\Gamma$ is a based or base-free tree, ${\mf s}$ is   a scale on $\Gamma$, ${\mf m}$ is a metric type on   $(\Gamma, {\mf s})$, and ${\mf p}$ is a decoration. The domain type $(\Gamma, {\mf s}, {\mf m} , {\mf p})$ is called {\it stable} if   $(\Gamma, {\mf s})$ is stable.
\end{defn}

Isomorphisms of domain types are isomorphisms of the underlying graphs preserving the scales, metric types, and decorations. Denote by ${\bf T}$ the set of isomorphism classes of domain types and ${\bf T}^{\rm st} \subset {\bf T}$ the subset of stable ones. Without making confusion, we drop ``isomorphism class'' and call an element of ${\bf T}$ a domain type.

\subsection{Degeneration and broken trees}

The possible degenerations of scaled metric trees involving
``degenerating'' an edge to obtain a broken tree, or extending
an edge of length zero to one with positive length. The balanced
condition in Definition \ref{defn34} imposes some restrictions on such
operations.

\begin{defn}\label{defn37} {(\rm Elementary transformations of domain types)} Let $\Gamma', \Gamma\in {\bf T}$ be domain types such that $\Gamma$ is {\it unbroken} and let ${\mf s}'$, ${\mf s}$ be their scales respectively. We say that $\Gamma$ is obtained from $\Gamma'$ by an {\it elementary transformation} if one of the following situations holds.\footnote{The labellings of these cases indicate that the first two transformations are ``interior,'' hence of codimension two; the cases (F1)---(F4) are of codimension one and correspond to ``fake boundaries'' of one-dimensional moduli spaces; the cases (T1) and (T2) are ``true boundaries.''}

\begin{itemize}

\item[(I1)] {\rm (Sphere bubbling)} There is a morphism $\rho: \Gamma' \to \Gamma$ that  collapses exactly one
  edge in $E_{\Gamma'} \setminus E_{\uds \Gamma'}$ that is not a
  special edge. Geometrically the morphism corresponds to bubbling off
  or gluing a holomorphic sphere.
 
\item[(I2)] {\rm (Interior affine vortex bubbling)} There is a morphism $\rho: \Gamma' \to \Gamma$ and a vertex $v_\alpha' \in \partial V_{\Gamma'}^\infty \setminus V_{\uds \Gamma'}$, such that $\rho$ collapses (and only collapses) the subtree consisting of $v_\alpha'$ and all $v_\beta'$ with $v_\beta' \succ v_\alpha'$ to a single vertex $v_\alpha \in V_\Gamma^1 \setminus V_{\uds \Gamma}$. This kind of morphism corresponds to the gluing or degeneration of stable affine vortices over ${\bm C}$.

\item[(F1)] {\rm (Shrinking the length of a non-special edge to zero)} There is a tree isomorphism $\rho: \Gamma' \to \Gamma$  which preserves all extra structures except that there is one edge $e' \in E_{\uds \Gamma'}^0$ identified with an edge $e \in E_{\uds \Gamma}^+$. 

\item[(F2)] {\rm (Disk bubbling)}   There is a morphism   $\rho: (\Gamma', {\mf s}') \to (\Gamma, {\mf s})$ such that $\rho$ collapses exactly one edge in $E_{\uds \Gamma'}^0$ which is not a   special edge. Geometrically the morphism corresponds to inclusion of  a stratum corresponds to bubbling off or gluing a holomorphic disk.

\item[(F3)] {\rm (Shrinking the lengths of a collection of special edges to zero)} There is a tree isomorphism $\rho: \Gamma' \to \Gamma$ preserving scales, decorations, and metric types on all edges with the following exceptions. There is a vertex $v_{\uds\alpha}' \in \partial^* V_{\uds \Gamma'}^1$ such that 
\[ v_\alpha = \rho_V(v_\alpha') \in \partial^- V_{\uds \Gamma} .\] 
This implies that for all $v_{\uds\beta}' \succ v_{\uds\alpha}'$,
$e_{\uds\beta\uds\alpha}' \in E_{\uds \Gamma'}^*$ and the
corresponding edge $e_{\uds\beta \uds \alpha}$ is in
$E_{\uds \Gamma}^+$.

\item[(F4)] {\rm (Boundary affine vortex bubbling)} There is a morphism $\rho: \Gamma' \to \Gamma$ and a vertex   $v_{\uds\alpha}' \in \partial^* V_{\uds \Gamma'}^\infty$, such that   $\rho$ collapses (and only collapses) the subtree consisting of   $v_{\uds\alpha}'$ and all $v_{\beta}'$ with   $v_\beta' \succ v_{\uds\alpha}'$ to a vertex
  $v_{\uds \alpha} \in V_{\uds \Gamma}^1$. Geometrically the morphism
  corresponds to the gluing or degeneration of stable affine vortices
  over ${\bm H}$.

\item[(T1)] {\rm (Breaking one edge)} $\Gamma$ is obtained from $\Gamma'$ by gluing a breaking connecting two unbroken parts of $\Gamma'$ such
  that the edge obtained from gluing is in $E_\Gamma^0$.
  
\item[(T2)] {\rm (Breaking a collection of edges at infinity)} Both $\Gamma'$ and $\Gamma$ are of scale $1$ or $\infty$. $\Gamma'$ has unbroken parts $\Gamma_1', \ldots, \Gamma_m', \Gamma_\infty'$ where $\Gamma_i'$ are of scale $1$ and $\Gamma_\infty'$ is of scale $\infty$. $\Gamma$ is obtained from $\Gamma'$ by simultaneously gluing the breakings connecting each $\Gamma_i'$ with $\Gamma_\infty'$.
\end{itemize}

\end{defn}

Elementary transformations induce a partial order on the set of domain types as follows.  For $\Pi, \Gamma \in {\bf T}$ with $\Gamma$ unbroken, denote $\Pi \leq \Gamma$, if there are domain types 
\[ \Pi = \Gamma_0', \Gamma_1', \ldots, \Gamma_{k-1}', \Gamma_k' =
\Gamma \in {\bf T} \]  
such that $\Gamma_i'$ is obtained from $\Gamma_{i-1}'$ by an
elementary transformation. This notion can be extended to the case
that $\Gamma$ is broken.  The proof of the following lemma is left to the reader:

\begin{lemma}\label{lemma38}
The relation $\Pi \leq \Gamma$ is a partial order. 
\end{lemma}

\subsection{Treed disks}\label{subsection33}

The domains of the configurations of vortices in our compactification
are constructed from scaled trees by replacing each vertex with a
nodal disk or sphere, or an affine space or half-space in the case of
a scaled vertex.

\begin{defn}\label{defn39} {\rm (Treed disks)}  
  Given an unbroken domain type $\Gamma = (\Gamma, {\mf s}, {\mf m}, {\mf p})$ a {\it treed disk} modelled on $\Gamma$ consists of a collection $(\Sigma_\alpha)$ of surfaces indexed by $v_\alpha \in V_\Gamma$, a collection of markings and nodes $Z$, and a metric $\ell$ of type ${\mf m}$ such that 
\begin{itemize}

\item If $v_\alpha \in V_{\uds \Gamma}$, then $\Sigma_\alpha = {\bm H}$; otherwise $\Sigma_\alpha = {\bm C}$.
\end{itemize}
Each $\Sigma_\alpha$ admits a compactification to a disk $\ol{\Sigma}_\alpha \cong {\bm D}^2$ or sphere $\ol{\Sigma}_\alpha \cong {\bm S}^2$ by adding a ``point at infinity.''  The metric, markings, and nodes form a tuple
\beqn
{\mc C} = \Big( \ell,\ \uds {\bm z} = (\uds z_j)_{t_j \in T_\Gamma^{\rm in}},\ {\bm z} = (z_i)_{l_i \in L_\Gamma},\ {\bm w}= (w_{\alpha \alpha'})_{e_{\alpha \alpha'} \in
    E_\Gamma } \Big) 
\eeqn 
where $\uds{\bm z}$ (resp. ${\bm z}$) is the collection of boundary (resp. interior) marked points, and ${\bm w}$ is the collection of nodes. These data are required to satisfy the following conditions/conventions.
\begin{enumerate}
  
\item (Nodes) If $e_{\uds\alpha \uds\alpha'}\in E_{\uds \Gamma}$, then $w_{\uds\alpha\uds\alpha'}\in \partial \Sigma_{\uds\alpha'}$; otherwise $w_{\alpha \alpha'}\in {\rm Int}\Sigma_{\alpha'}$.
    
\item (Marking) $\uds{z}{}_j \in \partial \Sigma_{\uds\alpha{}_j}$, $z_i \in {\rm Int} \Sigma_{\alpha_i}$.
\end{enumerate}
For each $v_\alpha \in V_\Gamma$, the collection of special
points 
\beqn
W_\alpha:= \big\{ \uds{z}{}_j \ |\ \uds\alpha{}_j = \alpha \big\} \cup \big\{ z_i \ |\ \alpha_i = \alpha \big\} \cup \big\{ w_{\beta\alpha} \ |\ e_{\beta\alpha} \in E_\Gamma \big\} \subset \Sigma_\alpha
\eeqn
are required to be distinct, and the order of boundary markings and nodes $w_{\uds\alpha\uds\alpha'}\in \partial \Sigma_{\uds\alpha'}, \uds{z}{}_j \in \partial \Sigma_{\uds\alpha{}_j}$ corresponding to edges meeting any vertex respect the ribbon structure of $\uds\Gamma$. If $\Gamma$ is broken and has unbroken parts $\Gamma_1, \ldots$, then a treed disk modelled on $\Gamma$ is a collection of treed disks ${\mc C}_s$  modelled on every $\Gamma_s$. Note that a unbroken part of ${\mc C}$ could be an infinite edge.

The set of treed disks naturally forms a category.  If ${\mc C}$ and ${\mc C}'$ are treed disks modelled on $\Gamma$ and $\Gamma'$ respectively, an {\it isomorphism} from ${\mc C}'$ to ${\mc C}$ consists of an isomorphism $\rho: \Gamma' \to \Gamma$ which identifies each $v_\alpha \in V_\Gamma$ with $v_\alpha' \in V_{\Gamma'}$, translations on infinite edges, and biholomorphic maps $\varphi_\alpha: \Sigma_\alpha' \to \Sigma_{\alpha}$ such that the nodes and markings transform correspondingly. We require that when $v_\alpha \in V_\Gamma^0 \sqcup V_\Gamma^\infty$, $\varphi_\alpha$ is a M\"obius transformation fixing the infinity; hen $v_\alpha \in V_\Gamma^1$, then $\varphi_\alpha$ is a translation of $\Sigma_\alpha$ (which is either ${\bm C}$ or ${\bm H}$). It follow immediately from the definition the automorphism group of a treed disk ${\mc C}$ modelled on $\Gamma \in {\bf T}$ is trivial if and only if $\Gamma$ is stable (see Definition \ref{defn33}). Here ends this definition.
\end{defn}

We think of treed disks as topological spaces via the following {\it realization} construction.  Given a treed disk ${\mc C}$ for each $e \in E_{\uds \Gamma}$ let $I_e \subset {\mb R}$ be a closed interval of length $\ell (e)$; to each boundary tail $t \in T_{\uds \Gamma}$ not belong to an infinite edge, let $I_t$ be a semi-infinite interval being either $(-\infty, 0]$
or $[0, +\infty)$; to an infinite edge $\{ t_{\rm in}, t_{\rm out} \}$ we assign $I_{t_-} = I_{t_+} = (-\infty, +\infty)$. These intervals and the surfaces $\Sigma_\alpha$ are glue together in a natural way, and form a connected topological space called the {\it  realization} of ${\mc C}$. In the realization, all intervals with positive or infinite lengths are replaced by a finite closed interval and zero-length intervals are replaced by a point. So the realization is well-defined up to homeomorphism. (see Figure \ref{figure3} for illustration.)

We denote the moduli spaces of stable treed disks as follows. For $\Gamma \in {\bf T}^{\rm st}$, let ${\mc W}_{\Gamma}$ be the set of isomorphism classes of stable treed disks of domain type $ \Gamma$. Denote by the bar notation the union over subordinate domain s
\beqn 
\ov{\mc W}_{\Gamma}:= \bigsqcup_{\substack{\Pi \leq \Gamma \\ \Pi\ {\rm stable}}} {\mc W}_{\Pi}.  
\eeqn

\begin{figure}[ht]
    \centering
    \includegraphics[width=\textwidth]{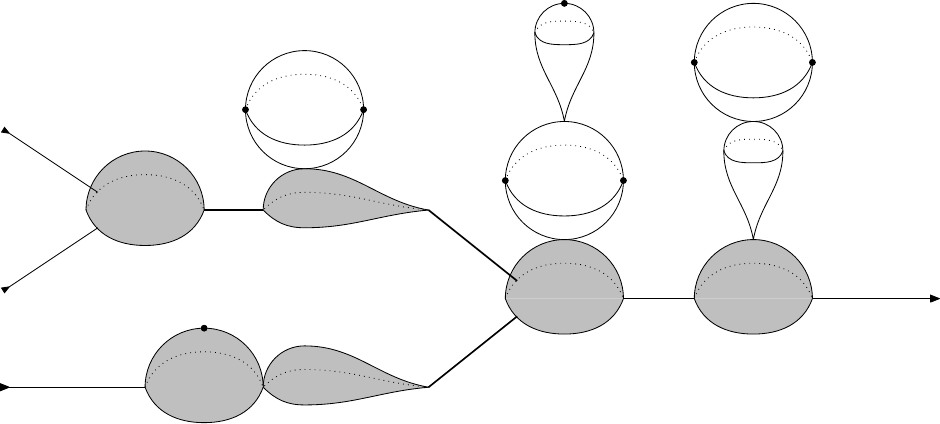}
    \caption{A typical treed disk of scale $1$. It has three inputs (located on the left of the picture), 11 vertices (two-dimensional components), and 8 markings (leaves). The teardrop-like components are of scale $1$. The gray parts together with edges connecting them form the base. This treed disk is stable.}
    \label{figure3}
\end{figure}

We introduce a topology on the moduli space of treed disks by the following notion of sequential convergence, which generalizes that for stable genus zero curves with $k + 1$ markings in McDuff--Salamon \cite{McDuff_Salamon_2004} and for trees  in Boardman--Vogt \cite{Boardman_Vogt}. We first recall the simple case where $(\Gamma, {\mf s})$ is the scaled base-free tree with a single vertex $v_\infty \in V_\Gamma$ whose scale is $1$ and interior leaves $l_1, \ldots, l_k$ ($k \geq 1$). Let ${\mf p}$ be any decoration. Let ${\mc C}_n$ be a sequence of treed disks modelled on $( \Gamma, {\mf s})$, which are equivalent to distinct points $z_{n, 1}, \ldots, z_{n, k} \in {\bm C}$ modulo translations. As $n$ goes to $\infty$, the points $z_{n,i}$ may come together or separate from each other. Consider $ \Pi \leq \Gamma$ and let ${\mc C}_\infty$ be a tree disk modelled on $ \Pi$, which is described as in Definition \ref{defn39}. We temporarily fix the following notations. Let $v_\alpha \in V_\Pi^0$. Then there exists a unique vertex $v_\alpha^1 \in V_\Pi^1$ connecting $v_\alpha$ to the root $v_\infty$ (see Definition \ref{defn31}), and a point $w_\alpha^1 \in {\bm C}$ corresponding to the node in the path connecting $v_\alpha$ and $v_\alpha^1$.

\begin{defn}\label{defn310}
A sequence ${\mc C}_n$ {\it converges} to ${\mc C}_\infty$ if the following conditions hold.
\begin{itemize}
\item For each $v_\alpha \in V_\Pi^0$, there exist a sequence of M\"obius transformations
  $\phi_{n, \alpha}: (\ov\Sigma_\alpha, \infty) \to (\ov {\bm C}, \infty)$ that converge to the constant map with value $w_\alpha^1$ uniformly with all derivatives away from nodes on $\ov\Sigma_\alpha$, such that for each leaf $l_i \in L_{v_\alpha}$, $\phi_{n, \alpha}^{-1} (z_{n, i})$ converges to $z_{\infty, i} \in \Sigma_\alpha$.
  
\item For each $v_\alpha \in V_\Pi^\infty$, there exist a sequence of M\"obius transformations $\phi_{n, \alpha}: (\ov{\Sigma}_\alpha, \infty) \to (\ov{{\bm C}}, \infty)$ that converges to the constant map with value $\infty$ away from nodes on $\ov\Sigma_\alpha$. 

\item For each $v_\alpha \in V_\Pi^1$, there exist a sequence of translations $\phi_{n, \alpha}: \Sigma_\alpha \to {\bm C}$ such that for each leaf $l_i \in L_{v_\alpha}$, $\phi_{n, \alpha}^{-1}(z_{n, i})$ converges to $z_{\infty, i} \in \Sigma_\alpha$.
\end{itemize} 
The sequences of M\"obius transformations satisfy the following condition. For each edge $e_{\alpha \beta} \in E_\Pi$ having a corresponding node $w_{\alpha \beta} \in \Sigma_\beta$, the sequence of maps
\beqn
(\phi_{n, \beta})^{-1} \circ \phi_{n, \alpha}: \Sigma_\alpha \to \Sigma_\beta
\eeqn
converges uniformly with all derivatives to the constant $w_{\alpha \beta}$ on compact subsets. Here ends this definition. 
\end{defn}

The sequential convergence of based treed disks is defined similarly. Consider a scaled based tree $(\Gamma, {\mf s})$ with a single vertex whose scale is $1$, boundary tails $t_1, \ldots, t_l$ and interior leaves $l_1, \ldots, l_k$. The stability condition is equivalent to $l + 2k \geq 1$. Let ${\mf m}$ be the trivial metric type. Let ${\mf p}$ be any decoration. For $\Gamma = (\Gamma, {\mf s}, {\mf m}, {\mf p})$, the topology of $\ov{\mc W}_{ \Gamma}$ is defined in a way similar to Definition \ref{defn310} (cf. \cite[Section 2]{Xu_glue} for detailed discussion of a special case). We omit the details.

For a general stable scaled based or base-free tree $\Gamma$, the notion of sequential convergence in $\ov{\mc W}_\Gamma$ can be obtained from the above two special cases combined with the notion of convergence of stable marked spheres or stable marked disks, and the notion of convergence of metric trees. Again we omit the details. The sequential convergence actually determines a compact Hausdorff topology, because of the existence of local distance functions as in McDuff--Salamon \cite{McDuff_Salamon_2004}. We leave it to the reader
  to check the following statement.  

\begin{lemma}\label{lemma311}
The moduli space $\ov{\mc{W}}_{\Gamma}$ is compact and Hausdorff with respect to the topology defined in Definition \ref{defn310}.
\end{lemma}

We give a formula for the dimension of the moduli spaces of stable treed disks. Define 
\beqn
d^\lambda(l,k) = \left\{ \begin{array}{cc} l + 2k - 2,\ &\ \lambda = 0\ {\rm or}\ \infty;\\
                                           l + 2k - 1,\ &\ \lambda = 1.\end{array}\right.
\eeqn
Then $d^\lambda(l,k)$  is the dimension of ${\mc W}_{\Gamma}$ where $\Gamma$ is a top
stratum stable domain type of scale $\lambda$ with $l$ inputs
and $k$ leaves.  For a general stable $\Pi$, there is a unique top
stratum stable domain type $\Gamma$ with $\Pi \leq \Gamma$. One has  
\begin{multline}\label{eqn32}
{\rm dim} {\mc W}_\Pi = {\rm dim} {\mc W}_\Gamma - {\rm codim} \Pi\\
= d^\lambda(l, k) - \# E_{\uds \Pi}^* - 2 \#  \Big( E_\Pi^0  \cup E_\Pi^\infty \setminus E_{\uds \Pi} \Big) - \# \partial^* V_{\uds \Pi}^\infty - 2 \# \partial V_{\Pi^{\rm sup}}^\infty - {\rm b} (\uds{\Pi}) .
\end{multline}
Here ${\rm b}(\uds{\Pi} )$ is a number characterizing how many breakings $\uds{\Gamma}$ has,  defined as
\beqn
{\rm b}( \uds{\Pi} ) = \left\{ \begin{array}{l} \# {\rm breakings\ in\ }\uds{\Pi}, \hfill {\rm if\ all\ unbroken\ parts\ are\ of\ scale\ }0\ {\rm or}\ \infty;\\
                                               \# {\rm unbroken\ parts\
                                    of\ scale \ }0 + \# {\rm unbroken\
                                    parts\ of\ scale\ }\infty, \
                                    \hspace{1cm} 
{\rm in\ other\ cases\ }.     \end{array} \right.
\eeqn
In particular, if $\uds\Pi$ is unbroken, then ${\rm b}(\uds\Pi) = 0$.

\subsection{Orientations}\label{subsection34}

To define the signed counts one needs to specify orientations on the moduli spaces of domains. Consider the moduli space ${\mc M}_{k+1, l}$ of marked disks with $k+1$ boundary markings and $l$ interior markings. Each point $p \in {\mc M}_{k+1, l}$ can be represented uniquely by a configuration of points
\beqn
0 = x_1 < x_2 < \cdots < x_{k-1} < x_k = 1,\ z_1, \ldots, z_l \in {\rm Int} {\bm H}
\eeqn
where we view the $0$-th boundary marking as $\infty \in {\bm D}\cong {\bm H} \cup \{\infty\}$. Via the above representation, one can identify
\beqn
T_p {\mc M}_{k+1, l} \cong {\mb R}^{k-2} \oplus {\mb C}^l.
\eeqn
Here the $i$-th ${\mb R}$-factor represents the direction of deformations of $x_{i+1}$. We orient ${\mc M}_{k+1, l}$ via the standard orientation of ${\mb R}^{k-2} \oplus {\mb C}^l$. 


On the other hand one consider marked {\it scaled disks}, namely configurations in the upper half plane modulo translation. Let ${\rm Aff}_{k, l}$ be the moduli space of configurations of $k$ boundary points and $l$ interior points of ${\bm H}$ modulo translation. A point $p \in {\rm Aff}_{k,l}$ can be uniquely represented by a configuration 
\beqn
0 = x_1 < \cdots < x_k < \infty,\ z_1, \ldots, z_l \in {\rm Int} {\bm H}.
\eeqn
Its tangent spaces are naturally identified with ${\mb R}^{k-1} \oplus {\mb C}^l$ hence are oriented. 

We extend the orientation specified above to treed configurations. In fact we only need to orient moduli spaces ${\mc M}_\Gamma$ of domain types belonging to the top strata. First consider domain types of scale $0$ and fix integers $k, l \ge 0$. Let $\Gamma$ be a stable unbroken domain type of scale $0$ with $k$ boundary inputs and $l$ interior leaves. Suppose the metric type of $\Gamma$ requires exactly one edge have length zero. The moduli space ${\mc M}_\Gamma$ is the codimension one boundary stratum of two moduli spaces ${\mc M}_{\Gamma'}$ and ${\mc M}_{\Gamma''}$ of top dimension $2l + k-2$. We may assume that $\Gamma'$ has one fewer vertex than $\Gamma''$. 
For the inductive step, suppose one has oriented ${\mc M}_{\Gamma'}$, then we orient ${\mc M}_{\Gamma''}$ in such a way that the induced boundary orientations on ${\mc M}_\Gamma$ from ${\mc M}_{\Gamma'}$ and ${\mc M}_{\Gamma''}$ are opposite. The union of moduli spaces of top dimensions by gluing common codimension one boundary strata becomes an oriented manifold. One can orient moduli spaces of treed disks of scale $1$ and $\infty$ in a similar way; we leave the details to the reader.

\section{Perturbations}\label{section4}

In this section we define the notion of perturbation data as certain functions defined over the universal curves.

\subsection{The universal curves}\label{subsection41}

Our perturbation scheme requires us to study the geometry and topology of the universal curve of stable treed disks. Given a stable domain type $\Gamma = (\Gamma, {\mf s}, {\mf m}, {\mf p})$, the universal curve $\ov{\mc U}_\Gamma$ is a compact Hausdorff topological space admitting a projection map $\pi_\Gamma: \ov{\mc U}_\Gamma \to \ov{\mc W}_\Gamma$ such that for each stable treed disk ${\mc C}$ representing a point $p \in \ov{\mc W}_{\Gamma}$, there is a homeomorphism $\pi_{\Gamma}^{-1} (p) \cong {\mc C}$. The space $\ov{\mc U}_\Gamma$ is independent of the decoration ${\mf p}$. Furthermore, we may write the universal curve as a union
\beqn
\ov{\mc U}_\Gamma = \ov{\mc U}{}_\Gamma^{\rm 1D} \cup \ov{\mc U}{}_\Gamma^{\rm 2D},
\eeqn
where $\ov{\mc U}{}_\Gamma^{\rm 1D}$ resp. $\ov{\mc U}{}_\Gamma^{\rm 2D}$ is the one-dimensional resp. two-dimensional part. Denote 
\beqn
{\mc U}_\Gamma = \pi_\Gamma^{-1}({\mc W}_\Gamma),\ {\mc U}_\Gamma^{\rm 1D} = {\mc U}_\Gamma \cap \ov{\mc U}{}_\Gamma^{\rm 1D},\ {\mc U}_\Gamma^{\rm 2D} = {\mc U}_\Gamma \cap \ov{\mc U}{}_\Gamma^{\rm 2D}.
\eeqn
The restrictions of $\pi_\Gamma$ to ${\mc U}{}_\Gamma^{\rm 1D}$ or ${\mc U}{}_\Gamma^{\rm 2D}$ are fibre bundles. 

There are certain inclusion maps among universal curves. If $\Pi \leq \Gamma$, then there is a natural map 
\beq\label{eqn41}
\ov{\mc U}{}_\Pi \hookrightarrow \ov{\mc U}{}_\Gamma. 
\eeq

\subsection{Locality of maps}\label{subsection42}

In the argument of Cieliebak--Mohnke, perturbations are chosen  that on each component depend only on the special points on that component; we call this property {\it locality}.   

To state this property, we first introduce the following notations. Let $\Gamma$ be a stable domain type. For each vertex $v \in V_\Gamma$ (resp. edge $e \in V_\Gamma$) let
\beqn
\ov{\mc U}_{\Gamma, v} \subset \ov{\mc U}_\Gamma\ \left( {\rm resp.}\ \ov{\mc U}{}_{\Gamma, e} \subset \ov{\mc U}{}_\Gamma \right)
\eeqn
be the closed set corresponding to points on the $v$-component (resp. points on the edge $e$). More generally, for a subgraph $\Pi$ (not necessarily connected), define
\beqn
\ov{\mc U}{}_{\Gamma, \Pi} = \bigcup_{v \in V_\Pi} \ov{\mc U}{}_{\Gamma, v}  \cup \bigcup_{e \in E_\Pi} \ov{\mc U}{}_{\Gamma, e} 	\subset \ov{\mc U}_{\Gamma}.
\eeqn
Let $\Gamma(\Pi)$ be the domain type obtained by removing from $\Gamma$ all vertices not in $V_\Pi$ and all edges not in $E_\Pi$.\footnote{There is a canonical way of assigning a scale, a decoration, and a metric type on $\Gamma(\Pi)$.} Then there is a natural forgetful map 
\beq\label{eqn42}
\ov{\mc U}_{\Gamma, \Pi} \to \ov{\mc U}_{\Gamma(\Pi)}.
\eeq

\begin{defn}\label{defn41}
Let $Z$ be a set. Given a stable domain type $\Gamma \in  {\bf T}^{\rm st}$. A map $f_\Gamma: \ov{\mc U}{}_\Gamma \to Z$ is called {\it  local} if the following conditions are satisfied. 

\begin{enumerate}

\item \label{locala} For each $v \in V_{\Gamma^{\rm sup}}$, let
  $\uds\Gamma(v)$ be the subgraph consisting of all vertices in the
  base $\uds\Gamma$ and the vertex $v$ and all edges connecting
  them. Then there exists a map
  $f_v: \ov{\mc U}_{\Gamma(\uds\Gamma(v))} \to Z$ such that the
  restriction of $f_\Gamma$ to $\ov{\mc U}_{\Gamma, \uds\Gamma(v)}$ is
  equal to the pullback of $f_v$ via the map
  $\ov{\mc U}_{\Gamma,\uds\Gamma(v)} \to \ov{\mc
    U}_{\Gamma(\uds\Gamma(v))}$.

\item \label{localb} For the base $\uds \Gamma$, there exists a map
  $f_{\uds \Gamma}: \ov{\mc U}{}_{\Gamma(\uds\Gamma)} \to Z$ such that
  the restriction of $f_\Gamma$ to $\ov{\mc U}{}_{\Gamma, \uds\Gamma}$
  is equal to the pullback of $f_{\uds \Gamma}$ via the forgetful map
  $\ov{\mc U}{}_{\Gamma, \uds\Gamma} \to \ov{\mc
    U}{}_{\Gamma(\uds\Gamma)}$.

\item For any $\Pi \leq \Gamma$, the restriction of $f_\Gamma$ to
  $\ov{\mc U}_\Pi$ via \eqref{eqn41} still satisfies \eqref{locala}
  and \eqref{localb}.

\end{enumerate}
\end{defn}

\subsection{Another forgetful map}\label{subsection43}

Before we discuss the notion of perturbation data, we describe a kind of forgetful map which is related to the method of regularizing ``crowded'' configurations and different from the forgetful map \eqref{eqn42}.

\begin{defn} \label{defn42}
  Consider a stable domain type $\Gamma$. For each subset $W \subset V_{\Gamma^{\sup}}$, let $L_W \subset L_\Gamma$ be the set
  of leaves attached to vertices in $W$. Given $W$, denote by
  $\Gamma_W$ the scaled tree obtained by the following two-step
  operation.
\begin{enumerate}
\item In the first step, for each connected component $W_a \subset W$
  with $L_{W_a} \neq \emptyset$, we forget all leaves in $L_{W_a}$ but
  the one with the largest index; then we stabilize and obtain a
  stable tree $\Gamma_{W, +}$. The set $V_\Gamma \setminus W$ clearly
  corresponds to a subset of $V_{\Gamma_{W, +}}$, whose complement is
  denoted by $\bar W_+$.
\item In the second step, for each
  $v \in V_{\Gamma_{W, +}^{\rm sup}}^1 \cap \bar W_+$ which is maximal
  in $\Gamma_{W, +}$ and has only one leaf, we replace $v$ by a new
  leaf and obtain a new tree $\Gamma_W$. The set
  $V_\Gamma \setminus W$ canonically corresponds to a subset of
  $V_{\Gamma_W}$ whose complement is denoted by $\bar W$. (See
  illustration in Figure \ref{figure4}.).
\end{enumerate}

\begin{figure}[ht]
    \centering
    \includegraphics[width=\textwidth]{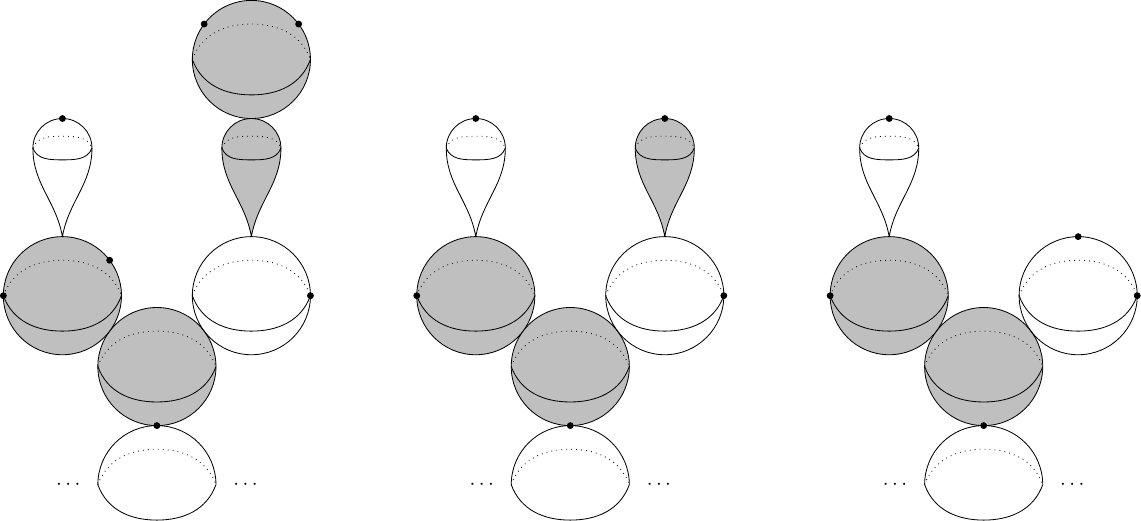}
    \caption{The forgetful map $\Gamma \mapsto \Gamma_W$. The gray components in the configuration on the left are labelled by elements of $W$. In the first step we forget all but one leaf in each connected component of $W$ and then stabilize, obtaining a stable tree (the configuration in the middle). In the second step we collapse certain vertices of scale $1$. The gray components in the configuration on the right are labelled by elements in $\bar W$.}
    \label{figure4}
\end{figure}

The graph $\Gamma_W$ inherits a metric type ${\mf m}_W$ and decoration
${\mf p}_W$ on $\Gamma_W$ as follows.

\begin{enumerate}

\item The metric type ${\mf m}_W$ is automatically induced from ${\mf m}$. 

\item The decoration ${\mf p}_W$ is defined as follows. There is a natural decomposition 
\beqn
L_{\Gamma_W} = L_{\Gamma_W}^{\rm old} \sqcup L_{\Gamma_W}^{\rm new}
\eeqn
where naturally old leaves in $L_{\Gamma_W}^{\rm old}$ corresponds to leaves in $L_\Gamma \setminus L_W$ and the values of the ole decoration ${\mf p}$ descend. Each new leaf $l \in L_{\Gamma_W}^{\rm new}$ corresponds to a connected component $W_a \subset W$ and define 
\beqn
{\mf p}_W( l ) = \prod_{l_j^a \in L_{W_a}} {\mf p}( l_j^a) \in ( \{0\} \cup {\mb N})^{\{ D_1,\ldots, D_h\}}.
\eeqn
\end{enumerate}

This construction results in a possibly empty subset $\bar W \subset \Gamma_W$. Each element of $\bar W$ corresponds to a vertex in $\Gamma$ which is not collapsed but whose valence or the number of leaves attached is changed. 
\end{defn}

\begin{rem} 
  A particular case of the above construction is the operation of
  removing certain sphere components. If $W = V_{\Gamma^{\rm sup}}^0$,
  then we denote the tree $\Gamma_W$ by $\check\Gamma$. There is then
  a contraction map \beqn \ov{\mc U}_\Gamma \to \ov{\mc
    U}{}_{\check\Gamma}.  \eeqn
\end{rem}

\subsection{Neighborhoods of nodes}\label{subsection44}

Our perturbation data are certain functions defined on
$\ov{\mc U}_\Gamma$ that vanish near nodes. We need to specify
neighborhoods of nodal points and breakings in a coherent way. In the
total space of the universal curve, let
$\ov{\mc U}{}_{\Gamma}^{\rm node}$
 denote the union of the nodes, all
edges in $E_{\underline \Gamma}^0$, and all leaves in $L_\Gamma$.  Let
$\ov{\mc U}{}_\Gamma^{\rm break}$ denote the closed subset consisting
of all broken points of edges and all infinities of infinite edges.
We would like to fix certain neighborhoods of these subsets.

\begin{defn}(Nodal neighborhoods) Assume $\Gamma \in {\bf T}^{\rm st}$.

\begin{enumerate}

\item A {\it  nodal neighborhood} in the universal curve $\ov{\mc
    U}_\Gamma$ is an open neighborhood 
$\ov{\mc U}{}^o_\Gamma$
of $\ov{\mc U}{}_\Gamma^{\rm node} \cup \ov{\mc U}{}_\Gamma^{\rm break}$ such that the characteristic function of this neighborhood is a local map (see Definition \ref{defn41}) from $\ov{\mc U}{}_\Gamma$ to $\{0, 1\}$.

\item Denote the intersection of the nodal neighborhood with $\ov{\mc U}{}_\Gamma^{\rm 2D}$ and $\ov{\mc U}{}_\Gamma^{\rm 1D}$ by 
\begin{align*}
  &\ \ov{\mc U}{}_{\Gamma}^{\rm thin} := \ov{\mc U}{}^o_\Gamma \cap \ov{\mc U}{}_\Gamma^{\rm 2D}
    \subset \ov{\mc U}{}_\Gamma^{\rm 2D},\ &\ \ov{\mc U}{}_\Gamma^{\rm
                                             long} := \ov{\mc U}^o_\Gamma \cap 
                                             \ov{\mc U}{}_\Gamma^{\rm 1D} 
                                             \subset \ov{\mc U}{}_\Gamma^{\rm 1D}
\end{align*}
and call them the thin part and the long part respectively. Denote 
\begin{align*}
&\ \ov{\mc U}{}_\Gamma^{\rm thick}:= \ov{\mc U}{}_\Gamma^{\rm 2D} \setminus \ov{\mc U}{}_\Gamma^{\rm thin},\ &\ \ov{\mc U}{}_\Gamma^{\rm short}:= \ov{\mc U}{}_\Gamma^{\rm 1D} \setminus \ov{\mc U}{}_\Gamma^{\rm long}
\end{align*}
and call them the thick part and the short part. For each $v \in V_\Gamma$, denote
\begin{align*}
&\ \ov{\mc U}{}_{\Gamma, v}^{\rm thin}:= \ov{\mc U}{}_\Gamma^{\rm thin} \cap \ov{\mc U}{}_{\Gamma, v},\ &\ \ov{\mc U}{}_{\Gamma, v}^{\rm thick}:= \ov{\mc U}{}_\Gamma^{\rm thick}\cap \ov{\mc U}{}_{\Gamma, v}.
\end{align*}

\item Given $A>0$ and a nodal neighborhood, we denote 
\beqn
|\ov{\mc U}{}_\Gamma^{\rm thick}| > A,
\eeqn
if for each $p \in {\mc W}_\Gamma$ represented by ${\mc C}_p$ and each
$v \in V_\Gamma^1\setminus V_{\uds\Gamma}$, there holds
\beqn {\rm
  Area}\Big( {\mc C}_p \cap \ov{\mc U}{}_{\Gamma, v}^{\rm thick} \Big)
> A.  \eeqn 
Here the area is taken with respect to an identification
${\mc C}_{p, v} \cong {\bm C}$. Such an identification is canonical
up to translation so the area is well-defined.
\end{enumerate}
\end{defn}

\begin{lemma}\label{lemma45}
Given $\delta > 0$, there exist a collection of nodal neighborhoods
\beqn
\Big\{ \ov{\mc U}{}_\Gamma^\delta\ |\ \Gamma \in {\bf T}^{\rm st}  \Big\}
\eeqn
that satisfy the following properties. 

\begin{enumerate}

\item For any $ \Pi \leq \Gamma$, we have 
\beqn
\ov{\mc U}{}_\Pi^\delta = \ov{\mc U}{}_\Gamma^\delta \cap \ov{\mc U}{}_{\Pi}.
\eeqn

\item For each nodal neighborhood $\ov{\mc U}{}_\Gamma^\delta$, for each vertex $v\in V_\Gamma^1$, there holds
\beq\label{eqn43}
| \ov{\mc U}{}_{\Gamma, v}^{\rm thick} | > \delta^{-2} \cdot {\rm deg}^{\rm max} \Gamma(v).
\eeq
Here ${\rm deg}^{\rm max}$ is defined by \eqref{eqn31}.

\item For each $e \in E_{\uds\Gamma}^+$ and each fibre ${\mc C}_p \subset {\mc U}_\Gamma$, ${\mc C}_{p, e}^{\rm short}:= {\mc C}_p \cap \ov{\mc U}{}_{\Gamma, e}^{\rm short} \neq \emptyset$.

\item For each $v \in V_{\uds\Gamma}$ and each fibre ${\mc C}_p \subset {\mc U}_\Gamma$, ${\mc C}_{p, v}^{\rm thick}:= {\mc C}_p \cap \ov{\mc U}{}_{\Gamma, v}^{\rm thick}$ has nonempty intersection with the boundary of the $v$-component. 

\item For each $v \in V_{\Gamma^{\rm sup}}^0$, $\ov{\mc U}{}_{\Gamma, v}^{\rm thin} = \ov{\mc U}{}_{\Gamma, v}$. 

\end{enumerate}
\end{lemma}

This lemma can be proved via an induction argument. We leave the
tedious proof to the reader. Having chosen these open sets, one can
define the notion of smoothness on functions defined on the universal
curve $\ov{\mc U}_\Gamma$ which vanish on
$\ov{\mc U}{}_\Gamma^{\rm thin} \cup \ov{\mc U}{}_\Gamma^{\rm long}$.
This is because the complement of the closure of
$\ov{\mc U}{}_\Gamma^{\rm thin} \cup \ov{\mc U}{}_\Gamma^{\rm long}$
is the union of two smooth manifolds (the 1D and 2D
components). Furthermore, by choosing certain Riemannian metrics on
the 1D and 2D components and a particular sequence $\epsilon_l$ of
positive numbers converging to zero, one can define Floer's
$C^\epsilon$-norm (see \cite{Floer_unregularized}) on such smooth
functions (possibly infinite) . More precisely, for a function $f$
vanishing on $\ov{\mc U}{}_\Gamma^\delta$, define its norm as \beqn \|
f \|_{C^\epsilon} = \| f \|_{C^\epsilon( \ov{\mc U}{}_\Gamma^{\rm
    1D})} + \| f \|_{C^\epsilon( \ov{\mc U}{}_\Gamma^{\rm 2D})}.
\eeqn We skip the precise definition of Floer's norm; all one needs to
know is that the space of $C^\epsilon$-functions contain certain
bump functions supported in arbitrarily small balls.

\subsection{Space of perturbations}\label{subsection45}

Domain-dependent perturbations are certain maps from the universal curves to an infinite-dimensional space. Now we specialize
this infinite-dimensional space ${\mc P}$. We have a space of smooth
$K$-invariant almost complex structures ${\mc J}_D$ (see
\eqref{eqn211}).  Choose a Morse function $F_L \in C^\infty(L)$ with a
unique maximum ${\bm x}_M \in L$.  The space of perturbations
$h_L \in C^\infty(L)$ of the Morse function $F_L \in C^\infty(L)$ is
denoted by ${\mc F}_L$. Define
\beqn
{\mc P} = {\mc F}_L \times {\mc J}_D.
\eeqn
We consider local functions  
\beqn
P_\Gamma = (F_\Gamma, J_\Gamma): \ov{\mc U}_\Gamma \to {\mc P}
\eeqn
whose values are in a small neighborhood of $(F_L, J_V)$. Such a map
$P_\Gamma = (F_\Gamma, J_\Gamma)$ is called a {\it  domain-dependent
  perturbation}, or simply a perturbation. There is a particular map, denoted by
$P_\Gamma^\circ$ which is equal to the constant $(F_L, J_V)$. If
$P_\Gamma$ agrees with $P_\Gamma^\circ$ over any subset of
$\ov{\mc U}_\Gamma$, we also say that $P_\Gamma$ {\it  vanishes} over
that subset.

Now we define our set of perturbations. 

\begin{defn}\label{defn46}
Given $\Gamma \in {\bf T}^{\rm st}$ and a nodal neighborhoods $\ov{\mc U}{}_\Gamma^\delta$. Let ${\mc P}_\Gamma$ be the space of smooth local maps $P_\Gamma = (F_\Gamma, J_\Gamma): \ov{\mc U}_\Gamma \to {\mc P}$ satisfying the following conditions.

\begin{enumerate}

\item $P_\Gamma = P_\Gamma^\circ$ on $\ov{\mc U}{}_\Gamma^\delta$.

\item $P_\Gamma - P_\Gamma^\circ$ has finite $C^\epsilon$-norm.

\item The restriction of $J_\Gamma$ to $\ov{\mc U}{}_\Gamma^\infty \times D$ is equal to $J_V$. 

\end{enumerate}
\end{defn}

\begin{rem} \label{rem47} 
The forgetful morphism in Subsection \ref{subsection43}
  extends to perturbations as follows.  Suppose
  $\Gamma \in {\bf T}^{\rm st}$ and $W \subset V_{\Gamma^{\rm sup}}$.
  Let $\Gamma_W$ be the domain type obtained by the operation
  defined in Subsection \ref{subsection43}. The nodal
  neighborhood $\ov{\mc U}{}_\Gamma^\delta$ induces a nodal
  neighborhood of $\ov{\mc U}_{\Gamma_W}$ defined as follows. Notice
  that there is a natural bijection
  $V_\Gamma \setminus W \cong V_{\Gamma_W} \setminus \bar W$ and an
  isomorphism $\uds\Gamma \cong \uds{\Gamma}_W$, for each
  $v \in V_{\Gamma_W} \setminus \bar W$.  Define
\beqn
\ov{\mc U}{}_{\Gamma_W, v}^{\rm thin} = \ov{\mc U}{}_{\Gamma, v}^{\rm thin}
\eeqn
and
\beqn
\ov{\mc U}{}_{\Gamma_W, \bar W}^\delta:= \prod_{v \in V_{\Gamma_W}\setminus \bar W} \ov{\mc U}{}_{\Gamma_W, v}^{\rm thin} \times \prod_{v \in \bar W} \ov{\mc U}{}_{\Gamma_W, v} \times \prod_{e \in E_{\Gamma_W}^+\cup T_{\Gamma_W}} \ov{\mc U}{}_{\Gamma, e}^{\rm long}. 
\eeqn
  If we denote by ${\mc P}_{\Gamma_W, \bar W}$ the space of 
  perturbations with respect to the above induced nodal neighborhood,
  then there is a natural map 
\beq\label{eqn44}
{\mc P}_\Gamma \to {\mc P}_{\Gamma_W, \bar W},\ P_\Gamma \mapsto P_{\Gamma_W}
\eeq 
defined as follows. For each $v \notin \bar W$ resp. an edge $e \in E_{{\uds \Gamma}_W}$, the restriction of $P_{\Gamma_W}$ on to the component corresponding to $v$ resp. the edge corresponding to $e$ is inherited from $P_\Gamma$; for each $v \in \bar W$, the restriction of $P_{\Gamma_W}$ to the component corresponding to $v$ is $J_V$. The locality condition implies that this map is well-defined. A right inverse of this map is given by choosing zero perturbation on the collapsed components.  Since the map is surjective and linear, the preimage of a comeager subset is also  comeager.   This fact will play a role in the proof of Proposition \ref{prop611}). 
\end{rem}

Perturbation data for different domain types need to satisfy compatibility condition. 

\begin{defn}\label{coherent}
A system of perturbation data 
\beqn
{\bm P} = \{ P_\Gamma \in {\mc P}_\Gamma \ |\ \Gamma \in {\bf T}^{\rm st}\}.
\eeqn
is {\it  coherent} if the following conditions are satisfied:
\begin{enumerate}

\item If $\Pi \leq \Gamma$, then with respect to the natural inclusion $\ov{\mc U}_\Pi \hookrightarrow \ov{\mc U}_\Gamma$ (see \eqref{eqn41}), $P_\Pi$ agrees with the restriction of $P_\Gamma$.

\item If $\Pi$ is an unbroken part of $\Gamma$, then the restriction
  of $P_\Gamma$ to $\ov{\mc U}{}_{\Gamma, \Pi}$ is equal to the
  pullback of $P_\Pi$ via the forgetful map
  $\ov{\mc U}{}_{\Gamma, \Pi} \to \ov{\mc U}{}_{\Gamma(\Pi)} \cong \ov{\mc U}_\Pi$ (see
  \eqref{eqn42}). In other words, $P_\Gamma$ is determined by $P_\Pi$
  for all its unbroken parts.
\end{enumerate}
\end{defn}

\section{Moduli Spaces}\label{section5}

In this section we construct moduli space of pseudoholomorphic curves and vortices appearing in definitions of the $A_\infty$
algebras and the $A_\infty$ morphism.

\subsection{Map types}\label{subsection51}

The combinatorial type of a ``treed vortex'' is defined now. Recall from Section \ref{subsection45} that we have chosen a Morse function $F_L: L \to {\mb R}$ which has a unique maximum ${\bm x}_M$ (since $L$ is connected).  For ${\bm x} \in {\rm Crit} F$ we assign the degree
\beqn 
|{\bm x}| = {\rm dim} L - {\rm Morse\ index\ of\ } {\bm
  x} \in {\mb Z} \cap [0,\dim L]  \eeqn 
On the other hand, consider the commutative diagram 
\beqn
\vcenter{ \xymatrix{ H_2(X ) \ar[r] \ar[d] & H_2^K(V ) \ar[d]\\
    H_2( X, L ) \ar[r] & H_2^K ( V, L_V)} } \eeqn 
with horizontal maps given by pullback.  We use $B$ to denote an
element of $H_2^K(V, L_V)$. We say that a holomorphic sphere in $X$
(resp. a holomorphic disk in $X$ resp. an affine vortex over ${\bm C}$
resp. an affine vortex over ${\bm H}$) {\it represents} $B$ if its
homology class in $H_2(X)$ (resp.  $H_2( X, L)$ resp. $H_2^K(V)$
resp. $H_2^K(V, L_V)$) is mapped to $B$.  The Chern number
of an absolute class 
\beqn 
c_1(B) \in {\mb Z}  
\eeqn
is one half of the Maslov index of a relative
class.

We label constraints at interior markings by the following kind of
data. Let
\beqn V^\us = \bigsqcup_{l=1}^m V_l^{\us}  \eeqn 
be a stratification of the unstable locus $V^\us$ by smooth complex submanifolds. 
Define
\beqn {\mc O}:= {\mc O}^{\rm stable} \sqcup {\mc O}^{\rm unstable}
\sqcup {\mc O}^{\rm tangential} \eeqn 
where \beqn {\mc O}^{\rm
  unstable}:= \{ V_1^{\rm us}, \ldots, V_m^{\rm us} \}, \eeqn 
and
\beqn {\mc O}^{\rm stable}:= \{ D_T^{\rm st}\ |\ T \subset \{D_1,
\ldots, D_h \} \} \cup \{ D_a^{\rm st} \cap S_i\ |\ D_a \in \{D_1,
\ldots, D_h\}, S_i \in \{S_1, \ldots, S_N\}\}, \eeqn 
where
\begin{itemize}
    \item for every subset $T \subset \{D_1, \ldots, D_h\}$,
\beqn
D_T = \bigcap_{D_a\in T} D_a.
\eeqn
\item the divisors
$S_1, \ldots, S_N$ are those in the condition (S3) in Definition \ref{defn12}), \end{itemize}
and
\beqn
{\mc O}^{\rm tangential}:= \{ D_a^m\ |\ D_a \in \{D_1, \ldots, D_h\}, m \geq 2 \}.
\eeqn
An element of ${\mc O}$ is called a {\it  constraint}; it is called a stable, unstable, or tangential if it lies in the corresponding subset of ${\mc O}$. Each constraint $O \in {\mc O}$ also defines a smooth $G$-invariant submanifold $V_O \subset V$ which descends to a smooth (possibly empty) submanifold $\bar V_O \subset X$. There is a particular constraint
\beqn
1 \in {\mc O}^{\rm st}
\eeqn
which corresponds to the open submanifold $V^{\rm st}$. For $O \in {\mc O}$, define its {\it degree} as
\beq \label{condeg}
\delta_O = \left\{\begin{array}{cc} {\rm codim} V_O,\ &\ O \notin {\mc O}^{\rm tangential};\\
                                    2m,\ & \ O = D_a^m \in {\mc O}^{\rm tangential}. \end{array}  \right.
\eeq

\begin{defn}\label{defn51}(Map types)

\begin{enumerate}

\item An unbroken {\it map type} is a tuple ${\bm \Gamma} = (\Gamma, {\bf B}, {\bf O}, {\bf X})$ where
\begin{itemize}
\item $\Gamma \in {\bf T}$ is an unbroken domain type. 

\item ${\bf B}: V_\Gamma \to H_2^K(V, L_V; {\mb Z})$ is a collection of curve classes $B_\alpha$.

\item ${\bf O}: L_\Gamma \to {\mc O}$ is a collection of constraints. 

\item ${\bf X}: T_\Gamma \to {\rm Crit} F_L$ is a sequence of elements in the set of critical points of $F$. It is usually denoted as ${\bf X} = ({\bm x}_1, \ldots, {\bm x}_k; {\bm x}_\infty)$ if $\Gamma$ has $k$ inputs. 

\end{itemize}
The notion of unbroken map types can be easily extended to the notion of {\it  broken} map types. We skip the details. Every map type ${\bm \Gamma}$ has an underlying domain type $\Gamma$ and the correspondence ${{\bm \Gamma}} \to \Gamma$ will be used frequently without explicit explanation. The map type ${\bm \Gamma}$ is then called a {\it   refinement} of $\Gamma$. 

\item The total energy $\omega ({\bm \Gamma})$ and the total Chern number $c_1 ({\bm \Gamma})$ of $\bm \Gamma$ is defined as 
\begin{align*}
&\ \omega({{\bm \Gamma}}):=\sum_{v_\alpha \in V_\Gamma} \omega (B_\alpha),\ &\ c_1({{\bm \Gamma}}): =  \sum_{v_\alpha \in V_\Gamma} c_1(B_\alpha).%
\end{align*}

\item Define 
\beqn
|{\bf X}| = |{\bm x}_\infty| - \sum_{t_j \in T_\Gamma} |{\bm x}_j|.
\eeqn

\item A map type ${\bm \Gamma}$ is {\it  stable} if 
\begin{itemize}
\item for every unstable $v_\alpha \in V_\Gamma$, $B_\alpha \neq 0$;

\item for every infinite edge in $\Gamma$, the labeling critical points at the two ends are different. 
\end{itemize}

\item For two map types ${{\bm \Gamma}}$, ${{\bm \Gamma}}'$, denote ${{\bm \Gamma}}' \leq {{\bm \Gamma}}$ if the following conditions are satisfied:
\begin{itemize}
\item $\Gamma' \leq \Gamma$, which implies a tree map $\rho: \Gamma' \to \Gamma$. 

\item For any boundary tail $t_{\uds i}' \in  T_{\Gamma'}$ and $t_{\uds i} = \rho_T (t_{\uds i}) \in T_\Gamma$, ${\bm x}_{\uds i}' = {\bm x}_{\uds i}$.

\item For any vertex $v_\alpha \in V_\Gamma$, we require

\beqn
B_\alpha = \sum_{v_\beta' \in \rho_V^{-1}(v_\alpha)} B_\beta'.
\eeqn

\item For $l_j' \in L_{\Gamma'}$ and $l_i = \rho_L ( l_j') \in L_\Gamma$ with corresponding constraints $O_j'$ and $O_i$, if $O_j' \in {\mc O}^{\rm unstable}$, then $V_{O_j'} \cap \ov{V_{O_j}} \neq \emptyset$; otherwise $O_j' = O_i$.
\end{itemize}
\end{enumerate}

\end{defn}

\subsection{Treed vortices}\label{subsection52}

A treed vortex is a combination of disks, spheres, and vortices, sometimes with Lagrangian boundary condition, pseudoholomorphic with
respect to a given collection of domain-dependent almost complex
structures, together with gradient segments of the given
domain-dependent Morse function. Given $\Gamma \in {\bf T}$ and a
perturbation
$P_{\Gamma^{\rm st}} = (J_{\Gamma^{\rm st}}, F_{\Gamma^{\rm st}}) \in
{\mc P}_{\Gamma^{\rm st}}$
where $\Gamma^{\rm st}$ is the stabilization of $\Gamma$, for each
treed disk ${\mc C}$ of domain type $\Gamma$, there are an induced family of
almost complex structures $J_{{\mc C}}: {\mc C} \to {\mc J}_D$ and an
induced family of domain-dependent smooth functions
$F_{{\mc C}}: {\mc C} \to {\mc F}$.
Then for each
$v_\alpha \in V_\Gamma$, $J_{\mc C}$ induces a map $J_\alpha: \Sigma_\alpha \to {\mc J}_D$; for each
$e \in T_\Gamma \cup E_{\uds\Gamma}^+$, ${\mc F}_C$ induces a map $F_e: I_e \times L \to {\mb R}$.

\begin{defn}\label{defn52}(Treed scaled vortex)
Given $\Gamma$ and $P_{\Gamma^{\rm st}}$ as above.

\begin{enumerate}

\item A {\it  treed scaled vortex} of domain type $\Gamma$ is a collection
\beqn
{\mc V}:= \big\{ {\mc C}, ({\bm v}_\alpha;v_\alpha \in V_\Gamma), (x_e; e \in E_{\uds\Gamma}^+ \cup T_\Gamma) \big\} 
\eeqn
where 
\begin{itemize}
\item ${\mc C}$ is a treed disk with domain type $\Gamma$;

\item For $v_\alpha \in V_\Gamma^0$, ${\bm v}_\alpha$ is a $K$-orbit of $J_\alpha$-holomorphic maps from $\Sigma_\alpha$ to $V$ with boundary mapped into $L_V$;

\item For $v_\alpha \in V_\Gamma^1$, ${\bm v}_\alpha$ is a gauge equivalence class of affine vortices over $\Sigma_\alpha$;

\item For $v_\alpha \in V_\Gamma^\infty$, ${\bm v}_\alpha = u_\alpha$ is a $I_\alpha$-holomorphic map from $\Sigma_\alpha$ to $X$ with boundary mapped into $L$. Here $I_\alpha$ is the domain-dependent almost complex structure on $X$ induced from $J_\alpha$;
\item For $e \in E_{\uds\Gamma}^+ \cup T_\Gamma$, $x_e: I_e\to L$ is a solution to the gradient flow equation 
\beqn
x_e'(t) + \nabla F_e(t) (x_e(t)) = 0;
\eeqn
\end{itemize}
such that the following conditions are satisfied:
\begin{itemize}

\item There are obvious matching conditions at nodes and points where one-dimensional components are attached to two-dimensional components. We omit the details.

\item When $t_-, t_+\in T_\Gamma$ form an infinite edge, $x_{t_-} = x_{t_+}$ as maps.

\end{itemize}

\item Let ${\bm \Gamma}$ be a map type refining $\Gamma$. We say that a treed scaled vortex ${\mc V}$ defined as above is of map type ${\bm \Gamma}$ if the following conditions are satisfied.
\begin{itemize}

\item For each $v_\alpha \in V_\Gamma$, ${\bm v}_\alpha$ represents the curve class $B_\alpha$;

\item For each $e \in T_\Gamma$, $x_e(t)$ converges to the prescribed critical point;

\item (Contact order condition) For each leaf $l_i \in L_\Gamma$ corresponding to the marking $z_i \in \Sigma_\alpha$ in ${\mc C}$, if $v_\alpha \in V_\Gamma^0 \cup V_\Gamma^1$, then the image of ${\bm v}_\alpha$ at $z_i$ is contained in $V_{O_i}$; if $v_\alpha \in V_\Gamma^\infty$, then the image of ${\bm v}_\alpha$ at $z_i$ is contained in $\bar V_{O_i}$. Moreover, if ${\bm v}_\alpha$ is not contained in $D_a$ and the constraint $O_i = D_a^{m_i}$ is tangential, then the local intersection between ${\bm v}_\alpha$ with $D_a$ (or $\bar D_a$) is at least $m_i$. 
\end{itemize}

\item An isomorphism between treed scaled vortices is a domain
  isomorphism intertwining the maps and vortices up to gauge
  symmetry. Given a map type ${\bm \Gamma}$, let
  ${\mc M}_{{\bm \Gamma}}( P_{\Gamma^{\rm st}})$ be the set of
  isomorphism classes of treed scaled vortices of map type ${\bm \Gamma}$.
  In the next section we equip this set with a topology and call it
  the moduli space of treed vortices.

\end{enumerate}
\end{defn}

\subsection{Convergence and compactness}\label{subsection53}
 
We describe a notion of Gromov convergence for sequences of vortices. We first consider the convergence of affine vortices over ${\bm C}$ defined by Ziltener \cite{Ziltener_book}. Let $(\Gamma, {\mf s})$ be the scaled base-free tree of scale $1$ and no leaves. Such map types model stable affine vortices with no markings. Let
\beqn
{\mc V}_\infty = \big\{ {\mc C}_\infty, ( {\bm v}_{\infty, \alpha}; v_\alpha \in V_\Gamma) \big\}
\eeqn
be a stable affine vortex of domain type $\Gamma$ as defined in Definition \ref{defn52}. Recall that each vertex $v_\alpha \in V_\Gamma$ gives a surface $\Sigma_\alpha$ biholomorphic to a copy of ${\bm C}$, and each edge $e_{\alpha \beta}$ corresponds to a point $w_{\alpha\beta} \in \Sigma_\beta$ (a node); the set of nodes attached to $\Sigma_\beta$ is denoted by $W_\beta$. For simplicity we also fix the almost complex structure to be $J_V$. 

\begin{defn} {\rm (Gromov convergence of vortices)\cite{Ziltener_book}}  
Let ${\bm v}_k = (u_k, \phi_k, \psi_k)$ be a sequence of affine vortices over ${\bm C}$. We say that ${\bm v}_k$ {\it converges} to ${\mc V}_\infty$ if there exist sequences of M\"obius transformations $\varphi_{k,\alpha}: \Sigma_\alpha \to {\bm C}$ (for all $v_\alpha \in V_\Gamma$) and sequences of gauge transformations $\zeta_{k,\alpha}: \Sigma_\alpha \to K$ (for all $v_\alpha \in V_\Gamma^0 \cup V_\Gamma^1$) satisfying the following.
\begin{enumerate}

\item For $v_\alpha \in V_\Pi^0$, suppose the $K$-orbit of holomorphic maps ${\bm v}_{\infty, \alpha}$ is represented by a holomorphic map $u_{\infty, \alpha}: \Sigma_\alpha \to V$. We require that $\zeta_{k,\alpha}^* {\bm v}_k \circ \varphi_{k,\alpha}$ converges in c.c.t. (see Subsection \ref{subsection22}) to $(u_{\infty, \alpha}, 0, 0)$.

\item For $v_\alpha \in V_\Gamma^1$, $\varphi_{k, \alpha}$ are translations of the complex plane ${\bm C}$ and $\zeta_{k,\alpha}^* {\bm v}_k \circ \varphi_{k,\alpha}$ converges in c.c.t. to ${\bm v}_{\infty, \alpha}$. 

\item For $v_\alpha \in V_\Pi^\infty$, we require that
  $\mu( u_k \circ \varphi_{k,\alpha})$ converges to zero uniformly on any
  compact subset $Q \subset \Sigma_\alpha \setminus W_\alpha$.
  Consider the sequence of maps $\pi \circ u_k \circ \varphi_{k,\alpha}$
  from $Q$ to $X$ where $\pi: V^{\rm st} \to X$ is the projection. We
  require further that $\pi \circ u_k \circ \varphi_{k,\alpha}$ converges
  uniformly to $u_{\infty, \alpha}$ over all $Q$.

\item For each edge $e_{\alpha\beta} \in E_\Gamma$, the sequence of M\"obius transformations $\varphi_{\beta, k}^{-1} \circ \varphi_{\alpha, k}$ converges to the constant $w_{\alpha\beta}\in \Sigma_\beta$.

\item $E({\bm v}_k)$ converges to $E({\bm v}_\infty)$. 
\end{enumerate}
\end{defn} 

One can generalize the above notion to include markings or with
boundary conditions and define convergence for 1) affine vortices with
markings; 2) affine vortices over ${\bm H}$ with both interior and
boundary markings. We omit the details.

We expand the definition to a notion of convergence for unbroken
stable treed scaled vortices.  Let $\Gamma\in {\bf T}^{\rm st}$ be an
unbroken stable domain type and $P_\Gamma \in {\mc P}_\Gamma$
be a perturbation datum on $\ov{\mc U}_\Gamma$. Let $\Pi \leq \Gamma$
be another (not necessarily stable) domain type. Then
$P_\Gamma$ induces a domain-dependent perturbation $P_{\mc C}$ for all
treed disks ${\mc C}$ of domain type $\Pi$. Let \beqn {\mc V}_k = \big\{ {\mc
  C}_k, ({\bm v}_{k,\alpha}; v_\alpha \in V_\Gamma), (x_{k,e}; e \in
E_{\uds\Gamma}^+ \cup T_\Gamma) \big\},\ k = 1, 2, \ldots \eeqn be a
sequence of stable treed scaled vortices of domain type $\Gamma$. Let \beqn
{\mc V}_\infty = \big\{ {\mc C}_\infty, ({\bm v}_{\infty,\alpha};
v_\alpha \in V_\Pi), (x_{\infty,e}; e\in E_{\uds\Pi}^+ \cup T_\Pi)
\big\} \eeqn be a stable treed scaled vortex of domain type $\Pi$ for the perturbation $P_{\mc C}$.

\begin{defn}\label{defn54}
We say that ${\mc V}_k$ converges to ${\mc V}_\infty$ if the following are satisfied.

\begin{enumerate}

\item Let ${\mc C}_\infty^{\rm st}$ be the stabilization of ${\mc C}_\infty$. Then we require that ${\mc C}_k$ converges to ${\mc C}_\infty^{\rm st}$ in the sense of Definition \ref{defn310}.

\item Let $\rho_V: V_\Pi \to V_\Gamma$ be the map induced from the relation $\Pi \leq \Gamma$. For each $v_\alpha \in  V_\Gamma$, consider the sequence of gauged maps ${\bm v}_{k,\alpha}$ with domain $\Sigma_\alpha$. The preimage of $v_\alpha$ under $\rho_V$ corresponds to a subtree $\Pi_\alpha$ of $\Pi$, which also corresponds to a stable affine vortex ${\mc V}_{\infty, \alpha}$ of domain type $\Pi_\alpha$. Then we require that the sequence ${\bm v}_{k, \alpha}$ converges to ${\mc V}_{\infty, \alpha}$, in the sense we discussed before this definition. 

\item Let $\rho_E: E_{\uds \Pi} \cup T_\Pi \to E_{\uds \Gamma} \to T_\Gamma$ be the map induced from the relation $\Pi \leq \Gamma$. For an edge of positive length or boundary tail $e \in E_{\uds \Gamma}^+ \sqcup T_\Gamma$, its preimage under $\rho_E$ corresponds to a (chain of) edges in $\Pi$, which in ${\mc V}_\infty$ corresponds to a (broken) perturbed gradient line (with possibly finite ends). If the lengths of $e$ in ${\mc C}_k$ do not converge to zero, then the sequence of perturbed gradient lines $x_{k,e}$ converges uniformly in all derivatives to the (broken) perturbed gradient line in ${\mc V}_\infty$. If the lengths of $e$ converge to zero, then there is no extra requirement for the sequence of perturbed gradient lines $x_{k,e}$. 
\end{enumerate}
\end{defn}

The above definition extends in a straightforward way to the case that the domain type $\Gamma$ of the sequence is broken. The uniqueness of the sequential limit up to isomorphism can be proved in a similar manner as the case of pseudoholomorphic maps. We omit the details.

\begin{prop}[Uniqueness of sequential limit up to isomorphism] The sequential limits defined as in Definition \ref{defn54} are unique up to isomorphism.
\end{prop}

The tangency order of the limiting objects to the stabilizing divisor may increase but not decrease. Hence if one fixes the tangency orders of all elements of a sequence, the limit has the same tangency order. Other kinds of constraints are also preserved. 

\begin{prop}[Tangency orders and constraints are preserved under convergence]
If in Definition \ref{defn54} ${\mc V}_k \in {\mc M}_{\bm\Gamma}(P_\Gamma)$ for a common map type ${\bm \Gamma}$ for all $k$, then there is a unique refinement ${\bm \Pi}$ of $\Pi$ such that ${\bm \Pi} \leq {\bm \Gamma}$ and ${\mc V}_\infty \in {\mc M}_{\bm\Pi}(P_{\Pi^{\rm st}})$.
\end{prop}

The sequential compactness theorem for treed vortices we will need is
the following.

\begin{thm}\label{thm57} {\it (Compactness of fixed domain type)} Let ${\mc V}_k$ be a sequence of stable treed scaled vortices with isomorphic stable underlying domain type $\Gamma \in {\bf T}^{\rm st}$, while the equations are defined by a common perturbation $P_{\Gamma} \in {\mc P}_\Gamma$. Suppose the energy $E ({\mc V}_k)$ are uniformly bounded. Then there is a subsequence of ${\mc V}_k$ converging to a stable treed scaled vortex ${\mc V}_\infty$.
\end{thm}

 We do not provide the detailed proof, which is a combination of results on compactness for broken Morse trajectories, Gromov compactness for holomorphic spheres and disks, and affine vortices due to Ziltener in \cite{Ziltener_book}; the compactness result of affine vortices over ${\bm H}$ is essentially proved by Wang--Xu in \cite{Wang_Xu} (in Wang--Xu \cite{Wang_Xu} the authors actually 
proved the compactness of vortices over the unit disk under adiabatic  limit, and the case of vortices over ${\bm H}$ can be reproduced from the same argument). 

The following notation will be used for moduli spaces of vortices.   Consider a stable domain type $\Gamma$
and a refinement ${{\bm \Gamma}}$. Consider the moduli space
${\mc M}_{{\bm \Gamma}}(P_{\Gamma})$. Define %
\beq\label{eqn52}
\ov{{\mc M}_{{\bm \Gamma} }(P_{\Gamma})}:= \bigsqcup_{{\bm \Pi} \leq
  {{\bm \Gamma}}} {\mc M}_{\bm \Pi} (P_{\Pi^{\rm st}}).  \eeq 
Here $P_{\Pi^{\rm st}}$ is induced from $P_\Gamma$ by restricting to lower
stratum of the universal curve. Notice that there might be stable
refinements ${{\bm \Pi}}$ appearing in the disjoint union on the right
hand side whose underlying $\Pi$ is unstable. 

To show the sequential compactness of $\ov{ {\mc M}_{{\bm \Gamma} }(P_{\Gamma})}$, first notice that there are only finitely many
${\bm \Pi}$ appearing in the disjoint union on the right hand side of
\eqref{eqn52}. Second, given a sequence of elements in
${\mc M}_{\bm \Pi}( P_{\Pi^{\rm st}})$ with $\Pi$ unstable, we can add
a fixed number of markings to stabilize all unstable components of
$\Pi$. Then Theorem \ref{thm57} implies that a subsequence of the
stabilized sequence converge to a limit. It is routine to show that
after removing the added marked points, the convergence still
holds. Then the notion of sequential convergence puts a compact
Hausdorff topology on the moduli space \eqref{eqn52} as for the case
of holomorphic curves in \cite{McDuff_Salamon_2004}.

\begin{thm}\label{thm58}
Let $\Gamma$ be a stable domain type and ${\bm \Gamma}$ be a refinement. Let $P_{\Gamma}\in {\mc P}_\Gamma$ be a perturbation data on $\ov{\mc U}_{\Gamma}$. Then the moduli space $\ov{ {\mc M}_{{\bm \Gamma} }(P_{\Gamma})}$ has a unique compact and Hausdorff topology for which sequential convergence is the same as the sequential convergence defined by Definition \ref{defn54}.
\end{thm}

\section{Transversality and orientation}\label{section6}

In this section we regularize various moduli spaces and equip them with orientations. We construct a coherent collection of perturbation data which make relevant moduli spaces regular and which imply good compactifications of moduli spaces of expected dimensions at most one. We also equip these moduli spaces with coherent orientations so that zero-dimensional moduli spaces have well-defined counts and one-dimensional moduli spaces can be glued together along fake boundaries.

\begin{defn}\label{controlled}
Let $\Gamma$ be a domain type. 
\begin{enumerate}
\item A {\it bouquet} of $\Gamma$ is a maximal subtree $\Upsilon$ with  
  root in the set $V_{\Gamma^{\rm sup}}^1$ of vertices corresponding  
  to affine vortices over ${\bm C}$.  
\item ${\bm \Gamma}$ is {\it controlled} if for each bouquet ${\bm  
  \Upsilon}$, we have the inequality 
\beq\label{eqn61}
\omega({\bm \Upsilon}) \leq {\rm deg}^{\rm max} \Upsilon.  
\eeq  
(See \eqref{eqn31} for the definition of the notation ${\rm deg}^{\rm max}\Upsilon$.)  
\end{enumerate}
\end{defn}

\begin{defn}\label{uncrowded}
Let ${{\bm \Gamma}} = (\Gamma, {\bf B}, {\bf O}, {\bf X}) $ be a map type (see Definition \ref{defn51}).
\begin{enumerate}
\item A {\it ghost vertex} of ${{\bm \Gamma}}$ is a vertex $v_\alpha$ with $B_\alpha = 0$. A {\it ghost tree} is a subtree whose vertices are all ghost vertices.
\item A map type ${{\bm \Gamma}}$ is called {\it uncrowded} if for each ghost tree $\Xi$, there is either no leaves attached to $\Xi$, or one leaf whose constraint is nontangential. Otherwise ${\bm \Gamma}$ is called {\it crowded}.
\end{enumerate}
\end{defn}

\begin{defn}\label{regularity} {\it (Regularity and strong regularity)} Let $\Gamma$ be a stable domain type and $P_\Gamma \in {\mc P}_\Gamma$ be a perturbation. 

\begin{enumerate}
\item Assume that $V_{\Gamma^{\rm sup}}^0 = \emptyset$. We say that configuration ${\mc V} \in {\mc M}_{\bm \Gamma}^*(P_\Gamma)$ is {\it regular} if for some $W_i \subset {\mc W}_\Gamma$ with ${\mc V} \in {\mc M}_{\bm \Gamma}^*(P_\Gamma, W_i)$, the linearization of the section at ${\mc V}$ (defined later in \eqref{eqn64}) is surjective onto the fibre 
of ${\mc E}_{\bm \Gamma}(W_i)$. This condition is independent of the 
choice of $W_i$. We say that $P_\Gamma$ is regular if for all 
controlled and uncrowded refinements ${\bm \Gamma}$ of $\Gamma$, every 
element of ${\mc M}_{\bm \Gamma}^*(P_\Gamma)$ is regular. 
\item To treat transversality for crowded map types we need the following concept. Let $W \subset V_{\Gamma^{\rm sup}}$ be a subset containing  
  $V_{\Gamma^{\rm sup}}^0$ and $\Gamma_W$ the stable domain type obtained by  applying the forgetful construction described in Subsection  \ref{subsection42}. The set $W$ is supposed to be the set of ghost spherical components. Let $P_{\Gamma_W}$ be the induced perturbation on $\ov{\mc U}_{\Gamma_W}$ (see \eqref{eqn44}). We say that $P_\Gamma$ is {\it  $W$-regular} (or  
  $P_{\Gamma_W}$ is $\bar W$-regular, with notation as in Definition  
  \ref{defn42} ) if for all uncrowded and controlled refinement  
  ${\bm \Gamma}_W$ of $\Gamma_W$ with $B_v = 0$ for all  
  $v \in \bar W$, every element in  
  ${\mc M}_{{\bm \Gamma}_W}^* (P_{\Gamma_W})$ is regular.  
\item $P_\Gamma$ is called {\it  strongly regular} if it is $W$-regular for all subsets $W \subset V_{\Gamma^{\rm sup}}$ containing $V_{\Gamma^{\rm sup}}^0$.
\end{enumerate}
\end{defn}

The main result of this section is the existence of coherent systems of strongly regular perturbations:

\begin{thm}\label{thm64}
There exists a coherent system of perturbations (see Definition \ref{coherent})
\beqn
{\bm P} = ( P_\Gamma)_{\Gamma \in {\bf T}^{\rm st}},\ P_\Gamma \in {\mc P}_\Gamma
\eeqn
such that each $P_\Gamma$ is strongly regular (see Definition \ref{regularity}).
\end{thm}

\begin{rem}
  For each two such collections, ${\bm P}_\Gamma$ and
  ${\bm P}_\Gamma'$, one considers homotopies between them, i.e., for
  each $\Gamma$, a family of perturbations $P_{\Gamma}(t)$
  parametrized by $t \in [0,1]$. One can define a similar notion of
  strong regularity for families. One can also interpolate between
  choices of nodal neighborhoods, etc. The definition of strong
  regularity for families and the proof of the fact that a generic
  homotopy is strongly regular are left to the reader. Using regular
  homotopies one can prove that the homotopy classes of the $A_\infty$
  algebras and $A_\infty$ morphisms we will construct are independent
  of various choices.
\end{rem}

\subsection{Local model with tangency conditions}

The proof of Theorem \ref{thm64} is a Sard--Smale argument applied to a universal moduli space involving tangential constraints at the markings. In \cite{Cieliebak_Mohnke}, one can use higher Sobolev space $W^{k,p}$ to model maps and consider the $l$-th order derivative of any maps of regularity $W^{k, p}$, thanks to the Sobolev embedding. However if we would like to use higher Sobolev spaces to model vortices, we need to extend the result of the second name author and Venugopalan \cite{Venugopalan_Xu} which is about the local model with $W^{1,p}$-Sobolev norms only. Instead, we still use the $W^{1,p}$ norm and use the twisting trick which appeared in \cite{Zinger_2011}. We thank A. Zinger for explaining this method to us.

We define a Banach space of sections that vanish to a certain order at a given point as follows.  Recall $p\in (2, 4)$. Let $\Sigma$ be a Riemann surface, and $E \to \Sigma$ be a complex line bundle equipped with a metric and a metric connection. Let $Z = o_1 z_1 + \cdots + o_m z_m$ be an effective divisor on $\Sigma$ with multiplicities $o_i \geq 1$ for all $i$. Let $k \geq 0$. Denote by $W_Z^{k,p}(\Sigma, E) \subset W^{k,p}_{\rm loc} (\Sigma \setminus Z, E)$ the space of sections $s$ of $E$ such that for every $z_i \in Z$ there
exists a neighborhood $U_i$ of $z_i$ and a coordinate $w: U_i \to {\mb C}$ such that
\begin{align*}
&\ w(z_i) = 0,\ &\  \frac{\nabla^l s}{ w^{k - l + o_i-1}} \in L^p (U_i, E),\ \forall l = 0, \ldots, k. 
\end{align*}
In other words, $s$ is of class $W^{k, p, o_i - \frac{2}{p}}$ with respect to the cylindrical metric near $Z$. It is straightforward to check that if $s$ is holomorphic near $Z$ with respect to certain local holomorphic structure of $E$ and $s(w) = O(|w|^{o_i})$ near all $z_i$, then $s$ lies in the space $W^{k, p}_Z(\Sigma, E)$. Moreover, the following lemma is standard. 
\begin{lemma}\label{lemma66}
$W_Z^{k,p}(\Sigma, E)$ is a Banach space. Moreover, if $\Sigma$ is a compact Riemann surface without boundary, $D: \Omega^0(E) \to \Omega^{0,1}(E)$ is a real linear Cauchy--Riemann operator, then for all $k \geq 0$, $D$ induces a Fredholm operator 
\beqn
D: W_Z^{k+1, p}(\Sigma, E) \to W_Z^{k, p}( \Sigma, \Lambda^{0,1}\otimes E),
\eeqn
whose index is 
\beqn
{\rm index} D = 2 - 2g + 2 {\rm deg} E - 2 {\rm deg} Z.
\eeqn
\end{lemma}

Now we use the above Sobolev space to model vortices or holomorphic
curves with tangency conditions. We first introduce the following
notations.  Let
$\lambda \in \{0, 1, \infty\}$ and $\Sigma$ be a Riemann surface. We
only consider the following situations in this paper.  \beqn (\lambda,
\Sigma) \in \{ (0, {\bm D}^2), (1, {\bm H}), (1, {\bm C}), (\infty,
{\bm D}^2), (\infty, {\bm S}^2) \}.  \eeqn Let $J$ be a $K$-invariant
almost complex structure on $V$, possibly parametrized by points on
$\Sigma$. Let $B$ be a curve class in $H_2^K(V, L_V; {\bf
  Z})$.
Consider Fredholm triples, i.e., triples
$ ({\mc B}, {\mc E}, {\mc F}) = ({\mc B}_\Sigma^\lambda(B), {\mc
  E}_\Sigma^\lambda(B), {\mc F}_\Sigma^\lambda(B) )$
of Banach manifolds, Banach space bundles, and Fredholm sections
specified as follows.

\begin{enumerate}

\item If $(\lambda, \Sigma) = (0, {\bm D}^2)$, consider the Banach manifold $\tilde {\mc B}$ of $W^{1,p}$-maps from ${\bm D}^2$ to $V$ with boundary mapped into $L_V$ representing the class $B$. There is a Banach space bundle $\tilde {\mc E}$ whose fibre over an element of this Banach manifold $u: {\bm D}^2  \to V$ is the space $L^p(\Sigma, u^* TV)$. For any $K$-invariant almost complex structure $J$, possibly parametrized by points on ${\bm D}^2$, the $J$-holomorphic curve equation defines a section $\tilde {\mc F}: \tilde {\mc B} \to \tilde {\mc E}$. Notice that all geometric data are $K$-invariant and there is a free $K$-action on $\tilde {\mc B}$. Then after quotienting by the $K$-action, one obtains a Fredholm triple $({\mc B}, {\mc E}, {\mc F})$. 

\item If $(\lambda, \Sigma) = (1, {\bm A})$ with ${\bm A} = {\bm C}$ or ${\bm H}$, then ${\mc B}$, ${\mc E}$, and ${\mc F}$ are the objects ${\mc B}_{\bm A}^{\rm vor}(B)$, ${\mc E}_{\bm A}^{\rm vor}(B)$, and ${\mc F}_J$ described in Theorem \ref{localmodel}.

\item If $(\lambda, \Sigma) = (\infty, {\bm D}^2)$ or $(\infty, {\bm S}^2)$, then ${\mc B}$ is the space of $W^{1,p}$-maps from $\Sigma$ to $X$ with $u(\partial \Sigma) \subset L$ representing the class $B$, and ${\mc E}$ is the bundle whose fibre over a map $u$ is $L^p(\Sigma, u^* TX)$. The section ${\mc F}$ is defined by 
Cauchy-Riemann operator with respect to the almost complex structure on $X$ induced from $J$. 

\end{enumerate}

We add tangency conditions on the level of Banach
manifolds. We will only do this in detail for
$(\lambda, \Sigma) = (1, {\bm C})$ as other cases are either similar
or simpler. Recall $h$ is the number of irreducible components of the
stabilizing divisor $D$ in \eqref{Dh}, and let $Z_1, \ldots, Z_h$
be disjoint effective divisors on ${\bm C}$ and denote
$Z = (Z_1, \ldots, Z_h)$.  Let
\beqn
{\mc B}_{{\bm C},Z}^1(B) \subset {\mc B}_{\bm C}^1(B) \eeqn 
be the subset of (gauge equivalence classes of) gauged maps
${\bm v} = (u, \phi, \psi)$ satisfying: 
\[ u(Z_1 \cup \cdots \cup Z_h) \subset V^{\rm st} \]
 and if
$D_a \subset V$ is the vanishing locus of
$s_a \in H^0( \tilde V{}^{k_a})$, then 
\beqn u^* s_a \in W_{Z_a}^{1,p}
( U_a, u^* \tilde V{}^{k_a} ) \eeqn 
where $U_a \subset {\bm C}$ is a bounded open set containing
$Z_a$. Define the Banach space bundle
${\mc E}^1_{{\bm C},Z}(B) \to {\mc B}^1_{{\bm C},Z}(B)$ whose fibre
over ${\bm v}$ is the space 
\beqn \Big\{ {\bm \varsigma} = (\varsigma,
\varsigma') \in {\mc E}^1(B)|_{\bm v} \ \left|\ d s_a \circ \varsigma
  \in W_{Z_a}^{0,p} ( U_a, u^* \tilde V{}^{k_a} ),\ 0 \leq a \leq h
\right. \Big\}.  \eeqn 
Given $J \in {\mc J}_D$, consider the section \beqn {\mc F}_J: {\mc
  B}_{\bm C}^1(B) \to {\mc E}_{\bm C}^1(B) \eeqn induced by the affine
vortex equation associated to $J$. We claim that for each
${\bm v} = (u, \phi, \psi) \in {\mc B}^1_{{\bm C},Z} (B)$,
${\mc F}_J ({\bm v}) \in {\mc E}^1_{{\bm C},Z}(B)$. Indeed, since  $s_a$ is holomorphic with respect to $J_V$, we see
\begin{multline*}
ds_a \circ ( \partial_s u + {\mc X}_\phi + J ( \partial_t u + {\mc X}_\psi) ) = d s_a \circ (\partial_s u + J \partial_t u) \\
= d s_a \circ ( \partial_s u + J_V \partial_t u)+ d s_a \circ ( J - J_V) \partial_t u = \ov\partial  (u^* s_a) + d s_a \circ ( J - J_V) \partial_t u.
\end{multline*}
By the condition $u^* s_a \in W^{1,p}_{Z_a} ( U, u^* \tilde V{}^{k_a} )$, the first term $\ov\partial( u^* s_a)$ is in the space $W_{Z_a}^{0, p}$. To estimate the second term $d s_a \circ ( J - J_V) \partial_t u$, let $z$ be a local holomorphic coordinate on the domain centered at the marking. Because near the marking $u$ approaches to $D_a$ in a certain rate $|z|^o$ and $J = J_V$ along $D_a$ (see \eqref{eqn211}), one can see that $d s_a \circ ( J - J_V) \partial_t u \in W_{Z_a}^{0, p}$. Then we have a smooth Fredhom section
\beqn 
{\mc F}_{{\bm C}, J}^1: {\mc B}_{{\bm C},Z}^1(B) \to {\mc E}_{{\bm C},Z}^1(B)
\eeqn
of Banach space bundles. Lemma \ref{lemma66} implies
that its index is equal to 
\beq \label{eqn62} {\rm index} {\mc F}_{{\bm C},J}^1 =
{\rm dim} X + 2 c_1(B) - 2 \sum_{a=1}^h {\rm deg} Z_a.  \eeq
Moreover, every point in the zero locus is an affine vortex which has
the prescribed tangency conditions at the interior markings.

%
%

\subsection{Definition of regularity}\label{subsection62}

\subsubsection{Fredholm problem setup for singular domains}

We set up the Fredholm problem for a general moduli space of
affine vortices.  We
impose the following assumption in effect for this section only:

\vspace{0.2cm}

\noindent {\sc Running Assumption.} $\Gamma$ is stable and the set of
vertices 
$V_{\Gamma^{\rm sup}}^0$ corresponding to holomorphic spheres in $V$
is empty. 

\vspace{0.2cm}

We define a Banach vector bundle and section so that the moduli
  space is locally homeomorphic to the zero set. Let ${\bm \Gamma}$
be a refinement of $\Gamma$ and ${\mc C} = {\mc C}_p$ be a
stable\footnote{We only need to consider the Fredholm theory for
  stable domains.} treed disk representing a point
$p \in {\mc W}_{\Gamma}$. For each vertex $v_\alpha \in V_\Gamma$
whose scale is $\lambda \in \{0, 1, \infty\}$, one has a pair
$(\lambda_\alpha, \Sigma_\alpha)$ belonging to the categories just
discussed. Further, all tangential constraints at markings on
$\Sigma_\alpha$ defines an $h$-tuple of effective divisors
$Z_\alpha = (Z_{1, \alpha}, \ldots, Z_{h, \alpha})$. Then we define
the Fredholm triple \beqn ({\mc B}_\alpha, {\mc E}_\alpha, {\mc
  F}_\alpha) \eeqn where
${\mc B}_\alpha \subset {\mc B}_{\Sigma_\alpha,
  Z_\alpha}^{\lambda_\alpha}(B_\alpha)$
is the Banach manifold of $K$-orbits of maps into $V$ resp. gauge
equivalence classes of gauged maps to $V$ resp. maps into $X$ which
satisfy the pointwise constraints at markings labelled by
non-tangential constraints. More precisely, for each
$l_i \in L_{v_\alpha}$, if $O_i\in {\mc O}$ is non-tangential, then
the valued of the (gauged) map is contained in $V_{O_i}$
resp. $V_{O_i} \qu G$. This requirement makes sense thanks to the
Sobolev embedding $W^{1,p} \hookrightarrow C^0$; this requirement is
also gauge invariant since all constraints are $K$-invariant. Lastly,
at each interior node or the infinity at $\Sigma$ there is a smooth
evaluation map from ${\mc B}_\alpha$ to $X$; at each boundary node
there is a smooth evaluation map from ${\mc B}_\alpha$ to $L$ (see
Theorem \ref{localmodel}). Let ${\mc E}_\alpha$ be the
restriction of
${\mc E}_{\Sigma_\alpha, Z_\alpha}^{\lambda_\alpha}(B_\alpha)$ to
${\mc B}_\alpha$ and let
${\mc F}_\alpha: {\mc B}_\alpha \to {\mc E}_\alpha$ the restriction of
$F_{\Sigma_\alpha}^{\lambda_\alpha}$ to ${\mc B}_\alpha$.

For each edge or tail there is also the Fredholm theory associated to gradient lines. For each $e \in E_{\underline\Gamma}^+ \sqcup T_\Gamma$, over the interval $I_e$ there is the Banach manifold ${\mc B}_e$ consisting of $W^{1,p}$-maps from $I_e$ to $L$; moreover, when $I_e$ has infinite end(s), we also require that elements in ${\mc B}_e$ satisfy the asymptotic condition specified by the data in ${\bm \Gamma}$. Over ${\mc B}_e$ there is the bundle whose fibre over a map $x_e$ is $L^p(I_e, x_e^* TL)$. The perturbed gradient line e provides a section ${\mc F}_e: {\mc B}_e \to {\mc E}_e$. For each finite end of $I_e$ there is a smooth evaluation map ${\mc B}_e \to L$. 

Combining the bundles over the one- and two-dimensional components
gives a section of a Banach bundle cutting out the moduli space
locally. Define
\beq \label{eqn63}
{\mc B}_{{{\bm \Gamma}}}({\mc C}) \subset \Big( \prod_{v_\alpha\in V_\Gamma} {\mc
  B}_\alpha \Big) \times \Big( \prod_{e \in E_{\underline{\Gamma}}^+ \sqcup T_\Gamma} {\mc B}_e \Big) 
\eeq 
to be the open subset of the product such that for each node of ${\mc
  C}$, the distance between the evaluations at the two sides of the
nodes is less than a fixed very small constant; for example, a number
smaller than the injectivity radii of $X$ and $L$. An element of ${\mc
  B}_{\bm \Gamma}({\mc C})$ is denoted by $(({\bm v}_\alpha),
(x_e))$. Notice that there is no matching condition required for ${\mc B}_{\bm \Gamma}({\mc C})$. To include the matching
condition, denote by 
\[ \Delta_X \subset X \times X, \quad \Delta_L
\subset L \times L\] 
the diagonals. Denote 
\beqn
X_{\Gamma} = (X \times X)^{E_\Gamma \setminus E_{\uds \Gamma}} \times (L \times L)^{E_{\underline \Gamma}^*} \times (L \times L)^{\partial E_{\underline \Gamma}^+}. 
\eeqn 
Then there is a smooth evaluation map 
\beqn 
{\rm ev}_\Gamma : {\mc B}_{\bm \Gamma} ({\mc C}) \to X_\Gamma
\eeqn 
defined by taking evaluations at all nodes and ends of edges. 
By the definition in \eqref{eqn63}, the image of this evaluation map is contained in a small neighborhood $U(\Delta_\Gamma)$ of the ``diagonal''
\beqn 
\Delta_{\Gamma} = (\Delta_X)^{ E_\Gamma \setminus E_{\uds \Gamma}} \times (\Delta_L)^{E_{\underline \Gamma}^*}  \times (\Delta_L)^{\partial E_{\underline \Gamma}^+}.
\eeqn 
We may assume that $U(\Delta_\Gamma)$ can be
embedded into the normal bundle $N_\Gamma \to \Delta_\Gamma$. Denote
the composition of ${\rm ev}_\Gamma$ with the projection
$U(\Delta_\Gamma) \to \Delta_\Gamma$ by $\ov{\ev}_\Gamma$. Then let
${\mc E}_{\bm\Gamma}({\mc C}) \to {\mc B}_{\bm\Gamma} ({\mc C})$ be
the Banach vector bundle whose fibre over $( ({\bm v}_\alpha), (x_e))$
is 
\beqn 
\Big( \bigoplus_{v_\alpha \in V_\Gamma} {\mc E}_\alpha|_{{\bm v}_\alpha} \Big) \oplus \Big( \bigoplus_{e \in E_{\underline \Gamma}^+} {\mc E}_e|_{x_e} \Big) \oplus \ov{\ev}_\Gamma^* N_{\Gamma}.  
\eeqn 
Lastly, we obtained a section 
\beqn 
{\mc F}_{\bm \Gamma} ({\mc C}): {\mc B}_{\bm \Gamma} ({\mc C}) \to {\mc E}_{\bm   \Gamma} ({\mc C}).  
\eeqn 
The regularity results for vortices (see \cite[Section 3]{Cieliebak_Gaio_Mundet_Salamon_2002}) and holomorphic curves imply that each point in the zero locus of
${\mc F}_{\bm \Gamma}({\mc C})$ is a gauge equivalence class of stable
treed scaled vortex over ${\mc C}$, hence defines an element of
${\mc M}_{\bm \Gamma}(P_\Gamma)$.
  Define
\beqn {\mc B}_{\bm \Gamma}^* ({\mc C})
\subset {\mc B}_{\bm \Gamma}({\mc C}) \eeqn 
to be the subset of objects
${\mc V} = ({\mc C}, ({\bm v}_\alpha), (x_e))$ such that whenever
$B_\alpha \neq 0$, ${\bm v}_\alpha$ is not contained in $D$ or
$\bar D$.

\subsubsection{Fredholm problem setup for varying domains}

We use local trivializations of the universal curve to extend the previous setup to varying domains. For each $p \in {\mc W}_\Gamma$ with fibre ${\mc C}_p$ in the universal curve, one can find an open neighborhood $W_p\subset {\mc W}_\Gamma$ and a smooth local trivialization 
\beqn
{\mc U}_\Gamma|_{W_p} \cong W_p \times {\mc C}_p.
\eeqn
Each fibre in ${\mc U}_\Gamma|_{W_p}$ can be
identified with ${\mc C}_p$ with fixed positions of markings, a
smoothly varying domain complex structure, and varying length
functions on all positive length edges. One can cover ${\mc U}_\Gamma$
by a countably many such charts $W_{p_i}$. Abbreviate $W_{p_i} = W_i$
and ${\mc C}_i = {\mc C}_{p_i}$. For each $i$, define 
\beqn
{\mc B}_{\bm \Gamma} (W_i):=W_i \times {\mc B}_{\bm \Gamma} ({\mc C}_i),
\eeqn 
and
\beqn 
{\mc E}_{\bm \Gamma} (W_i):= W_i \times {\mc E}_{\bm \Gamma}({\mc C}_i).
\eeqn
The equations (for the varying domain complex structures and lengths) and the matching condition at nodes defines a section
\beq\label{eqn64}
{\mc F}_{\bm \Gamma} (W_i): {\mc B}_{\bm \Gamma} (W_i) \to {\mc E}_{\bm \Gamma} (W_i).
\eeq
Define
\beqn
{\mc M}_{\bm\Gamma} (P_{\Gamma}, W_i) = {\mc F}_{\bm \Gamma} (W_i)^{-1}(0). 
\eeqn
Then there is a natural equivalence relation $\sim$ on the disjoint union of all ${\mc M}_{\bm\Gamma} (P_{\Gamma}, W_i)$. Consider the topological space 
\beq\label{eqn65} 
\bigsqcup_i {\mc M}_{\bm\Gamma} (P_{\Gamma}, W_i) / \sim.  
\eeq
The usual regularity results for holomorphic curves and the regularity of vortices (see \cite[Section 3]{Cieliebak_Gaio_Mundet_Salamon_2002}) imply that this space is naturally identified with the moduli space ${\mc M}_{\bm \Gamma}(P_\Gamma)$ defined in Section \ref{section5}. Moreover, it is a standard to prove that the notion of convergence defined in Section \ref{section5} is equivalent to the convergence induced from the topology of the Banach manifolds. Hence the identification is actually a homeomorphism. Lastly, define 
\beqn
{\mc M}_{{\bm \Gamma}}^*(P_\Gamma) \subset {\mc M}_{{\bm \Gamma}}(P_\Gamma)
\eeqn
to be the open subset of elements which do not have nonconstant components mapped into the stabilizing divisor $D = D_1 \cup \cdots \cup D_h$. 
For regular perturbations the moduli space ${\mc M}_{{\bm \Gamma}}^*(P_\Gamma)$ is a smooth manifold, whose dimension is given by a formula derived from the Riemann--Roch theorem for real Cauchy--Riemann operators in
\cite[Appendix]{McDuff_Salamon_2004}. We leave the verification to the reader.

\begin{prop}\label{prop67}
Suppose $\Gamma \in {\bf T}^{\rm st}$ and $V_{\Gamma^{\rm sup}}^0 = \emptyset$. Suppose $P_\Gamma$ is regular. Let ${\bm \Gamma}$ be a controlled and uncrowded map type ${\bm \Gamma}$ refining $\Gamma$. Also suppose for all ghost vertices $v_\alpha$ and any leaf $l_i$ attached to $v_\alpha$, the constraint $O_i$ is nontangential. Then one has
\beqn
{\rm index} {\bm \Gamma}:= {\rm dim} {\mc M}_{{\bm \Gamma}}^*(P_\Gamma) = {\rm dim} {\mc W}_\Gamma + 2 c_1( {\bm \Gamma}) + |{\bf X}| - \sum_{l_i \in L_\Gamma} \delta_{O_i}.
\eeqn
(See the definitions of relevant terms in Definition \ref{defn51}, \eqref{eqn52}, 
and \eqref{condeg}.)
\end{prop}

\subsection{The induction construction}\label{subsection63}

Before we start the inductive construction of coherent regular
  perturbation we introduce the
following notation. Define an equivalence relation $\Gamma\sim \Pi$
among stable domain types to be generated by the relation 
\beqn \Pi \leq
\Gamma\ {\rm and}\ \rho|_{\uds\Pi}: \uds \Pi \to \uds \Gamma\ {\rm is\
  an\ isomorphism} \eeqn 
where $\rho: \Pi \to \Gamma$ is the tree map induced by the relation
$\Pi \leq \Gamma$. Roughly speaking, $\Gamma \sim \Pi$ if $\Gamma,\Pi$ have the ``same'' base. Denote by $[\Gamma]$
the equivalence class of $\Gamma$. It is not hard to verify that
The partial order $\leq$ among stable domain types descends to a partial
order among equivalence classes, which is still denoted by $\leq$.

Our inductive construction is based on the partial order among
equivalence classes. Let $[\Gamma]$ be an equivalence class. As the
induction hypothesis, we assume that we have constructed $P_\Pi$ for
all $\Pi$ in all equivalence classes $[\Pi]$ with
$[\Pi] \lneqq [\Gamma]$. Let $\Gamma$ be any domain type in
$[\Gamma]$. Define a Banach manifold as follows. The pre-chosen
perturbation data $P_\Pi$ determine the value of $P_\Gamma$ over
lower strata $\ov{\mc U}{}_{\Gamma, \Pi}$ for all $\Pi \lneqq \Gamma$
and $[\Pi] \lneqq [\Gamma]$. This gives a Banach manifold 
\beqn {\mc
  P}_\Gamma^* \subset {\mc P}_\Gamma \eeqn 
consisting of perturbations whose values over all these subsets agree
with the pre-chosen ones. We would like to use the Sard--Smale theorem
to find a comeager subset of ${\mc P}_\Gamma^*$ consisting of
perturbations satisfying the regularity condition.

\begin{lemma}\label{lemma68}
For $\delta > 0$ sufficiently small, there exists a comeager subset ${\mc P}_{\Gamma}^{*, {\rm s.reg}} \subset {\mc P}_{\Gamma}^*$ whose
  elements are strongly regular. \end{lemma}

\begin{proof}
By definition, the strong regularity of a perturbation $P_\Gamma$ consists of finitely many $W$-regularity properties. We would like to prove the following claim.

\vspace{0.2cm}

\noindent {\it  Claim.} For every $W \subset V_{\Gamma^{\rm sup}}$
containing $V_{\Gamma^{\rm sup}}^0$, there is a comeager subset ${\mc
  P}_{\Gamma_W, \bar W}^{*, {\rm reg}}$ consisting of $\bar W$-regular
perturbations, with notation as in Definition \ref{defn42}.
\vspace{0.2cm}

Once this claim is proved, then via the forgetful map 
\beqn
{\mc P}_\Gamma^* \to {\mc P}_{\Gamma_W, \bar W}^*
\eeqn
(see Remark \ref{rem47} and \eqref{eqn44}) the comeager subsets provided by this claim can all be pulled back to comeager subsets of ${\mc P}_\Gamma^*$. Since there are only finitely many, their intersection is still comeager, denoted by ${\mc P}_\Gamma^{*, {\rm s.reg}}$. 

Now we need to prove the above claim. Since the arguments are basically the same, we only prove the case that $W = \emptyset$, and assume that $V_{\Gamma^{\rm sup}}^0 = \emptyset$. Consider a controlled and uncrowded refinement ${\bm \Gamma}$. We would like to consider a ``universal'' version of the section \eqref{eqn64} by adding deformations of perturbation data. We omit $W_i$ from the notation. Consider the universal section
\beq\label{eqn66} 
{\mc F}_{\bm \Gamma}: {\mc P}_\Gamma^* \times {\mc B}_{\bm\Gamma}^* \to {\mc E}_{\bm\Gamma}.
\eeq
We would like to prove that ${\mc F}_{\bm \Gamma}$ is transversal to the moduli space. In fact this follows from the claim that for each fibre ${\mc C} = {\mc C}_p$ of ${\mc U}_\Gamma \to {\mc W}_\Gamma$, the section
\beq\label{eqn67}
{\mc F}_{\bm \Gamma}({\mc C}): {\mc P}_\Gamma^*  \times {\mc B}_{\bm \Gamma}^*({\mc C}) \to {\mc E}_{\bm \Gamma}({\mc C})
\eeq
is transverse to the zero section. 

Now we consider the more precise description of the linearization of ${\mc F}_{\bm \Gamma}({\mc C})$. Consider a zero point $(P_\Gamma, {\mc V})$. There is a decomposition 
\beqn
T_{P_\Gamma} {\mc P}_\Gamma^*   = \bigoplus_{ v_\alpha \in V_\Gamma} T_\alpha^{\mc P} \oplus \bigoplus_{e \in E_{\uds\Gamma}^+} T_e^{\mc P}
\eeqn
where each summand is the subspace of infinitesimal perturbations that vanishes on all but one component. There is also a decomposition
\beqn
T_{\mc V} {\mc B}_{\bm \Gamma}^*({\mc C}) = \bigoplus_{v_\alpha \in V_\Gamma} T_\alpha^{\mc B} \oplus \bigoplus_{e \in E_{\uds \Gamma}^+} T_e^{\mc B}.
\eeqn
Let $\mathring{T}_\alpha^{\mc B} \subset T_\alpha^{\mc B}$ be the
subspace of deformations that vanish at nodes (but not markings) on
the $v_\alpha$-components and let
$\mathring{T}_e^{\mc B} \subset T_e^{\mc B}$ be the space of
deformations that vanish at the boundaries of the interval
$I_e$. There are also summands
${\mc E}_\alpha \subset {\mc E}_{\bm \Gamma}({\mc C})|_{\mc V}$ and
${\mc E}_e \subset {\mc E}_{\bm \Gamma}({\mc C})|_{\mc V}$. Then the
linearization of ${\mc F}_{\bm \Gamma}({\mc C})$ at ${\mc V}$ induces
linear maps
\beq\label{eqn68}
{\mc D}_\alpha: T_\alpha^{\mc P} \oplus \mathring{T}_\alpha^{\mc B} \to {\mc E}_\alpha
\eeq
\beq\label{eqn69}
{\mc D}_e: T_e^{\mc P}\oplus \mathring{T}_e^{\mc B} \to {\mc E}_e.
\eeq
For the same reason as in Cieliebak--Mohnke's approach in \cite{Cieliebak_Mohnke}, since the refinement is uncrowded, one only needs to show that for all components $v_\alpha$ with $B_\alpha \neq 0$ (resp. all edges $e \in E_{\uds\Gamma}^+\cup T_\Gamma$), the linear map \eqref{eqn68} (resp. the map \eqref{eqn69}) is surjective. These different cases are treated as follows.

\begin{enumerate}

\item First we consider transversality for gradient flow lines.  For each edge $e \in E_{\uds \Gamma}^+$, denote the variables of ${\mc D}_e$
  by $f_e \in T_e^{\mc P}$ (deformation of
  $F_e: I_e \times L \to {\mb R}$) and
  $\xi_e \in \mathring{T}_e^{\mc B}$ (deformation of
  $x_e: I_e \to L$). Then 
\beqn {\mc D}_e (f_e, \xi_e) = D_e (\xi_e) +
  \nabla f_e(x_e) \in {\mc E}_e = L^p (I_e, x_e^* T L ).  \eeqn 
By
  item (c) of Lemma \ref{lemma45}, $f_e$ can be nontrivial over a
  nonempty open subset of $I_e$. Then ${\mc D}_e$ is surjective since any element of the cokernel satisfies an ordinary differential equation but vanishes on an open subset.

\item Next, consider transversality for holomorphic spheres or disks
  in the symplectic quotient $X$.  Suppose
  $v_\alpha \in V_\Gamma^\infty$ so
  ${\bm v}_\alpha = u_\alpha: \Sigma_\alpha \to X$ is a holomorphic
  sphere or holomorphic disk. We have assumed that $u_\alpha$ is
  nonconstant. Notice that the thick part
  $\Sigma_\alpha^{\rm thick} \subset \Sigma_\alpha$, which is the
  region where one can deform the almost complex structure, is
  nonempty. Then by basic holomorphic curve theory in
    \cite{McDuff_Salamon_2004}, the map ${\mc D}_\alpha$ is
  surjective.

\item Next consider transversality for vortices over the half-space.
  Suppose $v_\alpha \in V_{\uds\Gamma}^1$ so ${\bm v}_\alpha$ is a
  gauge equivalence class of affine vortices over
  $\Sigma_\alpha \cong {\bm H}$. As explained at the end of
  Subsection \ref{subsection23}, surjectivity of the augmented
  linearization \eqref{eqn27} is equivalent to surjectivity of
  the operator $D_\alpha$. Choose a smooth representative
  $(u_\alpha, \phi_\alpha, \psi_\alpha)$. Then we would like to show
  that the operator \beqn {\mc D}_\alpha^+: T_\alpha {\mc P}\oplus
  \mathring{T}_{\alpha, +}^{\mc B} \to \tilde L^p( {\bm H}, u_\alpha^*
  TV \oplus {\mf k} \oplus {\mf k}) \eeqn defined by \beqn {\mc
    D}_\alpha^+( h_\alpha, \xi_\alpha) = \Big( D_\alpha'(\xi_\alpha) +
  h_\alpha( \partial_t u_\alpha + {\mc X}_{\psi_\alpha}), D_\alpha'' (
  \xi_\alpha), D_\alpha'''(\xi_\alpha) \Big).  \eeqn is
  surjective. Here $\mathring{T}_{\alpha, +}^{\mc B}$ is the subspace
  of triples
\[ \xi_\alpha = (\xi_\alpha', \xi_\alpha'', \xi_\alpha''')\in \tilde
  L_m^{1,p}({\bm H}, u_\alpha^* TV \oplus {\mf k}\oplus {\mf
    k})_{L_V} \] 
  where $\xi'$ vanishes at nodes, and
  $D_\alpha', D_\alpha'', D_\alpha'''$ are the three components of
  the augmented linearization \eqref{eqn27}. By the Fredholm property
  of the vortex equation, $D_\alpha^+$ has closed range. Hence to
  prove that ${\mc D}_\alpha^+$ is surjective, it suffices to show
  that its range is dense. Suppose this is not the case, so that there is an element
 \[ \varsigma_\alpha = (\varsigma_\alpha', \varsigma_\alpha'',
  \varsigma_\alpha''') \in L^{q, -\tau_p}( \Sigma_\alpha, u_\alpha^*
  TV \oplus {\mf k} \oplus {\mf k}) \] 
  which is $L^2$-orthogonal to the range of ${\mc D}_\alpha^+$. Here
  $q$ is the conjugate of $p$. Notice that by item (d) of Lemma
  \ref{lemma45}, $\Sigma_\alpha^{\rm thick} \subset \Sigma_\alpha$
  contains a portion of the boundary of $\Sigma_\alpha$, while
  $u_\alpha(\partial \Sigma_\alpha) \subset L_V \subset \mu^{-1}(0)$.
  Hence there exists a nonempty open subset
  $O \subset \Sigma_\alpha^{\rm thick}$ with $u_\alpha(O) \subset U$
  where $U$ is the neighborhood of $\mu^{-1}(0)$ specified in
  \eqref{eqn210}. Over $O$ one can deform $J_\alpha$ in a
  domain-dependent and $K$-invariant way. On the other hand, since
  ${\bm v}_\alpha$ is nonconstant, the subset of $O$ where
  $\partial_t u_\alpha + {\mc X}_{\psi_\alpha} \neq 0$ is open and
  dense (see \cite[Lemma
  2.5]{Cieliebak_Gaio_Mundet_Salamon_2002}). Then using the
  deformation $h_\alpha$, one sees that
  $\varsigma_\alpha'|_O \equiv 0$. Furthermore, choose
  $\xi_\alpha = ( \xi_\alpha', 0, 0)$, then \beqn 0 = \langle {\mc
    D}_\alpha^+(0, \xi_\alpha), \varsigma_\alpha \rangle = \langle
  d\mu(u_\alpha) \cdot \xi_\alpha', \varsigma_\alpha'' \rangle +
  \langle d\mu(u_\alpha) \cdot J_\alpha \xi_\alpha',
  \varsigma_\alpha''' \rangle.  \eeqn Since $u_\alpha(O)$ is contained
  in $U$ where the $K$-action is free, by taking $\xi_\alpha'$
  properly one can prove that
  $\varsigma_\alpha''|_O = \varsigma_\alpha'''|_O \equiv 0$. Then
  $\varsigma_\alpha|_O = 0$.  Unique continuation of solutions to first order
  elliptic equations implies that  $\varsigma_\alpha \equiv 0$ over
  $\Sigma_\alpha$, which is a contradiction. Hence ${\mc D}_\alpha^+$
  has dense range and hence surjective.

\item Lastly consider transversality for vortices over ${\bm C}$,
  using the controlled property of map types and the assumptions on
  the size of the nodal neighborhoods.  Suppose
  $v_\alpha \in V_{\Gamma^{\rm sup}}^1$, namely ${\bm v}_\alpha$ is a
  nonconstant affine vortex over $\Sigma_\alpha \cong {\bm C}$. By
  the argument in the above case, one only needs to show that
\beq\label{eqn610} u_\alpha^{-1}( U ) \cap \Sigma_\alpha^{\rm thick}
  \neq \emptyset.  \eeq 
In fact, by the controlledness property, 
\beqn E({\bm v}_\alpha) = \omega(B_\alpha) \leq {\rm
    deg}^{\max} (\Gamma(v_\alpha)).
    \eeqn 
On the other hand, by the definition of the energy, one has 
\beqn E({\bm v}_\alpha) \geq \int_{\Sigma_\alpha} |\mu(u_\alpha)|^2 \geq | \Sigma_\alpha^{\rm thick}| \inf_{\Sigma_\alpha^{\rm thick}}
  |\mu(u_\alpha)|^2 \geq  \delta^{-2} \inf_{\Sigma_\alpha^{\rm thick}}
  |\mu(u_\alpha)|^2 {\rm deg}^{\max} (\Gamma(v_\alpha)).  
  \eeqn
  Therefore, when $\delta$ is sufficiently small, there holds \beqn
  \inf_{\Sigma_\alpha^{\rm thick}} |\mu(u_\alpha)| \leq \delta <
  \delta_U.  \eeqn So \eqref{eqn610} is true and hence
  ${\mc D}_\alpha$ is surjective.
\end{enumerate}

Therefore the operator \eqref{eqn67} is surjective. So is
\eqref{eqn66}. Then the zero locus ${\mc F}_{\bm \Gamma}^{-1}(0)$ is a
smooth Banach submanifold of
${\mc P}_\Gamma^* \times {\mc B}_{\bm \Gamma}^*$. Since the projection
onto ${\mc P}_\Gamma^* $ is Fredholm, by the Sard--Smale theorem,
the set of regular values is comeager. The proof is then finished by
using the construction provided at the beginning.
\end{proof}

\begin{proof}[Proof of Theorem \ref{thm64}]
To begin the induction one needs to show that for the basic domain types that do not have degenerations, transversality can be achieved. These domain types are Y-shapes and $\Phi$-shapes (see Definition \ref{defn32}), and the strong regularity for these shapes can be achieved in the same way as in Lemma \ref{lemma68}. 

Now we carry out the inductive step. Given an equivalence class $[\Gamma]$, we assume that we have constructed coherent perturbations $P_\Pi$ for all stable domain type $\Pi$ with $[\Pi] \lneqq [\Gamma]$, which produces a Banach manifold $P_\Gamma^* \subset P_\Gamma$ for all $\Gamma$ in the class $[\Gamma]$. Now we construct a comeager subset ${\mc P}_\Gamma^{**}$ inductively for all $\Gamma$ in this equivalece class. First, if $\Gamma$ is a smallest element in $[\Gamma]$, then define 
\beqn
{\mc P}_\Gamma^{**}:= {\mc P}_\Gamma^{*, {\rm s.reg}}
\eeqn
where $P_\Gamma^{*, {\rm s.reg}} \subset P_\Gamma^*$ is constructed by Lemma \ref{lemma68}. Suppose we have constructed a comeager subset ${\mc P}_\Pi^{**} \subset {\mc P}_\Pi^*$ for all $\Pi$ within the same equivalence class and $\Pi \lneqq \Gamma$. Then consider the restriction maps
\beqn
{\mc P}_\Gamma^* \to {\mc P}_\Pi^* \ \ \forall \Pi \lneqq \Gamma,\ [\Pi ] = [\Gamma].
\eeqn
These maps all admit right inverses, hence the preimages of ${\mc
  P}_\Pi^{**}$ together with the comeager set $P_\Gamma^{*, {\rm s.reg}}$ given by Lemma \ref{lemma68} have a comeager intersection, which is defined to be
${\mc P}_\Gamma^{**}$. Lastly, for every maximal elements $\Gamma$ in
the equivalence class $[\Gamma]$, we choose one element $P_\Gamma \in
{\mc P}_\Gamma^{**}$. They restrict to strongly regular perturbations
$P_\Pi$ for all $\Pi$ in this equivalence class. 

For Morse trees, see similar construction in \cite{abouzaid_plumbing} and \cite{Mescher_2018}.
\end{proof}

\begin{rem}\label{rem69} We arrange the induction
  procedure so that we first construct perturbations for trees of
  scale zero and scale infinity, and then extend these perturbations
  to trees of scale $1$. The constructions for scale zero trees and
  scale infinity trees are independent. However, one can make their
  restrictions to domain types without interior leaves identical. Indeed, let
  ${\bf T}_0 \subset {\bf T}^{\rm st}$ be the set of stable domain types
  which have no interior leaves, and
  ${\bf T}_0^0, {\bf T}_0^\infty \subset {\bf T}_0$ the subset of
  those domain types of scale $0$ and $\infty$ respectively. Then there is a
  one-to-one correspondence ${\bf T}_0^0 \to {\bf T}_0^\infty$,
  denoted by $\Gamma^0 \mapsto \Gamma^\infty$. There are
  identifications
  $\ov{\mc W}_{\Gamma^0} \cong \ov{\mc W}_{\Gamma^\infty}$ between
  moduli spaces and
  $\ov{\mc U}_{\Gamma^0} \cong \ov{\mc U}_{\Gamma^\infty}$ between
  universal curves. Then during our construction, we may require that
  with respect to these identifications, the two sets of perturbations
  are identical on those domain types without interior leaves.

Furthermore, let ${\bf T}_0^1\subset {\bf T}_0$ be the subset of {\it  unbroken} stable domain types of mixed scale. For each $\Gamma^1 \in {\bf T}_0^1$ and $\lambda \in \{0, \infty\}$, there is an associated $\Gamma^\lambda \in {\bf T}_0^\lambda$. This is defined by including the domain symmetry of scale $1$ pieces (which is ${\mb R}$) to the group of real M\"obius transformations fixing the infinity. This induces a map $\ov{\mc W}_{\Gamma^1} \to \ov{\mc W}_{\Gamma^\lambda}$, which is a generically a fibration of relative dimension $1$. There is also the associated map $\ov{\mc U}_{\Gamma^1} \to \ov{\mc U}_{\Gamma^\lambda}$. Therefore, in our construction, we may require that for $\Gamma^1 \in {\bf T}_0^1$, the perturbation data $P_\Gamma$ is the pullback of $P_{\Gamma^\lambda}$. A consequence of this requirement is that the $A_\infty$ morphism we will define is a higher order deformation of the identity. 
\end{rem}

\subsection{Refined Compactness}

The transversality of various moduli spaces implies a stronger compactness result for moduli spaces of expected dimension at most one. To state the precise form our the refined compactness theorem, we introduce the notions of {\it essential} map types. In index zero these map types constitute configurations to be counted to define the Fukaya algebras and the $A_\infty$ morphisms. 

\begin{defn}\label{essential}
A map type ${{\bm \Gamma}}$ is called {\it  essential} if the following are satisfied.
\begin{enumerate}

\item $\Gamma$ is stable, unbroken, has no vertices corresponding to holomorphic spheres, has no edge of length zero, and the set of bordered vertices $\partial^* V_{\uds \Gamma}^\infty $ is empty. This is equivalent to the condition that $\Gamma$ is in ``top stratum.'' In terms of dimension, this condition is equivalent to ${\rm dim} {\mc W}_\Gamma = d^\lambda(k, l)$ (see \eqref{eqn32}).

\item All leaves are attached to vertices in $V_\Gamma^0$ or $V_\Gamma^1$.

\item ${{\bm \Gamma}}$ is uncrowded.

\item For each $l_i \in L_\Gamma$, $p_i = D_{a_i}$ for some $D_{a_i} \in \{D_1, \ldots, D_h\}$ and $O_i = D_{a_i}^{\rm st}$.

\item For each $v_\alpha \in V_\Gamma$ with $B_\alpha \neq 0$, the
  number of leaves which are attached to $v_\alpha$ and are decorated
  with $D_a$ is $\omega(B_\alpha) {\rm deg} D_a $. This condition
  implies that ${\bm \Gamma}$ is controlled in the sense of Definition 
\ref{controlled}.
\end{enumerate}
\end{defn}

The last condition of the above definition also implies that for every element in ${\mc M}_{\bm\Gamma}^*(P_\Gamma)$, every intersection with the stabilizing divisor is at the marking, contained in the stable locus, and is transverse. 

The following establishes sequential compactness for moduli spaces defined using these perturbations.  
From now on we fix a coherent system of strongly regular perturbation data $P_\Gamma$ for all stable domain type $\Gamma$ granted by Theorem \ref{thm64} and only consider the corresponding moduli spaces.

\begin{prop}\label{prop611}
Let ${\bm \Gamma}$ be an essential map type. 
\begin{enumerate}

\item When ${\rm index} {\bm \Gamma} = 0$, one has
\beqn
\ov{ {\mc M}_{{\bm\Gamma}}^*(P_\Gamma)} = {\mc M}_{\bm\Gamma}^*(P_\Gamma).
\eeqn

\item When ${\rm index} {\bm \Gamma} = 1$, one has
\beqn
\ov{ {\mc M}_{{\bm\Gamma}}^*(P_\Gamma)} \setminus {\mc M}_{\bm\Gamma}^*(P_\Gamma) = \bigsqcup {\mc M}_{\bm\Pi}^*(P_\Pi)
\eeqn
where the disjoint union is taken over all map types ${\bm \Pi}$ obtained from ${\bm \Gamma}$ by applying exactly one of the following elementary transformations (see explanations of the following cases in Definition \ref{defn37}).

\begin{enumerate}

\item[{\rm (F1)}] Shrinking the length of a non-special edge to zero.

\item[{\rm (F2)}] Disk bubbling.

\item[{\rm (F3)}] Shrinking the length of a collection of special edges to zero. 

\item[{\rm (F4)}] Boundary affine vortex bubbling.

\item[{\rm (T1)}] Breaking one edge.

\item[{\rm (T2)}] Breaking a collection of edges at infinity.
\end{enumerate} 
\end{enumerate}
\end{prop}

\subsection{Proof of Proposition \ref{prop611}}\label{subsection65}

We first fix a few notations. Suppose ${\mc V}$ represents a point in the closure $\ov{ {\mc M}_{{\bm \Gamma}}^*(P_\Gamma)}$ which is the limit of a sequence ${\mc V}_k$ representing a sequence of points in the uncompactified part ${\mc M}_{{\bm \Gamma}}^*(P_\Gamma)$. Then by the definition of convergence, there is a surjective map
\beqn 
\rho_V: V_\Pi \to V_\Gamma 
\eeqn 
For each $v_\alpha \in V_\Pi$, let the image of $v_\alpha$ be $v_{\check\alpha} \in V_\Gamma$.

We first describe the intersection points of the limiting
configuration with the stabilizing divisor.

\begin{lemma}\label{lemma612}
  Suppose $v_\alpha \in V_\Pi$ and the corresponding vortex or
  holomorphic map is ${\bm v}_\alpha$. Suppose ${\bm v}_\alpha$ has
  positive energy and $z_\alpha \in \Sigma_\alpha$ is an isolated
  intersection of ${\bm v}_\alpha$ with $D$.  Then $z_\alpha$ is
  either a node or an interior marking.
\end{lemma}

\begin{proof}
Suppose $z_\alpha \in {\bm v}_\alpha^{-1}(D)$ and $z_\alpha$ is not a node. Take a small disk $U_\alpha \subset \Sigma_\alpha$ containing $z_\alpha$ that is disjoint from other points mapped into $D$ as well as special points on $\Sigma_\alpha$. Because the map ${\bm v}_\alpha|_{U_\alpha}$ is pseudoholomorphic, the local intersection number of ${\bm v}_\alpha|_{U_\alpha}$ is positive (see Lemma \ref{lemma211}). By the fact that ${\mc V}_k$ converges to ${\mc V}$, for $k$ sufficiently large, there is a small disk $U_{\check\alpha,k} \subset \Sigma_{\check\alpha, k} \subset {\mc C}_k$ such that ${\bm v}_{\check\alpha, k}$ converges (after applying the reparametrization $U_{\check\alpha, k}\cong U_\alpha$ and a sequence of gauge transformations) to ${\bm v}_\alpha|_{U_\alpha}$. In particular the image of the boundary of $U_{\check\alpha, k}$ is disjoint from $D$ and the intersection number between ${\bm v}_{\check\alpha, k}|_{U_{\check\alpha, k}}$ with $D$ is equal to the local intersection number between ${\bm v}_\alpha$ with $D$ at $z_\alpha$. In particular, there is at least one intersection point between ${\bm v}_{\check\alpha, k}$ and $D$ within $U_{\check\alpha, k}$. Since the map type ${\bm \Gamma}$ is essential, this intersection point must be a marking. Then by the convergence of the domain ${\mc C}_k \to {\mc C}$, $z_\alpha$ must be a marking as well. 
\end{proof}

Next we show that in the limiting configuration each bouquet (Definition \ref{controlled}) has enough markings. This result implies that the limiting map type is controlled, a prerequisite for achieving transversality.

\begin{lemma}\label{lemma613}
For each bouquet $\Upsilon$ in $\Pi$, there exists $D_a \in \{D_1, \ldots, D_h\}$ such that 
\beq\label{eqn611}
{\rm deg}_{D_a} {\mf p}_\Upsilon = {\rm deg} D_a \cdot \omega({\bm \Upsilon}).
\eeq
\end{lemma}

\begin{proof}
Let the root component of $\Upsilon$ be $v_\alpha \in V_\Pi^1$ which supports an affine vortex ${\bm v}_\alpha$. Recall that the evaluation of ${\bm v}_\alpha$ at infinity is a point in $X$ while the intersection of all $D_a \qu G$ is empty. Therefore there exists an element $D_a \in \{D_1, \ldots, D_h\}$ such that the evaluation is not in $D_a \qu G$. Choose such a $D_a$.

On the other hand, $v_{\check\alpha}$ must be of scale $1$ and one has an isomorphism $\Sigma_{\check\alpha}$ with either ${\bm C}$ or ${\bm H}$. The convergence of ${\bm v}_k$ implies that there exists a sequence of disks $O_{k, \check\alpha} \subset \Sigma_{\check\alpha}$ identified with $B_{r_k}(z_k) \subset \Sigma_{\check\alpha}$ with $r_k \to \infty$ such that one can extend ${\bm v}_k|_{O_{k, \check\alpha}}$ to a smooth gauged map ${\bm v}_k':= (P_k, A_k, u_k)$ from ${\bm S}^2$ to $V$ satisfying the following conditions.

\begin{enumerate}

\item The restriction of ${\bm v}_k'$ to the complement of $B_{r_k}(z_k)$ converges as $k \to \infty$ to the constant which is the evaluation of ${\bm v}_\alpha$ at infinity. Hence for $k$
  large, the restriction of ${\bm v}_k'$ to the complement of
  $B_{r_k}(z_k)$ is uniformly away from $D_a$.

\item For large $k$, the equivariant homology class of ${\bm v}_k'$ coincides with the total homology class of ${\bm \Upsilon}$, denoted by $B_\Upsilon \in H_2^K(V)$.

\end{enumerate}

Therefore, by the definition of convergence of treed disks (see Definition \ref{defn310}), ${\rm deg}_{D_a} \Upsilon$ is equal to the number of markings in ${\mc V}_k$ which are decorated by $D_a$ and which are contained in $O_{k, \check\alpha}$. By (a) above, this number is also equal to the intersection number between ${\bm v}_k'$ and $D_a$, which is equal to the pairing between $B_\Upsilon$ and ${\rm deg} D_a \cdot [\omega_V + \mu]$. Therefore \eqref{eqn611} is true.
\end{proof}

\begin{proof}[Proof of Proposition \ref{prop611}]
 We would like to modify the limiting triple $(\Pi, {{\bm \Pi}}, {\mc V})$ to a triple for which transversality holds. Consider 
\beqn
\tilde W:= V_{\Pi^{\rm sup}}^0 \cup \Big\{ v_\alpha \in V_{\Pi^{\rm sup}}^1\ |\ {\bm v}_\alpha \subset D \Big\} \cup \Big\{ v_\alpha\in V_{\Pi^{\rm sup}} \ |\ B_\alpha = 0 \Big\}.
\eeqn
We decompose $\tilde W$ into connected components
\beqn
\tilde W = \tilde W_1 \sqcup \cdots \sqcup \tilde W_l.
\eeqn
Then by the monotonicity of $V$ and the condition (S1) in Definition \ref{defn12}), one has
\beqn
c_1(\tilde W_j):= \sum_{v_\alpha \in W_j} c_1(B_\alpha) \geq 0.
\eeqn
For each $\tilde W_i$, define $\tilde W_j'$ as follows: if $\tilde W_j$ is a ghost tree, then $\tilde W_j' = \tilde W_j$; otherwise, namely when 
\beqn
E(\tilde W_j):= \sum_{v_\alpha \in \tilde W_j} E(B_\alpha) > 0,
\eeqn
let $\tilde W_j'$ be the minimal subtree of $\Gamma^{\rm sup}$ containing $\tilde W_j$ that
is also non-separating, i.e., the complement of $\tilde W_j'$  in
$\Gamma$ is connected. Then define 
\beqn
\tilde W' = \bigcup_{1\leq j \leq l} \tilde W_j'.
\eeqn
Furthermore, there might be isolated vertices $v_\alpha \in \tilde
W'\cap V_{\Gamma^{\rm sup}}^1$ which have only one leaf. Removing such
vertices in $\tilde W'$ and define the set of remaining vertices $W
\subset V_{\Gamma^{\rm sup}}$. Notice that $V_{\Gamma^{\rm sup}}^0
\subset W$. Then by applying the forgetful map $\Pi \mapsto \Pi_W$
defined in Subsection \ref{subsection43}, one obtains a domain type $\Pi_W$ with a decomposition 
\beqn
L_{\Pi_W} = L^{\rm old} \sqcup L^{\rm new}.
\eeqn
Moreover, there is a one-to-one correspondence between new leaves and
connected components which have at least one leaf. Let 
\beqn
W = W_1 \sqcup \cdots \sqcup W_k \sqcup W_{k+1} \sqcup \cdots \sqcup W_{k+s}
\eeqn
be the decomposition into connected components, such that
$L_{W_1}, \ldots, L_{W_k}$ are nonempty and
$L_{W_{k+1}} = \cdots = L_{W_{k+s}} = \emptyset$. Let
$l_1, \ldots, l_k$ be the corresponding new leaves of $\Pi_W$
corresponding to the components $W_1, \ldots, W_k$. By the forgetful
property of the perturbations in Remark \ref{rem47}, there is an
induced perturbation $P_{\Pi_W} \in {\mc P}_{\Pi_W}$ such that, by
forgetting the corresponding components, one obtains a treed disk
${\mc C}_W$ supporting a treed scaled vortex ${\mc V}_W$ of domain type
$\Pi_W$ with respect to the perturbation data $P_{\Pi_W}$. Let
$z_i \in {\mc C}_W$ be the new marking corresponding to the new leaf
$l_i$. Now we would like to upgrade $\Pi_W$ to a map type ${\bm \Pi}_W$ so that
${\mc V}_W \in {\mc M}_{{\bm \Pi}_W}^*(P_{\Pi_W})$. Obviously, one
only needs to assign constraints at new markings of ${\mc C}_W$, since
these are the only data of a map type which do not descend
from ${\bm \Pi}$ canonically. To define the constraints, we decompose
the set of new leaves
\beqn L^{\rm new} = L^a \sqcup L^b \sqcup L^c \sqcup L^d \eeqn 
where classes $a, b, c, d$ are defined by the properties of the
corresponding connected components of $W$ as follows.
\begin{enumerate}
\item $l_i \in L^a$ if $E(W_i) = 0$.
\item $l_i \in L^b$ if $E(W_i)>0$ and $c_1(W_i) = 0$.
\item $l_i \in L^c$ if $c_1(W_i) = 1$.
\item $l_i \in L^d$ if $c_1(W_i) \geq 2$.
\end{enumerate}
We first claim that $L^b = L^c = L^d = \emptyset$. Define a constraint
$O_i \in {\mc O} $ for every new leaf $l_i \in L^{\rm new}$ as
follows.

\begin{enumerate}

\item For each $l_i \in L^a$ corresponding to a ghost tree with at
  least one leaf, the marking $z_i$ is mapped into $D$ (or $\bar D$).
  Then define $O_i$ to be the stratum of $D$ or $\bar D$ where the
  value of $z_i$ belongs to. So the degree of the constraint defined
  in \eqref{condeg} is $\delta_{O_i} \geq 2$.
\item For each $l_i \in L^b$, by the semi-Fano conditions
  (S1)-(S3), every nonconstant component in $W_i$ must be an
  $J_V$-affine vortex contained in $D \cap S$. The value of
  ${\mc V}_W$ at $z_i$ belongs to $\bar D \cap \bar S$. Define $O_i$
  to be the corresponding (stable) stratum of $D \cap S$. So
  $\delta_{O_i} \geq 4$.

\item For each $l_i \in L^c$, by the semi-Fano condition, we know that
  every nonconstant component in $W_i$ has to be an $J_V$-affine
  vortex contained in $D$. Then define $O_i$ to be the constraint
  given by corresponding stable stratum of $D$ where the evaluation at $z_i$ belongs to. So $\delta_{O_i} \geq 2$.

\item For each $l_i \in L^d$, define $O_i  = 1$ which means no constraint. So $\delta_{O_i} = 0$.

\end{enumerate}
The above assignments define a map type ${\bm \Pi}_W$. We claim that 
\beqn
{\mc V}_W \in {\mc M}_{{\bm \Pi}_W}^*(P_{\Pi_W}).
\eeqn
Indeed, the definitions of the constraints at the new markings imply that ${\mc V}_W \in {\mc M}_{{\bm \Pi}_W}(P_{\Pi_W})$. Moreover, all components supporting nonconstant affine vortices which are contained in $D$ are removed by the forgetful map $\Pi \mapsto \Pi_W$. Further, by Theorem \ref{thm25} and item (c) of Definition \ref{defn46}, in ${\mc V}_W$ there are no nonconstant holomorphic spheres contained in $\bar D$. Hence ${\mc V}_W \in {\mc M}_{{\bm \Pi}_W}^* (P_{\Pi_W})$. The construction of $\Pi_W$ shows that ${\bm \Pi}_W$ is uncrowded. Lemma \ref{lemma613} shows ${\bm \Pi}_W$ is controlled. Then the above moduli space is regular and hence (by the dimension formula \eqref{eqn32} and Proposition \ref{prop67})
\beqn
0 \leq {\rm index} {\bm \Pi}_W = d^\lambda(0, k) - {\rm codim} \Pi_W + |{\bf X}| + 2c_1({\bm \Pi}_W) + \sum_{l_i \in L_{\Pi_W}} ( 2- \delta_{O_i}).
\eeqn
By taking the difference ${\rm index} {\bm \Gamma} - {\rm index} {\bm \Pi}_W$, one has 
\begin{multline*}
1 \geq {\rm index} {\bm \Gamma} \geq {\rm codim} \Pi_W + 2 \left( c_1 ({\bm\Gamma}) - c_1({\bm \Pi}_W) \right) - \sum_{l_i\in L_{\Pi_W}} (2 - \delta_{O_i} ) \\
 = {\rm codim} \Pi_W +  \sum_{l_i \in L^b \cup L^c \cup L^d } \left( 2 c_1(W_i ) - ( 2 -\delta_{O_i}) \right) - \sum_{l_i\in L_{\Pi_W}^{\rm old} \cup L^a } ( 2 - \delta_{O_i})\\
 \geq  \sum_{l_i \in L^b \cup L^c \cup L^d} \left( 2 c_1(W_i ) - ( 2 -\delta_{O_i}) \right)
\end{multline*}
By the definition of $O_i$ for new leaves, we know that 
\beqn
2 c_1(W_i) + \delta_{O_i} \geq 4,\ \forall l_i \in L^b \cup L^c \cup L^d.
\eeqn
Therefore $L^b = L^c = L^d = \emptyset$. Moreover, because ${\rm
  codim} \Pi_W \leq 1$, there is no separating connected components of
$W$, because in the construction of $\Pi_W$, every separating
component of $W$ will create at least one vertex $v_\alpha \in
V_{\Pi_W^{\rm sup}}^\infty$ which increases the codimension by at
least two. 

By the previous paragraph, $W_i$ are ghost trees that are
non-separating and in the remainder of the proof we rule them out as
well.  To do so we redefine ${\bm \Pi}_W$ to rule out tangencies and
other high codimension phenomena. For each
$l_i \in L^{\rm new} = L^a$, we know that the value of $z_i$ is in $D$
or $\bar D$. If the value at $z_i$ is in a lower stratum of $D$ or
$\bar D$, then define $O_i$ to be the corresponding stratum of $D$
(whose codimension is at least four). Otherwise, the value at $z_i$ is
in a top stratum, say $D_a$. The construction of $W$ implies that
$z_i$ is an isolated intersection with $D_a$. Then define
$O_i = D_a^{m_i}$ where $m_i$ is the local intersection number. Notice
that $m_i \geq 2$ since the collapsed ghost tree has at least two markings. So one has \beqn \delta_{O_i} \geq 4\
\forall l_i \in L^{\rm new}.  \eeqn Then we have
\begin{multline*}
0 \leq {\rm index} {\bm \Pi}_W = d^\lambda(0, k) - {\rm codim} \Pi_W + c_1({\bm \Pi}_W) + \sum_{l_i \in L_{\Pi_W}} ( 2- \delta_{O_i})\\
= {\rm index} {\bm \Gamma} - {\rm codim} \Pi_W + \sum_{l_i \in L^{\rm old}} ( 2- \delta_{O_i}) + \sum_{l_i\in L^{\rm new}}( 2- \delta_{O_i}). 
\end{multline*}
Therefore, $L^{\rm new} = \emptyset$ which implies $W = \emptyset$. So
$\Pi_W = \Pi$. Furthermore, if ${\rm index} {\bm \Gamma} = 0$, then
${\rm codim}\Pi_W = {\rm codim} \Pi = 0$ which means $\Pi$ is in top
stratum. Since $\Pi \leq \Gamma$ we must have $\Pi =\Gamma$ and ${\bm
  \Pi} = {\bm \Gamma}$. If ${\rm index} {\bm \Gamma} = 1$, then besides
the possibility ${\bm \Pi} = {\bm \Gamma}$, $\Pi$ can also be in a
codimension one stratum of $\ov{\mc W}_\Gamma$, which belongs to the
list of degenerations in the statement of Proposition
\ref{prop611}. Therefore the proof of Proposition \ref{prop611} is
complete. \end{proof} 

\subsection{Orienting moduli spaces}\label{subsection66}

To count elements of zero-dimensional moduli spaces with signs, we need to specify the orientations of determinant line bundles of these moduli spaces. Recall that for a Fredholm operator $T: A \to B$ between separable Banach spaces, the determinant line is 
\beqn
\det T:= (\det {\rm coker} T)^*\otimes \det {\rm ker} T.
\eeqn
An orientation of $T$ is an ${\mb R}^*$-orbit in $\det T$. However, we often regard an orientation of $T$ as a nonzero element of $\det T$. With $T$ oriented, we call an element of $\det T$ {\it positive} if it is a positive multiple of the preferred orientation.

To orient moduli spaces we need to specify orientations of the unstable/stable manifolds of critical points. For each critical point ${\bm x}$ of $F_L: L \to {\mb R}$, fix (arbitrarily) an orientation of the unstable manifold
\beqn
W_{\bm x}^- \subset L.
\eeqn
Let $\sigma_{\bm x}^- \in \det {T_{\bm x} W_{\bm x}^-}$ be a positive orientation class. An orientation on the stable manifold is then induced in the following way: an orientation class $\sigma_{\bm x}^+\in \det T_{\bm x} W_{\bm x}^+$ is positive if and only if the orientation class $\sigma_{\bm x}^+ \wedge \sigma_{\bm x}^-$ is positive with respect to the natural isomorphism 
\beqn
\wedge: \det T_{\bm x} W_{\bm x}^+ \otimes \det T_{\bm x} W_{\bm x}^- \cong  \det T_{\bm x} W_{\bm x}^+ \oplus T_{\bm x} W_{\bm x}^- \cong \det T_{\bm x} L.
\eeqn

Now we orient the moduli spaces of treed scaled vortices. Let ${\bm \Gamma}$ be a map type with a regular moduli space ${\mc M}_{\bm \Gamma}^*$. Because borderless domains do not affect the orientation issue, we assume that ${\bm \Gamma}$ has no spherical components and hence the underlying tree $\Gamma$ is a ribbon tree. For any element ${\mc V} \in {\mc M}_{\bm \Gamma}^*$ whose underlying treed disk is ${\mc C}$, the linearized operator is a Fredholm operator 
\beqn
\tilde D: T_{\mc C} {\mc M}_\Gamma \oplus {\mc B} \to {\mc E}
\eeqn
whose restriction to the second summand is abbreviated temporarily by
\beqn
D: {\mc B} \to {\mc E}.
\eeqn
This is the linearized operator on the fixed treed disk. Since the tangent space $T_{\mc C} {\mc M}_\Gamma$ is finite dimensional, $\tilde D$ can be homotopic via Fredholm operators to the operator $0 \oplus D: T_{\mc C} {\mc M}_\Gamma \oplus {\mc B} \to {\mc E}$. Let $\sigma_\Gamma \in \det T_{\mc C} {\mc M}_\Gamma$ be the positive orientation class specified in Subsection \ref{subsection34}. We will define the positive orientation class of $\tilde D'$ (which determines one for $\tilde D$) as an element 
\beqn
\sigma_\Gamma \wedge \sigma_D \in \det T_{\mc C} {\mc M}_\Gamma \otimes \det D \cong \det \tilde D'\cong \det \tilde D.
\eeqn
Indeed, for each vertex $v_\alpha \in V_{\Gamma}$, let $D_\alpha: {\mc B}_\alpha \to {\mc E}_\alpha$ be the linearization of the nonlinear Cauchy--Riemann equation or the affine vortex equation depending on the position of the vertex $v_\alpha$. The orientation and the spin structure on $L$ equip a natural orientation class $\sigma_\alpha \in \det D_\alpha$ (see Subsection \ref{subsection27}). On the other hand, for each boundary edge $e \in E_\Gamma$ whose length is $\ell(e)$, the perturbation datum $P_\Gamma$ induces a perturbed negative gradient flow of times for a time interval $[0, \ell(e)]$. The time-$\ell(e)$ map is a diffeomorphism $\varphi_e: L \to L$ whose graph
\beqn
\Delta_e \subset L \times L
\eeqn
is isotopic to the diagonal. Let the normal bundle to $\Delta_e$ be $N_e$. Suppose $e$ connects vertices $v_{\alpha}$ and $v_{\beta}$ where $v_\beta$ is closer to the output. Then the matching condition is 
\beqn
(u_\alpha(z_\alpha), u_\beta(z_\beta)) \in \Delta_e. 
\eeqn
Then the total linearization at ${\mc V}$ (without deforming ${\mc C}$) is equivalent to 
\beqn
D: \bigoplus_{v_\alpha \in V_\Gamma}  \to \bigoplus_{v_\alpha \in V_\Gamma} {\mc E}_\alpha \oplus \bigoplus_{e \in E_\Gamma} N_e \oplus \bigoplus_{i=1}^k W_{{\bm x}_i}^+ \oplus W_{{\bm x}_\infty}^-.
\eeqn
Since the summands $N_e$, $W_{{\bm x}_i}^+$, and $W_{{\bm x}_\infty}^-$ are all finite-dimensional, the operator $D$ is homotopic through Fredholm operators to the direct sum of all $D_\alpha$ and the zero operators to these finite-dimensional vector spaces. Moreover, $\det D_\alpha$, $\det N_e$, $\det W_{{\bm x}_i}^+$, and $\det W_{{\bm x}_\infty}^-$ all have preferred positive orientation classes. Hence one can orient $D$ using these preferred orientation classes on components by specifically ordering elements in $V_\Gamma \cup E_\Gamma \cup T_\Gamma$.

We define the order as follows. First suppose $\Gamma$ has only one vertex $v$ with output $t_{\rm out}$ and inputs $t_1, \ldots, t_k\in T_\Gamma$. Then we order $V_\Gamma \cup E_\Gamma \cup T_\Gamma$ as 
\beqn
t_{\rm out}, v, t_1, \ldots, t_k.
\eeqn
More generally, suppose we have defined the order of elements for all ribbon trees with at most $m$ vertices. Then for $\Gamma$ with $m+1$ vertices, let $\Gamma_1, \ldots, \Gamma_s$ be the connected components of the complement of the output $t_{\rm out}$ and the root $v_\infty$. Then we order elements of $V_\Gamma \cup E_\Gamma \cup T_\Gamma$ as 
\beqn
t_{\rm out}, v_\infty, \underbrace{a_1, \ldots, a_{r_1}}_{V_{\Gamma_1} \cup E_{\Gamma_1} \cup T_{\Gamma_1}}, \cdots, \underbrace{ a_1, \ldots, a_{r_s}}_{V_{\Gamma_s}\cup E_{\Gamma_s} \cup T_{\Gamma_s}}.
\eeqn
See Figure \ref{ordering} for an example. 

\begin{figure}[ht]
    \centering
    \includegraphics{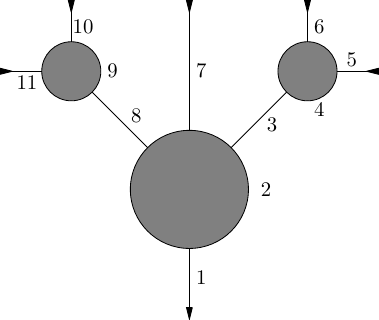}
    \caption{Ordering vertices, edges, and tails of a ribbon tree.}
    \label{ordering}
\end{figure}

With the order of tree elements specified above, we define the positive orientation class of the linearized operator as follows. Suppose $V_\Gamma \cup E_\Gamma \cup T_\Gamma$ has $n_\Gamma$ elements and $\sigma_i$ is the positive orientation class of the $i$-th element. Then define 
\beqn
\sigma_D:= (-1)^{|{\bm x}_\infty|(n-|{\bm x}_\infty|)} \bigwedge_{i=1}^{n_\Gamma} \sigma_i\in \det D.
\eeqn
Here $|{\bm x}_\infty| = {\rm dim} W_{{\bm x}_\infty}^+$ and the sign $(-1)^{|{\bm x}_\infty|(n-|{\bm x}_\infty|)}$ is to correct the orientation class $\sigma_{{\bm x}_\infty}^- \in \det T_{{\bm x}_\infty} W_{{\bm x}_\infty}^-$ as the orientation class one needs is the orientation of the normal bundle to the stable manifold $W_{{\bm x}_\infty}^+$.  With the orientation of the moduli space of domains, one obtains an orientation of the moduli space ${\mc M}_{\bm \Gamma}^*$. When ${\bm \Gamma}$ is an essential map type with index zero, ${\mc M}_{\bm \Gamma}^*$ is discrete and each element is equipped with a sign. 

\subsection{Fake and true boundaries of one-dimensional moduli spaces}\label{subsection67}

In this section we study the boundary of the moduli spaces. Consider any essential map type ${{\bm \Gamma}}$ with ${\rm index} {\bm \Gamma} \in \{0, 1\}$. Proposition \ref{prop611} gives a description of the compactification of the moduli space ${\mc M}_{\bm \Gamma}^*$ (from now on we omit the perturbation $P_\Gamma$ we have chosen from the notations). 

\begin{lemma}\label{lemma614}
Suppose ${{\bm \Gamma}}$ is an essential map type with ${\rm index} {{\bm \Gamma}} = 1$. Then $\ov{ {\mc M}_{{\bm \Gamma}}^*}$ is a compact one-dimensional manifold with boundary. Moreover, the following are true. 
\begin{enumerate}
\item If ${{\bm \Gamma}}$ is of scale $0$, then
\beqn
\partial \ov{ {\mc M}_{{\bm \Gamma}}^* } = \partial_{\rm T1} \ov{   {\mc M}_{{\bm \Gamma}}^* } \sqcup \partial_{\rm F1} \ov{ {\mc M}_{{\bm \Gamma}}^* } \sqcup \partial_{\rm F2} \ov{ {\mc
    M}_{{\bm \Gamma}}^*}.  \eeqn

\item If ${{\bm \Gamma}}$ is of scale $\infty$, then 
\beqn
\partial \ov{ {\mc M}_{{\bm \Gamma}}^* } = \partial_{\rm T2} \ov{ {\mc M}_{{\bm \Gamma}}^* } \sqcup \partial_{\rm F1} \ov{ {\mc M}_{{\bm \Gamma}}^* } \sqcup \partial_{\rm F2} \ov{ {\mc M}_{{\bm \Gamma}}^* }. 
\eeqn

\item If ${{\bm \Gamma}}$ is of scale 1, then 
\beqn
\partial \ov{ {\mc M}_{{\bm \Gamma}}^* } = \partial_{\rm T1} \ov{ {\mc M}_{{\bm \Gamma}}^* } \sqcup \partial_{\rm T2} \ov{ {\mc M}_{{\bm \Gamma}}^* } \sqcup  \partial_{\rm F1} \ov{ {\mc M}_{{\bm \Gamma}}^* } \sqcup \partial_{\rm F2} \ov{ {\mc M}_{{\bm \Gamma}}^* } \sqcup \partial_{\rm F3} \ov{ {\mc M}_{{\bm \Gamma}}^* } \sqcup \partial_{\rm F4} \ov{ {\mc M}_{{\bm \Gamma}}^* }. 
\eeqn
\end{enumerate}
Here $\partial_{\rm T1} \ov{{\mc M}_{{\bm \Gamma}}^*}$ etc. are the boundary components corresponding to the map types listed in Proposition \ref{prop611}. 
\end{lemma}

\begin{proof}
  Transversality implies that the interior of
  $\ov{ {\mc M}_{{\bm \Gamma}}^* }$ is a one-dimensional
  manifold. Proposition \ref{prop611} implies that the boundary has
  the above description. It remains to put boundary charts on
  $\ov{ {\mc M}_{{\bm \Gamma}}^* }$. The existence of these charts is
  the corollary of various gluing results. Here gluing of holomorphic disks (see for example \cite{FOOO_Book}) and gluing of gradient flow lines are standard (see for example \cite{Schwarz_book}), which give boundary charts near type (F2), type (T1), and type (T2) boundary components. For type (F1) and type (F3) boundary components, using the transversality, one can also put boundary charts parametrized by the lengths of the edges that are shrunk to zero. Lastly, for type (F4) boundary components, one can put boundary charts using the gluing result of \cite{Xu_glue}.
\end{proof}

The difference between true and fake boundary components is explained as follows.  For two different essential map type
${{\bm \Gamma}}_1$ and ${{\bm \Gamma}}_2$ of index $1$, the
compactified moduli spaces $\ov{ {\mc M}_{{{\bm \Gamma}}_1}^*}$ and
$\ov{ {\mc M}_{{{\bm \Gamma}}_2}^*}$ may share common boundary
components. Indeed,
$\partial_{\rm T1} \ov{ {\mc M}_{{{\bm \Gamma}}}^*}$ and $\partial_{\rm T2} \ov{ {\mc M}_{{\bm \Gamma}}^* }$ are components
that cannot be glued in two different ways and are called {\it true boundaries}; other boundary components in the
list of Proposition \ref{prop611} can be glued in two different ways
and are called {\it fake boundaries} (as indicated by their
labellings). More precisely, by using two different kinds of gluings,
the fake boundary points are the ends of one-dimensional components of
the moduli space in two different ways. The various one-dimensional
strata glue together to a topological one-manifold whose boundary is
the union of true boundary points:

\begin{lemma}\label{lemma615}
For each essential map type ${{\bm \Gamma}}$ with ${\rm index} {{\bm \Gamma}} = 1$ and $i \in \{1, \ldots, 4\}$, then there exists exactly one essential map type ${{\bm \Gamma}}_i$
which has the same number of leaves and the same number of inputs labelled by the same sequence of critical points, and which satisfies
\beqn
\partial_{\rm Fi} \ov{ {\mc M}_{{\bm \Gamma}}^* } = \left\{ \begin{array}{cc} \partial_{\rm Fi+1} \ov{ {\mc M}_{ {{\bm \Gamma}}_l}^*},\ &\ {\rm if\ } i = 1, 3;\\
\partial_{\rm Fi-1} \ov{ {\mc M}_{{{\bm \Gamma}}_l}^*},\ &\ {\rm if\ } i = 2,4.
\end{array} \right.
\eeqn
Moreover, if we take the union $\ov{ {\mc M}_{{\bm \Gamma}}^* } \cup \ov{ {\mc M}_{{{\bm \Gamma}}_i}^* }$ by identifying their common boundaries, then the union is an oriented one-dimensional manifold (with or without boundary) whose orientation agrees with the orientations on $\ov{ {\mc M}_{{\bm \Gamma}}^*}$ and $\ov{ {\mc M}_{{{\bm \Gamma}}_i}^* }$. 
\end{lemma}

\begin{proof} 
Up to orientation the statement is a special case of Lemma \ref{lemma614}. Since the way we orient the linearized operators for both holomorphic disks and affine vortices over ${\bm H}$ is coherent with respect to gluing (see Remark \ref{coherence}), over the common fake boundary of two different moduli spaces, the induced boundary orientations are exactly opposite. Therefore the union $\ov{{\mc M}_{\bm \Gamma}^*} \cup \ov{{\mc M}_{{\bm \Gamma}_i}^*}$ is still oriented.
\end{proof} 

We introduce the following notation for moduli spaces.  or each $B \in H_2^K (V, L_V; {\mb Z})$, $i \in \{0, 1\}$, and $\lambda \in \{0, 1, \infty\}$, let
\beqn
\ov{ {\mc M} ( B, {\bf X})_i^\lambda},
\eeqn
be the union of all ${\mc M}_{{\bm \Gamma}}^*$ for all essential map types ${{\bm \Gamma}}$ whose inputs and outputs labelled by ${\bf X} = ({\bm x}_1, \ldots, {\bm x}_l; {\bm x}_\infty)$, whose index is $i$, and whose total homology class is $B$. Notice that when $i = 0$, this is a disjoint union of discrete points; when $i = 1$, we identify boundary components using Lemma \ref{lemma615}. Since the class $B$ fixes the total energy and hence the number of interior markings, $\ov{ {\mc M}( B, {\bf X} )_i^\lambda}$ consists of finitely many different essential map types. By the compactness result for fixed domain type (Theorem \ref{thm57}), $\ov{{\mc M}(B, {\bf X})_i^\lambda}$ is compact $i$-dimensional manifold with boundary. When $i = 1$, we can write 
\beq\label{eqn612}
\partial \ov{ {\mc M} (  B, {\bf X} )_1^\lambda} = \partial_{{\rm T1}} \ov{ {\mc M} ( B, {\bf X} )_1^\lambda}  \sqcup \partial_{{\rm T2}} \ov{ {\mc M} ( B,  {\bf X})_1^\lambda}.
\eeq
The two components correspond to ``breaking upstairs'' and ``breaking downstairs.''

\section{Fukaya Algebras}\label{section7}

In this section we wrap up all the previous analytical constructions
and discuss their algebraic implications. In particular we complete
the definitions of the $A_\infty$ algebras $\fuk^0 (L)$,
$\fuk^\infty (L)$ and the $A_\infty$ morphism between them. We also
prove certain properties of these structures.

\subsection{\texorpdfstring{$A_\infty$}{Lg} algebras}\label{subsection71}

 Recall that ${\bm \Lambda}$ is the Novikov field of formal Laurent series in ${\bm q}$
\beqn
{\bm \Lambda} = \Set{ \sum_{i \geq i_0} a_i {\bm q}^i\ | \  a_i \in {\mb C} },
\eeqn
and ${\bm \Lambda}_0$ resp. ${\bm \Lambda}_+$ are the subring of series with only nonnegative resp. positive powers. Let ${\bm \Lambda}^* \subset {\bm \Lambda}_0$ be the group of invertible  elements. 

We recall the basic notions about $A_\infty$ algebras over the Novikov field. Let $\uds{A}$ be a finite-dimensional ${\mb Z}$-graded ${\mb C}$-vector space, and denote $A = \uds A \otimes {\bm \Lambda}$. The grading of a homogeneous element of $a \in A$ is denoted by $|a|\in {\mb Z}$. For $k \in {\mb Z}$, $A[k]$ is the shifted graded module whose degree $l$ component is $A[k]^l = A^{k+l}$.

\begin{defn}\label{defn71}
\begin{enumerate}

\item A (curved) $A_\infty$ algebra structure over $A$ consists of a sequence of graded, ${\bm \Lambda}$-linear maps
\beqn
{\bm m}_l: \overbrace{ A \underset{\bm \Lambda}{\otimes} \cdots \underset{\bm \Lambda}{\otimes} A}^l \to A[2- l],\ l \geq 0
\eeqn
satisfying the $A_\infty$ relation: for any $l \geq 1$ and a $l$-tuple $({\bm a}_1, \ldots, {\bm a}_{l})$ of homogeneous elements of $A$, denoting $n_j  = |{\bm a}_1| + \cdots + |{\bm a}_j| + j$, then
\beq \label{ainftyrel}
0 = \sum_{r = 0}^l \sum_{j = 0}^{l - r} (-1)^{n_j } \m_{l-r+1} \big( {\bm a}_1, \ldots, {\bm a}_j, \m_r \big( {\bm a}_{j + 1}, \ldots, {\bm a}_{j + r} \big), {\bm a}_{j+ r+1}, \ldots, {\bm a}_l \big).
\eeq

\item A {\it  strict unit} of an $A_\infty$ algebra is an element ${\bm e} \in A$ of degree zero satisfying
\beqn
\m_2( {\bm e},{\bm a}) = (-1)^{|{\bm a}|} \m_2 ({\bm a}, {\bm e})  = {\bm a},\ \forall {\bm a} \in A;
\eeqn
\beqn
\m_l ( {\bm a}_1, \ldots, {\bm a}_{j - 1}, {\bm e}, {\bm a}_{j + 1}, \ldots, {\bm a}_l ) = 0,\ \forall l \geq 2,\ {\bm a}_1, \ldots, {\bm a}_{j-1}, {\bm a}_{j+1}, \ldots, {\bm a}_l \in A.
\eeqn
An $A_\infty$ algebra with a strict unit is called a {\it  strict unital} $A_\infty$ algebra.
\end{enumerate}
\end{defn}

\subsubsection{Maurer--Cartan equation and potential function}

Denote $A_+ = \uds A \otimes {\bm \Lambda}_+$. The {\it  weak Maurer--Cartan} equation for an $A_\infty$ algebra ${\mc A}$ reads 
\beqn 
{\bm m}_0^{\bm b}(1):=\sum_{l \geq 0} \m_{l} ( {\bm b}, \ldots, {\bm b} ) \equiv 0\ {\rm mod}\ {\bm \Lambda} {\bm e},\ \ {\rm where}\ {\bm b} \in A_+ \cap
A^{\rm odd}.  
\eeqn 
A solution ${\bm b}$ is called a {\it  weakly bounding cochain}. Let $MC ({\mc A})$ be the set of all weakly bounding cochains. Define the {\it  potential function} ${\bf W}_{\mc A}: MC ({\mc A}) \to {\bm \Lambda}$ by 
\beqn 
{\bf W}_{\mc A} ( {\bm b} ) {\bm e} = \m_0^{\bm b}(1).  
\eeqn

\subsubsection{$A_\infty$ morphisms}

\begin{defn}\label{defn72}
Let ${\mc A}_i = (A, {\bm m}_{i, 0}, {\bm m}_{i, 1}, \ldots)$, $i = 1, 2$ be two $A_\infty$ algebras over $A$. 
\begin{enumerate}
\item 
A {\it  morphism} from ${\mc A}_1$ to ${\mc A}_2$ is a collection of ${\bm \Lambda}$-linear maps 
\beqn
{\bm \kappa} = \Big( \kappa_l: \overbrace{ A \underset{\bm \Lambda}{\otimes} \cdots \underset{\bm \Lambda}{\otimes} A }^{l} \to A[1- l],\ l = 0, 1, \ldots \Big)
\eeqn
satisfying the $A_\infty$ relation
\begin{multline}\label{eqn72}
0 = \sum_{i + j \leq l} (-1)^{n_j} \kappa_{l - i + 1} \big( {\bm a}_1, \ldots, {\bm a}_j, {\bm m}_{1, i}\big( {\bm a}_{j+1}, \ldots, {\bm a}_{j+i} \big), {\bm a}_{j+i+1}, \ldots, {\bm a}_{l} \big) \\
+ \sum_{r \geq 1} \sum_{ i_1, \ldots, i_r \geq 0}^{ 
i_1 + \cdots + i_r = l} {\bm m}_{2, r} \big( \kappa_{i_1} \big( {\bm a}_1, \ldots, {\bm a}_{i_1} \big), \ldots, \kappa_{i_r} \big( {\bm a}_{l - i_r +1}, \ldots, {\bm a}_l \big) \big).
\end{multline}

\item If ${\mc A}_1$ and ${\mc A}_2$ are strictly unital with strict units ${\bm e}_1, {\bm e}_2$, then a morphism ${\bm \kappa}: {\mc A}_1 \to {\mc A}_2$ is called {\it  unital} if 
\beqn
\kappa_1({\bm e}_1) = {\bm e}_2,\ \kappa_l \big( {\bm a}_1, \ldots, {\bm a}_i, {\bm e}_1, {\bm a}_{i+2}, \ldots, {\bm a}_l \big) = 0,\ \forall l \neq 1,\ 0 \leq i \leq l - 1.
\eeqn

\end{enumerate}
\end{defn}

From now on we make the following assumption in the abstract discussion of $A_\infty$ algebras. We assume that ${\mc A}_1, {\mc A}_2$ be two $A_\infty$ algebras over the same finite-dimensional ${\bm \Lambda}$-module $A = \uds A \otimes {\bm \Lambda}$ with a common strict unit ${\bm e} \in \uds A \subset A$ and ${\bm \kappa}: {\mc A}_1 \to {\mc A}_2$ be a unital $A_\infty$ morphism. An $A_\infty$ morphism ${\bm \kappa}: {\mc A}_1 \to {\mc A}_2$ is called a {\it  higher order deformation of the identity}, if 
\begin{align*}
&\ \kappa_1 ({\bm a}) - {\bm a} \in A_+,\ \forall {\bm a} \in \uds A;\ &\ \kappa_l \big( \overbrace{\uds A \underset{\mb C}{\otimes} \cdots \underset{\mb C}{\otimes} \uds A }^{l} \big) \subset A_+,\ \forall l  \neq 1.
\end{align*}

\begin{lemma}\label{lemma73}
Suppose ${\bm \kappa}: {\mc A}_1 \to {\mc A}_2$ is a higher order deformation of the identity. Then there is a well-defined map 
\beq\label{eqn73}
\uds \kappa: A_+ \to A_+,\ \uds \kappa( {\bm a}) = \sum_{l \geq 0} \kappa_{l}( \underbrace{{\bm a}, \ldots, {\bm a}}_{l} ).
\eeq
Moreover, $\kappa_1: A \to A$ is a linear isomorphism and $\uds \kappa: A_+ \to A_+$ is a bijection.
\end{lemma}

\begin{proof}
The convergence of $\uds \kappa$ follows from the definition. $\kappa_1$ is invertible because it differs from the identity by higher order terms. $\uds \kappa$ is bijective because we may always solve $\uds \kappa ({\bm a}) = {\bm b}$ order by order. 
\end{proof}

Now we consider how $A_\infty$ morphisms interact with potential functions. 

\begin{prop}\label{prop74}
Let ${\mc A}_1, {\mc A}_2$ be $A_\infty$ algebras over the same ${\bm \Lambda}$-module $A$ with a common strict unit ${\bm e}$. Let ${\bm \kappa}: {\mc A}_1 \to {\mc A}_2$ be a unital $A_\infty$ morphism that is a higher order deformation of the identity. Then the map $\uds \kappa: A_+ \to A_+$ defined by \eqref{eqn73} maps $MC ({\mc A}_1)$ into $MC({\mc A}_2)$. Moreover one has
\beqn
{\bf W}_{{\mc A}_1} ({\bm b}_1) =  {\bf W}_{{\mc A}_2} (\uds \kappa ({\bm b}_1)),\ \forall {\bm b}_1 \in MC({\mc A}_1).
\eeqn
\end{prop}

\begin{proof}
The map $\uds \kappa$ sends odd degree elements to odd degree elements. Moreover, by the $A_\infty$ axiom of the morphism ${\bm \kappa}$ and its unitality, one has
\beqn
\doublespacing
\begin{split}
 &\ \sum_{l \geq 0} {\bm m}_{2, l} \big( \uds \kappa ( {\bm b}_1), \ldots, \uds \kappa ({\bm b}_1) \big) \\
= &\ \sum_{l \geq 0} \sum_{i_1, \ldots, i_l \geq 0} {\bm m}_{2, l} \big( \kappa_{i_1}( {\bm b}_1, \ldots, {\bm b}_1 ), \ldots, \kappa_{i_l}( {\bm b}_1, \ldots, {\bm b}_1) \big)\\
= &\ \sum_{r \geq 0} \sum_{l \geq 0} \sum_{i_1, \ldots, i_l \geq 0}^{ i_1 + \cdots + i_l = r} {\bm m}_{2, l} \big( \kappa_{i_1}( {\bm b}_1, \ldots, {\bm b}_1 ), \ldots, \kappa_{i_l}(  {\bm b}_1 , \ldots, {\bm b}_1   ) \big) \\
= &\ \sum_{r \geq 0} \sum_{s = 0}^{r} \sum_{j = 0}^{r - s} (-1)^{n_j} \kappa_{ r - s + 1} \big( \underbrace{ {\bm b}_1, \ldots, {\bm b}_1}_{j}, {\bm m}_{1, s} ( {\bm b}_1, \ldots, {\bm b}_1), \underbrace{ {\bm b}_1, \ldots, {\bm b}_1}_{r - s - j} \big)\\
= &\ \sum_{j \geq 0} (-1)^{n_j} \sum_{l \geq j}	\kappa_{l + 1	} \big( \underbrace{{\bm b}_1, \ldots, {\bm b}_1 }_{j}, \sum_{s \geq 0} {\bm m}_{1, s}( {\bm b}_1, \ldots, {\bm b}_1), \underbrace{{\bm b}_1, \ldots, {\bm b}_1}_{l - j} \big) \\
= &\ \sum_{j \geq 0} (-1)^{n_j} \kappa_{l + 1} \big( \underbrace{ {\bm b}_1, \ldots, {\bm b}_1}_{j}, {\bf W}_{{\mc A}_1} ({\bm b}_1 ) {\bm e}, \underbrace{ {\bm b}_1, \ldots, {\bm b}_1 }_{l - j} 	\big)\\
= &\ \kappa_1 \big( {\bf W}_{{\mc A}_1} ({\bm b}_1) {\bm e} \big)\\
= &\ {\bf W}_{{\mc A}_1} ({\bm b}_1) {\bm e}.
\end{split}
\eeqn
This finishes the proof. 
\end{proof}

\subsubsection{Deformations and cohomology}

We continue with the assumption that ${\mc A}$ is an $A_\infty$ algebra over $A$ with a strict unit ${\bm e}$. Suppose ${\bm b} \in MC({\mc A})$. Define 
\beqn
{\bm m}_l^{\bm b} ({\bm a}_1, \ldots, {\bm a}_l ) = \sum_{i_0, \ldots, i_l} {\bm m}_{l + i_0 + \cdots + i_l} ( \underbrace{{\bm b}, \ldots, {\bm b}}_{i_0}, {\bm a}_1, \underbrace{ {\bm b}, \ldots, {\bm b}}_{i_1}, \cdots, {\bm a}_l, \underbrace{{\bm b}, \ldots, {\bm b}}_{i_l} ).
\eeqn
Similarly, suppose ${\bm \kappa}:{\mc A}_1 \to {\mc A}_2$ is a unital
$A_\infty$ morphism that is a higher order deformation of the
identity, and ${\bm b}_1 \in MC({\mc A}_1)$. Define 
\beqn \kappa_l^{\bm b} ({\bm a}_1, \ldots, {\bm a}_l) = \sum_{i_0,
  \ldots, i_l} \kappa_{l + i_0 + \cdots + i_l} ( \underbrace{{\bm b},
  \ldots, {\bm b}}_{i_0}, {\bm a}_1, \underbrace{ {\bm b}, \ldots,
  {\bm b}}_{i_1}, \cdots, {\bm a}_l, \underbrace{{\bm b}, \ldots, {\bm
    b}}_{i_l} ).  \eeqn

\begin{lemma}\label{lemma75}
Let ${\mc A}_1$ and ${\mc A}_2$ be $A_\infty$ algebras over $A$. Let ${\bm \kappa}: {\mc A}_1 \to {\mc A}_2$ be an $A_\infty$ morphism that is a higher order deformation of the identity. Let ${\bm b}_1 \in MC({\mc A}_1)$ be a weakly bounding cochain and denote ${\bm b}_2= \uds \kappa ({\bm b}_1) \in MC({\mc A}_2)$. 
\begin{enumerate}

\item For $i = 1, 2$, the maps ${\bm m}_{i, 0}^{{\bm b}_i},  {\bm m}_{i, 1}^{{\bm b}_i}, \ldots$ define a unital $A_\infty$ algebra ${\mc A}_i^{{\bm b}_i}$ over $A$. 

\item The maps $\kappa_{0}^{{\bm b}_1}, \kappa_1^{{\bm b}_1}, \ldots$ define a unital $A_\infty$ morphism ${\bm \kappa}^{{\bm b}_1}: {\mc A}_1^{{\bm b}_1} \to {\mc A}_2^{{\bm b}_2}$. 

\item For $i = 1, 2$, one has ${\bm m}_{i, 1}^{{\bm b}_i} \circ {\bm m}_{i, 1}^{{\bm b}_i} = 0$; moreover, $\kappa_1^{{\bm b}_1}$ is an isomorphism of cochain complexes $\kappa_1^{{\bm b}_1}: ( A, {\bm m}_{1, 1}^{{\bm b}_1} ) \cong ( A, {\bm m}_{2, 1}^{{\bm b}_2} )$ and hence induces an isomorphism in cohomology
\beqn
H({\bm \kappa}, {\bm b}_1): H({\mc A}_1, {\bm b}_1) \cong H({\mc A}_2, {\bm b}_2). 
\eeqn
\end{enumerate}
\end{lemma}

\begin{proof}
It is routine to check the first two items. For the third item, by the $A_\infty$ axiom and the fact that ${\bm b}_i \in MC({\mc A}_i)$, one sees that
\beqn
\m_{i,1}^{{\bm b}_i} \big( \m_{i,1}^{{\bm b}_i} ({\bm a}) \big) = \m_{i, 2} \big( \m_{i,0}^{{\bm b}_i} (1), {\bm a} \big) - (-1)^{|{\bm a}|} \m_{i,2}\big( {\bm a}, \m_{i,0}^{{\bm b}_i} (1) \big) = 0.
\eeqn
Hence $(A, \m_{i,1}^{{\bm b}_i} )$ is a cochain complex. Further, by the $A_\infty$ relations, we have (since the degree of ${\bm b}_1$ is odd, the signs in the \eqref{eqn72} are all positive)
\beqn
\begin{split}
&\ \kappa_1^{{\bm b}_1} \big( \m_{1, 1}^{{\bm b}_1} ({\bm a}) \big) \\
= &\ \sum_{i_0, i_1\geq 0} \kappa_{i_0 + i_1 + 1} \big( \underbrace{{\bm b}_1, \ldots, {\bm b}_1}_{i_0}, \sum_{j_0, j_1 \geq 0} \m_{1, j_0 + j_1 + 1} \big( \underbrace{{\bm b}_1, \ldots, {\bm b}_1}_{j_0}, {\bm a}, \underbrace{ {\bm b}_1, \ldots, {\bm b}_1}_{j_1} \big), \underbrace{{\bm b}_1, \ldots, {\bm b}_1}_{i_1} \big)  \\
= &\ - \sum_{r_0, r_1, r_2 \geq 0} \sum_{j_0\geq 0} \kappa_{ r_0 + r_1 + r_2 + 1} \big( \underbrace{{\bm b}_1, \ldots, {\bm b}_1}_{r_0},  \m_{1, j_0} \big( {\bm b}_1, \ldots, {\bm b}_1 \big), \underbrace{{\bm b}_1, \ldots, {\bm b}_1}_{ r_1}, {\bm a}, \underbrace{{\bm b}_1, \ldots, {\bm b}_1}_{r_2} \big) \big) \\
&\ - \sum_{s_0, s_1, s_2 \geq 0} \sum_{j_1\geq 0} \kappa_{ s_0 +s_1 + s_2 + 1}  \big( \underbrace{{\bm b}_1, \ldots, {\bm b}_1}_{s_0}, {\bm a}, \underbrace{{\bm b}_1, \ldots, {\bm b}_1}_{s_1}, \m_{1, j_1} \big(  {\bm b}_1, \ldots, {\bm b}_1 \big), \underbrace{{\bm b}_1, \ldots, {\bm b}_1}_{s_2} \big) \\
 &\ + \sum_{r, s \geq 0} \sum_{i_1, \ldots, i_r \geq 0} \sum_{j_1, \ldots, j_s \geq 0}  \m_{2, r + s + 1} \big( \kappa_{i_1}(  {\bm b}_1, \ldots, {\bm b}_1 ), \ldots, \kappa_{i_{r}} (  {\bm b}_1, \ldots, {\bm b}_1  ) \big.\\
&\ \big. , \sum_{l_0, l_1 \geq 0} \kappa_{l_0 + l_1 + 1} \big( \underbrace{{\bm b}_1, \ldots, {\bm b}_1}_{l_0}, {\bm a}, \underbrace{{\bm b}_1, \ldots, {\bm b}_1}_{l_1} \big), \kappa_{j_0} (  {\bm b}_1, \ldots, {\bm b}_1) , \ldots, \kappa_{j_s} (  {\bm b}_1, \ldots, {\bm b}_1  ) \big) \\
= &\ - \sum_{r_0, r_1, r_2 \geq 0} \kappa_{r_0 + r_1 + r_2 + 1} \big( \underbrace{{\bm b}_1, \ldots, {\bm b}_1}_{r_0},  {\bf W}_{{\mc A}_1} ({\bm b}_1) {\bm e}_1, \underbrace{{\bm b}_1, \ldots, {\bm b}_1}_{r_1}, {\bm a}, \underbrace{{\bm b}_1, \ldots, {\bm b}_1}_{r_2} \big) \big) \\
&\ - \sum_{s_0, s_1, s_2 \geq 0} \kappa_{ s_0 + s_1 + s_2 + 1}  \big( \underbrace{{\bm b}_1, \ldots, {\bm b}_1}_{s_0}, {\bm a}, \underbrace{{\bm b}_1, \ldots, {\bm b}_1}_{s_1}, {\bf W}_{{\mc A}_1} ({\bm b}_1) {\bm e}_1, \underbrace{{\bm b}_1, \ldots, {\bm b}_1}_{s_2} \big) \\
&\ +    \sum_{r, s \geq 0} \m_{2, r + s + 1} \big( \underbrace{{\bm b}_2, \ldots, {\bm b}_2}_{r}, \kappa_1^{{\bm b}_1} ( {\bm a}), \underbrace{{\bm b}_2, \ldots, {\bm b}_2}_{s} \big)\\
= &\ \m_{2, 1}^{{\bm b}_2} \big( \kappa_1^{{\bm b}_1}({\bm a}) \big).
\end{split}
\eeqn
Here to obtain the last equality we used the unitality of ${\bm \kappa}$. Hence $\kappa_1^{{\bm b}_1}$ is a chain map. It is an isomorphism because it is an invertible linear map (see Lemma \ref{lemma73}).
\end{proof}

\subsection{The Fukaya algebra on zero scale}\label{subsection72}

Define compositions and morphisms on the Floer cochain group as follows. The space of Floer cochains is  the ${\mb C}$-vector space
\beqn
CF^*(L) = \bigoplus_{{\bm x}\in {\rm Crit} F_L} {\mb C} \{{\bm x}\}
\eeqn
and
is graded by $| {\bm x} | = {\rm dim} L - {\rm index} {\bm x}$.

Orientations lead to an assignment of sign to each rigid vortex. Recall from Subsection \ref{subsection27} moduli spaces of treed scaled vortices are oriented. Then for each essential map type ${{\bm \Gamma}}$ of index zero, each element ${\mc V} \in {\mc M}_{{{\bm \Gamma}}}^* (P_\Gamma)$ is equipped with a sign ${\bf Sign}({\mc V}) \in \{\pm 1\}$. 

The local systems affect the counts in the following way. We assume we have chosen a ${\bm \Lambda}$-local system on $L$, that is, a flat ${\bm \Lambda}^*$-bundle over $L$ with structure group being the group of units ${\bm \Lambda}^*$ of $\Lambda$. We use $b \in H^1(L; {\bm \Lambda}_0)$ to represents a local system hence for any curve class $B$, one has the associated holonomy \beqn e^{\langle
  b, \partial B\rangle} \in {\bm \Lambda}.  \eeqn

Combining the weights above, we have a well-defined count of holomorphic disks as follows.  Fix a collection of generators ${\bf X} = ({\bm x}_1, \ldots, {\bm x}_l; {\bm x}_\infty)$, and consider the zero-dimensional moduli space ${\mc M} (B, {\bf X})_0^0$. Define
\beqn
\langle {\bm x}_1, \ldots, {\bm x}_l | {\bm x}_\infty \rangle_B^0 = \widetilde {\#} {\mc M} ( B, {\bf X} )_0^0:= (-1)^{\heartsuit( {\bf X} ) }  \sum_{{\mc V} \in {\mc M} ( B, {\bf X} )_0^0 }  {\bf Sign} ({\mc V}) \cdot e^{\langle b,  \partial B\rangle} \in {\mb C}. 
\eeqn
where 
\beqn
\heartsuit( {\bf X} ) := \sum_{i = 1}^l i|{\bm x}_i| \in {\mb Z};
\eeqn
Here for each equivariant class $B \in H_2^K(V, L_V; {\mb Z})$ , the boundary $\partial B$ denotes the image in $H_1^K(L_V; {\mb Z}) \cong H_1(L; {\mb Z})$. For each $B$, denote
\beqn
k_B = ( k_{1, B}, \ldots, k_{h, B}) = ( {\rm deg} D_1 \cdot \omega(B), \ldots, {\rm deg} D_h \cdot \omega(B))
\eeqn
and 
\beqn
k_B! = k_{1, B}! \cdots k_{h, B}!.
\eeqn
Define 
\beq\label{eqn74}
\langle {\bm x}_1, \ldots, {\bm x}_l | {\bm x}_\infty\rangle^0 = \sum_B \frac{{\bm q}^{\omega(B)}}{ k_B!} \langle {\bm x}_1, \ldots, {\bm x}_l |{\bm x}_\infty \rangle_B^0 \in {\bm \Lambda},
\eeq

The composition maps are defined by adding corrections to the Morse boundary operator. For $l = 1$, define 
\beqn
\m_1^0 ({\bm x}) = (-1)^{|{\bm x}|} \delta_{\rm Morse}({\bm x}) +  \displaystyle \sum_{{\bm x}_\infty \in \crit} \langle{\bm x}| {\bm x}_\infty \rangle^0 \cdot {\bm x}_\infty.
\eeqn
For $l \neq 1$, define
\beqn
\m_l^0 \big( {\bm x}_1, \ldots, {\bm x}_l \big) = \sum_{ {\bm x}_\infty \in \crit} \langle  {\bm x}_1, \ldots, {\bm x}_l | {\bm x}_\infty \rangle^0 \cdot {\bm x}_\infty.
\eeqn
The $l = 1$ case has additional terms coming from Morse
trajectories because by definition an infinite edge is not a stable domain type.

\begin{thm}\label{thm76} 
$\fuk^0 (L):= ( CF^*(L) \otimes {\bm \Lambda}, \m_0^0, \m_1^0, \ldots)$ is an $A_\infty$ algebra. 
\end{thm}

\begin{proof}
  It suffices to verify the $A_\infty$ relation \eqref{ainftyrel} for   all ${\bm a}_i$ being generators ${\bm x}_i$. Then it is equivalent   to verify that all curve class $B$ and critical points   ${\bm x}_\infty$, one has 
\beqn 
\sum_{r = 0}^l \sum_{j = 0}^{l- r}   (-1)^{ n_j} \langle {\bm x}_1, \ldots, {\bm x}_{j -1}, \m_r^0 ( {\bm     x}_{j+1}, \ldots, {\bm x}_{j + r} ), {\bm x}_{j + r + 1}, \cdots,   {\bm x}_l; {\bm x}_\infty \rangle_B^0 = 0.  
\eeqn 
We first verify these relations up to signs. First, the $\lambda = 0$
  case of \eqref{eqn612} implies \beqn
\partial \ov{ {\mc M}(B, {\bf X})_1^0} = \partial_{\rm T1} \ov{ {\mc M}(B, {\bf X})_1^0}.
\eeqn
The configurations in the boundary are those that split along an edge of infinite length. Let $\txt{BU}$ be the set of ``broken upstairs'' pairs $( {\bf X}_1, {\bf X}_2 )$ where for some $j, r$
and ${\bm z} \in \crit$,
\begin{align*}
&\ {\bf X}_1 = ( {\bm x}_{j+1}, \ldots, {\bm x}_{j +  r}; {\bm z}),\ &\ {\bf X}_2 = ( {\bm x}_1, \ldots, {\bm x}_j, {\bm z}, {\bm x}_{j+ r + 1}, \ldots, {\bm x}_l; {\bm x}_\infty). 
\end{align*}
(in the notation $\txt{BU}$ we omit the dependence on ${\bf X}$). Then we have
\beqn
\partial_{\rm T1} \ov{ {\mc M} (B, {\bf X})_1^0} \cong \bigsqcup_{({\bf X}_1, {\bf X}_2 )  \in \txt{BU}} \bigsqcup_{B_1 + B_2 = B} \frac{ k_B!}{ k_{B_1} ! k_{B_2}!} \Big( {\mc M} (B_1, {\bf X}_1)_0^0 \times  {\mc M}(B_2, {\bf X}_2 )_0^0 \Big). 
\eeqn
Then we have (modulo signs)
\begin{multline*}
\sum_{ r =0}^l \sum_{j = 1}^{l -  r} \langle {\bm x}_1, \ldots, {\bm x}_j, {\bm m}_r^0 ( {\bm x}_{j+1}, \ldots, {\bm x}_{j+ r} ), {\bm x}_{j+ r + 1}, \ldots, {\bm x}_l | {\bm x}_\infty \rangle_B^0  \\
= \sum_{( {\bf X}_1, {\bf X}_2 ) \in \txt{BU}} \sum_{B_1 + B_2 = B} \langle {\bm x}_{j+1}, \ldots, {\bm x}_{j+  r} | {\bm z} \rangle_{B_1}^0 \cdot \langle {\bm x}_1, \ldots, {\bm x}_j, {\bm z}, {\bm x}_{j+ r +1}, \ldots, {\bm x}_l |  {\bm x}_\infty \rangle_{B_2}^0 \\
= \sum_{ ( {\bf X}_1,{\bf X}_2) \in \txt{BU}} \sum_{ B_1 + B_2 = B}  \left(  \frac{ \tilde {\#}  {\mc M} ( B_1, {\bf }_1)_0^0}{ k_{B_1} !} \cdot \frac{ \tilde {\#}  {\mc M} ( B_2, {\bf X}_2 )_0^0}{ k_{B_2}!} \right) = \frac{1}{ k_B !} \Big( \tilde {\#} \big( \partial \ov{ {\mc M}  (B, {\bf X})_1^0} \big) \Big) = 0.
\end{multline*}
Indeed, the last identity follows from the fact that the restriction of our perturbation data to a boundary stratum which has one breaking is equal to the product of the perturbation data on the two unbroken parts. The verification of signs can be done following the same argument as in the proof of \cite[Theorem 3.6]{Woodward_toric} by carefully comparing the boundary orientation class and the product orientation class on codimension one boundary strata.
\end{proof}

\subsection{The Fukaya algebra on the infinite scale}\label{subsection73}

The $A_\infty$-relation for the infinite scale Fukaya algebra follows in a similar way. Consider essential map types of scale
$\infty$ with inputs labelled by ${\bf X} = ({\bm x}_1, \ldots, {\bm x}_j; {\bm x}_\infty)$. For each
$B \in H_2^K(V, L_V; {\mb Z})$, consider the zero-dimensional moduli
space ${\mc M}(B, {\bf X})_0^\infty$. Using the orientation on the
moduli space induced from the spin structure, define the counts
\beqn
\langle {\bm x}_1, \ldots, {\bm x}_j | {\bm x}_\infty \rangle_B^\infty:= (-1)^{\heartsuit({\bf X})} \sum_{{\mc V} \in {\mc M}(B, {\bf X})_0^\infty} {\bf Sign}({\mc V}) \cdot e^{\langle b, \partial B\rangle} \in {\mb C}
\eeqn
and the correlators
\beq\label{eqn75}
\langle {\bm x}_1, \ldots, {\bm x}_j |{\bm x}_\infty \rangle^\infty:= \sum_{B} \frac{{\bm q}^{\omega(B)}}{k_B!}  \langle {\bm x}_1, \ldots, {\bm x}_j | {\bm x}_\infty \rangle_B^\infty  \in {\bm \Lambda}.
\eeq
Then we define the compositions ${\bm m}_j^\infty$. For $j = 1$, define
\beqn
{\bm m}_1^\infty ({\bm x}) = (-1)^{|{\bm x}|} \delta_{\rm Morse}({\bm x}) +  \displaystyle \sum_{{\bm x}_\infty} \langle {\bm x} | {\bm x}_\infty \rangle^\infty \cdot {\bm x}_\infty.
\eeqn
For $j \neq 1$, define
\beqn
{\bm m}_j^\infty \big( {\bm x}_1, \ldots, {\bm x}_j \big) = \sum_{{\bm x}_\infty} \langle {\bm x}_1, \ldots, {\bm x}_j | {\bm x}_\infty \rangle^\infty.
\eeqn

\begin{thm}\label{thm77}
$\fuk^\infty (L):= ( CF^* (L) \otimes {\bm \Lambda}, {\bm m}_0^\infty, {\bm m}_1^\infty, \ldots )$ is an $A_\infty$ algebra over ${\bm \Lambda}$. 
\end{thm}

\begin{proof}
The proof is similar to Theorem \ref{thm76} and left to the reader (the verification of signs is exactly the same).
\end{proof}

\begin{rem}\label{rem78}
  We remark why $\fuk^\infty(L)$ can be viewed as what
  Fukaya-Oh-Ohta-Ono \cite{FOOO_Book} call the {\it bulk-deformed Fukaya algebra} $\fuk(L; {\mf c})$ for a certain cycle
  ${\mf c} \in C^*( X; {\bm \Lambda}_+)$. Indeed, by considering
  domain-dependent almost complex structures over universal curves for
  base-free trees, one can prove transversality for vortices over ${\bm C}$ for generic
    perturbations. Then for each curve class
  $B \in H_2^K(V; {\mb Z})$ with $\omega(B) \geq 0$, there is only one
  essential base-free map type ${\bm \Gamma}_B$. The transversality
  result implies that the moduli space ${\mc M}_{{\bm \Gamma}_B}^*$ is
  a smooth oriented manifold of dimension ${\rm dim} X + 2 c_1(B) - 2$
  and the evaluation map 
\beqn \ev: {\mc M}_{{\bm \Gamma}_B}^* \to X
  \eeqn 
  is an oriented pseudocycle. One may define   a Novikov-weighted cycle 
\beqn 
{\mf c} = \sum_{B \in H_2^K ( V; {\mb Z})} \frac{ {\bm q}^{\omega(B)}}{ k_B!}    ( \ev( {\mc M}_{{\bm \Gamma}_B}^* ) )   \eeqn 
whose homology class is independent of all choices. 
 Then formally one may view $\fuk{}^\infty(L)$ as the Fukaya algebra
  deformed by ${\mf c}$.
\end{rem}

\subsection{The open quantum Kirwan map}\label{subsection74}

Now we define a morphism between the two Fukaya algebras constructed
above.  By Definition \ref{defn72} an $A_\infty$ morphism is given by a sequence of maps 
\beqn
\kappa_l: (CF^*(L)\otimes {\bm \Lambda})^{\otimes l} \to CF^* (L) \otimes {\bm \Lambda},\ l = 0, 1, \ldots.
\eeqn
The coefficients of these maps for 
${\bf X} = ( {\bm x}_1, \ldots, {\bm x}_l; {\bm x}_\infty)$ are given as follows. For each
$B \in H_2^K(V, L_V; {\mb Z})$, consider the zero-dimensional
moduli space ${\mc M}(B; {\bf X})_0^1$. Define 
\beqn
\langle {\bm x}_1, \ldots, {\bm x}_l | {\bm x}_\infty \rangle_k^1:= (-1)^{\heartsuit({\bf X})} \sum_{{\mc V}
  \in {\mc M}(B, {\bf X})_0^1} {\bf Sign}({\mc V}) \cdot e^{\langle
  b, \partial B\rangle} \in {\mb C} \eeqn 
and 
\beqn \kappa_l \big( {\bm x}_1, \ldots, {\bm x}_l \big) =
\sum_{{\bm x}_\infty} \sum_{B} \frac{{\bm
    q}^{\omega(B)} }{k_B!} \langle {\bm x}_1, \ldots, {\bm x}_l | {\bm
  x}_\infty \rangle_B^1 \cdot {\bm x}_\infty.  \eeqn

\begin{thm}\label{thm79}
The collection of multilinear maps $\kappa_l$ is an $A_\infty$ morphism from $\fuk^0 (L)$ to $ \fuk^\infty (L)$. Moreover, when the perturbation satisfies the properties described in Remark \ref{rem69}, the morphism is a higher order deformation of the identity. 
\end{thm}

\begin{proof}
Recall that the $A_\infty$ axiom for ${\bm \kappa}$ holds if for all positive integers $l$ and homogeneous elements ${\bm a}_1, \ldots, {\bm a}_l$ of $CF(L)$,
\begin{multline}\label{eqn76}
\sum_{0 \leq r \leq l} \sum_{0 \leq j \leq l-r} (-1)^{ n_j} \kappa_{l -  r + 1} \big( {\bm a}_1, \ldots, {\bm a}_j, {\bm m}_r^0 \big( {\bm a}_{j+1}, \ldots, {\bm a}_{j + r} \big), {\bm a}_{j + r + 1}, \ldots, {\bm a}_l \big) \\
= \sum_{s\geq 1} \sum_{r_1 + \cdots + r_s = l} {\bm m}_s^\infty \big( \kappa_{ r_1} \big( {\bm a}_1, \ldots, {\bm a}_{ r_1} \big), \ldots, \kappa_{ r_s } \big({\bm a}_{l - r_s -1}, \ldots,  {\bm a}_l \big) \big).
\end{multline}
To prove \eqref{eqn76} up to signs, it suffices consider the case that all
${\bm a}_i = {\bm x}_i$ are arbitrary elements of ${\bf crit}$.
Choose ${\bm x}_\infty \in \crit$ and denote
${\bf X} = ({\bm x}_1, \ldots, {\bm x}_{\uds k}; {\bm x}_\infty)$. For
each curve class $B$, consider the specialization of \eqref{eqn612} to
the case $\lambda = 1$. To better describe these two boundary
components, introduce the following notations.
\begin{enumerate}

\item Let $\txt{BU}$ (stands for ``broken upstairs'') be the set
  of pairs $( {\bf X}_1, {\bf X}_2)$, where ${\bf X}_1$ and
  ${\bf X}_2$ are sequences of elements of $\crit$ of the forms
 \begin{align*}
&\ {\bf X}_1 = ({\bm x}_{j+1}, \ldots, {\bm x}_{j+  r}; {\bm z}),\ &\ {\bf X}_2 = ({\bm x}_1, \ldots, {\bm x}_j, {\bm z}, {\bm x}_{j + r + 1}, \ldots, {\bm x}_l; {\bm x}_\infty)
\end{align*}
for all ${\bm z} \in \crit$, $r \in \{0, 1, \ldots,l \}$, and
$j \in \{ 0, 1, \ldots, l - r \}$. Then modulo orientations, we have
\beqn
\partial_{\rm T1}  \ov{ {\mc M}(B, {\bf X})_1^1}  \cong \bigsqcup_{ ({\bf X}_1, {\bf X}_2 ) \in \txt{BU}} \bigsqcup_{B_1 + B_2 = B}  \frac{k_B!}{ k_{B_1} !  k_{B_2} ! }   \Big( {\mc M}( B_1, {\bf X}_1 )_0^0 \times {\mc M} ( B_2, {\bf X}_2 )_0^1 \Big).
\eeqn

\item Let $\txt{BD}$ (stands for ``broken downstairs'') be the set of $(s + 1)$-tuples $( {\bf X}_1, \ldots, {\bf X}_s; {\bf X}_\infty )$ where
\beqn
{\bf X}_i = ( {\bm x}_{j_i}, \ldots, {\bm x}_{j_i + r_i}; {\bm z}_i),\ i = 1, \ldots, l,\ j_{i+1} = j_i + r_i + 1;\ {\bf X}_\infty = ({\bm z}, \ldots, {\bm z}_s; {\bm x}_\infty)
\eeqn
for all $s \in {\mb N}$, all partitions $l =  r_1 + \cdots + r_s$ into $s$ nonnegative integers, and all choices of $s$-tuple of critical points $({\bm z}_1, \ldots, {\bm z}_s )$. Then similar to the above case, modulo orientation one has 
\begin{multline*}
\partial_{\rm T2}  \ov{ {\mc M}(B, {\bf X})_1^1}  \cong \\
 \bigsqcup_{({\bf X}_1, \ldots, {\bf X}_s; {\bf X}_\infty )\in \txt{BD}}  \bigsqcup_{ B_1 + \cdots + B_s + B_\infty = B}  \frac{k_B!}{ k_{B_1}!\cdot \cdots \cdot k_{B_s} ! \cdot k_{B_\infty}!} \left(  {\mc M}( B_\infty; {\bf X}_\infty)_0^\infty \times \prod_{i=1}^s {\mc M}(B_i; {\bf X}_i)_0^1 \right).
 \end{multline*}
\end{enumerate}

Now we prove the $A_\infty$ relation (up to sign). For all $k \geq 0$, one has up to sign
\beqn\doublespacing
\begin{split}
&\ \sum_{ r =0}^l \sum_{j = 1}^{l - r} \langle {\bm x}_1, \ldots, {\bm x}_j, {\bm m}_r^0 ( {\bm x}_{j+1}, \ldots, {\bm x}_{j+ r} ), {\bm x}_{j+ r + 1}, \ldots, {\bm x}_l ); {\bm x}_\infty \rangle_B^1 \\
= &\ \sum_{( {\bf X}_1, {\bf X}_2) \in \txt{BU}} \sum_{B_1 + B_2 = B} \langle {\bm x}_{j+1}, \ldots, {\bm x}_{j+ r}; {\bm z} \rangle_{B_1}^0 \cdot \langle {\bm x}_1, \ldots, {\bm x}_j, {\bm z}, {\bm x}_{j+ r +1}, \ldots, {\bm x}_l ;  {\bm x}_\infty \rangle_{B_2}^1 \\
= &\ \sum_{ ( {\bf X}_1, {\bf X}_2) \in \txt{BU}} \sum_{ B_1 + B_2 = B}  \left(  \frac{ \tilde {\#} {\mc M}(B_1, {\bf X}_1)_0^0 }{k_{B_1} !} \cdot \frac{ \tilde {\#} {\mc M}( B_2, {\bf X}_2 )_0^1 }{ k_{B_2} !} \right) \\
= &\ \frac{1}{k!} \Big( \tilde {\#} \big( \partial_{\rm T1} \ov{ {\mc M} (B, {\bf X})_1^1} \big) \Big)\\
= &\ \frac{1}{k!} \Big( \tilde {\#} \big( \partial_{\rm T2} \ov{ {\mc M} (B, {\bf X})_1^1}  \big) \Big) \\
= &\ \sum_{( {\bf X}_1, \ldots, {\bf X}_s; {\bf X}_\infty ) \in \txt{BD} } \sum_{B_1 + \cdots + B_s + B_\infty = B } \left(  \frac{ \tilde {\#} {\mc M}( B_\infty, {\bf X}_\infty )_0^\infty}{k_{B_\infty}!} \times \prod_{i =1}^s \frac{ \tilde {\#}  {\mc M}( B_i; {\bf X}_i)_0^1}{ k_{B_i} !} \right) \\
= &\ \sum_{(  {\bf X}_1, \ldots, {\bf X}_s; {\bf X}_\infty ) \in \txt{BD}} \sum_{B_1 + \cdots + B_s + B_\infty = B} \langle {\bm z}_1, \ldots, {\bm z}_s;  {\bm x}_\infty \rangle_{B_\infty}^\infty \cdot  \prod_{i=1}^s \langle {\bm x}_{j_i}, \ldots, {\bm x}_{j_i + r_i};  {\bm z}_i \rangle_{B_i}^1 \\
= &\ \sum_{ r_1 + \cdots + r_s = l} \langle \kappa_{ r_1}({\bm x}_1, \ldots, {\bm x}_{ r_1}), \cdots, \kappa_{ r_s} ( {\bm x}_{ r_1 + \cdots + r_{s-1} + 1}, \ldots, {\bm x}_l ); {\bm x}_\infty   \rangle_B^\infty.
\end{split} \eeqn

The verification of signs in the axioms of the $A_\infty$ morphism can be done in the same way as in the proof of \cite[Theorem 3.14]{Woodward_toric}. Indeed, one needs to compare the boundary orientations and the product orientations of one-dimensional moduli spaces. As affine vortices over ${\bm H}$ is combinatorially similar to quilted disks, the signs appearing are the same as in the proof of \cite[Theorem 3.14]{Woodward_toric}. We omit the details. This finishes the proof of the $A_\infty$ relation for the morphism.
\end{proof}

\subsection{The positivity conditions}\label{subsection75}

We consider special properties of the $A_\infty$ algebras and $A_\infty$ morphisms under the following positivity assumptions.
\vspace{0.2cm}

\begin{defn}\label{poscond} Define the following conditions:
\begin{itemize}
\item[] (P1) All nonconstant $J_V$-holomorphic
  disks in $V$ have positive Maslov indices.
\item[] (P2) All nonconstant $I_X$-holomorphic disks in $X$ have positive Maslov indices.
\item[] (P3) All nonconstant $J_V$-affine vortices
  over ${\bm H}$ have positive Maslov indices.
\end{itemize}
\end{defn}

\begin{prop}\label{prop711}
  Assume that $(V, J_V, L_V)$ satisfies (P1). Then there exists a function
  $W^0: H^1(L; {\bm \Lambda}_0) \to {\bm \Lambda}$ such that for the
  $A_\infty$ algebra $\fuk^0(L)$ defined for each local system $\exp(b)$, $b \in H^1(L; {\bm \Lambda}_0)$ there holds 
  \beq\label{eqn77} 
  {\bm m}_0^0(1) = W^0(b)  {\bm x}_M.  
  \eeq
\end{prop}

\begin{proof}
We first prove the following claim using Gromov compactness for pseudoholomorphic disks:

\vspace{0.2cm}

\noindent {\it  Claim.} For any $E>0$, there exists $\theta^0 (E)>0$ satisfying the following condition. Let $J: {\bm D}^2 \to {\mc J}_D$ be a smooth map such that 
\beqn
\| J - J_V\|_{C^2 ({\bm D}^2 \times V)} \leq \theta^0 (E),
\eeqn
and $u: {\bm D}^2 \to V$ is a $J$-holomorphic disk with $E(u) \leq E$. Then the Maslov index of $u$ is positive. 		
\vspace{0.2cm}

\noindent {\it  Proof of the claim.} Suppose the claim is not true.  Then there exists a sequence of families of almost complex structures $J_i$ parametrized by $z \in {\bm D}^2$ converging in the $C^2$-topology to $J_V$ and a sequence of $J_i$-holomorphic disks $u_i: {\bm D}^2 \to V$ with $E(u_i) \leq E$ with nonpositive Maslov indices. By Gromov compactness, 
a subsequence of $u_i$ converges to a stable $J_V$-holomorphic disk. The limit has positive energy hence positive Maslov index. This is a contradiction since the homology class is preserved in Gromov convergence. \hfill {\it  End of the proof of the claim.}

\vspace{0.2cm}

When inductively construct the coherent system of strongly regular perturbation data, one can ask the following condition be satisfied. For each stable domain type $\Gamma$ and each subtree $\Pi$ of $\Gamma$, recall that ${\rm deg}^{\rm max}(\Pi)$ defined in \eqref{eqn31} is roughly the expected energy of components in $\Pi$ in a stable treed scaled vortex of domain type $\Gamma$. If all edges of $\Pi$ have zero length, we require that $\| J_\Gamma - J_V \|_{C^\epsilon(  \ov{\mc U}_{\Gamma, \Pi} \times V)}$ being sufficiently small so that for each treed disk ${\mc C}$ of type $\Pi$, over each disk component which can be identified with the standard disk ${\bm D} \subset {\bm C}$, one has
\beqn
\| J_\Gamma - J_V \|_{C^2({\bm D} \times V)} < \theta^0 ( {\rm deg}^{\rm max}\Pi ).
\eeqn
This condition can be preserved step by step during the induction. Then by the above claim, only quasidisks with positive Maslov indices contribute to the composition maps. In particular, one has that
\beqn
{\bm m}_0^0(1) =  \left( \sum \tilde{\#} {\mc M}_{{\bm \Gamma}_B}^*(P_{\Gamma_B}) \right) \cdot {\bm x}_M
\eeqn
where the sum is taken over all $B \in H_2(V, L_V; {\mb Z})$ having positive area and Maslov index two; ${\bm \Gamma}_B$ is the essential map type which has only one vertex $v_\alpha$, no input, and one output, with $B_\alpha = B$. Hence \eqref{eqn77} follows. \end{proof}

\begin{prop}\label{prop712}
Suppose $(V, L_V, J_V)$ satisfies (P2). Then there exists a function $W^\infty: H^1(L; {\bm \Lambda}_0) \to {\bm \Lambda}$ such that for the $A_\infty$ algebra $\fuk^\infty(L)$ defined for each local system $b \in H^1( L; {\bm \Lambda}_0)$ there holds
\beq\label{eqn78}
{\bm m}_0^\infty(1) = W^\infty(b) {\bm x}_M.
\eeq
\end{prop}

\begin{proof}
Similar to the proof of Proposition \ref{prop711}, we first prove the following claim.

\vspace{0.2cm}

\noindent {\it  Claim.} For any $E>0$, there exists $\theta^\infty(E)>0$ satisfying the following conditions (notice that $S \cap L_V = \emptyset$ and for the standard almost complex structure, all nonconstant objects with zero Chern number are contained in $S$).
\begin{enumerate}

\item Let $J: {\bm D}^2 \to {\mc J}_D$ be a smooth map with $\| J - J_V \|_{C^2({\bm D}^2 \times V)} \leq \theta^\infty(E)$ and $u: {\bm D}^2 \to X$ be an $I$-holomorphic disk with $E(u) \leq E$ where $I$ is the family of almost complex structures on $X$ induced from $J$.  Then the Maslov index of $u$ is positive.

\item Let $J: {\bm S}^2 \to {\mc J}_D$ be a smooth map with $\| J - J_V \|_{C^2({\bm S}^2 \times V)} \leq \theta^\infty(E)$ and $u: {\bm S}^2 \to X$ be an $I$-holomorphic sphere with $E(u) \leq E$ where $I$ is the family of almost complex structures on $X$ induced from $J$, then the Chern number of $u$ is nonnegative. Moreover, when the Chern number is zero, the image of $u$ is disjoint from $L$.

\item Let $J: {\bm C} \to {\mc J}_D$ be a smooth map with $\| J - J_V \|_{C^2({\bm C}\times V)} \leq \theta^\infty(E)$ and ${\bm v}$ be a $J$-affine vortex over ${\bm C}$ with $E({\bm v}) \leq E$, then the Chern number of ${\bm v}$ is nonnegative. Moreover, if the Chern number is zero, then the image of ${\bm v}$ is disjoint from $L_V$.

\end{enumerate}

\vspace{0.2cm}

When constructing the coherent system of perturbation data ${\bm P} = \{ P_\Gamma \}$, we can require that for all stable domain type $\Gamma$ and all subtrees $\Pi \subset \Gamma$ with only zero length edges, the norm $\| J_\Gamma - J_V \|_{C^\epsilon ( \ov{\mc U}_{\Gamma, \Pi} \times V)}$ is sufficiently small so that for each treed disk ${\mc C}$ of type $\Gamma$, for each disk component which can be identified with the standard disk ${\bm D}^2$, there holds
\beqn
\| J_\Gamma - J_V \|_{C^2( {\bm D}^2 \times V)} < \theta^\infty( {\rm deg}^{\rm max} \Pi).
\eeqn
Then for dimensional reason, only Maslov two or Maslov zero stable treed disks of scale $\infty$ contribute to ${\bm m}_0^\infty(1)$. Moreover, since all Maslov index zero affine vortices are contained in a subset that is disjoint from the Lagrangian, objects with total Maslov index zero cannot contribute to ${\bm m}_0^\infty(1)$. Then \eqref{eqn78} follows. 
\end{proof} 

\begin{prop}\label{prop713}
  Suppose $(V, L_V)$ satisfies all three positivity conditions in Definition \ref{poscond}. Then $W^0(b) = W^\infty(b)$.
\end{prop}

\begin{proof}
By the $A_\infty$ relation for the $A_\infty$ morphism ${\bm \kappa}$, there holds 
\beqn
\kappa_1 ( {\bm m}_0^0(1)) = \sum_{l=0}^\infty {\bm m}_l^\infty ( \kappa_0(1), \ldots, \kappa_0(1)).
\eeqn
By Proposition \ref{prop711} and \ref{prop712}, the result follows from 
\beqn
\kappa_0(1) = 0
\eeqn
and
\beqn
\kappa_1({\bm x}_M) = {\bm x}_M.
\eeqn
Similar to the proofs of Proposition \ref{prop711} and \ref{prop712}, when the perturbation data is sufficiently close to the standard one, by dimension counting and the positivity conditions, there is no nonempty moduli spaces contributing to $\kappa_0(1)$. On the other hand, the only objects that contribute to $\kappa_1({\bm x}_M)$ are those with Maslov index zero. If such an object has a positive energy, then the energy is concentrated on spherical components. However, since nonconstant holomorphic spheres in $X$ or affine vortices with zero Chern number never intersect the Lagrangian, only the object with zero energy contributes to $\kappa_1({\bm x}_M)$, implying $\kappa_1({\bm x}_M) = {\bm x}_M$. Hence $W^0(b) = W^\infty(b)$.
\end{proof}

\subsection{Strict unitality}\label{subsection77}

Strict unitality is an important property for $A_\infty$ algebras. It
plays a similar role as the {\it Fundamental Class} axiom of
Gromov--Witten invariants. However, due to the complexity of the chain
level theory in Lagrangian Floer theory the naive construction usually
does not grant strict units. A usual method is to construct a weaker
alternative, called a {\it homotopy unit} (see \cite[Section 3]{FOOO_Book}).  We state our extension as
follows. The detailed construction and proof are given in Section
\ref{sectiona1}.

\begin{thm}\label{thm714}
Let $\fuk^0 (L)$ and $\fuk^\infty (L)$ be the $A_\infty$ algebras and ${\bm \kappa}: \fuk^0 (L) \to \fuk^\infty (L)$ be the $A_\infty$ morphism, all of which have been constructed by choosing a coherent system of perturbation data ${\bm P}$. Let 
\beqn
\wt{CF}(L; {\bm \Lambda}) = CF( L) \oplus {\bm \Lambda} {\bm e} \oplus {\bm \Lambda} {\bm p}
\eeqn
be the graded ${\bm \Lambda}$-vector space where the degree of ${\bm e}$ is $0$ and the degree of ${\bm p}$ is $-1$. Then there exist $A_\infty$ algebra structures $\wt{\fuk}{}^0 (L)$ and $\wt{\fuk}{}^\infty (L)$ over $\wt{CF} (L; {\bm K})$ satisfying the following properties. 
\begin{enumerate}

\item The composition maps $\tilde {\bm m}_l^0$ of $\wt{\fuk}{}^0 (L)$ resp. the composition maps $\tilde {\bm m}_l^\infty$ of $\wt{\fuk}{}^\infty(L)$ extend ${\bm m}_l^0$ resp. ${\bm m}_l^\infty$, namely
\beqn
\tilde {\bm m}_l ({\bm a}_1, \ldots, {\bm a}_l ) = {\bm m}_l ({\bm a}_1, \ldots, {\bm a}_l ),\ \forall l \geq 0,\ {\bm a}_1, \ldots, {\bm a}_l \in CF (L) \subset \wt{CF} (L),
\eeqn
where $(\tilde {\bm m}_l, {\bm m}_l)$ is either $(\tilde {\bm m}_l^0, {\bm m}_l^0 )$ or $(\tilde {\bm m}_l^\infty, {\bm m}_l^\infty )$. 

\item ${\bm e}$ is a strict unit for both $\wt{\fuk}{}^0 (L)$ and $\wt{\fuk}{}^\infty (L)$. 

\item There is a unital $A_\infty$ morphism $\tilde {\bm \kappa}: \wt{\fuk}{}^0 (L) \to \wt{\fuk}{}^\infty ( L)$ extending ${\bm \kappa}$, i.e.,
\beqn
\tilde  \kappa_l  ({\bm a}_1, \ldots, {\bm a}_l ) = \kappa_l ({\bm a}_1, \ldots, {\bm a}_l ),\ \forall l \geq 0,\ {\bm a}_1, \ldots, {\bm a}_l \in CF (L) \subset \wt{CF}(L).
\eeqn
Moreover, $\tilde {\bm \kappa}$ is a higher order deformation of the identity. 

\item Under condition (P1) there holds
\beqn
\tilde {\bm m}_l^0( {\bm p}, \ldots, {\bm p}) = \left\{ \begin{array}{cc}  {\bm e} - {\bm x}_M,\ &\ l = 1,\\
                                                                                            0,\ &\ l \geq 2. \end{array} \right.
\eeqn
\end{enumerate}
\end{thm}

The detailed proof is based on the construction of a coherent system of perturbation data on the universal curves of moduli spaces of {\it  weighted} treed disks (see Section \ref{sectiona1}). With the strict units one can define the Maurer--Cartan spaces $MC(\wt{\fuk}{}^0(L))$ and $MC(\wt{\fuk}{}^\infty)$ and the potential functions 
\begin{align*}
&\ {\bf W}^0: MC(\wt{\fuk}{}^0(L)) \to {\bm \Lambda},\ &\ {\bf W}^\infty: MC( \wt{\fuk}{}^\infty(L) \to {\bm \Lambda}.
\end{align*}
A straightforward corollary to Theorem \ref{thm714} and Proposition \ref{prop74} is that the open quantum Kirwan map induces a map between the Maurer--Cartan spaces.

\begin{cor}\label{cor715}
The map 
\beqn
\uds{\tilde \kappa}: \wt{CF}(L; {\bm \Lambda}_+) \to \wt{CF}(L; {\bm \Lambda}_+),\ \ \uds{\tilde \kappa} ({\bm b}) = \sum_{l \geq 0} \tilde \kappa_l ({\bm b}, \cdots, {\bm b})
\eeqn
maps $MC( \wt{\fuk}{}^0 (L))$ into $MC( \wt{\fuk}{}^\infty (L))$ and for each ${\bm b} \in MC( \wt{\fuk}{}^0 (L))$, one has
\beqn
{\bf W}^0 ({\bm b}) = {\bf W}^\infty ( \uds{\tilde \kappa} ({\bm b}))	\in {\bm \Lambda}.
\eeqn
\end{cor}

On the other hand, given a weakly bounding cochain ${\bm b}^0 \in MC( \wt{\fuk}{}^0 (L))$ or ${\bm b}^\infty \in MC(\wt{\fuk}{}^\infty (L))$, one can define Floer cohomologies respectively. Using the terminology of \cite{Woodward_toric}, we call the first Floer cohomology the {\it  quasimap Floer cohomology}, denoted by 
\beqn
QHF^*((L, b, {\bm b}^0); {\bm \Lambda}).
\eeqn
We denote the cohomology out of $\wt{\fuk}{}^\infty(L)$ and a weakly bounding cochain ${\bm b}^\infty$ by 
\beqn
HF^* ((L, b, {\bm b}^\infty); {\bm \Lambda}).
\eeqn
When the weakly bounding cochain or the local systems are understood from the context we also omit them from the notations. Since $\tilde {\bm \kappa}$ is a higher order deformation of the identity, the degree one component of $\tilde {\bm \kappa}$ is an invertible linear map, and hence an isomorphism of cochain complexes. Then by Lemma \ref{lemma75} we obtain the isomorphism of the Floer cohomology groups.

\begin{cor}\label{cor716}
For any weakly bounding cochain ${\bm b}^0$ of $\wt{\fuk}{}^0(L)$ and the corresponding weakly bounding cochain ${\bm b}^\infty = \uds{\tilde {\kappa}}({\bm b}^0)$, there is an isomorphism
\beqn
QHF^* ( (L, {\bm b}^0 ); {\bm \Lambda}) \cong HF^* ( (L, {\bm b}^\infty); {\bm \Lambda}).
\eeqn
\end{cor}

Lastly, under certain positivity assumptions, the previous results have stronger implications. 

\begin{prop}\label{prop717}
Suppose $(V, J_V, L_V)$ satisfies (P1). Then for the $A_\infty$ algebra $\wt{\fuk}{}^0(L)$ defined for each local system $b \in H^1(L; {\bm \Lambda}_0)$, there holds
\begin{align}\label{eqn79}
&\ W^0 (b) {\bm p} \in MC( \wt{\fuk}{}^0 (L)),\ &\ {\bf W}^0 ( W^0 (b) {\bm p} ) = W^0 (b).
\end{align}
\end{prop}

\begin{proof}
By item (d) of Theorem \ref{thm714}, one has 
\beqn
\sum_{l \geq 0} \tilde {\bm m}_l^0( W^0(b) {\bm p}) = \tilde {\bm m}_0^0(1) + \tilde {\bm m}_1^0( W^0(b) {\bm p}) = W^0(b) {\bm x}_M + W^0(b) ( {\bm e} - {\bm x}_M) = W^0(b) {\bm x}_M.
\eeqn
Hence \eqref{eqn79} follows. 
\end{proof}

Under (P1), the weakly bounding
cochain $W^0(b) {\bm p}$ is called the {\it canonical weakly bounding
  cochain}. This notion provides a natural section \beqn H^1(L; {\bm
  \Lambda}_0) \to \bigsqcup_{b} MC(\wt{\fuk}{}^0(L, b)).  \eeqn
Composing with the potential function one obtains a
${\bm \Lambda}$-valued function on $H^1(L; {\bm \Lambda}_0)$. By abuse
of notations, we still denote this function by \beqn {\bf W}^0: H^1(L;
{\bm \Lambda}_0) \to {\bm \Lambda}, \eeqn called the {\it restricted
  potential function}. It coincides with $W^0$. Similarly, under condition (P2) there is also a family of canonical weakly
bounding cochains $W^\infty(b) {\bm p}$ of $\wt{\fuk}{}^\infty(L)$
inducing a restricted potential function \beqn {\bf W}^\infty: H^1(L;
{\bm \Lambda}_0) \to {\bm\Lambda}.  \eeqn With only (P1), the canonical weakly bounding
cochain of $\wt{\fuk}{}^0(L)$ is always mapped to a certain weakly
bounding cochain of $\wt{\fuk}{}^\infty(L)$. When all three positivity
conditions hold, the image is actually the canonical weakly
bounding cochain of $\wt{\fuk}{}^\infty(L)$, and in addition
${\bf W}^0 = {\bf W}^\infty$.

\section{Toric Manifolds}\label{section8}

In this section we consider the example of toric manifolds. Let $X$ be an $n$-dimensional compact toric manifold. Its moment polytope $P \subset {\mb R}^n$ gives a canonical presentation of $X$ as GIT quotient. We consider the Lagrangian Floer theory of the preimage of an interior point of $P$ under the natural projection $X \to P$. Preceding our results, there have been many mathematical works about Lagrangian Floer theory of this kind of Lagrangians, see for example \cite{Cho_Oh}, \cite{Cho_Poddar}, \cite{FOOO_toric_1, FOOO_toric_2, FOOO_mirror}, \cite{Woodward_toric}, \cite{Chan_Lau_2014}, \cite{Gonzalez_Iritani}, \cite{CLLT}. Our result gives a conceptual explanation of the work of Chan--Lau--Leung--Tseng \cite{CLLT}.

\subsection{Compact toric manifolds as GIT quotients}\label{subsection81}

Let us first recall the basics about compact toric manifolds. We introduce toric manifolds using Batyrev's approach \cite{Batyrev_1993} but change a few notations. Let $M$ be the lattice ${\mb Z}^n$ and $M^\vee = {\rm Hom}_{\mb Z}(M, {\mb Z})$ its dual lattice. Denote $M_{\mb R} = M \otimes_{\mb Z} {\mb R}$, $M_{\mb Z}^\vee = M^\vee \otimes_{\mb Z}{\mb R}$. Let $k \geq 1$. A convex subset $\sigma \subset M_{\mb R}$ is a {\it regular $k$-dimensional cone} if there are $k$ linearly independent vectors $v_1, \ldots, v_k \in M$ such that 
\beqn
\sigma = \big\{ a_1 v_1 + \cdots + a_k v_k\ |\ a_1, \ldots, a_k \geq 0 \big\}
\eeqn
and $v_1, \ldots, v_k$ can be extended to a ${\mb Z}$-basis of $M$. A regular cone $\sigma'$ is a {\it  face} of another regular cone $\sigma$, denoted by $\sigma' \prec \sigma$, if the set of generators of $\sigma'$ is a subset of that of $\sigma$. A {\it  complete $n$-dimensional fan} is a finite collection $\Sigma = \{ \sigma_1, \ldots, \sigma_s \}$ of regular cones in $M_{\mb R}$ satisfying the following conditions:
\begin{enumerate}
\item If $\sigma \in \Sigma$ and $\sigma' \prec \sigma$, then $\sigma' \in \Sigma$.

\item If $\sigma, \sigma' \in \Sigma$, then $\sigma \cap \sigma' \prec \sigma$, $\sigma \cap \sigma' \prec \sigma'$. 

\item $M_{\mb R} = \sigma_1 \cup \cdots \cup \sigma_s$. 
\end{enumerate}

Given a complete $n$-dimensional fan $\Sigma$, let $G(\Sigma) = \{
v_1, \ldots, v_N \} \subset M$ be the set of generators of
1-dimensional fans in $\Sigma$. $\gamma  = \{ v_{i_1}, \ldots,
v_{i_p}\} \subset G(\Sigma)$ is called a {\it  primitive collection}
if $\{ v_{i_1}, \ldots, v_{i_p}\}$ does not generate a $p$-dimensional
cone in $\Sigma$ but any proper subset of $\gamma$ generate a cone in
$\Sigma$. 

Associated to the combinatorial data above we associate a quotient of affine space as follows.  Consider the exact sequence 
\beqn
\xymatrix{0 \ar[r] & {\bf K}_\Sigma \ar[r] & {\mb Z}^N \ar[r]^{e_i \mapsto v_i} & {\mb Z}^n \ar[r] & 0}.
\eeqn
Let $\tilde K = (S^1)^N$ be the $N$-dimensional torus and $\tilde G = ({\mb C}^*)^N$ be its complexification. Then ${\bf K}_\Sigma \cong {\mb Z}^{N - n}$ generates a subgroup $G: = G_\Sigma \subset \tilde G$, which acts on ${\mb C}^N$ via weights ${\bm w}_1, \ldots, {\bm w}_N \in ({\mb Z}^{N-n})^\vee$. The unstable locus of this action is 
\beq\label{eqn81}
V^\us = \bigcup_{\gamma} V^\gamma = \bigcup_\gamma \big\{ (x_1, \ldots, x_N)\in {\mb C}^N\ |\ x_i = 0\ \forall\ v_i \in \gamma \big\}.
\eeq
Let $V = {\mb C}^N$. Then the $G$-action on $V^\ss = V \setminus
V^\us$ is free and the quotient $X:= V^\ss/ G$ is a compact toric
manifold, with a residual $({\mb C}^*)^n$-action. On the other hand, a
moment map for the $\tilde K$-action is 
\beqn
\tilde \mu_0 (x_1, \ldots, x_N) = \Big( \tilde \mu_{0,1}, \ldots, \tilde \mu_{0, N} \Big) = \Big( -\frac{{\bf i}}{2} |x_1|^2 , \ldots, - \frac{{\bf i}}{2} |x_N|^2 \Big).
\eeqn
Here we identify ${\bf i} {\mb R} \cong {\rm Lie} S^1$ with its dual by the metric $|{\bf i}| = 1$. Then via the map ${\fk} \hookrightarrow \tilde {\mf k}$, $\tilde \mu_0$ restricts to a moment map of the $G$-action, which is 
\beqn
\mu_0  (x_1, \ldots, x_N) = - \frac{{\bf i}}{2} \sum_{i=1}^N |x_i|^2 {\bm w}_i \in {\fk}^*.	
\eeqn

Lagrangian branes are defined as generic orbits of the standard Hamiltonian torus action.  Choose positive real numbers
$c_1, \ldots, c_N$. Then the submanifold \beq\label{eqn82} L_V =
\Big\{ (x_1, \ldots, x_N) \in V \ \Big| \ | x_i |^2 = \frac{c_i}{\pi}
\Big. \Big\} \cong (S^1)^N, \eeq which is the level set of
$\tilde \mu_0$ at \beqn \tilde {\mb C} := \left( - \frac{{\bf i}}{2
    \pi } c_1, \ldots, - \frac{{\bf i}}{2\pi} c_N \right), \eeqn is a
$K$-invariant Lagrangian submanifold of $V$. Via the inclusion
${\mf k} \subset \tilde {\mf k}$, $\tilde {\mb C}$ reduces to
${\mb C} \in {\mf k}^*$. Denote
\begin{align*}
&\ \tilde \mu = \tilde \mu_0 - \tilde {\mb C},\ &\ \mu = \mu_0 - {\mb C}.
\end{align*}
We can construct a linearization of the $G$-action, such that the induced moment map coincides with $\mu$. Assume $c_i \in  {\mb Q}$ for all $i$. We define a linearization with respect to the $\tilde G$-action as follows. Let $V_i \cong {\mb C}$ be the $i$-th factor of $V \cong {\mb C}^N$ with coordinate $x_i$. Let $\tilde V_{i,0} \to V_i$ be the trivial line bundle with a constant Hermitian metric, equipped with the ${\mb C}^*$-action of weight $1$. Then the moment map on $\tilde V_{0,i}$ is $\tilde \mu_{0, i}$. On the other hand, since $c_i$ is rational, one can choose $k>0$ such that $k c_i \in {\mb Z}$ for all $i$. Let $\tilde V_i \to V_i$ be the $k$-th tensor power of $\tilde V_{0,i}$ and twist the ${\mb C}^*$-action by additional weight $k c_i$. Then the moment map on $\tilde V_i$ is 
\beqn
\tilde \mu_{0,i} - \frac{ {\bf i}}{2\pi}  k c_i. 
\eeqn
Define 
\beqn
\tilde V = \tilde V_1 \boxtimes \cdots \boxtimes \tilde V_N \to V.
\eeqn
It is a linearization of the $\tilde G$-action whose equivariant first Chern form
\beqn
k \big( \omega_V + \tilde \mu_0 - \tilde {\mb C} \big).
\eeqn
Restricting from $\tilde G$ to $G$, we obtain a linearization of the $G$-action on $V$.  The unstable locus of the $G$-action with respect to this linearization is still $V^{\rm us}$ given by \eqref{eqn81}. Moreover, 
\beqn
X \cong \mu^{-1}(0) / K.
\eeqn 
The Lagrangian $L_V$ then descends to a Lagrangian torus $L \subset X$. 
 
\begin{lemma}
The Lagrangian submanifold $L \subset X$ is strongly rational in the sense of Definition \ref{defn27}.
\end{lemma}

\begin{proof}
Because $k c_i \in {\mb Z}$, the connection $\nabla^{\tilde V_i}$ on $\tilde V_i$ restricted to the $i$-th factor $L_i$ of $L_V$ has a trivial holonomy. Therefore, $(\tilde V, \nabla^{\tilde V})$ descends to a line bundle with connection whose restriction to $L$ is trivial. Hence Definition \ref{defn27} (a) is satisfied. Moreover, an equivariant holomorphic section $s_i$ of $\tilde V_i$ which can be written as a monomial in $x_i$ vanishes at the origin of $X_i$ and hence is nowhere vanishing over $L_i$. The trivialization of $\tilde V_i|_{L_i}$ induced by $s_i|_{L_i}$ agrees with the above trivialization of $\tilde V_i|_{L_i}$ up to homotopy. Then $s_1 \boxtimes \cdots \boxtimes s_N$ is a $\tilde G$-equivariant, hence $G$-equivariant holomorphic section of $\tilde V$ that is nowhere vanishing over $L_V$. This section descends to a holomorphic section of $\tilde X \to X$ that is nowhere vanishing on $L$ and belongs to the same homotopy class of trivializations. Hence Definition \ref{defn27} item (b) is satisfied. 
\end{proof}

\subsection{Proof of Theorem \ref{thm18}}\label{subsection82}

The reader may verify that the target space $(V, G, \mu)$ satisfies  Definition \ref{defn11}. We also rescale the symplectic form so that all disk classes have integral areas. Moreover, we assume that (S1)---(S4) of Definition \ref{defn12} are satified if we choose  
\beqn
S = S_1 \cup \cdots \cup S_N
\eeqn
to be the union of coordinate hyperplanes. As explained in Example \ref{example13}, these conditions can be described combinatorially.

\begin{rem}
This semi-positivity condition does not only depends on the toric manifold $X$ but also depends on how to present $X$ as a GIT quotient. 
\end{rem}

Under the above hypothesis, by Theorem \ref{thm76}, Theorem
\ref{thm77}, and Theorem \ref{thm79}, $\fuk{}^0(L)$, $\fuk^\infty(L)$,
and the morphism ${\bm \kappa}: \fuk^0(L) \to \fuk^\infty(L)$ are
defined. Theorem \ref{thm714} enhance them to strictly unital
$A_\infty$ algebras $\wt{\fuk}{}^0 (L)$, $\wt{\fuk}{}^\infty (L)$, and
a unital morphism
$\tilde {\bm \kappa}: \wt\fuk{}^0(L) \to \wt{\fuk}{}^\infty(L)$.

Next we would like to show that both $A_\infty$ algebras are weakly
unobstructed by applying Proposition \ref{prop717}. Recall the
following well-known result.

\begin{prop}[Blaschke's Formula]\label{blaschke}
Every $J_V$-holomorphic quasidisk $u: {\bm D}^2 \to V$ is of the form 
\beqn
{\bm D}^2 \ni z \mapsto \big[ u_j(z) \big]_{j=1}^N := \Big[     \frac{ c_j e^{{\bf i} \theta_j}}{\pi}   \prod_{k=1}^{d_j} \Big( \frac{ z - \alpha_{j, k}}{ 1 - \ov{\alpha}_{j, k} z} \Big) \Big]_{j=1}^N.
\eeqn
Here $\alpha_{j, k}$ is in the interior of the unit disk. Furthermore, the Maslov index of such a holomorphic disk is equal to $2 ( d_1 + \ldots + d_N)$.
\end{prop}

One can see that there are $N$ basic disk classes $B_i \in H_2(V, L_V)$ of Maslov index two such that in each class there is essentially one holomorphic disk. It also follows that all $J_V$-holomorphic disks have nonnegative Maslov indices. Thus we have the following corollary to Proposition \ref{prop717}. 

\begin{thm}
For any local system $\wt{\fuk}{}^0 (L)$ and $\wt{\fuk}{}^\infty (L)$ are weakly unobstructed.
\end{thm}

\subsection{Calculation of the potential function}

We need to calculate the potential function of the equivariant Fukaya algebra. The potential function as a function defined over $H^1(L; {\bm \Lambda}_0)$ is expected to coincides with the Givental--Hori--Vafa potential. However, the calculation made in \cite{Woodward_toric} cannot be carried over directly since we are using a different perturbation scheme.

\begin{lemma}\label{lemma85}
For every $E>0$, there exists $\epsilon(E) >0$ such that for any smooth family of $K$-invariant almost complex structures $J$ parametrized by points on ${\bm D}^2$ with $\| J - J_V \|_{C^2({\bm D} \times V)}  \leq \epsilon(E)$, for any $J$-holomorphic disk $u: {\bm D}^2 \to V$ with boundary mapped into $L_V$ and $E(u) \leq E$, the class of $u$ is equal to $\sum d_i B_i$ with $d_i \geq 0$.
\end{lemma}

\begin{proof}
The proof is similar to the proof of the claim made in Proposition \ref{prop711}. Suppose this is not true, then there exist a sequence of domain-dependent almost complex structures $J_i$ parametrized by points on the disk such that $J_i$ converges to $J_V$ in $C^2$-topology and a sequence of $J_i$-holomorphic disks $u_i: {\bm D}^2 \to V$ such that $E(u_i) \leq E$ such that the homology classes of $u_i$ are not of the desired type. By Gromov compactness, 
a subsequence (still indexed by $i$) converges to a stable $J_V$-holomorphic disk. By Proposition \ref{blaschke}, the homology class of the limiting stable disk has to be $d_1 B_1 + \cdots + d_N B_N$ with all $d_i$ nonnegative. This is a contradiction.
\end{proof}

We also need certain properties satisfied by a generic stabilizing divisor.

\begin{lemma}\label{lemma86}
There exists $k_0>0$ such that for each natural number $h$ bigger than the real dimension of $X$, for each tuple 
\beqn
(k_1, \ldots, k_h),\ k_a \geq k_0
\eeqn
there exists a comeager subset of 
\beqn
\prod_{a=1}^h \Gammait_{k_a}(L)
\eeqn
consisting of elements $(s_1, \ldots, s_h)$ satisfying the following conditions. 
\begin{enumerate}

\item The divisor $D = D_1 \cup \cdots \cup D_h$ satisfies the conditions for a stabilizing divisor in Lemma \ref{lemma210}. In particular, they intersect transversely in the stable locus.

\item For each $s_a$ and each disk class $B_i$, the moduli space of
  $J_V$-holomorphic disks of class $B_i$ with one interior marking
  mapped nontransversely to $D_a^{\rm st}$ has the expected dimension
  $N - 3$. 

\item For each pair $(s_a, s_b)$ in the $h$-tuple, the moduli spaces of $J_V$-holomorphic disks of class $B_i$ with one interior marking mapped into $D_a^{\rm st}$ and $D_b^{\rm st}$ both transversely has the expected dimension $N-3$.
\end{enumerate}
\end{lemma}

The proof of Lemma \ref{lemma86} is provided in the next subsection. Now we fix a divisor $D = D_1 \cup \cdots D_h$ satisfying the conditions of Lemma \ref{lemma86}. Further we would like to use Lemma \ref{lemma85} such that if the perturbed almost complex structure is sufficiently closed to $J_V$, then at least for computing the potential function, only disks in the basic disk classes contribute. 

\begin{prop}\label{prop87}
There exists a coherent system of strongly regular perturbation data ${\bm P}$ for all stable domain types $\Gamma$ of scale $0$ satisfying the following condition. 

\begin{enumerate}

\item For any treed disk ${\mc C}$ of type $\Gamma$ and a disk component which can be identified with the standard disk ${\bm D} \subset {\bm C}$, suppose this disk component is contained in a maximal subtree $\Pi$ whose lengths all have length zero, then 
\beqn
\| J_\Gamma - J_V \|_{C^2( {\bm C} \times V)} \leq \epsilon ( {\rm deg}^{\rm max}\Pi ).
\eeqn

\item For each basic disk class $B_i$, let ${\bm \Gamma}_i$ be the essential map type of scale $0$ which has no input, one output labelled by ${\bm x}_M$, and which has the homology class $B_i$, the moduli space ${\mc M}_{{\bm \Gamma}_i}^*(P_{\Gamma_i})$ is regular and consists of $k_{B_i} !$ points, where
\beqn
k_{B_i}! = \prod_{a=1}^h k_{B_i, a}!,\ k_{B_i, a} = {\rm deg} D_a \cdot \omega(B_i).
\eeqn
\end{enumerate}
\end{prop}

\begin{proof}
  First, for all the essential map types (finitely many) mentioned
  in item (b) of this proposition, let $P_{\Gamma_i}^*$ be the
  perturbation data which is equal to standard almost complex
  structure $J_V$ on the disk and which is equal to the function $F_L$
  on the output. Then Lemma \ref{lemma86} implies that the moduli
  space ${\mc M}_{{\bm \Gamma}_i}^* (P_{\Gamma_i}^*)$ is regular and
  consists of $k_{B_i}!$ points. Then by the transversality condition,
  there exists a neighborhood
  ${\mc P}_{\Gamma_i}^* \subset {\mc P}_{\Gamma_i}$ of
  $P_{\Gamma_i}^*$ such that for each perturbation $P_{\Gamma_i}$ in
  this neighborhood, there are exactly the same number of points which
  are still regular. Second, during the inductive construction of the
  coherent system of perturbation data, the condition (a) can be
  satisfied. Notice that for each $\Gamma_i$ there are only finitely
  many stable domain types $\Pi$ with $\Pi \leq \Gamma_i$. Hence in
  constructing the coherent system of strongly regular perturbation
  data, one can require both conditions to be satisfied.
\end{proof}

Now we can compute the potential function for the equivariant Fukaya
algebra. Each element $b \in H^1(L; {\bm \Lambda}_0)$ giving rise to a local system $\exp b$. We use $(b_1, \ldots, b_N) \in {\bm \Lambda}_0^N$ as local coordinates on the space of local systems.

\begin{cor}
For the Fukaya algebra $\fuk^0(L)$ defined for any local system $b \in H^1(L; {\bm \Lambda}_0)$, there holds
\beqn
{\bm m}_0^0(1) = W^0 (b) {\bm x}_M
\eeqn
where $W^0$ is the Givental--Hori--Vafa potential, i.e., 
\beq\label{eqn83}
W^0 (b) =  \sum_{i=1}^N \exp(b_i) {\bm q}^{\omega(B_i)}.
\eeq
\end{cor}

\begin{proof}
By Proposition \ref{prop711} and its proof, there is a function $W^0(b)$ with ${\bm m}_0^0(1) = W^0(b) {\bm x}_M$ where $W^0(b)$ is the weighted counts of Maslov two disks passing through the point ${\bm x}_M$. By the conditions on the perturbation data stated in Proposition \ref{prop87}, the only nonempty moduli spaces which contribute to ${\bm m}_0^0(1)$ are those which represent the basic disk classes $B_i$. Moreover, for each $B_i$, item (b) of Proposition \ref{prop87}, the count coincides with the count as if using the standard almost complex structure. Moreover the holonomy of the local system along the boundary of $B_i$ is exactly $\exp(b_i)$. One can also check the orientations as did in \cite{Woodward_toric} and each disk contributes positively to the counting. Hence by the definition of the composition maps,  \eqref{eqn83} follows.
\end{proof}

By Proposition \ref{prop717}, $W^0(b)$ coincides with the potential
function of $\wt{\fuk}{}^0(L)$ evaluated at the canonical weakly
bounding cochain $W^0(b) {\bm p}$. Hence we proved item (b) of Theorem
\ref{thm18}. To prove item (c) of Theorem \ref{thm18}, we would like
to use Proposition \ref{prop712} and \ref{prop713}. The condition (P2) follows from Blaschke's formula
(Proposition \ref{blaschke}). Then by Proposition \ref{prop712}, the
restricted potential function
${\bf W}^\infty: H^1(L;{\bm \Lambda}_0) \to {\bm \Lambda}$ is
defined. To get the identity ${\bf W}^0 = {\bf W}^\infty$, one needs
to verify the condition (P3).

\begin{lemma}
Nonconstant $J_V$-affine vortices over ${\bm H}$ have positive Maslov indices. 
\end{lemma}

\begin{proof}
Let ${\bm v} = (u, \phi, \psi)$ be such an affine vortex. By the removal of singularity theorem for affine vortices over ${\bm H}$ (Proposition \ref{prop22}), after an appropriate gauge transformation, we may assume that $u$ converges at infinity and hence extends to a continuous map $u: {\bm D}^2 \to V$. We write $u$ in coordinates as $(u_1, \ldots, u_N)$, where $u_i: {\bm D}^2 \to V_i \cong {\mb C}$ is a continuous map. Notice that $\phi, \psi$ takes value in ${\mf k}$ which is canonically embedded into the Lie algebra of $T^N$. Hence we can also write $\phi$ and $\psi$ in coordinates as $(\phi_1, \ldots, \phi_N)$ and $(\psi_1, \ldots, \psi_N)$. Since the action is induced from the standard torus action on ${\mb C}^N$, we see for each $i$, $(u_i, \phi_i, \psi_i)$ satisfies the equation
\beq
\partial_s u_i + {\mc X}_{\phi_i} (u_i) + {\bf i}( \partial_t u_i + {\mc X}_{\psi_i}(u_i)) = 0
\eeq
(but not necessarily the vortex equation). We can identify $u_i$ with a map $\tilde u_i: {\bm D}^2 \to {\bm D}^2 \times V_i$ which is $J_{\phi_i, \psi_i}$-holomorphic, where $J_{\phi_i, \psi_i}$ is the almost complex structure on ${\bm D}^2 \times V_i$ induced from the standard complex structure on $V_i$ and the Hamiltonian vector fields generated by $\phi_i$ and $\psi_i$. Then since these Hamiltonian vector fields vanish at the origin, the divisor ${\bm D}^2 \times \{0\}\subset {\bm D}^2 \times V_i$ is $J_{\phi_i,\psi_i}$-almost complex. By the proof of \cite[Proposition 7.1]{Cieliebak_Mohnke}, the intersection number between $\tilde u_i$ and ${\bm D}^2 \times \{0\}$ is nonnegative. Therefore, the winding number of the boundary restriction of $u_i$ is nonnegative. On the other hand, the Maslov index of ${\bm v}$ is equal to twice of the sum of the winding numbers of the boundary restrictions of $u_i$. Therefore the Maslov index is nonnegative. Moreover, when the Maslov index is zero, which implies all $u_i$ do not intersect with the origin of $V_i$, the class of ${\bm v}$ is necessarily zero, meaning that ${\bm v}$ is a constant vortex. 
\end{proof}

Therefore by Proposition \ref{prop713}, item (c) of Theorem \ref{thm18} is proved. 

\subsection{The stabilizing divisor}\label{subsection84}

Now we prove Lemma \ref{lemma86}. Recall that $\tilde V \to V$ is the
line bundle linearizing the $G$-action, which descends to a
holomorphic line bundle $\tilde X \to X$. For any $k \geq 1$, denote
$\Gammait_k = H^0_G(V, \tilde V^{\otimes k})$ the space of
$G$-equivariant holomorphic sections of $\tilde V{}^{\otimes k}$,
which are certain subspace of polynomials on ${\mb C}^N$. For
sufficiently large $k$, $\Gammait_k$ is isomorphic to $\bar
\Gammait_k: = H^0(X, \tilde X^{\otimes k})$. Since the bundle $\tilde X$ is positive, there exists $k_X \geq 1$ such that for all $k \geq
k_X$, for all $x \in X$, one has the surjectivity of two linear maps 
\beq\label{eqn84} 
\bar \Gammait_k \mapsto \tilde X_x^{\otimes k},\ \bar s \mapsto \bar s( x); 
\eeq
\beq\label{eqn85} 
\bar \Gammait_{k, x} \mapsto T_x^* X \otimes \tilde X_x^{\otimes k},\ \bar s \mapsto d \bar s(x).
\eeq
Here $\bar \Gammait_{k, x} \subset \bar \Gammait_k$ is the subset of sections that vanish at $x$. We assume $k \geq k_X$. Given $f \in \Gammait_k$, denote by $D_f = f^{-1}(0)$ the equivariant stabilizing divisor and $D_f^\ss:= D_f \cap V^\ss$ the semi-stable part. Let $\Gammait_k (L) \subset \Gammait_k$ be the subset of $f$ satisfying $D_f \cap L_V = \emptyset$.

The following lemma is used to show that the strata of maps meeting both the stabilizing divisor and the boundary divisor at the same time are expected dimensions. For each $i = 1, \ldots, N$, denote
\beqn
Z_i = \left\{ (x_1, \ldots, x_N) \in {\mb C}^N \ |\ \| x_j \|^2 = \frac{c_j}{\pi}\ \forall j \neq i \right\}
\eeqn
where $c_i$ are the constants from \eqref{eqn82}. 

\begin{lemma}\label{lemma810}
For each $i$ and $p \in Z_i$, the tangent vector $\frac{\partial}{\partial x_i}(p)$ is not tangent to the $G$-orbit through $p$.
\end{lemma}

\begin{proof}
Let $p = (x_1, \ldots, x_N) \in Z_i$. If $x_i = 0$, then since the $G$-action is linear, the infinitesimal $G$-actions at $p$ are all orthogonal to $\frac{\partial}{\partial x_i}$, hence $\frac{\partial}{\partial x_i}$ is not an infinitesimal $G$-action. Now assume $x_i \neq 0$. If the statement is not true, then because $x_j \neq 0$ for all $j \neq i$, there is a one-parameter subgroup of $G$ which fixes all coordinates but $x_i$. Then the hyperplane defined by $x_i = 0$ is fixed by this one-parameter subgroup, and is necessarily unstable. This contradicts the fact that the unstable locus has complex codimension at least two.
\end{proof}

Now we consider tangencies between holomorphic disks and a divisor. For each basic disk class $B_i$, let $\widetilde {\mc M}_1(B_i)$ be the set of pairs $(u, z_0)$ where $u$ is a holomorphic disk and $z_0$ is an interior point of its domain. Consider
\beq\label{eqn87}
\widetilde {\mc M}_1^t (f;B_i) = \Big\{ (u, z_0) \in \widetilde {\mc M}_1(B_i) \ |\ u(z_0) \in D_f^\ss,\ \frac{\partial f}{\partial x_i}(u(z_0)) = 0 \Big\}.
\eeq
By Proposition \ref{blaschke}, the above set contains all holomorphic disks tangent to $D_f$ at the point $z_0$. 

\begin{lemma}
There is a comeager subset $\Gammait_k^{\rm reg}(L) \subset \Gammait_k(L)$ such that for all $f \in \Gammait_k^{\rm reg}(L)$ and all basic disk class $B_i$, $\widetilde {\mc M}_1^t (f; B_i)$ is a smooth manifold of dimension $N$.
\end{lemma}

\begin{proof}
Consider the smooth map 
\beqn
\begin{split}
{\mc F}: \Gammait_k(L) \times \widetilde {\mc M}_1(B) \to &\ {\mb C} \times {\mb C},\\
                               (f, u, z_0) \mapsto &\ \left( f(u(z_0)),\ \frac{\partial f}{\partial x_i}( u(z_0)) \right).
\end{split}
\eeqn
Let the infinitesimal deformations of $(f, u, z_0)$ be denoted by $(s, \xi, w)$. Suppose $(f, u, z_0) \in {\mc F}^{-1}(0)$ and $p_0: = u(z_0) \in D_f^\ss$. The partial derivative of the above map at $(f, u, z_0)$ in the $f$ direction is 
\beq\label{eqn88} 
s \mapsto \left( s(p_0),\ \frac{\partial s}{\partial   x_i}(p_0) \right) 
\eeq 
By Lemma \ref{lemma810}, $\frac{\partial}{\partial x_i}(p_0) \notin G_X|_{p_0}$. Then by the surjectivity of \eqref{eqn84} and \eqref{eqn85}, the linear map \eqref{eqn88} is surjective if $p_0 \in D_f^\ss$. So ${\mc F}^{-1}(0)$ is a smooth manifold in the locus where $p_0 \in D_f^\ss$. Then consider the projection ${\mc F}^{-1}(0) \to \Gammait_k(L)$. By Sard's theorem,  the set of regular values in $\Gammait_k(L)$ is a comeager subset $\Gammait_k^{{\rm reg}}(B_i)$. Hence all $f \in \Gammait_k^{\rm reg}(B_i)$, the condition for $\widetilde {\mc M}_1^t (f; B_i)$ holds.
\end{proof}

Now define 
\beqn
\widetilde {\mc M}_1^*(f; B_i) = \widetilde {\mc M}_1(f; B_i) \setminus \widetilde {\mc M}_1^t (f; B_i).
\eeqn
This is the moduli space of disks in class $B$ with one marked point such that the disk intersects transversely with $D_f^{\rm st}$ at the marked point. 

\begin{proof}[Proof of Lemma \ref{lemma86}]
The first two conditions on $(f_1, \ldots, f_h)$ are clearly satisfied by a generic $h$-tuple $(f_1, \ldots, f_h)$. To fulfill the third condition, for each pair $(f, f')\in \Gammait_k(L) \times \Gammait_{k'}(L)$ and each basic disk class $B_i$, consider the set 
\beqn
\widetilde {\mc M}_1^c (f, f'; B_i) = \Big\{ (u, z_0) \in \widetilde {\mc M}_1^* (f; B_i) \cap \widetilde {\mc M}_1^* (f'; B) \ |\ u(z_0) \in D_f \cap D_{f'} \Big\}.
\eeqn
For each element $(u, z_0)$, since $u(z_0)$ is in the stable locus and $u$, $D_f$, $D_{f'}$ are transverse,  Sard's theorem  implies  that for a generic pair $(f, f')$, the set $\tilde {\mc M}_1^c(f, f'; B_i)$ has the expected dimension $N$. Then for a generic $h$-tuple $(f_1, \ldots, f_h)$ the third condition of Lemma \ref{lemma86} is satisfied.
\end{proof}

\section{Quotients of products of spheres}\label{section9}

In this section we consider non-abelian quotients of products of spheres as an example of the theory developed so far. These examples were considered in Kirwan's book \cite{Kirwan_book}, and are quotients of products of spheres by a diagonal action of the  group of Euclidean rotations $SO(3)$, also called {\it polygon     spaces} in the literature.  We expect that the technique works for   many other examples; for example $SO(3) = PSU(2)$ to $PSU(k)$ which acts on ${\bf CP}^{k-1}$.  By considering dimensions, the diagonal   $PSU(k)$-action on the product of $k+1$ copies of ${\bf CP}^{k-1}$   has zero dimensional GIT quotient. Therefore, one can cook up a   Lagrangian in the GIT quotient of the diagonal $PSU(k)$-actions on   the product of $2l + k+1$ copies of ${\bf CP}^{k-1}$, where the   Lagrangian is diffeomorphic to the product of $l$ copies of   ${\bf CP}^{k-1}$. (cf. \cite[Section 5]{Jack_Smith_thesis}.)

To construct the polygon spaces, let $S^2$ be the unit sphere in
${\mb R}^3$ equipped with the Fubini--Study form. Here we distinguish
from the notation ${\bm S}^2$ which is used for domains. Viewing each
point of $S^2$ as a unit vector in ${\mb R}^3$, then $ K = SO(3)$ acts
diagonally on the product
\beqn V: = V_{2l +3}: = (S^2)^{2l +3}.
\eeqn 
Since $V$ is K\"ahler and the $SO(3)$-action preserves the complex
structure, the action extends to an action by $G = PSL(2)$. A moment
map of the $SO(3)$-action is 
\beqn \mu(x_1, \ldots, x_{2l}, x_{2l+1},
x_{2l +2}, x_{2l + 3}) = \sum_{i=1}^{2l+3} x_i \in {\mb R}^3 \cong
\mf{so}(3).  \eeqn 
Then the GIT quotient $X$ can be viewed as the moduli space of
equilateral $(2l +3)$-gons in ${\mb R}^3$ modulo rigid body motion.

A natural family of Lagrangians is given by specifying the lengths of diagonals. Let
$L_V \subset V_{2l+3}$ be the Lagrangian
\beqn
L_V:= (\bar \Delta_{S^2})^l \times \Delta_3
\eeqn
where each $\bar \Delta_{S^2}$ is the anti-diagonal of the product of the $2j-1$ and $2j$-th factors of $V$, and $\Delta_3 \subset S^2 \times S^2 \times S^2$ is the set of vectors $(x_{2l+1}, x_{2l +2}, x_{2l +3})$ such that $x_{2l +1} + x_{2l +2} + x_{2l +3} = 0$.  The subset $L_V$ is a compact $SO(3)$-invariant Lagrangian contained in $\mu^{-1}(0)$. Then $L_V$ descends to $L = L_V / SO(3) \subset X$.  The $SO(3)$-action on the $\Delta_3$ factor is free and transitive. Hence $L$ is diffeomorphic to the product of $l$ two-spheres in the polygon space $X$.

\begin{lemma}\label{lemma91}
  There is a linearization $\tilde V \to V$ of the $PSL(2)$-action on
  $V$ such that $L$ is strongly rational with respect to the induced
  line bundle $\tilde X \to X$.
\end{lemma}

\begin{proof}  First we treat the diagonals in the two-fold products.
  Let $V_i \cong S^2 \cong \bf{CP}^1$ be the $i$-th factor in
  $V$. Let $\tilde V_i \to V_i$ be the line bundle ${\mc O}(1)$
  equipped with the standard Fubini--Study Hermitian metric. Then the
  $PSL(2)$-action can be lifted to ${\mc O}(1)$ while the
  Fubini--Study metric is $SO(3)$-invariant. Then the moment map on
  $\tilde V_i$ coincides with the inclusion
  $\tilde \mu_i: V_i \to {\mb R}^3$. Then the bundle \beqn \tilde
  V_{2i-1} \boxtimes \tilde V_{2i} \to V_{2i-1} \times V_{2i} \eeqn
  provides a linearization of the diagonal $PSL(2)$-action on
  $V_{2i-1} \times V_{2i}$. Moreover, the anti-diagonal
  $\bar \Delta_{S^2}$ is the level set of the moment map
  $\tilde \mu_{2i-1} + \tilde \mu_{2i}$. Let $\tau: S^2 \to S^2$ be
  the anti-podal map $x \mapsto -x$ which is an anti-holomorphic
  involution. Then there is a canonical $SO(3)$-equivariant
  isomorphism
\beqn \tau^* {\mc O}(1) \cong \ov{{\mc O}(1)}.  \eeqn
  Then at every point
  $(x, -x) \in \bar \Delta_{S^2} \subset V_{2i-1} \times V_{2i}$,
  there is a canonical identification
 \beqn \tilde V_{2i-1}|_x \otimes
  \tilde V_{2i}|_{-x} \cong {\mc O}(1)|_x \otimes \ov{{\mc O}(1)}|_x
  \cong {\mb C} \eeqn 
which gives an $SO(3)$-equivariant
  trivialization of
  $\tilde V_{2i-1} \boxtimes \tilde V_{2i}|_{\bar
    \Delta_{S^2}}$.
  Since $\bar\Delta_{S^2}$ is the zero locus of the moment map, this
  trivialization is also parallel with respect to the connection
  $\nabla^{\tilde V_{2i-1} \boxtimes \tilde V_{2i}}$.

Next we consider the Lagrangians given by closed triangles. 
The factor  $\Delta_3 \subset V_{2l+1} \times V_{2l+2} \times V_{2l+3}$ is an orbit for the diagonal $SO(3)$-orbit. Hence the restriction of the bundle $\tilde V_{2l+1} \boxtimes \tilde V_{2l+2} \boxtimes \tilde V_{2l+3}$ to $\Delta_3$ has a trivialization given by the $SO(3)$-action. Since $\Delta_3$ is also the zero locus of the moment map, this trivialization is also parallel with respect to the Chern connection. 

We combining factors as follows. Define
\beqn
\tilde V = \tilde V_1 \boxtimes \cdots \boxtimes \tilde V_{2l+3} \to V_{2l+3}
\eeqn
with the product Hermitian metric and the product connection. This is
a linearization of the diagonal $PSL(2)$-action on $V_{2l+3}$. The
properties obtained above imply that the pair
$(\tilde V, \nabla^{\tilde V})$ descends to a bundle with connection
$(\tilde X, \nabla^{\tilde X})$ on $X$ whose restriction to $L$ is a
trivial bundle with the trivial connection. Hence Definition
\ref{defn27} item (a) is satisfied. By the construction of
\cite{Auroux_Gayet_Mohsen} (the holomorphic version), for sufficiently
large $k$, there exists holomorphic sections of $\tilde X{}^k$ which
are nowhere vanishing on $L$. Since $L$ is homeomorphic to a product
of spheres, there is only one homotopy class of trivializations of
$\tilde X|_L$. So Definition \ref{defn27} item (b) is also
satisfied. Therefore $L$ is strongly rational.
\end{proof}

To carry out our construction we also need the monotonicity condition.

\begin{lemma}\label{lemma92}
The pair $(V, L_V)$ is monotone with minimal Maslov index two.
\end{lemma}

\begin{proof}
It is standard knowledge that $(S^2 \times S^2, \bar\Delta_{S^2})$ is monotone and has minimal Maslov index four. Consider $H_2(S^2 \times S^2 \times S^2, \Delta_3)$. Since $\Delta_3 \cong SO(3) \cong {\bf RP}^3$, the homology exact sequence
\beqn
H_2(\Delta_3) \to H_2( S^2 \times S^2 \times S^2) \to H_2( S^2 \times S^2 \times S^2, \Delta_3) \to H_1(\Delta_3) \to H_1(S^2 \times S^2 \times S^2)
\eeqn
in integer coefficients is isomorphic to 
\beqn
0 \to {\mb Z} \oplus {\mb Z} \oplus {\mb Z} \to H_2(S^2\times S^2 \times S^2, \Delta_3) \to {\mb Z}_2 \to 0.
\eeqn
So for any $B \in H_2(S^2 \times S^2 \times S^2, \Delta_3)$, $2B$ can
be lifted to a spherical class in $S^2 \times S^2 \times S^2$.
Therefore, $(S^2 \times S^2 \times S^2, \Delta_3)$ is also monotone,
having minimal Maslov index two, and its monotonicity constant is the
same as that of $S^2$, and hence the same as that of
$(S^2\times S^2, \bar \Delta_{S^2})$.
\end{proof}

Moreover, we know that the unstable locus $V^{\rm us}$ consists of
points $(x_1, \ldots, x_{2l +3})$ in which at least $l +2$ coordinates
are equal. Since $l \geq 1$, the real codimension of $V^{\rm us}$ is
at least four. So all conditions of Definition \ref{defn11} are
satisfied. Lastly, conditions (S1)---(S4) of Definition
\ref{defn12} are verified by the following lemma (one can take
$S = \emptyset$).

\begin{lemma}\label{lemma93}
  If $B \in H_2^{SO(3)}(V_{2l+3}; {\mb Z})$ is a spherical class with
  $\langle [\omega_V + \mu], B \rangle > 0$, then the equivariant
  Chern number of $B$ is at least 1.
\end{lemma}

\begin{proof}
Since $\pi_1(SO(3)) \cong {\mb Z}_2$, there are two isomorphism classes of $SO(3)$-bundles over ${\bm S}^2$, the trivial and the nontrivial ones. Suppose $B$ is represented by a pair $(P, u)$ where $P \to {\bm S}^2$ is an $SO(3)$-bundle and $u$ is a section of $P (V_{2l+3})$, then $2B$ is contained in the image of 
\beqn
H_2( V_{2l+3} ; {\mb Z}) \to H_2^{SO(3)}(V_{2l+3}; {\mb Z}). 
\eeqn
Then for some $A \in H_2( V_{2l+3}; {\mb Z})$ one has
\begin{align*}
&\ 2 \langle  [\omega_V + \mu], B \rangle = \langle [\omega_V], A \rangle,\ &\ 2 \langle c_1^{SO(3)}(TV_{2l+3}), B \rangle = \langle c_1( TV_{2l+3}), A \rangle.
\end{align*}
Then the claim follows from the monotonicity of $V_{2l+3}$ itself.
\end{proof}

Therefore, all prerequisites of our abstract construction, i.e., (T1)---(T3) of Definition \ref{defn11}, (S1)---(S4) of Definition \ref{defn12}, and (L1)---(L2) of Definition \ref{defn27}, are hence verified.
One can construct the $A_\infty$ algebras $\wt{\fuk}{}^0(L)$, $\wt{\fuk}{}^\infty(L)$, and the morphism 
\beqn
\tilde {\bm \kappa}: \wt{\fuk}{}^0(L) \to \wt{\fuk}{}^\infty(L).
\eeqn
Notice that because $h^1(L) = 0$, there is only one local system hence
we do not have the dependence on $b$. Moreover, by Proposition
\ref{prop711}, Proposition \ref{prop717}, and the monotonicity
condition of $L_V$ (Lemma \ref{lemma92}), one obtains the following result.

\begin{thm}\label{thm93}
Both $\wt{\fuk}{}^0 (L)$ and $\wt{\fuk}{}^\infty (L)$ are weakly unobstructed.
\end{thm}

It is possible to use the standard complex structure $J_V$ on $V_{2l +3}$ to define the equivariant Fukaya algebra $\wt{\fuk}{}^0 (L)$. Nevertheless, the current situation is enough to determine the Floer cohomology, if we perturb $J_V$. 

\begin{thm}\label{thm94}
For any weakly bounding cochain the Floer cohomology of $L$ is isomorphic to the homology of the product of $l$ spheres. 
\end{thm}

\begin{proof}
Since $L$ is diffeomorphic to the product of two spheres, one can choose a Morse function on $L$ such that the number of critical points is equal to the sum of Betti numbers of $L$. Moreover, since the Maslov indices are all even, the Floer differential must be odd. Hence the Floer differential is actually zero and the Floer cohomology is isomorphic to the cochain complex, which is the same as the homology of the product of $l$ spheres.
\end{proof}

\appendix

\section{Strict Units}\label{sectiona1}

In this section we describe the combinatorics necessary to equip our Fukaya algebras with strict units. The construction follows from the idea of Charest--Woodward \cite{Charest_Woodward_2017} using weighted trees. Using this construction we prove Theorem \ref{thm714}.

\subsection{Weighted treed disks}

We use an extra set of notations to define weighted trees. Consider a commutative monoid with three elements $\circ, \bullet, \weight$. The multiplication is defined in such a way that $\circ$ is the unit, $\bullet$ is zero, and $\weight$ is idempotent.

\begin{defn}\label{defna1}
Let $(\Gamma, {\mf s})$ be an unbroken scaled tree. A {\it  weighting type} is a map ${\mf w}: T_\Gamma \cup L_\Gamma \to \{ \circ, \bullet, \weight\}$ such that ${\mf w} \equiv \bullet$ on $L_\Gamma$ and 
\beq\label{eqna1}
{\mf w}(t_{\rm out}) = \prod_{t \in T_\Gamma^{\rm in} \cup L_\Gamma} {\mf w}(t) \in \{ \circ, \bullet, \weight\}.
\eeq
The weighting type induces a partition $T_\Gamma = T_\Gamma^\circ \sqcup T_\Gamma^\bullet \sqcup T_\Gamma^\weight$. This requirement implies that there are nine weighted Y-shapes and three weighted $\Phi$-shapes, as shown in Figure \ref{figure6} and Figure \ref{figure7}. 

\begin{figure}[ht]
    \centering
    \includegraphics[width=\textwidth]{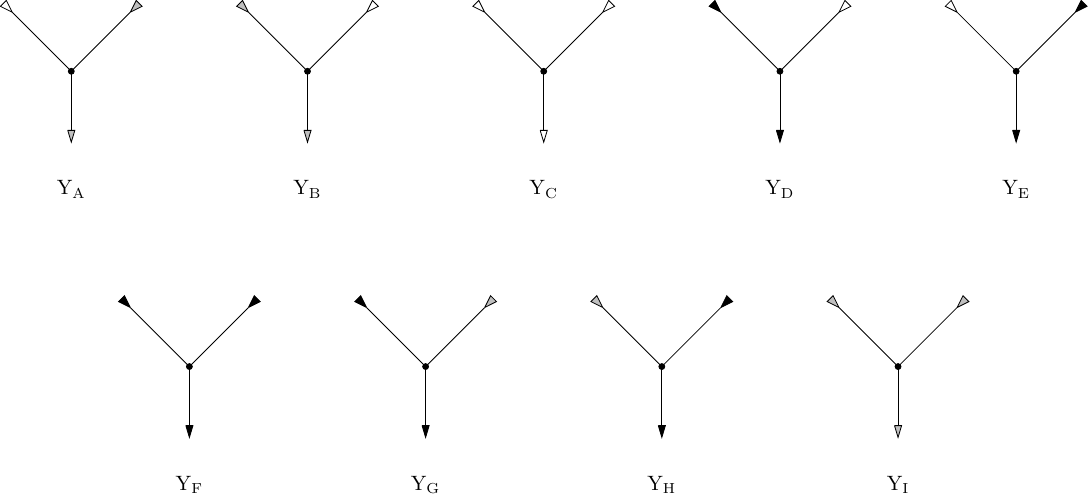}
    \caption{The allowed weighted Y-shapes. The grayscales indicate the weighting types on the boundary tails.}
    \label{figure6}
\end{figure}

\begin{figure}[ht]
    \centering
    \includegraphics{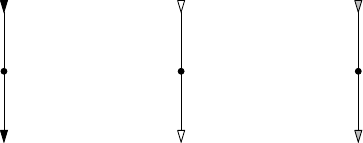}
    \caption{The allowed weighting types on a $\Phi$-shape. }
    \label{figure7}
\end{figure}

\end{defn}

Equation \eqref{eqna1} implies that the values of ${\mf w}$ on boundary inputs determines its value on the output. Hence the notion of weighting types naturally generalizes to the case of possibly broken trees.

\begin{defn}\label{defna2} A {\it weighted domain type} is a tuple 
\beqn
\hat\Gamma = (\Gamma, {\mf s}, {\mf p}, {\mf m},  {\mf w}),
\eeqn
where $\Gamma = (\Gamma, {\mf s}, {\mf p}, {\mf m})$ is an unweighted domain type and ${\mf w}$ is a weighting type on $(\Gamma, {\mf s})$. $\hat\Gamma$ is {\it  stable} if $(\Gamma, {\mf s})$ is stable. Let $\hat{\bf T}$ be the set of isomorphism classes of weighted domain types and let $\hat{\bf T}^{\rm st} \subset \hat{\bf T}$ be the subset of stable ones. 
\end{defn}

\begin{defn}\label{defna3} (Weighted trees) 
\begin{enumerate}
\item A {\it  weighting} on an {\it  unbroken} weighted domain type $\hat\Gamma$ is a function ${\bm \omega}: T_\Gamma \to [0, 1]$ that satisfies 
\beqn
{\bm \omega} |_{T_\Gamma^\circ} = 1,\ {\bm \omega} |_{ T_\Gamma^\bullet} = 0,\ {\bm \omega} |_{ T_\Gamma^\weight} \in (0, 1),\ {\bm \omega} (t_{\rm out}) = \prod_{t \in T_\Gamma^{\rm in}} {\bm \omega} (t). 
\eeqn
A weighting on a broken weighted domain type $\hat\Gamma$ is a collection of weightings on all its unbroken parts which coincide at breakings.

\item Let $\hat\Gamma$ and $\hat\Gamma'$ be two weighted domain types and ${\bm \omega}$, ${\bm \omega}'$ be weightings on them. $(\hat\Gamma, {\bm \omega} )$ is said to be {\it  equivalent} to $(\hat\Gamma', {\bm \omega}')$ if 
\begin{enumerate}
\item either $\hat\Gamma$ and $\hat\Gamma'$ have weighted output and there exists $a>0$ such that
\beqn
{\bm \omega} (t)^a = {\bm\omega}' ( \rho_T(t)),\ \forall t \in T_\Gamma,\ \ {\rm or}
\eeqn

\item the outputs of $\hat\Gamma$ and $\hat\Gamma'$ are not weighted and ${\bm \omega} = {\bm\omega}' \circ \rho_T$.
\end{enumerate}

\item Given $\hat\Gamma \in \hat{\bf T}$. A weighted treed disk of domain type $\hat\Gamma$ is a treed disk of domain type $\Gamma$ together with a weighting ${\bm\omega}$ on $\hat\Gamma$. Two weighted treed disks of domain type $\hat\Gamma$ are isomorphic if the treed disks are isomorphic and the weightings are equivalent. 
\end{enumerate}
\end{defn}

Define a partial order among weighted domain types as follows. 

\begin{defn}\label{defna4}
Let $\hat\Gamma'$, $\hat\Gamma$ be weighted domain types. We denote $\hat\Gamma' \leq \hat\Gamma$ if $\Gamma' \leq\Gamma$, and for each boundary tail $t' \in T_{\hat \Gamma'}$ with corresponding boundary tail $t \in T_{\hat \Gamma}$, one has
\begin{align*}
&\ t \in T_{\hat \Gamma}^\circ \Longrightarrow t' \in T_{\hat \Gamma'}^\circ;\ &\  t \in T_{\hat \Gamma}^\bullet \Longrightarrow t' \in T_{\hat \Gamma'}^\bullet.
\end{align*}
\end{defn}
Using Lemma \ref{lemma38} it is not hard to see that $\leq$ is still a partial order. 

\subsubsection{Moduli spaces}

For a {\it stable} weighted domain type $\hat\Gamma$, let ${\mc W}_{\hat\Gamma}$ be the set of all isomorphism classes of stable weighted treed disks modelled on $\hat\Gamma$. Define
\beqn
\ov{\mc W}_{\hat\Gamma}:= \bigsqcup_{ \substack{ \hat\Pi \leq \hat\Gamma \\ \hat\Pi\ {\rm stable}}} {\mc W}_{\hat\Pi}. 
\eeqn
The topology on $\ov{\mc W}_{\hat\Gamma}$ is defined via the following notion of sequential convergence. 

\begin{defn}\label{defn5}
Let ${\mc C}_\nu$ be a sequence of stable weighted treed disks of a stable weighted domain type $\hat\Gamma$ and ${\mc C}_\infty$ is another stable weighted treed disk of a stable weighted domain type $\hat\Gamma_\infty \leq \hat\Gamma$. We say that ${\mc C}_\nu$ converges to ${\mc C}_\infty$ if the convergence hold for the underlying unweighted treed disks, and in addition, the following conditions are satisfied.
\begin{enumerate}
\item If the outputs $t_{\rm out}$, $t_{\rm out}'$ of $\hat\Gamma$, $\hat\Gamma_\infty$ are not weighted, then the weights on boundary inputs converge. More precisely, let $\rho_T: T_{\Gamma_\infty} \to T_\Gamma$ be the bijection between tails (whose existence is implied by the convergence of unweighted treed disks), then for each $t_i \in T_\Gamma$, $w_\nu(t_i)  \to w_\infty(\rho_T^{-1}(t_i))$.
\item If the outputs $t_{\rm out}$, $t_{\rm out}'$ of $\hat\Gamma$, $\hat\Gamma_\infty$ are weighted, then there exist real numbers $a_\nu$, $a_\infty$ such that
\beqn
\big[ {\bm\omega}_\nu(t_{\rm out}) \big]^{a_\nu} = \big[ {\bm\omega}_\infty(t_{\rm out}') \big]^{a_\infty} = \frac{1}{2}.
\eeqn
We require that for all boundary inputs $t_{\uds i}$,
\beqn
\lim_{\nu \to \infty} \big[ {\bm\omega}_\nu(t_{\uds i}) \big]^{a_\nu} = \big[ {\bm\omega}_\infty(t_{\uds i}') \big]^{a_\infty} \in (0, 1).
\eeqn
\end{enumerate}
\end{defn}

To better understand the topology of $\ov{\mc W}_{\hat\Gamma}$, we look at a few special cases. 
\begin{enumerate}
\item When $\hat\Gamma$ is a Y-shape (see Figure \ref{figure6}), by Definition \ref{defna3}, for the first six configurations of Figure \ref{figure6}, ${\mc W}_{\hat\Gamma}$ is an isolated point. ${\mc W}_{{\rm Y}_{\rm G}}, {\mc W}_{{\rm Y}_{\rm H}}$ and ${\mc W}_{{\rm Y}_{\rm I}}$ are all homeomorphic to an open interval parametrized by the weightings on the weighted inputs. One can compactify them by adding boundary configurations as described by Figure \ref{figure8} and Figure \ref{figure9}.

\begin{figure}[ht]
    \centering
    \includegraphics{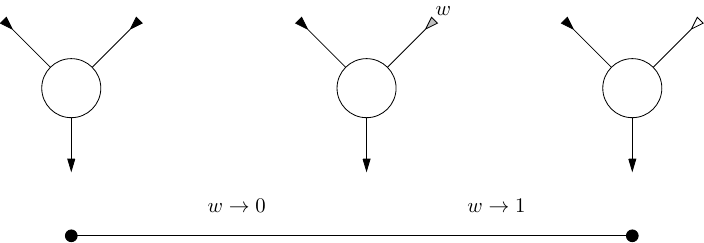}
    \caption{Moduli space ${\mc W}_{{\rm Y}_{\rm G}}$ and its compactification. The case of ${\mc W}_{{\rm Y}_{\rm H}}$ is similar.}
    \label{figure8}
\end{figure}

\begin{figure}[ht]
    \centering
    \includegraphics{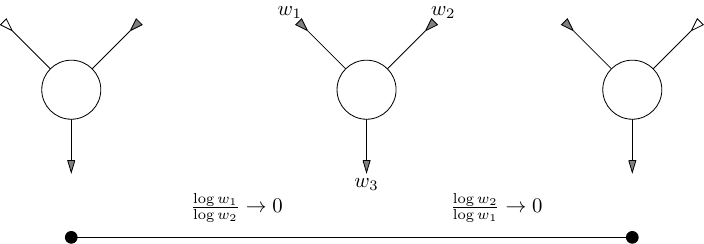}
    \caption{Moduli space ${\mc W}_{{\rm Y}_{\rm I}}$ and its compactification.}
    \label{figure9}
\end{figure}

\item When $\hat\Gamma$ is a $\Phi$-shape, by Definition \ref{defna1} and \ref{defna3}, ${\mc W}_{\hat\Gamma}$ is a single point.
\end{enumerate}

One can prove that moduli spaces of stable weighted treed disks are compact and Hausdorff. We left the proof to the reader. Lastly we give the dimension formula for a stable weighted domain type $\hat\Gamma$. Let $\Gamma$ be the underlying unweighted domain type. \eqref{eqn32} gives the dimension of ${\mc W}_{\Gamma}$, and we have
\beqn
{\rm dim} {\mc W}_{\hat\Gamma} = {\rm dim} {\mc W}_\Gamma + \# (T_{\hat \Gamma}^\weight \cap T_{\hat \Gamma}^{\rm in}) - \# ( T_{\hat \Gamma}^\weight \cap T_{\hat \Gamma}^{\rm out} ).
\eeqn

\subsubsection{Forgetting boundary tails}

We describe a forgetful map. Let $t$ be a forgettable incoming boundary tail of $\hat\Gamma$ and $\hat\Gamma_t$ be the domain type obtained by forgetting $t$ and stabilization. For any ${\mc C}$ representing a point in $\ov{\mc W}_{\hat\Gamma}$, this operation gives another stable weighted treed disk ${\mc C}_t$, and induces a continuous map $\pi_t: \ov{\mc W}_{\hat\Gamma} \to \ov{\mc W}_{\hat\Gamma_t}$. It also induces a contraction map ${\mc C} \to {\mc C}_t$, which gives a commutative diagram 
\beq\label{eqna2}
\vcenter{ \xymatrix{ \ov{\mc U}_{\hat\Gamma} \ar[r] \ar[d]_-{\tilde \pi_t} & \ov{\mc W}_{\hat\Gamma} \ar[d]^-{\pi_t} \\
    \ov{\mc U}_{\hat\Gamma_t} \ar[r] & \ov{\mc W}_{\hat\Gamma_t} } }.  
          \eeq

\begin{rem}\label{rema6}
Given a treed disk ${\mc C}$ of domain type $\hat\Gamma$ and an incoming forgettable boundary tail $t$ of $\hat\Gamma$. Suppose the unbroken part containing $t$ is not an infinite edge, a Y-shape, or a $\Phi$-shape. There are several possibilities of the change of shapes regarding the forgetful operation $\pi_t$.
\begin{enumerate}

\item When $v_t$, the vertex to which $t$ is attached, is still stable after forgetting $t$, only the interval $I_{t}$ is contracted by the stabilization.

\item If $v_t$ becomes unstable after forgetting $t$, then the two-dimensional component $\Sigma_{t}$ corresponding to $v_t$ is also contracted by the forgetful map. See Figure \ref{figure10}.

\begin{figure}[ht]
    \centering
    \includegraphics[width=\textwidth]{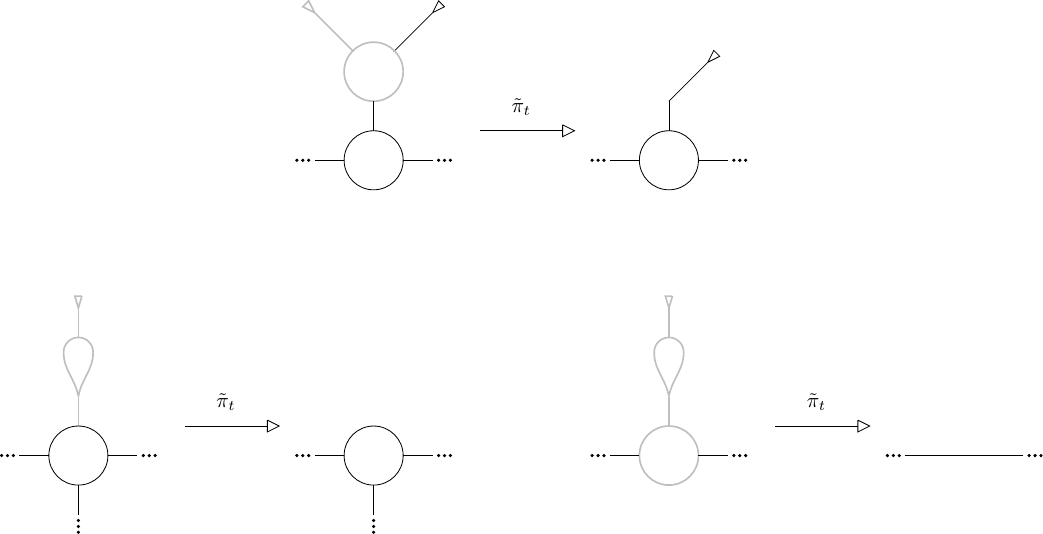}
    \caption{When forgetting a boundary tail, we regard the gray part of the treed disk is contracted.}
    \label{figure10}
\end{figure}

\item An extremal situation of the above case is described by the Figure \ref{figure11}, where we contract a whole unbroken part of a broken treed disk. 

\begin{figure}[ht]
    \centering
    \includegraphics{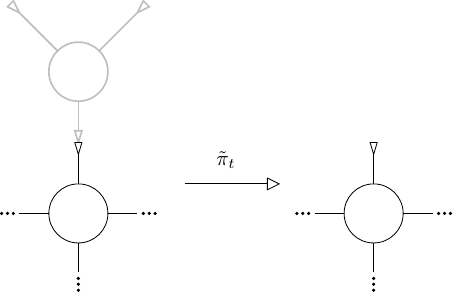}
    \caption{If we forgetting the upper-right boundary tail, then the gray part of the treed disk is contracted.}
    \label{figure11}
\end{figure}

\end{enumerate}
\end{rem}
 
\subsubsection{Coherent perturbations}

For a weighted domain type $\hat\Gamma$, let $\Gamma$ be its underlying unweighted domain type. Then there is a natural forgetful map $\ov{\mc U}_{\hat\Gamma}  \to \ov{\mc U}_\Gamma$ which covers another forgetful map $\ov{\mc W}_{\hat\Gamma} \to \ov{\mc W}_\Gamma$. Indeed, $\ov{\mc W}_{\hat\Gamma}$ is homeomorphic to the product of $\ov{\mc W}_\Gamma$ and a cube, while $\ov{\mc U}_{\hat\Gamma}$ is the pullback of $\ov{\mc U}_\Gamma \to \ov{\mc W}_\Gamma$. Then the collections of nodal neighborhoods we have used (see Lemma \ref{lemma45}) are pulled back to nodal neighborhoods
\beqn
\ov{\mc U}{}_{\hat\Gamma}^\delta \subset \ov{\mc U}{}_{\hat\Gamma}.
\eeqn
Define the space of perturbations ${\mc P}_{\hat\Gamma}$ in the same way as Definition \ref{defn46}. Moreover, we define the notion of coherent system of perturbation data as follows. 

\begin{defn} (cf. Definition \ref{coherent}) \label{defna7}
Let 
\beqn
\hat{\bm P} = \{ P_{\hat\Gamma} \in {\mc P}_{\hat\Gamma}\ |\ \hat\Gamma \in \hat{\bf T}^{\rm st} \}
\eeqn
be a system of perturbation data. We say this system is {\it  coherent} if it satisfies similar conditions to those in Definition \ref{coherent}. In addition, we require the following condition.

\begin{itemize}

\item Let $t \in T_{\hat\Gamma}^{\rm in}$ be the first forgettable input and assume that $\hat\Gamma{}_t$ is nonempty. Let $\ov{\mc U}_{\hat\Gamma, t} \subset \ov{\mc U}_{\hat\Gamma}$ be the set of points which are not contracted by the map $\tilde \pi_t$. Then 
\beqn
P_{\hat\Gamma}|_{\ov{\mc U}_{\hat\Gamma, t}} = P_{\hat \Gamma{}_t} \circ \tilde \pi_t.
\eeqn
\end{itemize}
\end{defn}

\begin{lemma}\label{lemmaa8}
Let $\hat\Gamma$ one of the first five ${\rm Y}$-shapes in Figure \ref{figure6}. Suppose $P_{\hat\Gamma} = (F_{\hat\Gamma}, J_{\hat\Gamma})$ belongs to a coherent system of perturbation data.
\begin{enumerate}

\item If $\hat\Gamma\in \{ {\rm Y}_{\rm A}, {\rm Y}_{\rm B}, {\rm Y}_{\rm D}, {\rm Y}_{\rm E}\}$, then $F_{\hat\Gamma} = F_L$ on the edges that are not forgettable. 

\item If $\hat\Gamma = {\rm Y}_{\rm C}$, then $F_{\hat\Gamma} = F_L$ on the outgoing and the second incoming edges. 
\end{enumerate}
\end{lemma}

\begin{proof}
We prove for the case $\hat\Gamma = {\rm Y}_{\rm D}$. The other cases are the same. Consider an unbroken domain type $\hat\Xi$ as shown in Figure \ref{figure12}. Consider a sequence of weighted treed disks ${\mc C}_k$ converge to a broken configuration by degenerating the edge $e$. In ${\mc C}_k$ we identify the interval $I_e$ with $[-\ell_k(e), 0]$ where $\ell_k(e)$ is the length of the edge $e$ which tends to infinity as $k \to +\infty$. Call the forgettable tail by $t$ and the other tail which is not collapsed by the forgetful map by $t'$. Identify the tail $I_{t'}$ with $(-\infty, -\ell_k(e)] \cong (-\infty, 0]$. 

Suppose $P_{\hat\Gamma}$ and $P_{\hat \Xi}$ belong to the same coherent system of perturbation data. Let $f_{k, t'}$ and $f_{k, e}$ be the restriction of $F_{\hat \Xi} - F_L$ to the tail $I_{t'}$ and the edge $I_e$. By Definition \ref{defna7}, after forgetting the tail $t$, the dotted part in Figure is collapsed and all $P_{\hat \Xi}$ induce the same perturbation data $P_{\hat \Xi_t}$. In particular, if denote the tail in $\hat \Xi_t$ corresponding to the surviving tail $t'$ and the edge $e$ by $t' \cup e$, and denote the restriction of $F_{\hat \Xi_t} - F_L$ to $I_{t' \cup e}$ by $f_{t' \cup e}$. Then for all $k$
\beqn
f_{k, t'} \cup f_{k, e}: (-\infty, -\ell_k(e)] \cup [-\ell_k (e), 0] \times L \to {\mb R}
\eeqn
is equal to $f_{t' \cup e}$. Since $f_{t' \cup e}$ vanishes near $-\infty$, we see for $k$ large, $f_{k, t'} \equiv 0$ and $f_{k, e}$ is supported in a fixed subinterval $[-a, 0] \subset [-\ell_k(e), 0]$. On the other hand, as $k \to \infty$, $P_{\hat \Xi}$ converges to the perturbation data on the limiting broken treed disk determined by the perturbation data $P_{\hat \Gamma}$ and the perturbation data on the other unbroken part. Then we see on the tails in $\hat \Gamma$ corresponding to the solid part in Figure \ref{figure12} (the tail $t'$ and the left piece of $e$), the difference $F_{\hat \Gamma} - F_L$ has to vanish. 
\end{proof}

\begin{figure}[ht]
    \centering
    \includegraphics{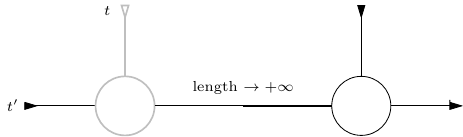}
    \caption{The gray part is collapsed by the forgetful map. The thick part in the edge $e$ is the support of the perturbation of the Morse function on the Lagrangian $L$.}
    \label{figure12}
\end{figure}

\subsection{Moduli spaces and transversality}

In this subsection we construct a strongly regular coherent system of perturbation data $\hat{\bm P}$ for all stable weighted domain types that extends the system ${\bm P}$ we have used for the unweighted case.

\subsubsection{Weighted map types}

Introduce two more generators of the Floer cochain complex
\begin{align*}
&\ {\bm x}_M^\circ,\ &\ {\bm x}_M^{\weight}.
\end{align*}
Here ${\bm x}_M^\circ$ is called the {\it  forgettable} ${\bm x}_M$ (which will be the strict units), and ${\bm x}_M^\weight$ is called the {\it  weighted} ${\bm x}_M$. To distinguish from these two elements, we re-denote the original ${\bm x}_M \in {\rm Crit} F_L$ by ${\bm x}_M^\bullet$ and call it the {\it  unforgettable} ${\bm x}_M$. Other elements of ${\rm Crit} F_L$ are also regarded as unforgettable. For the two new elements, define their degrees to be $| {\bm x}_M^\circ | = 0$ and $| {\bm x}_M^\weight | = -1$.

Now we extend the notion of map types, given by Definition \ref{defn51}, to {\it  weighted} map types. A weighted map type is a tuple $\hat{\bm \Gamma} = ( \hat\Gamma, {\bf B}, {\bf O}, {\bf X})$ where $\hat\Gamma \in \hat{\bf T}$ is a weighted domain type, ${\bf B}$ and ${\bf O}$ are defined in the same way as in Definition \ref{defn51}, and ${\bf X} = ({\bm x}_1, \ldots, {\bm x}_l; {\bm x}_\infty )$ is a sequence of elements in ${\rm Crit} F_L \cup \{ {\bm x}_M^\circ, {\bm x}_M^{\weight}\}$.

As in the unweighted case, the underlying domain type of a weighted map type $\hat{\bm\Gamma}$ is always denoted by the regular non-boldface symbol $\hat\Gamma$.  For any coherent system of perturbation data $\hat {\bm P}$, we can define the moduli spaces ${\mc M}_{\hat{\bm \Gamma}}(P_{\hat \Gamma})$ in the same way as Section \ref{section5}, and the open subset ${\mc M}_{\hat{\bm\Gamma}}^*(P_{\hat \Gamma}) \subset {\mc M}_{\hat{\bm \Gamma} }(P_{\hat\Gamma})$ consisting of configurations which do not have a nontrivial component mapped entirely into the stabilizing divisor $D$ (or $D \qu G$). In order to achieve transversality, we also look at special weighted map types which are uncrowded, controlled, and which have no components corresponding to holomorphic spheres in $V$. The notions of (strong) regularity are defined in exactly the same way as Definition \ref{regularity}.

\subsubsection{Perturbed gradient flows}

In order to achieve transversality for weighted map types while maintaining the coherence condition, especially the one about
forgetting forgettable boundary tails (see Definition \ref{defna7}),
we need to carefully choose certain perturbations to the gradient flow
equation.  Let $f^\circ \in C_c^\infty( (-\infty, 0] \times L )$ be a
time-dependent perturbation of the function $F_L: L \to {\mb R}$ and
denote $F^\circ = F_L + f^\circ$. Consider the perturbed gradient flow
equation \beq\label{eqna3} \dot{x}(s) + \nabla F^\circ (x(s)) = 0,\
-\infty < s \leq 0,\ \lim_{s \to -\infty} x(s) = {\bm x}_M.  \eeq
Denote the moduli space of solutions be ${\mc M}_\circ ({\bm x}_M)$.  The dimension of 
${\mc M}_\circ ({\bm x}_M)$
is ${\rm dim} L $, since ${\bm x}_M$ is the maximum.

\begin{lemma}\label{lemmaa9}
There exists $f^\circ \in C_c^\infty((-\infty, 0]\times L)$ satisfying the following condition. 
\begin{enumerate}
\item $f^\circ$ vanishes in a neighborhood of ${\bm x}_M$;

\item \label{lemmaa9b} For any ${\bm x} \in {\rm Crit} F_L$, there is a unique solution \eqref{eqna3} with $x(0) = {\bm x}$. 

\item The evaluation at time zero ${\mc M}_\circ ( {\bm x}_M) \to L$ is a submersion. 
\end{enumerate}
\end{lemma}

The proof is left to the reader. Choosing a function $F^\circ = F_L + f^\circ$ where $f^\circ$ satisfies the conditions of Lemma \ref{lemmaa9}, the perturbation data for certain simple weighted domain types are also determined by the coherence condition. The following two lemmata give these perturbation data explicitly and prove that they are regular. Their proofs are left to the reader.

\begin{lemma}\label{lemmaa10}
Choose $f^\circ$ satisfying the conditions of Lemma \ref{lemmaa9}. Let $\hat\Gamma$ be one of the first five ${\rm Y}$-shapes listed in Figure \ref{figure6}. Then the perturbation $P_{\hat\Gamma}$ that is equal to $f^\circ$ on the (first, if any) forgettable incoming edge and is trivial on other edges and trivial on the two-dimensional component is (strongly) regular. 

Moreover, if $\hat {\bm \Gamma}$ is an uncrowded and controlled refinement of $\hat\Gamma$, then $\hat {\bm \Gamma}$ is classified by the labellings of the boundary tails. Then we have

\begin{enumerate}
\item For $\hat\Gamma = {\rm Y}_{\rm A}, {\rm Y}_{\rm B}, {\rm Y}_{\rm C}$, there is only one such refinement $\hat {\bm \Gamma}$, where the boundary tails are labelled by ${\bm x}_M^\circ$ or ${\bm x}_M^{\weight}$ according to the weighting domain types of the boundary tails. In all cases ${\rm index} \hat{\bm \Gamma} = 0$ and ${\mc M}_{\hat {\bm \Gamma}}$ contains a single element represented by the trivial solutions. 

\item For $\hat\Gamma = {\rm Y}_{\rm D}, {\rm Y}_{\rm E}$, a refinement $\hat {\bm \Gamma}$ is essentially a pair of labelling $({\bm x}_{\rm in}^\bullet, {\bm x}_{\rm out}^\bullet)$ on the two unforgettable boundary tails. If ${\rm index} \hat{\bm \Gamma} = 0$, then $| {\bm x}_{\rm in}^\bullet| = | {\bm x}_{\rm out}^\bullet |$. In this case, the moduli space ${\mc M}_{\hat{\bm \Gamma}}$ is empty if ${\bm x}_{\rm in}^\bullet \neq {\bm x}_{\rm out}^\bullet$ and contains a single element if ${\bm x}_{\rm in}^\bullet = {\bm x}_{\rm out}^\bullet$. In the latter case, the only element is represented by the solution which is constant on the unforgettable edges and constant on the disk component, and is equal to the unique solution to \eqref{eqna3} with $x(0) = {\bm x}_{\rm in}^\bullet = {\bm x}_{\rm out}^\bullet$.   (Uniqueness follows from Lemma \ref{lemmaa9}). 
\end{enumerate}
\end{lemma}

\subsubsection{Canonical extension of perturbation data to weighted domain types}

Suppose we are given a coherent system of perturbation data for all unweighted domain types. We would like to extend this family to include weighted domain types
so that forgetful maps exist in the case of infinite weightings.  
Given an unbroken $\hat\Gamma \in \hat{\bf T}^{\rm st}$ and suppose $t \in T_{\hat\Gamma}$ is its first forgettable input. Consider the case that $\hat \Gamma_t \neq \emptyset$. Suppose we have chosen a perturbation $P_{\hat\Gamma_t}$. Then upon choosing $F^\circ$, there is a uniquely determined $P_{\hat\Gamma} = (F_{\hat\Gamma}, J_{\hat\Gamma})$, called the $F^\circ$-extension of $P_{\hat\Gamma_t}$, which, together with $P_{\hat\Gamma_t}$, respects the forgetful operation forgetting the boundary tail $t$. 

The construction of the perturbations on weighted types is defined in more detail as follows. For each treed disk ${\mc C}$ with underlying weighted domain type $\hat\Gamma$, let ${\mc C}_t$ be the treed disk obtained by forgetting the boundary tail $t$ and stabilizing. Then certain one-dimensional or two-dimensional components of ${\mc C}$ are contracted. The preserved components corresponds to components of ${\mc C}_t$. Then $P_{\hat\Gamma_t}$ determines the value of $P_{\hat\Gamma}|_{{\mc C}}$ on all the preserved components. To determine the value of $P_{\hat\Gamma}|_{{\mc C}}$ on the contracted components, we consider the following three cases.
\begin{enumerate}
\item If the forgetful operation only contracts a tail $I_t  \cong (-\infty, 0]$, then we define $F_{\hat\Gamma}|_{I_t}$ to be $F^\circ$. 

\item If the forgetful operation contracts a boundary tail $I_t$ and a vertex $v \in V_{\hat\Gamma}^0 \cup V_{\hat\Gamma}^\infty$ (see Figure \ref{figure10} and Figure \ref{figure11}), then we define $F_{\hat\Gamma}|_{I_t}$ in the same way as above and define $J_{\hat\Gamma}|_{\Sigma_{v}}$ to be $J_V$.

\item Suppose the forgetful operation contracts a vertex $v \in V_{\hat\Gamma}^1$. Then $v$ has no superstructure or leaf. Let $e$ be the edge starting from $v$ towards the root. Then $e$ has finite length. Then $I_t \cup I_e$ can be identified with $(-\infty, 0] \cup [0, -\ell (e)] \cong (-\infty, 0]$. Then define $F_{\hat\Gamma}|_{I_e \cup I_e}$ to be $F^\circ$ via this identification. Moreover, the forgetful map may or may not contract another vertex $v'\in V_{\hat\Gamma}^\infty$ i.e., the other end of $e$. In either case, define the restriction of $J_{\hat\Gamma}$ on the contracted two-dimensional component(s) to be $J_V$ (see Figure \ref{figure10}). 
\end{enumerate}

The following lemma shows that $P_{\hat\Gamma}$ is (strongly) regular as long as $F^\circ$ satisfies the conditions in Lemma \ref{lemmaa9} and $P_{\hat\Gamma_t}$ is (strongly) regular. 

\begin{lemma} \label{lemmaa11}
Suppose $F^\circ$ satisfies the conditions of Lemma \ref{lemmaa9} and $P_{\hat\Gamma_t}$ is a strongly regular perturbation, then the $F^\circ$-extension $P_{\hat\Gamma}$ of $P_{\hat\Gamma_t}$ is also strongly regular. 
\end{lemma}

\begin{proof}
The strong regularity of $P_{\hat\Gamma}$ follows from Item (c) of Lemma \ref{lemmaa9}.  
\end{proof}

Now we construct an extension of the previously chosen ${\bm P}$ for unweighted domain types to all weighted domain types. Let $\hat{\bf T}^\bullet \subset \hat{\bf T}^{\rm st}$ be the subset of weighted domain types whose inputs are all unforgettable. Then $\hat{\bf T}^\bullet \cong {\bf T}^{\rm st}$ and ${\bm P}$ induces a coherent system $\hat{\bm P}{}^\bullet$ for all domain types in $\hat{\bf T}^\bullet$. Now let $\hat{\bf T}^{\bullet \circ} \subset \hat{\bf T}^{\rm st}$ consisting of all stable weighted domain types whose boundary inputs are unforgettable or forgettable, but not weighted. By forgetting forgettable inputs successively, any unbroken $\hat\Gamma\in \hat{\bf T}^{\bullet \circ}$ can be reduced to some $\hat\Gamma' \in \hat{\bf T}^\bullet$ or one of ${\rm Y}_{\rm A}, {\rm Y}_{\rm B}, {\rm Y}_{\rm C}, {\rm Y}_{\rm D}, {\rm Y}_{\rm E}$. Then upon choosing $F^\circ$, we can extend ${\bm P}$ to a system $\hat{\bm P}{}^{\bullet \circ}$ to all $\hat\Gamma \in \hat{\bf T}^{\bullet \circ}$. It is routine to check that this new system is coherent in the sense of Definition \ref{defna7}.

Lastly, we can use a variant of Lemma \ref{lemma68} to inductively construct strongly regular perturbations for domain types which have weighted inputs. The base cases of the induction include choosing regular perturbations (independently) for the  domain types ${\rm Y}_{\rm G}, {\rm Y}_{\rm H}, {\rm Y}_{\rm I}$ (see Figure \ref{figure6}) and the domain type $\Phi_{\rm C}$ (see Figure \ref{figure7}). The induction procedure provides us a strongly regular coherent system of perturbation data $\hat {\bm P}$, which grants each stable weighted domain type $\hat\Gamma \in \hat{\bf T}^{\rm st}$ a strongly regular perturbation $P_{\hat\Gamma} = ( F_{\hat \Gamma}, J_{\hat\Gamma} )$. 

Given the strongly regular coherent system of perturbation data $\hat{\bm P}$, for any uncrowded, controlled weighted map type $\hat{\bm \Gamma}$ which has no holomorphic spheres no spheres upstairs, the subset ${\mc M}_{\hat{\bm \Gamma}}^*( P_{\hat\Gamma}) \subset {\mc M}_{\hat{\bm \Gamma}} (P_{\hat\Gamma})$ of stable weighted treed scaled vortices in which no nontrivial component is mapped entirely into $D$ is a smooth manifold of dimension
\beqn
{\rm dim} {\mc M}_{\hat{\bm \Gamma}}^*(P_{\hat\Gamma}) = {\rm index} \hat{\bm \Gamma}:= {\rm dim} {\mc W}_{\hat\Gamma} + |{\bf X}| + 2c_1( \hat{\bm \Gamma}) - \sum_{i=1}^k \delta_{O_i}.
\eeqn

We would like to extend the refined compactness result (Proposition \ref{prop611}) to include moduli spaces for weighted map types. The notions of essential map types (see Definition \ref{essential}) can be extended to include weighted map types without changing a word. Then we have the following extension about refined compactness (Proposition \ref{prop611}). Its proof is completely the same as before.

\begin{prop}\label{propa12}
Let $\hat{\bm \Gamma}$ be an essential weighted map type. 
\begin{enumerate}

\item When ${\rm index} \hat{\bm \Gamma} = 0$, one has
\beqn
\ov{ {\mc M}_{\hat{\bm \Gamma}}^* (P_{\hat\Gamma})} = {\mc M}_{\hat{\bm \Gamma}}^*(P_{\hat\Gamma}).
\eeqn

\item When ${\rm index} \hat{\bm \Gamma} = 1$, one has 
\beqn
\ov{ {\mc M}_{\hat{\bm\Gamma}}^*(P_{\hat\Gamma})} \setminus {\mc M}_{\hat{\bm\Gamma}}^*(P_{\hat\Gamma}) = \bigsqcup {\mc M}_{\hat{\bm \Pi}}^*(P_{\hat\Pi}),
\eeqn
where the disjoint union is taken over with all weighted map types $\hat{\bm \Pi}$ obtained from $\hat{\bm \Gamma}$ by applying exactly one operation listed in Proposition \ref{prop611}, plus the following one more possibility: 
\begin{itemize}
\item One weighted input of $\hat{\bm \Gamma}$ becomes forgettable or unforgettable while the weighting type of the output of $\hat{\bm \Pi}$ may or may not change according to the weighting rule \eqref{eqna1}.
\end{itemize}
\end{enumerate}
\end{prop}

\subsection{Strict unitality}

Now we start to equip the $A_\infty$ algebras with strict units and prove Theorem \ref{thm714}. Let $\widetilde{CF}{}^* (L; {\bm \Lambda})$ be the free ${\bm \Lambda}$-module
\beqn
\widetilde{CF} (L; {\bm \Lambda}) = CF ( L; {\bm \Lambda} ) \oplus {\bm \Lambda} {\bm e} \oplus {\bm \Lambda} {\bm p} = CF ( L; {\bm \Lambda} ) \oplus {\bm \Lambda} {\bm x}_M^\circ \oplus {\bm \Lambda} {\bm x}_M^\weight.
\eeqn
Then we extend the $A_\infty$ compositions ${\bm m}_l^0$, ${\bm m}_l^\infty$ and the $A_\infty$ morphism defined in Section \ref{section7} to $\widetilde{CF} (L; {\bm \Lambda})$. Indeed, by counting elements in zero-dimensional moduli spaces of scale $0$, $\infty$, and mixed scale, one can define in the same way these extensions, except for $\tilde {\bm m}_1^0 ( {\bm x}_M^\weight)$ and $\tilde {\bm m}_1^\infty ({\bm x}_M^\weight)$. We define
\beq\label{eqna4}
\tilde {\bm m}_1^0 ({\bm x}_M^\weight) = {\bm x}_M^\circ - {\bm x}_M^\bullet + \sum_{{\bm x}_\infty }  \langle {\bm x}_M^\weight | {\bm x}_\infty \rangle^0 \cdot {\bm x}_\infty
\eeq
where $\langle {\bm x}_M^\weight | {\bm x}_\infty \rangle^0$ is defined in the same way as \eqref{eqn74}. Similarly,
\beqn
\tilde {\bm m}_1^\infty ({\bm x}_M^\weight) = {\bm x}_M^\circ - {\bm x}_M^\bullet + \sum_{{\bm x}_\infty}  \langle {\bm x}_M^\weight | {\bm x}_\infty \rangle^\infty \cdot {\bm x}_\infty
\eeqn
where $\langle {\bm x}_M^\weight | {\bm x}_\infty \rangle^\infty$ is defined in the same way as \eqref{eqn75}.

The following proposition proves Theorem \ref{thm714} Item (a)--Item (c). 

\begin{prop}\label{propa13} $\wt{\fuk}{}^0 (L) = (\wt{CF} (L;{\bm \Lambda} ), \{ \tilde {\bm m}_l^0 \}_{l\geq 0} )$ and $\wt{\fuk}{}^\infty (L) = (\wt{CF} ( L; {\bm \Lambda} ), \{ \tilde {\bm m}_l^\infty \}_{l\geq 0} )$ are strictly unital $A_\infty$ algebras with ${\bm x}_M^\circ$ being their strict units. $\tilde {\bm \kappa} = ( \tilde \kappa_0, \tilde \kappa_1, \ldots)$ is a unital $A_\infty$ morphism which is a higher order deformation of the identity.
\end{prop}

\begin{proof}
The fact that $\{ \tilde {\bm m}_l^0 \}_{l \geq 0}$ and $\{ \tilde {\bm m}_l^\infty \}_{l \geq 0}$ define $A_\infty$ algebra structures over $\wt{CF} ( L; {\bm \Lambda})$ and the fact that $\{ \tilde \kappa_l \}_{l \geq 0}$ define an $A_\infty$ morphism follow from the same argument as before, in which the refined compactness theorem (Proposition \ref{propa12}) plays a crucial role. In the following we explain the unitality about ${\bm e} = {\bm x}_M^\circ$. 

We first prove the unitality of ${\bm x}_M^\circ$ for $\wt{\fuk}{}^0 (L)$. The first condition is $\tilde {\bm m}_1^0 ( {\bm x}_M^\circ) = 0$. Notice that by definition, there are two parts contributing to $\tilde {\bm m}_1^0({\bm x}_M^\circ)$, the Morse differential and the contribution from the counts of moduli spaces for essential map types. Since ${\bm x}_M^\circ$ is the unique maximum which has to represent a homology class, its Morse differential vanishes. On the other hand, consider essential map types $\hat{\bm \Gamma}$ of index zero with only one boundary input $t$ which is labelled by ${\bm x}_M^\circ$. For the underlying weighted domain type $\hat\Gamma$, recall that $\hat\Gamma_t$ is the weighted  domain type obtained by forgetting the boundary tail $t$ and stabilizing. If $\hat\Gamma_t$ is nonempty, we want to prove that the moduli space ${\mc M}_{\hat{\bm \Gamma}}^*(P_{\hat\Gamma})$ is empty. Indeed, if ${\mc V} \in {\mc M}_{\hat{\bm \Gamma}}^*(P_{\hat\Gamma})$, then since the perturbation data respects forgetting a forgettable tail (see Item (e) of Definition \ref{defna7}), by forgetting the tail $t$, we obtain an element ${\mc V}_t \in {\mc M}_{\hat{\bm \Gamma}_t}^* (P_{\hat\Gamma_t})$. However, the expected dimension of the latter moduli space is one less than that of ${\mc M}_{\hat{\bm \Gamma}}^*(P_{\hat\Gamma})$, which, by Item (e) of Definition \ref{defna7} and the transversality, is impossible. Hence ${\mc M}_{\hat{\bm \Gamma}}^*(P_{\hat\Gamma}) = \emptyset$. Therefore $\tilde m_0^1({\bm x}_M^\circ) = 0$. Next, we show that 
\beqn
\tilde {\bm m}_2^0 ({\bm x}_M^\circ, {\bm a}) = {\bm a} =  (-1)^{|{\bm a}|} \tilde {\bm m}_2^0 ( {\bm a}, {\bm x}_M^\circ),\ \forall {\bm a} \in \wt{\rm Crit} F_L.
\eeqn
If ${\bm a} = {\bm x}_M^\circ$, then the second identity is the same as the first one. If ${\bm a} \neq {\bm x}_M^\circ$, then the proof of the second identity is also similar to the proof of the first one. Hence we only show $\tilde {\bm m}_2^0 ({\bm x}_M^\circ, {\bm a}) = {\bm a}$. Indeed, suppose a configuration with map type $\hat{\bm \Gamma}$ and weighted domain type $\hat\Gamma$ contributes to $\tilde {\bm m}_2^0 ({\bm x}_M^\circ, {\bm a})$. Let $t$ be the first input of $\hat\Gamma$. Then for the same reason as above, $\hat\Gamma_t = \emptyset$, which implies that $\hat\Gamma$ is a {\rm Y}-shape. The condition that the index of $\hat{\bm \Gamma}$ is zero implies that the other two boundary tails are labelled by critical points of the same Morse index, say ${\bm x}'$ and ${\bm x}''$. By our construction of the perturbation
data on Y-shapes (see Lemma \ref{lemmaa9}), ${\mc M}_{\hat{\bm \Gamma}}^*(P_{\hat\Gamma})  = \emptyset$ if ${\bm x}' \neq {\bm x}''$, and contains a single element if ${\bm x}' = {\bm x}''$. The sign of this single element is positive. Hence it shows that $\tilde {\bm m}_2^{\it  eq} ({\bm x}_M^\circ, {\bm a}) = {\bm a}$ for any generator, and hence all element of $\wt{CF}(L; {\bm \Lambda})$. Similarly one has $\tilde {\bm m}_l^{\it eq} (\cdots, {\bm x}_M^\circ, \cdots) = 0$ for $l \geq 3$ because for any unbroken weighted domain type with more than three inputs, forgetting one input results in a nonempty weighted domain type. 

The unitality of ${\bm x}_M^\circ$ for $\wt{\fuk}{}^\infty (L)$ is
exactly the same. We now prove the unitality of the morphism, which
means
\begin{align}\label{eqna5}
&\ \tilde \kappa_1 \big( {\bm x}_M^\circ \big) = {\bm x}_M^\circ;\ &\ \tilde \kappa_l \big( \ \cdots, {\bm x}_M^\circ, \cdots \big) = 0,\ \forall l \geq 2.
\end{align}
To prove the first identity, consider any map type $\hat{\bm \Gamma}$ such that the moduli space ${\mc M}_{\hat{\bm \Gamma}}$ may contribute to $\tilde \kappa_1({\bm x}_M^\circ)$. Suppose $t$ is the boundary input labelled by ${\bm x}_M^\circ$. If $\omega( \hat{\bm \Gamma} ) > 0$, then $L_{\hat\Gamma} \neq \emptyset$. Therefore, $\hat\Gamma_t$ is nonempty, which results in a nonempty moduli space ${\mc M}_{\hat{\bm \Gamma}{}_t}$ with expected dimension $-1$. Hence $\omega(\hat{\bm \Gamma} ) = 0$. Then since ${\rm index} \hat{\bm \Gamma} = 0$, the output of $\tilde\Gamma$ must be labelled by ${\bm x}_M^\circ$ or ${\bm x}_M^\bullet$. However the latter was excluded by our requirement on weighting types (see Figure \ref{figure7}). Therefore $\tilde \kappa_1({\bm x}_M^\circ) = {\bm x}_M^\circ$. On the other hand, for all admissible weighted map type $\hat{\bm \Gamma}$ of mixed scale with $l \geq 2$ boundary inputs one of whose boundary tails is labelled by ${\bm x}_M^\circ$, forgetting the first forgettable input results in a nonempty map type $\hat{\bm \Gamma}_t$ with ${\rm index} \hat{\bm \Gamma}_t = -1$. Hence ${\mc M}_{\hat{\bm \Gamma}_t} = \emptyset$ by transversality and the coherence of the perturbation data. So the second identity of \eqref{eqna5} is proved. 

Lastly, one can still choose perturbation data satisfying properties in Remark \ref{rem69}, such that the $A_\infty$ algebras $\wt{\fuk}{}^0 (L)$ and $\wt{\fuk}{}^\infty  (L)$ are identical on the classical level and $\tilde {\bm \kappa}$ is a higher order deformation of the identity. This finishes the proof.
\end{proof}

It remains to prove item (d) of Theorem \ref{thm714}. Indeed, suppose (P1) holds. To compute
$\tilde {\bm m}_l^0 ({\bm p}, \ldots, {\bm p})$ we consider map types with all inputs being weighted. Using a perturbation data
satisfying the same condition as used in the proof of Proposition
\ref{prop711}, i.e. being sufficiently close to the standard
perturbation. Then for dimensional reason,
$\tilde {\bm m}_l^0( {\bm p}, \ldots, {\bm p})$ is zero unless $l = 1$
and objects contributing to $\tilde {\bm m}_0^1({\bm p})$ have zero
Maslov index. Then by the definition \eqref{eqna4}, \eqref{eqn77} is
proved. This finishes the proof of Theorem \ref{thm714}.

\bibliography{mathref}	

\bibliographystyle{amsalpha}

\end{document}